\ifx \firsttimeheader \undefined
\documentclass[a4paper,fontsize=12pt]{scrbook}

\KOMAoptions{BCOR=1cm}
\usepackage{import}
\usepackage[utf8x]{inputenc}
\usepackage{amsmath}
\usepackage{graphicx}
\usepackage{index}
\usepackage{mathdots}
\usepackage{enumitem}
\usepackage{stmaryrd}
\usepackage{multirow}
\usepackage{tabularx}
\usepackage[small,balance]{diagrams}
\newarrow{ToFro}<--->

\newindex{xsyms}{sdx}{snd}{Index of Symbols}

\newcommand{\screenornot}{}

\DeclareFontFamily{U}{matha}{}
\DeclareFontShape{U}{matha}{m}{n}{ <5> <6> <7> <8> <9> <10> <12> gen * matha <11> matha10}{}
\DeclareSymbolFont{matha}{U}{matha}{m}{n}
\DeclareMathSymbol{\hash}{\mathop}{matha}{"23}

\usepackage{empheq}
\numberwithin{equation}{section}

\usepackage{amsthm}
\theoremstyle{plain}
\newtheorem{thm}{Theorem}[section]
\newtheorem*{maintheorem}{Main Theorem}
\newtheorem*{ranktheorem}{Rank Theorem}
\newtheorem{lem}[thm]{Lemma}
\newtheorem{cor}[thm]{Corollary}
\newtheorem{prop}[thm]{Proposition}
\newtheorem{obs}[thm]{Observation}
\newtheorem{reminder}[thm]{Reminder}
\newtheorem{fact}[thm]{Fact}
\newtheorem*{ncor}{Corollary}

\newtheorem{apthm}{Theorem}[chapter]
\newtheorem{aplem}[apthm]{Lemma}
\newtheorem{apcor}[apthm]{Corollary}
\newtheorem{approp}[apthm]{Proposition}

\newtheorem{apfact}[apthm]{Fact}

\newenvironment{xrefthm}[1]{
 \def\thexref{\ref{#1}}
 \begin{thexrefthm}
}{
 \end{thexrefthm}
}
\newtheorem*{thexrefthm}{Theorem~\thexref}

\theoremstyle{definition}

\newtheorem{exmpl}[thm]{Example}

\theoremstyle{remark}
\newtheorem{rem}[thm]{Remark}
\newtheorem{aprem}[apthm]{Remark}

\usepackage{amsfonts}
\newcommand{\C}{\mathbb{C}}
\newcommand{\R}{\mathbb{R}}
\newcommand{\Q}{\mathbb{Q}}
\newcommand{\Z}{\mathbb{Z}}
\newcommand{\F}{\mathbb{F}}
\newcommand{\E}{\mathbb{E}}
\renewcommand{\H}{\mathbb{H}}
\renewcommand{\S}{\mathbb{S}}
\newcommand{\N}{\mathbb{N}}
\newcommand{\A}{\mathbb{A}}

\newcommand{\nin}{\notin}
\newcommand{\union}{\cup}
\newcommand{\Union}{\bigcup}
\newcommand{\Dunion}{\coprod}

\newcommand{\intersect}{\cap}
\newcommand{\Intersect}{\bigcap}
\newcommand{\into}{\hookrightarrow}

\newcommand{\defeq}{\mathrel{\mathop{:}}=}
\newcommand{\eqdef}{=\mathrel{\mathop{:}}}

\newcommand{\op}{\mathrel{\operatorname{op}}}
\newcommand{\opm}[1]{\operatorname{op}_{#1}}

\newcommand{\conv}{\operatorname{conv}}

\newcommand{\typ}{\operatorname{typ}}
\newcommand{\relint}{\operatorname{int}}
\newcommand{\gen}[1]{\left\langle #1\right\rangle}
\usepackage{buxsym}
\renewcommand{\Join}{\bigast}
\newcommand{\Subdiv}[1]{\mathring{#1}}
\newcommand{\Link}{\operatorname{lk}}
\newcommand{\Descending}{^{\downarrow}}
\newcommand{\direction}{\rightslice}
\usepackage{amssymb}
\newcommand{\horizontal}{\mathrel{\multimap}}
\newcommand{\abs}[1]{\left\lvert#1\right\rvert}
\newcommand{\PlaceOf}[1]{[#1]}
\newcommand{\GL}{\operatorname{GL}}
\newcommand{\SL}{\operatorname{SL}}
\newcommand{\Sp}{\operatorname{Sp}}
\newcommand{\SO}{\operatorname{SO}}
\renewcommand{\equiv}{\Longleftrightarrow}

\usepackage{color}

\usepackage[final,hyperindex,pageanchor=true,plainpages=false,pdftex,pdfpagelabels,hypertexnames=true]{hyperref}
\hypersetup{pdftitle={Finiteness Properties of Chevalley Groups over the Ring of (Laurent) Polynomials over a Finite Field}}
\hypersetup{pdfauthor={Stefan Witzel}}

\usepackage{stmaryrd}
\usepackage{xargs}
\usepackage{ifthen}
\newcommand{\ifnonempty}[2]{
  \ifthenelse{\equal{#1}{}}{}{#2}
}
\newcommand{\optionalsubsuper}[3]{
\ifthenelse{\equal{#2}{}}{
  \ifthenelse{\equal{#3}{}}{
    #1}{
    #1^{#3}}
  }{
  \ifthenelse{\equal{#3}{}}{
    #1_{#2}}{
    #1_{#2}^{#3}}}}
\newcommand{\optionalsub}[2]{
\ifthenelse{\equal{#2}{}}{
  #1}{
  #1_{#2}}
}
\newcommand{\optionaltwosubs}[3]{
\ifthenelse{\equal{#2}{}}{
  \ifthenelse{\equal{#3}{}}{
    #1}{
    #1_{#3}}
  }{
  \ifthenelse{\equal{#3}{}}{
    #1_{#2}}{
    #1_{#2,#3}}}}
\newcommand{\Group}{G}

\newcommand{\BigGroup}{E}
\newcommand{\AltGroup}{H}
\newcommand{\Another}[1]{{#1}'}

\newcommand{\Space}{X}
\newcommand{\AltSpace}{Y}
\newcommand{\Skeleton}[2]{{#1}^{(#2)}}

\newcommand{\Set}{A}
\newcommand{\ModelSpace}[1]{M_{#1}}
\newcommand{\ModelDiameter}[1]{D_{#1}}
\newcommand{\SpherPoint}{p}
\newcommand{\Point}{x}
\newcommand{\AltPoint}{y}
\newcommand{\YetAltPoint}{z}
\newcommand{\PosPoint}{\Point_+}
\newcommand{\NegPoint}{\Point_-}
\newcommand{\APosPoint}{\AltPoint_+}
\newcommand{\ANegPoint}{\AltPoint_-}
\newcommand{\TheNegPoint}{{a_-}}
\newcommand{\ThePosPoint}{{a_+}}
\newcommand{\PointAtInfty}{\xi}

\newcommand{\Vector}{v}
\newcommand{\AltVector}{w}
\newcommand{\InftyPoint}{\xi}
\newcommandx{\Disk}[1][1={}]{\optionalsubsuper{D}{}{#1}}
\newcommandx{\Sphere}[1][1={}]{\optionalsubsuper{S}{}{#1}}
\newcommand{\FP}{F\mkern-3mu{}P}

\newcommand{\scp}[2]{\langle #1 \mid #2 \rangle}
\newcommand{\scpb}[2]{( #1 \mid #2 )}
\newcommand{\Path}{\gamma}
\newcommand{\Cell}{\sigma}
\newcommand{\PosCell}{\Cell_+}
\newcommand{\NegCell}{\Cell_-}
\newcommand{\AltCell}{\tau}
\newcommand{\BigCell}{\tau}
\newcommand{\PosBigCell}{\BigCell_+}
\newcommand{\NegBigCell}{\BigCell_-}
\newcommand{\Ray}{\rho}
\newcommand{\Busemann}{\beta}
\newcommand{\Infty}[1]{#1^\infty}
\newcommand{\GlobalField}{k}
\newcommand{\LocalField}{K}
\newcommand{\ClosedField}{K}
\newcommand{\Haar}{\mu}
\newcommand{\Field}{k}
\newcommand{\ExtensionField}{K}
\newcommand{\Datum}{\mathcal{D}}
\newcommand{\Matrix}{A}
\newcommand{\Ring}{A}
\newcommand{\KMalgOf}[1]{L(#1)}
\newcommand{\LieAlg}{\mathfrak{g}}
\newcommand{\Units}[1]{#1^\times}
\newcommand{\HaarModule}{\operatorname{mod}}

\newcommand{\Valuation}{v}
\newcommand{\Place}{s}
\newcommand{\DiscreteValuation}{\nu}
\newcommandx{\Integers}[1][1={}]{\optionalsubsuper{\mathcal{O}}{#1}{}}
\newcommand{\MaxIdeal}{\mathfrak{m}}

\newcommand{\GroupScheme}{\mathbf{G}}
\newcommand{\Torus}{\mathbf{T}}
\newcommand{\Root}{\alpha}
\newcommand{\Apartment}{\Sigma}
\newcommand{\SApartment}{\Sigma}
\newcommand{\EApartment}{\tilde{\Sigma}}
\newcommand{\SBuilding}{\Delta}
\newcommand{\TwinApartment}{\Apartment}
\newcommand{\PosApartment}{\Apartment_+}
\newcommand{\NegApartment}{\Apartment_-}
\newcommand{\PNApartments}{(\PosApartment,\NegApartment)}
\newcommand{\AltPNApartments}{(\PosApartment',\NegApartment')}
\newcommand{\EBuilding}{X}
\newcommand{\Building}{X}
\newcommand{\PosBuilding}{\Building_+}
\newcommand{\NegBuilding}{\Building_-}
\newcommand{\PNBuildings}{(\Building_+,\Building_-)}
\newcommand{\Isom}{\operatorname{Isom}}
\newcommand{\Wall}{H}
\newcommand{\PosWall}{\Wall_+}
\newcommand{\NegWall}{\Wall_-}
\newcommand{\Weyl}{W}
\newcommand{\EWeyl}{\tilde{W}}
\newcommand{\InftyEWeyl}{\Weyl}
\newcommand{\weyl}{w}
\newcommand{\InftyChamber}{C}
\newcommand{\ModelChamber}{\Chamber_{\mathrm{mod}}}
\newcommand{\ModelProjection}{\pi}
\newcommand{\Chamber}{c}
\newcommand{\AltChamber}{d}
\newcommand{\YetAltChamber}{e}
\newcommand{\PosChamber}{\Chamber_+}
\newcommand{\NegChamber}{\Chamber_-}
\newcommand{\Vertex}{v}
\newcommand{\AltVertex}{w}
\newcommand{\PosVertex}{\Vertex_+}
\newcommand{\NegVertex}{\Vertex_-}

\newcommand{\Types}{I}
\newcommand{\AltTypes}{J}
\newcommand{\LeftMod}{\setminus}
\newcommand{\Type}{i}
\newcommand{\AltType}{j}
\newcommand{\CoxeterDiagram}{\Gamma}
\newcommand{\WeylDistance}{\delta}
\newcommand{\PosWeylDistance}{\WeylDistance_+}
\newcommand{\NegWeylDistance}{\WeylDistance_-}
\newcommand{\WeylCoDistance}{\WeylDistance^*}
\newcommand{\Retraction}[2]{\rho_{#1,#2}} 
\newcommand{\WeylProjection}{\operatorname{pr}}
\newcommand{\CoordinateProjection}{\operatorname{pr}}
\newcommandx{\ClosestPointProjection}[1][1={}]{\optionalsub{\operatorname{pr}}{#1}}
\newcommandx{\EuclProjection}[1][1={}]{\optionalsub{\operatorname{pr}}{#1}}
\newcommandx{\SpherProjection}[1][1={}]{\optionalsub{\operatorname{pr}}{#1}}
\newcommand{\ApartmentSystem}{\mathcal{A}}
\newcommand{\Sector}{S}
\newcommand{\Vertices}{\operatorname{vt}}
\newcommand{\Facets}{\operatorname{ft}}
\newcommand{\EuclDistance}{d}
\newcommand{\SpherDistance}{d}
\newcommandx{\EuclCoDistance}[2][1={},2={}]{\optionaltwosubs{\EuclDistance^*}{#1}{#2}}
\newcommandx{\SpherCoDistance}[2][1={},2={}]{\optionaltwosubs{\SpherDistance^*}{#1}{#2}}
\newcommandx{\TwinRay}[4][3={},4={}]{\optionaltwosubs{[#1,#2)}{#3}{#4}}
\newcommand{\TwinRayMap}[2]{\Ray_{#1}^{#2}}
\newcommand{\Dimension}{n}
\newcommand{\Hor}{^{\textup{hor}}}
\newcommand{\Ver}{^{\textup{ver}}}
\newcommand{\Min}{^{\textup{min}}}
\newcommand{\OpenHemi}{^{>\pi/2}}
\newcommand{\ClosedHemi}{^{\ge\pi/2}}
\newcommand{\Equator}{^{=\pi/2}}
\newcommand{\CofacePart}{_{\delta}}
\newcommand{\FacePart}{_{\partial}}
\newcommand{\ApproxHeight}{h'}
\newcommand{\Height}{h}
\newcommand{\Morse}{f}
\newcommand{\Gradient}{\nabla}
\newcommandx{\Zonotope}[1][1={}]{Z\ifthenelse{\equal{#1}{}}{}{(#1)}}
\newcommand{\NormalCone}[1]{N(#1)}
\newcommand{\Directions}{D}
\newcommand{\Direction}{z}
\newcommand{\ConvexSet}{C}
\newcommand{\FaceLattice}{\mathcal{F}}
\newcommand{\Up}{\nearrow}
\newcommand{\Down}{\searrow}
\newcommand{\Roof}[1]{\hat{#1}}
\newcommand{\Depth}{\operatorname{dp}}
\newcommand{\OneSpace}{\Space}
\newcommand{\TwoSpace}{\Space}
\newcommand{\DummyArg}{-}
\newcommand{\EIdent}{\iota}
\newcommand{\PosEIdent}{\EIdent_+}
\newcommand{\NegEIdent}{\EIdent_-}
\newcommand{\Difference}{\pi}
\newcommand{\Plus}{\hash}
\newcommand{\ZeroLevel}{\TwoSpace_0}
\newcommand{\Star}{\operatorname{st}}
\newcommand{\CAT}[1]{\textup{CAT}\textup{(}$\mathrm{#1}$\textup{)}}

\newcommandx{\Cay}[3][3={}]{\operatorname{Cay}(#1,#2)}

\newcommand{\UpSet}{L^\uparrow}
\newcommand{\UpConvex}{{\tilde{A}}}
\newcommand{\UpOpen}{{\tilde{B}}}
\newcommand{\UpLink}{B}
\newcommand{\UpFat}{U}

\newcommand{\Mod}{\operatorname{Mod}}
\newcommand{\Out}{\operatorname{Out}}

\newcommand{\headerlevel}[1]{
\ifx \footeronlevel \undefined
\def\footeronlevel{#1}
\else\fi}
\newcommand{\footerlevel}[1]{
\def\currentlevel{#1}}

\begin{document}
\def\firsttimeheader{\true}
\else\fi
\renewcommand{\screenornot}{_screen}

\headerlevel{1}

\frontmatter

\headerlevel{2}

\title{Finiteness Properties of Chevalley Groups over the Ring of (Laurent) Polynomials over a Finite Field}
\author{Stefan Witzel}

\date{}

\cleardoublepage 

\maketitle

\cleardoublepage
\thispagestyle{empty}
\vspace*{10.5pc}
\begin{center}
\LARGE
Für meine Eltern
\end{center}

\footerlevel{2}


\cleardoublepage
\phantomsection
\pdfbookmark{Contents}{pdfbook:toc}
\tableofcontents

\headerlevel{2}

\chapter*{Introduction}

\label{chap:intro}

In these notes we determine finiteness properties of two classes of groups whose most prominent representatives are the groups
\[
\SL_n(\F_q[t]) \quad \quad \text{and} \quad\quad \SL_n(\F_q[t,t^{-1}])\text{ ,}
\]
the special linear groups over the ring of polynomials, respectively Laurent polynomials, over a field with $q$ elements.

The finiteness properties we are interested in generalize the notions of being finitely generated and of being finitely presented. A group $\Group$ is generated by a subset $S$ if and only if the Cayley graph $\Cay{\Group}{S}[1]$ is connected. And $S$ is finite if and only if the quotient $\Group \backslash \Cay{\Group}{S}[1]$ is compact. That is, $\Group$ is finitely generated if and only if it admits a connected Cayley graph that has compact quotient modulo $\Group$.

Similarly, consider a set $R$ of relations in $\Group$, that is, words in the letters $S \union S^{-1}$ which describe the neutral element in $\Group$. The Cayley $2$-complex $\Cay{\Group}{S,R}[2]$ is obtained from the Cayley graph by gluing in a $2$-cell for every edge loop that is labeled by an element of $R$. The Cayley $2$-complex is $1$-connected (that is, connected and simply connected) if and only if $\gen{S \mid R}$ is a presentation of $\Group$. And it has a compact quotient modulo $\Group$ if both, $S$ and $R$, are finite. That is, $\Group$ is finitely presented, if and only if $\Group$ admits a $1$-connected, cocompact Cayley $2$-complex.

Since $\Group$ is described up to isomorphism by a presentation, this is how far the classical interest goes. From the
topological point of view, on the other hand, one can go on and ask whether it is possible to glue in $3$-cells along
``identities'' $I$ in such a way that the resulting complex $\Cay{\Group}{S,R,I}[3]$ is $2$-connected and cocompact.

Wall \cite{wall65, wall66} developed this topological point of view and introduced the following notion: a group $\Group$ is of type $F_n$ if it has a classifying space $\Space$ (that is, a contractible CW-complex on which $\Group$ acts freely) such that the quotient $\Group \backslash \Skeleton{\Space}{n}$ of the $n$-skeleton is compact modulo $\Group$. It is not hard to see that, indeed, a group is of type $F_1$ if and only if it is finitely generated, and is of type $F_2$ if and only if it is finitely presented. We say that a group is of type $F_\infty$ if it is of type $F_n$ for all $n$. This property is strictly weaker than that of being of type $F$, namely having a cocompact classifying space. In fact, a group that has torsion elements cannot be of type $F$. But if it is virtually of type $F$, that is, if it contains a finite index subgroup that is of type $F$, then it is still of type $F_\infty$.

In the decades following Wall's articles some effort has been put, on the one hand, in determining what finiteness properties certain interesting groups have, and on the other hand, in better understanding what the properties $F_n$ mean by producing separating examples.
We mention just some of the results not directly related to the present notes. Statements about arithmetic and related groups will be mentioned further below. It will be convenient to introduce the finiteness length of a group $\Group$ defined as
\[
\phi(\Group) \defeq \sup \{n \in \N \mid \Group\text{ is of type } F_n\}\text{ .}
\]
Finitely generated groups that are not finitely presented have been known since Neumann's article \cite{neumann37}. The first group known to be of type $F_2$ but not of type $F_3$ was constructed by Stallings \cite{stallings63}. Stallings's example is the case $n=3$ of the following construction due to Bieri: let $L^n$ be the direct product of $n$ free groups on two generators and let $K_n$ be the kernel of the homomorphism $L^n \to \Z$ that maps each of the canonical generators to $1$. In \cite{bieri76} Bieri showed that $\phi(K_n) = n-1$.
Abels and Brown \cite{abebro87} proved that the groups $\GroupScheme_n$ of upper triangular $n$-by-$n$ matrices with extremal diagonal entries equal to $1$ satisfy $\phi(\GroupScheme_n(\Z[\frac{1}{p}])) = n-1$ for any prime $p$. In \cite{brown87} Brown proved that Thompson's groups and some of their generalizations are of type $F_\infty$. For the group $\mathbf{B}_n$ of upper triangular matrices and a ring $\Integers[S]$ of $S$-integers of a global function field, Bux \cite{bux04} showed that $\phi(\mathbf{B}_n(\Integers[S])) = \abs{S}-1$.

The general pattern of proof to determine the finiteness properties of a group $\Group$ is the same in many cases: first one produces a contractible CW-complex $\Space$ on which $\Group$ acts with ``good'' (typically finite) stabilizers. This action will typically not be cocompact. One then constructs a filtration $(\Space_i)_i$ of $\Space$ by cocompact subcomplexes $\Space_i$ such that the inclusions $\Space_i \into \Space_j$, $i \le j$, preserve $n$-connectedness for some fixed $n$. Now $\Space_0$ would be the $(n+1)$-skeleton of a classifying space if the stabilizers were trivial instead of only ``good''. In this situation there is a famous criterion due to Brown \cite{brown87} stating not only that ``good'' stabilizers are good enough to conclude that $\Group$ is of type $F_{n}$, but also that the group is not of type $F_{n+1}$ provided the filtration does not preserve $n+1$-connectedness in an essential way.

In some cases an appropriate space $\Space$ for $\Group$ to act on has been known long before people were interested in higher finiteness properties. Thus Raghunathan \cite{raghunathan68} showed that arithmetic subgroups of semisimple algebraic groups over number fields, like $\SL_n(\Z)$, are virtually of type $F$. To this end he considered the action of the arithmetic group on the symmetric space $\Space$ of its ambient Lie group and constructed a Morse function on the quotient $\Group \backslash \Space$ with compact sublevel sets. It is noteworthy, that this proof fits into the general pattern described above. In fact, the filtrations mentioned before are often, and in these notes in particular, obtained by (a discrete version of) Morse theory. This reduces the problem to understanding certain local data, the descending links.

There are two classes of groups that are closely related to arithmetic groups: mapping class groups $\Mod(S_g)$ of closed surfaces and outer automorphism groups $\Out(F_n)$ of finitely generated free groups. The space for $\Mod(S_g)$ to act on is Teich\-mül\-ler space, likewise a very classical object. A proof that Teich\-mül\-ler space admits an invariant contractible cocompact subspace, and therefore $\Mod(S_g)$ is virtually of type $F$, can be found in \cite{ivanov91}. The right space to consider for $\Out(F_n)$ is outer space \cite{culvog86}. Unlike in the two previous cases, the space was not known before Culler and Vogtman constructed it to establish that $\Out(F_n)$ is virtually of type $F$. The cited proofs for mapping class groups and outer automorphism groups of free groups are very similar in spirit to the one for arithmetic groups and fit again into our general pattern. An alternative to exhibiting a highly connected cocompact subspace of the original space is to construct a cocompact partial 
compactification on which the group still acts properly discontinuously. This has been done by Borel and Serre \cite{borser73} for arithmetic groups and by Harvey \cite{harvey79} for mapping class groups.

A number theoretic generalization of arithmetic groups are $S$-arithmetic groups. To define them, we consider a number
field $\GlobalField$ and its set of places $T$, that is, a maximal set of inequivalent valuations. Let $T_\infty$ denote
the subset of Archimedean valuations, such as the usual absolute value. For an element $\alpha \in \GlobalField$ the
condition that $\Valuation(\alpha) \ge 0$ for all non-Archimedean places $\Valuation$ describes the ring of integers of
$\GlobalField$. If, instead, one imposes this condition for all but a finite set $S$ of non-Archimedean places, one
obtains the ring of $S$-integers $\Integers[S]$. Accordingly, $S$-arithmetic groups are matrix groups of $S$-integers.

The field $\GlobalField$ admits a completion $\GlobalField_\Valuation$ with respect to every valuation $\Valuation \in T$. An $S$-arithmetic group $\GroupScheme(\Integers[S])$ is a discrete subgroup of the locally compact group $\prod_{\Valuation \in T_\infty \union S} \GroupScheme(\GlobalField_\Valuation)$. For instance, the group $\SL_n(\Z[\frac{1}{2}])$ is a discrete subgroup of the group $\SL_n(\R) \times \SL_n(\Q_2)$.

If $\GroupScheme$ is a reductive $\GlobalField$-group, then $\GroupScheme(\Integers[S])$ acts properly discontinuously on the product of the spaces associated to the locally compact groups $\GroupScheme(\GlobalField_\Valuation)$, $\Valuation \in T_\infty \union S$. For the Archimedean valuations, this is again a symmetric space. For the non-Archimedean valuations the naturally associated space is a Bruhat--Tits building, that is, a locally compact cell complex with a piecewise Euclidean metric.

The action of an $S$-arithmetic subgroup of a reductive algebraic group over a number field described above has been used by Borel and Serre \cite[Théo\-rème~6.2]{borser76} to show that these groups are virtually of type $F$. Their proof is again by constructing a properly discontinuous action on a partial compactification rather than finding a highly connected subspace.

There is the notion of a global function field which parallels that of a number field. A global function field is a finite extension of a field of the form $\F_p(t)$ where $\F_p$ is the finite field with $p$ elements and $t$ is transcendental over $\F_p$. There is no strong formal ressemblance between number fields and global function fields, but it has turned out that they share many properties. In particular, the theory of places, completions and $S$-integers can be developed analogously for global function fields, with the exception that there are no Archimedean valuations.

Finiteness properties of $S$-arithmetic subgroups of semisimple groups over global function fields differ fundamentally from the analogous properties in the number field case we have seen above. This is apparent already from the first result in this class: Nagao \cite{nagao59} showed that the groups $\SL_2(\F_q[t])$ are not finitely generated. As another example Stuhler \cite{stuhler80} showed that $\SL_2(\F_q[t,t^{-1}])$ is finitely generated but not finitely presented.

In these notes we prove the following generalization of these two theorems. It concerns almost simple $\F_q$-groups of rank $n$. Examples are Chevalley groups such as $\SL_{n+1}$, $\Sp_{2n}$, $\SO_{2n+1}$, and $\SO_{2n}$.

\begin{maintheorem}
Let $\GroupScheme$ be a connected, noncommutative, almost simple $\F_q$-group of $\F_q$-rank $\Dimension \ge 1$. Then $\GroupScheme(\F_q[t])$ is of type $F_{n-1}$ but not of type $F_n$ and $\GroupScheme(\F_q[t,t^{-1}])$ is of type $F_{2n-1}$ but not of type $F_{2n}$.
\end{maintheorem}

The second part of the Main~Theorem, which is proved in Chapter~\ref{chap:two_places} as Theorem~\ref{thm:two_places_arithmetic}, is new for $n \ge 2$. The first part is proved in Chapter~\ref{chap:one_place} as Theorem~\ref{thm:one_place_arithmetic}. In the case where $q$ is large compared to $n$, it was shown before by Abels and Abramenko \cite{abeabr93} for $\GroupScheme = \SL_{n+1}$ and by Abramenko \cite{abramenko96} for $\GroupScheme$ a classical group.

The proof of the Main~Theorem is very specific to the groups in question. We first collect the general properties. The rings $\F_q[t]$ and $\F_q[t,t^{-1}]$ are rings of $S$-integers in $\F_q(t)$ where $S = \{\Valuation_0\}$ contains one place in the first case and $S = \{\Valuation_0,\Valuation_\infty\}$ contains two places in the second case. So the groups are $S$-arithmetic groups. As in the number field case, an $S$-arithmetic group $\GroupScheme(\Integers[S])$ is discrete as a subgroup of the locally compact group $\prod_{\Valuation \in S} \GroupScheme(\GlobalField_\Valuation)$ (recall that there are no Archimedean valuations). Since $\GroupScheme$ is almost simple, there is a Bruhat--Tits building $\Building_\Valuation$ associated to each of the factors $\GroupScheme(\GlobalField_\Valuation)$. Therefore the group $\GroupScheme(\Integers[S])$ acts properly discontinuously on the building $\Building \defeq \prod_{\Valuation \in S} \Building_\Valuation$. The action is not cocompact and the task, according 
to our general strategy, is therefore to construct a cocompact filtration which preserves high connectedness.

This can and has been done using Harder's reduction theory \cite{harder67, harder68, harder69}. However, since this is quite involved in higher rank, the results obtained in this way were restricted in one of two ways: Either they only held for global rank $1$ such as those by Stuhler \cite{stuhler80} and Bux--Wortman \cite{buxwor08}. Or they did not determine the full finiteness properties such as those by Behr \cite{behr98}.

What makes the groups of the Main Theorem so special, is that the group $\GroupScheme(\F_q[t,t^{-1}])$ happens to also be a Kac--Moody group. In terms of spaces this means that the two buildings $\Building_0$ and $\Building_\infty$ that the group acts on form a twin building. That is, there is a codistance between $\Building_0$ and $\Building_\infty$ measuring in some sense the distance between cells in the two buildings, and this codistance is preserved by $\GroupScheme(\F_q[t,t^{-1}])$. In fact one can define two kinds of codistance: one is a combinatorial codistance between the cells of $\Building_0$ and of $\Building_\infty$ and the other is a metric codistance between the points of $\Building_0$ and of $\Building_\infty$. The group $\GroupScheme(\F_q[t])$ is a stabilizer in $\GroupScheme(\F_q[t,t^{-1}])$ of a cell in $\Building_0$.

In \cite{abramenko96} Abramenko used the combinatorial codistance to define a Morse function on $\Building_\infty$ and partially obtain the first case of the Main Theorem as described above. To ensure that the filtration preserves connectedness properties, Abramenko had to study certain combinatorially described subcomplexes of spherical buildings, which arose as descending links.

In our proof we use the metric codistance in a similar way to Abramenko's use of the combinatorial codistance. The descending links that occur in our filtration are metrically described subcomplexes of spherical buildings. The connectedness properties of these have already been established by Schulz \cite{schulz}.

Since our proof makes heavy use of the piecewise Euclidean metric on the buildings $\Building_0$ and $\Building_\infty$ it is restricted to affine Kac--Moody groups. Abramenko's combinatorial proof on the other hand, making no reference to the metric structure of the twin building, generalizes to hyperbolic Kac--Moody groups.

That being said, we want to point that, even though the proof also uses the twin building structure, it is possible to generalize it to an arbitrary set of places $S$. The codistance function is then replaced by a Morse function constructed using Harder's reduction theory. However, since the theory of twin buildings is easier to handle than reduction theory, the proof of the general case is substantially more involved. On the other hand, in proving the Main~Theorem one already encounters (and has to solve) all problems of the general case that are not due to the use of reduction theory. One technique developed for this purpose is the flattening of level sets that is introduced in Sections~\ref{sec:zonotopes} and \ref{sec:height}. Another technique is the use of the depth function as a secondary height function in the flattened regions. It was introduced in \cite{buxwor08} and is generalized to reducible buildings in Section~\ref{sec:moves}. The proof to be presented here therefore allows us to study the 
geometry of the problem without having to struggle with the difficulties of reduction theory. The understanding of these isolated problems has enabled Kai-Uwe Bux, Ralf Köhl, and the author to prove the following result in \cite{buxgrawit10b}, which was posed as a question in \cite[p.~197]{brown89}:

\begin{ranktheorem}
Let $\GlobalField$ be a global function field. Let $\GroupScheme$ be a connected, noncommutative, almost simple $\GlobalField$-isotropic $\GlobalField$-group. Let $N \defeq \sum_{\Place \in S} \operatorname{rank}_{\GlobalField_\Place} \GroupScheme$ be the sum over the local ranks at places $\Place \in S$ of $\GroupScheme$. Then $\GroupScheme(\Integers[S])$ is of type $F_{N-1}$ but not of type $F_N$.
\end{ranktheorem}

The negative statement of the Rank Theorem was known before by work of Bux and Wortman \cite{buxwor07} and an
alternative proof of it has recently been given by Gandini \cite{gandini+11}. If $\GroupScheme$ has $\GlobalField$-rank
$0$, then the group in the theorem is virtually of type $F$, as was shown by Serre \cite[Théorème~4~(b)]{serre71}.

Note that if $\GroupScheme$ is a Chevalley group, then the rank $\Dimension$ is independent of the field. Thus, in this case the Rank Theorem can be expressed by the relation
\[
\phi(\GroupScheme(\Integers[S])) = \abs{S} \cdot \Dimension - 1 \text{ .}
\]

In Appendix~\ref{chap:adding_places} we show that the finiteness length of an almost simple $S$-arithmetic group can only grow as $S$ gets larger (a fact that was already used in \cite{abramenko96}). Though this is clear in presence of the Rank Theorem, it allows one to deduce finiteness properties (though not the full finiteness length) of some groups even without it. For example, the following is a consequence of our Main Theorem:

\begin{ncor}
Let $\GroupScheme$ be a connected, noncommutative, almost simple $\F_q$-group of $\F_q$-rank $\Dimension \ge 1$. Let $S$ be a finite set of places of $\F_q(t)$ and let $\Group \defeq \GroupScheme(\Integers[S])$. If $S$ contains $\Valuation_0$ or $\Valuation_\infty$, then $\Group$ is of type $F_{\Dimension-1}$. If $S$ contains $\Valuation_0$ and $\Valuation_\infty$, then $\Group$ is of type $F_{2\Dimension-1}$.
\end{ncor}

These notes are essentially the author's Ph.D. thesis \cite{witzel11} and based on the unpublished notes \cite{buxgrawit10} and \cite{witzel10}. Therefore Chapter~\ref{chap:one_place} contains many ideas of Kai-Uwe Bux and Ralf Köhl. Also, Appendix~\ref{chap:affine_kac-moody_groups} is mostly due to them.

\subsection*{Acknowledgements}

I am indebted to Kai-Uwe Bux and Ralf Köhl who have taught me a lot of what is used in these notes and have encouraged me while I wrote it. I would also like to thank Jan Essert, David Ghatei, Sven Herrmann, Karl Heinrich Hofmann, Michael Joswig, Alexander Kartzow, Andreas Mars, Petra Schwer, Julia Sponsel, and Markus-Ludwig Wermer for helpful discussions and comments on the manuscript. I gratefully acknowledge the support I received from the DFG and the Studienstiftung des deut\-schen Volkes and the hospitality of the University of Virginia, the SFB 701 of the Universität Bielefeld, and the Mathematisches Forschungsinstitut Oberwolfach.

\footerlevel{2}

\mainmatter

\headerlevel{2}

\chapter{Basic Definitions and Properties}

In this first chapter we introduce notions and statements that will be needed later on and that are more or less generally known. The focus is on developing the necessary ideas in their natural context, proofs are generally omitted. For the reader who is interested in more details, an effort has been made to give plenty of references. Where less appropriate references are known to the author, the exposition is more detailed.

An exception to this is Section~\ref{sec:buildings} on buildings: there are many excellent books on the topic but our point of view is none of the classical ones so we give a crash course developing our terminology along the way. For this reason even experts may want to skim through Section~\ref{sec:buildings}. In Section~\ref{sec:metric_spaces} on metric spaces some definitions are slightly modified and non-standard notation is introduced. Apart from that the reader who feels familiar with some of the topics is encouraged to skip them and refer back to them as needed.

\headerlevel{3}

\section{Metric Spaces}
\label{sec:metric_spaces}

In this section we introduce what we need to know about metric spaces, in particular about those which have bounded curvature in the sense of the \CAT{\kappa} inequality. We also define cell complexes in a way that will be convenient later. The canonical and almost exhaustive reference for the topics mentioned here is \cite{brihae} from which most of the definitions are taken. Other books include \cite{ballmann95} and \cite{papadopoulos05}.

\subsection*{Geodesics}

Let $\Space$ be a metric space. A \emph{geodesic} in $\Space$ is an isometric embedding $\Path \colon [a,b] \to \Space$ from a compact real interval into $\Space$; its image is a \emph{geodesic segment}. The geodesic \emph{issues at $\Path(a)$} and \emph{joins $\Path(a)$ to $\Path(b)$}. A \emph{geodesic ray} is an isometric embedding $\Ray \colon [a,\infty) \to \Space$ and is likewise said to \emph{issue at $\Ray(a)$}. Sometimes the image of $\Ray$ is also called a \emph{geodesic ray}.

A metric space is said to be \emph{geodesic} if for any two of its points there is a geodesic that joins them. It is \emph{($D$-)uniquely geodesic} if for any two points (of distance $<D$) there is a unique geodesic that joins them.

If $\Point, \AltPoint$ are two points of distance $<D$ in a $D$-uniquely geodesic space, then we write $[\Point,\AltPoint]$ for the geodesic segment that joins $\Point$ to $\AltPoint$.

A subset $\Set$ of a geodesic metric space is \emph{($D$-)convex} if for any two of its points (of distance $<D$) there is a geodesic that joins them and the image of every such geodesic is contained in $\Set$.

If $\Path \colon [0,a] \to \Space$ and $\Path' \colon [0,a'] \to \Space$ are two geodesics that issue at the same point, one can define the angle $\angle_{\Path(0)}(\Path,\Path')$ between them (see \cite[Definition~1.12]{brihae}). If $\Space$ is hyperbolic, Euclidean, or spherical space, this is the usual angle. If $\Space$ is $D$-uniquely geodesic and $\Point, \AltPoint, \YetAltPoint \in \Space$ are three points with $d(\Point,\AltPoint), d(\Point,\YetAltPoint) < D$, we write $\angle_{\Point}(\AltPoint,\YetAltPoint)$ to denote the angle between the unique geodesics from $\Point$ to $\AltPoint$ and from $\Point$ to $\YetAltPoint$.

\subsection*{Products and Joins}

The \emph{direct product} $\prod_{i=1}^n = X_1 \times \cdots \times X_n$ of a finite number of metric spaces $(X_i,d_i)_{1 \le i \le n}$ is the set-theoretic direct product equipped with the metric $d$ given by 
\[
d\big((\Point_1,\ldots,\Point_n),(\AltPoint_1,\ldots,\AltPoint_n)\big) = \big(d_1(\Point_1,\AltPoint_1)^2 + \cdots + d_n(\Point_n,\AltPoint_n)^2\big)^{1/2} \text{ .}
\]

The \emph{spherical join} $\Space_1 * \Space_2$\index[xsyms]{star@$\Space_1*\Space_2$} of two metric spaces $(\Space_1,d_1)$ and $(\Space_2,d_2)$ of diameter at most $\pi$ is defined as follows: as a set, it is the quotient $([0,\pi/2] \times \Space_1 \times \Space_2) / \sim$ where $(\theta,\Point_1,\Point_2) \sim (\theta',\Point_1',\Point_2')$ if either $\theta = \theta' = 0$ and $\Point_1 = \Point_1'$, or $\theta = \theta' = \pi/2$ and $\Point_2 = \Point_2'$, or $\theta = \theta' \nin \{0,\pi/2\}$ and $\Point_1 = \Point_1'$ and $\Point_2 = \Point_2'$. The class of $(\theta,\Point_1,\Point_2)$ is denoted $\cos \theta \Point_1 + \sin \theta \Point_2$ and, in particular, by $\Point_1$ or $\Point_2$ if $\theta$ is $0$ or $\pi/2$.

The metric $d$ on $\Space_1 * \Space_2$ is defined by the condition that for two points $\Point = \cos \theta \Point_1 + \sin \theta \Point_2$ and $\Point' = \cos \theta' \Point_1 + \sin \theta' \Point_2$ the distance $d(\Point,\Point')$ be at most $\pi$ and that
\begin{equation}
\label{eq:spherical_join}
\cos d(\Point,\Point') = \cos \theta \cos \theta' \cos d_1(\Point_1,\Point_1') + \sin \theta \sin \theta' \cos d_2(\Point_2,\Point_2') \text{ .}
\end{equation}
The maps $\Space_i \to \Space_1 * \Space_2, \Point_i \mapsto \Point_i$ are isometric embeddings and so we usually regard $\Space_1$ and $\Space_2$ as subspaces of $\Space_1 * \Space_2$. For three metric spaces $\Space_1$, $\Space_2$ and $\Space_3$ of diameter at most $\pi$, the joins $(\Space_1 * \Space_2) * \Space_3$ and $\Space_1 * (\Space_2 * \Space_3)$ are naturally isometric so there is a spherical join $\Join_{i = 1}^n \Space_i = \Space_1 * \cdots * \Space_n$\index[xsyms]{star@$\Join_i\Space_i$} for any finite number $n$ of metric spaces $\Space_i$ of diameter at most $\pi$.

\subsection*{Model Spaces}

We introduce the model spaces for positive, zero, and negative curvature, see \cite[Chapter~I.2]{brihae} or \cite{ratcliffe} for details. First let $\R^n$ be equipped with the standard Euclidean scalar product $\scp{\DummyArg}{\DummyArg}$. The set $\R^n$ together with the metric induced by $\scp{\DummyArg}{\DummyArg}$ is the \emph{$n$-dimensional Euclidean space} and as usual denoted by $\E^n$.

The \emph{$n$-dimensional sphere} (or \emph{spherical $n$-space}) $\S^n$ is the unit sphere in $\R^{n+1}$ equipped with the angular metric. That is, the metric $d_{\S^n}$ is given by $\cos d_{\S^n}(\Vector,\AltVector) = \scp{\Vector}{\AltVector}$.

Now let $\scpb{\DummyArg}{\DummyArg}$ be the Lorentzian scalar product on $\R^{n+1}$ that for the standard basis vectors $(e_i)_{1 \le i \le n+1}$ takes the values
\[
\scpb{e_i}{e_j} =
\left\{ 
\begin{array}{ll}
0& \text{ if }i \ne j\\
1& \text{ if }1 \le i = j \le n\\
-1& \text{ if } i = j = n+1\text{ .}
\end{array}
\right.
\]
The sphere of radius $\mathrm{i}$ with respect to this scalar product, i.e., the set
\[
\{\Vector \in \R^{n+1} \mid \scpb{\Vector}{\Vector} = -1\} \text{ ,}
\]
has two components. The component consisting of vectors with positive last component is denoted by $\H^n$ and equipped with the metric $d_{\H^n}$ which is defined by $\cosh d_{\H^n}(\Vector,\AltVector) = \scpb{\Vector}{\AltVector}$. The metric space $\H^n$ is the \emph{$n$-dimensional hyperbolic space}.

The $n$-sphere, Euclidean $n$-space, and hyperbolic $n$-space are the model spaces $\ModelSpace{\kappa}^n$ for curvature $\kappa = 1$, $0$ and $-1$ respectively. We obtain model spaces for all other curvatures by scaling the metrics of spherical and hyperbolic space: for $\kappa>0$ the model space $\ModelSpace{\kappa}^n$ is $\S^n$ equipped with the metric $d_\kappa \defeq 1/\sqrt{\kappa} d_{\S^n}$; and for $\kappa < 0$ the model space $\ModelSpace{\kappa}^n$ is $\H^n$ equipped with the metric $d_\kappa \defeq 1/\sqrt{-\kappa} d_{\H^n}$. For every $\kappa$ we let $\ModelDiameter{\kappa}$ denote the diameter of $\ModelSpace{\kappa}^n$ (which is independent of the dimension). Concretely this means that $D_\kappa = \infty$ for $\kappa \le 0$ and $D_\kappa = \pi/\sqrt{\kappa}$ for $\kappa > 0$. Each model space $\ModelSpace{\kappa}^n$ is geodesic and $D_\kappa$-uniquely geodesic.

By a \emph{hyperplane} in $\ModelSpace{\kappa}^n$ we mean an isometrically embedded $\ModelSpace{\kappa}^{n-1}$. The complement of a hyperplane has two connected components and we call the closure of one of them a \emph{halfspace} (in case $\kappa > 0$ also \emph{hemisphere}). A \emph{subspace} of a model space is an intersection of hyperplanes and is itself isometric to a model space (or empty).

\subsection*{\CAT{\kappa}-Spaces}

A \CAT{\kappa}-space is a metric space that is curved at most as much as $\ModelSpace{\kappa}^2$. The curvature is compared by comparing triangles. To make this precise, we define a \emph{geodesic triangle} to be the union of three geodesic segments $[p,q]$, $[q,r]$, and $[r,p]$ (which need not be the unique geodesics joining these points), called its \emph{edges}, and we call $p$, $q$, and $r$ its \emph{vertices}. If $\Delta$ is the triangle just described, a \emph{comparison triangle} $\bar{\Delta}$ for $\Delta$ is a geodesic triangle $[\bar{p},\bar{q}] \union [\bar{q},\bar{r}] \union [\bar{r},\bar{p}]$ in a model space $\ModelSpace{\kappa}^2$ such that $d(p,q) = d(\bar{p},\bar{q})$, $d(q,r) = d(\bar{q},\bar{r})$, $d(r,p) = d(\bar{r},\bar{p})$. If $\Point$ is a point of $\Delta$, say $\Point \in [p,q]$, then its comparison point $\bar{\Point} \in [\bar{p},\bar{q}]$ is characterized by $d(p,x) = d(\bar{p},\bar{x})$ so that also $d(q,x) = d(\bar{q},\bar{x})$.

Let $\kappa$ be a real number. A geodesic triangle $\Delta$

is said to satisfy the \emph{\CAT{\kappa} inequality} if
\[
d(\Point,\AltPoint) \le d(\bar{\Point},\bar{\AltPoint})
\]
for any two points $\Point,\AltPoint \in \Delta$ and their comparison points $\bar{\Point},\bar{\AltPoint} \in \bar{\Delta}$ in any comparison triangle $\bar{\Delta} \subseteq \ModelSpace{\kappa}$. The space $\Space$ is called a \CAT{\kappa} space if every triangle of perimeter $<2 D_\kappa$ satisfies the \CAT{\kappa} inequality (note that the condition on the perimeter is void if $\kappa \le 0$).

\begin{lem}
\label{lem:spherical_projection}
Let $\Space$ be a \CAT{\kappa}-space and let $\ConvexSet$ be a $D_\kappa$-convex subset. If $\Point \in \Space$ satisfies $d(\Point,\ConvexSet) < D_\kappa/2$, then there is a unique point $\ClosestPointProjection[\ConvexSet] \Point$ in $\ConvexSet$ that is closest to $\Point$. Moreover, the angle $\angle_{\ClosestPointProjection[\ConvexSet] \Point}(\Point,\AltPoint)$ is at least $\pi/2$ for every $\AltPoint \in \ConvexSet$.
\end{lem}

\begin{proof}
This is proven like Proposition~II.2.4~(1) in \cite{brihae}, see also Exercise~II.2.6~(1).
\end{proof}

\subsection*{Polyhedral Complexes}

An intersection of a finite (not necessarily non-zero) number of halfspaces in some $\ModelSpace{\kappa}^n$ is called an $\ModelSpace{\kappa}$-\emph{polyhedron} and if it has diameter $< D_\kappa$ it is called an $\ModelSpace{\kappa}$-\emph{polytope}. If $H^+$ is a halfspace that contains an $\ModelSpace{\kappa}$-polyhedron $\BigCell$, then the intersection $\Cell$ of the bounding hyperplane $H$ with $\BigCell$ is a \emph{face} of $\BigCell$ and $\BigCell$ is a \emph{coface} of $\Cell$. By definition $\BigCell$ is a face (and coface) of itself. The \emph{dimension} $\dim \BigCell$ of $\BigCell$ is the dimension of the minimal subspace that contains it. The \emph{(relative) interior} $\relint \BigCell$ is the interior as a subset of that space, it consists of the points of $\BigCell$ that are not points of a proper face. The \emph{codimension} of $\Cell$ in $\BigCell$ is $\dim \BigCell - \dim \Cell$. A face of codimension $1$ is a \emph{facet}.

An \emph{$\ModelSpace{\kappa}$-polyhedral complex} consists of $\ModelSpace{\kappa}$-polyhedra that are glued together along their faces. Formally, let $(\BigCell_\alpha)_\alpha$ be a family of $\ModelSpace{\kappa}$-polyhedra. Let $\AltSpace = \Dunion_\alpha \BigCell_\alpha$ be their disjoint union and $p \colon \AltSpace \to \Space$ be the quotient map modulo an equivalence relation. Then $\Space$ is an $\ModelSpace{\kappa}$-polyhedral complex if
\begin{enumerate}[label=(PC\arabic{*}), ref=PC\arabic{*},leftmargin=*]
\item for every $\alpha$ the map $p$ restricted to $\BigCell_\alpha$ is injective, and\label{item:injective}
\item for any two indices $\alpha,\beta$, if the images of the interiors of $\BigCell_{\alpha}$ and of $\BigCell_{\beta}$ under $p$ meet, then they coincide and the map $p|_{\relint \BigCell_\alpha}^{-1} \circ p|_{\relint \BigCell_\beta}$ is an isometry.
\end{enumerate}
An $\ModelSpace{\kappa}$-polyhedral complex is equipped with a quotient pseudo-metric. Martin Bridson \cite{bridson91} has shown that this pseudo-metric is a metric if only finitely many shapes of polyhedra occur, see \cite[Section~I.7]{brihae} for details. In the complexes we consider this will always be the case (in fact there will mostly be just one shape per complex).

When we speak of an $\ModelSpace{\kappa}$-polyhedral complex, we always mean the metric space together with the way it was constructed. This allows us to call the image of a face $\Cell$ of some $\BigCell_\alpha$ under $p$ a \emph{cell} (an \emph{$i$-cell} if $\Cell$ is $i$-dimensional), and to call a union of cells a \emph{subcomplex}. We write $\Cell \le \Cell'$ to express that $\Cell$ is a face of $\Cell'$ and $\Cell \lneq \Cell'$ if it is a proper face. The \emph{(relative) interior} of a cell $p(\BigCell_\alpha)$ is the image under $p$ of the relative interior of $\BigCell_\alpha$. The \emph{carrier} of a point $\Point$ of $\Space$ is the unique minimal cell that contains it; equivalently it is the unique cell that contains $\Point$ in its relative interior.

By a morphism of $\ModelSpace{\kappa}$-polyhedral complexes we mean a map that isometrically takes cells onto cells. Consequently, an isomorphism is an isometry that preserves the cell structure.

\begin{rem}
Our definition of $\ModelSpace{\kappa}$-polyhedral complexes differs from the definition in \cite{brihae} in two points: we allow the cells to be arbitrary polyhedra while in \cite{brihae} they are required to be polytopes. Since any polyhedron can be decomposed into polytopes, this does not affect the class of metric spaces that the definition describes, but only the class of possible cell structures on them. For example, a sphere composed of two hemispheres is included in our definition.

On the other hand our definition requires the gluing maps to be injective on cells which the definition in \cite{brihae} does not. Again this does not restrict the spaces one obtains: if an $\ModelSpace{\kappa}$-polyhedral complex does not satisfy this condition, one can pass to an appropriate subdivision which does. The main reasons to make this assumption here are, that it makes the complexes easier to visualize because the cells are actual polyhedra, and that the complexes we will be interested in satisfy it.
\end{rem}

If we do not want to emphasize the model space, we just speak of a \emph{polyhedron}, a \emph{polytope}, or a \emph{polyhedral complex} respectively.

If $\Space$ is a polyhedral complex and $\Set$ is a subset, the subcomplex \emph{supported by $\Set$} is the subcomplex consisting of all cells that are contained in $\Set$. If $\Cell_1$ and $\Cell_2$ are cells of $\Space$ that are contained in a common coface, then the minimal cell that contains them is denoted $\Cell_1 \vee \Cell_2$ and called the \emph{join of $\Cell_1$ and $\Cell_2$}.

By a \emph{simplicial complex}, we mean a polyhedral complex whose cells are simplices and whose face lattice is that of an abstract simplicial complex (see \cite[Chapter~3]{spanier} for an introduction to simplicial complexes). That is, each of its cells is a simplex (no two faces of which are identified by \eqref{item:injective}), and if two simplices have the same proper faces, then they coincide.

The \emph{flag complex of a poset $(P,\le)$} is an abstract simplicial complex that has $P$ as its set of vertices and whose simplices are finite flags, that is, finite totally ordered subsets of $P$. A simplicial complex is a \emph{flag complex} if the corresponding abstract simplicial complex is the flag complex of some poset. This is equivalent to satisfying the ``no triangles condition'': if $v_1,\ldots,v_n$ are vertices any two of which are joined by an edge, then there is a simplex that has $v_1,\ldots,v_n$ as vertices.

The \emph{barycentric subdivision} $\Subdiv{\Space}$ of a polyhedral complex $\Space$ is obtained by replacing each cell by its barycentric subdivision. This is always a flag complex, namely the flag complex of the poset of nonempty cells of $\Space$.

If $\Space$ is a simplicial complex and $V$ is a set of vertices, the \emph{full subcomplex of $V$} is the subcomplex of simplices in $\Space$ all of whose vertices lie in $V$. A subcomplex of $\Space$ is \emph{full} if it is the full subcomplex of a set of vertices, i.e., if a simplex is contained in it whenever all of its vertices are.

\subsection*{Links}

Let $\Space$ be a polyhedral complex and let $\Point \in \Space$ be a point. On the set of geodesics that issue at $\Point$ we consider the equivalence relation $\sim$ where $\Path \sim \Path'$ if and only if $\Path$ and $\Path'$ coincide on an initial interval; formally this means that if $\Path$ is a map $[a,b] \to \Space$ and $\Path'$ is a map $[a',b'] \to \Space$, then there is an $\varepsilon > 0$ such that $\Path(a+t) = \Path(a'+t)$ for $0 \le t < \varepsilon$.

The equivalence classes are called \emph{directions}. The direction defined by a geodesic $\Path$ that issues at $\Point$ is denoted by $\Path_\Point$\index[xsyms]{directiongeodesic@$\Path_\Point$}; we will also use this notation for geodesic segments writing for example $[\Point,\AltPoint]_\Point$\index[xsyms]{directionsegment@$[\Point,\AltPoint]_\Point$}.

The angle $\angle_\Point(\Path_\Point,\Path'_\Point) \defeq \angle_\Point(\Path,\Path')$ between two directions at a point $\Point$ is well-defined. Moreover, since $\Space$ is a polyhedral complex, two directions include a zero angle only if they coincide. Thus the angle defines a metric on the set of all directions issuing at a given point $\Point$ and this metric space is called the \emph{space of directions} or \emph{(geometric) link} of $\Point$ and denoted $\Link_\Space \Point$, or just $\Link \Point$ if the space is clear from the context.

The polyhedral cell structure on $\Space$ induces a polyhedral cell structure on $\Link \Point$. Namely, if $\Cell$ is a cell that contains $\Point$, we let $\Cell \direction \Point$ denote the subset of $\Link \Point$ of all directions that point into $\Cell$, i.e., of directions $\Path_\Point$ where $\Path$ is a geodesic whose image is contained in $\Cell$. Then $\Link \Space$ can be regarded as an $\ModelSpace{1}$-polyhedral complex whose cells are $\Cell \direction \Point$ with $\Cell \supseteq \Point$.

If $\Cell$ is a cell of $\Space$, then the links of all interior points of $\Cell$ are canonically isometric. The \emph{(geometric) link of $\Cell$}, denoted $\Link \Cell$ is the subset of any of these of directions that are perpendicular to $\Cell$. It is an $\ModelSpace{1}$-polyhedral complex whose cells are the subsets $\BigCell \direction \Cell$\index[xsyms]{taudirsigma@$\BigCell \direction \Cell$} of directions that point into a coface $\BigCell$ of $\Cell$.

If $\Point \in \Space$ is a point and $\Cell$ is its carrier, then the link of $\Point$ decomposes as
\begin{equation}
\label{eq:point_link_decomposition}
\Link \Point = (\Cell \direction \Point) * \Link \Cell
\end{equation}
where

$\Cell \direction \Point$ can be identified with the boundary $\partial \Cell$ in an obvious way and, in particular, is a sphere of dimension $(\dim \Cell - 1)$.

From a combinatorial point of view, the map $\BigCell \mapsto \BigCell \direction \Cell$ establishes a bijective correspondence between the poset of (proper) cofaces of $\Cell$ and the poset of (nonempty) cells of $\Link \Cell$. The poset of cofaces of $\Cell$ is therefore called the \emph{combinatorial link} of $\Cell$.

If $\Space$ is a simplicial complex and $\Cell \subseteq \BigCell \subseteq \Space$ are simplices, one sometimes writes\linebreak[4] $\BigCell \setminus \Cell$ to denote the complement of $\Cell$ in $\BigCell$ (this alludes to abstract simplicial complexes). Using this notation, there is a bijective correspondence $\BigCell \mapsto \BigCell \setminus \Cell$ between the combinatorial link and the subcomplex of $\Space$ of simplices $\Cell'$ which are such that $\Cell \intersect \Cell' = \emptyset$ but $\Cell \vee \Cell'$ exists.

\subsection*{Visual Boundary}

Let $(\Space,d)$ be a \CAT{0}-space. A geodesic ray $\Ray$ in $\Space$ defines a \emph{Busemann function} $\Busemann_\Ray$ by
\[
\Busemann_\Ray(\Point) = \lim_{t \to \infty} (t - d(\Point,\Ray(t)))
\]
(note the reversed sign compared to \cite[Definition~II.8.17]{brihae}). Two geodesic rays $\Ray, \Ray'$ in $\Space$ are \emph{asymptotic} if they have bounded distance, i.e., if there is a bound $R > 0$ such that $d(\Ray(t),\Ray'(t))<R$ for every $t \ge 0$. If two rays define the same Busemann function then they are asymptotic. Conversely the Busemann functions $\Busemann_\Ray$, $\Busemann_{\Ray'}$ defined by two asymptotic rays $\Ray$ and $\Ray'$ may differ by an additive constant. A \emph{point at infinity} is the class $\Infty\Ray$ of rays asymptotic to a given ray $\Ray$ or, equivalently, the class $\Infty\Busemann$ of Busemann functions that differ from a given Busemann function $\Busemann$ by an additive constant. The \emph{visual boundary} $\Infty\Space$\index[xsyms]{xinfty@$\Infty\Space$} consists of all points at infinity. It becomes a \CAT{1}-space via the angular metric
\[
d_{\Infty\Space}(\Infty\Ray,\Infty{\Ray'}) = \angle(\Ray,\Ray')
\]
(see \cite[Chapter~II.9]{brihae}).

We say that a geodesic ray $\Ray$ \emph{tends to} $\Infty\Ray$, or that $\Infty\Ray$ is the \emph{limit point} of $\Ray$, and that a Busemann function $\Busemann$ is \emph{centered at} $\Infty\Busemann$.

\begin{prop}[{\cite[Proposition~II.8.12]{brihae}}]
\label{prop:ray_from_point_to_point_at_infty}
Let $\Space$ be a \CAT{0}-space. If $\Point$ is a point and $\InftyPoint$ is a point at infinity of $\Space$, then there is a unique geodesic ray $\Ray$ that issues at $\Point$ and tends to $\InftyPoint$.
\end{prop}

In the situation of the proposition we denote the image of $\Ray$ by $[\Point,\InftyPoint)$.

\footerlevel{3}
\headerlevel{3}

\section{Spherical Geometry}
\label{sec:spherical_geometry}

In this section we discuss some spherical geometry, that is, geometry of spheres $\S^n$ of curvature $1$. We start with configurations that are essentially $2$-di\-men\-sio\-nal and then extend them to higher dimensions.

First we recall the Spherical Law of Cosines:
\begin{prop}[{\cite[I.2.2]{brihae}}]
Let $a$, $b$ and $c$ be points on a sphere, let $[c,a]$ and $[c,b]$ be geodesic segments that join $c$ to $a$ respectively $b$ (which may not be uniquely determined if $a$ or $b$ has distance $\pi$ to $c$), and let $\gamma$ be the angle in $c$ between these segments. Then
\[
\cos d(a,b) = \cos d(a,c) \cos d(b,c) + \sin d(a,c) \sin d(b,c) \cos \gamma \text{ .}
\]
\end{prop}

\subsection*{Spherical Triangles}

For us a spherical triangle is given by three points $a$, $b$ and $c$ any two of which have distance $< \pi$ and that are not collinear (i.e., do not lie in a common $1$-sphere). Note that this implies in particular that all angles and all edge lengths have to be positive. The \emph{spherical triangle} itself is the convex hull of $a$, $b$ and $c$.

\begin{obs}
\label{obs:spherical_triangles}
Let $a$, $b$ and $c$ be points on a sphere any two of which have distance $< \pi$. Write the respective angles as $\alpha = \angle_a(b,c)$, $\beta = \angle_b(a,c)$ and $\gamma = \angle_c(a,b)$.
\begin{enumerate}
\item If $d(a,b) = \pi/2$ and $d(b,c),d(a,c) \le \pi/2$, then $\gamma \ge \pi/2$. \label{item:edge_gives_angle}
\item If $d(a,b) = \pi/2$ and $\beta = \pi/2$, then $d(a,c) = \gamma = \pi/2$.\\
If $d(a,b) = \pi/2$ and $\beta < \pi/2$, then $d(a,c) < \pi/2$. \label{item:edge_and_angle_give_edge_and_angle}
\item If $d(a,b) = d(a,c)= \pi/2$ and $b \ne c$, then $\beta = \gamma = \pi/2$. \label{item:edges_give_angles}
\item If $\beta = \gamma = \pi/2$ and $b \ne c$, then $d(a,b) = d(a,c) = \pi/2$. \label{item:angles_give_edges}
\end{enumerate}
\end{obs}

\begin{proof}
All properties can be deduced from the Spherical Law of Cosines. But they can also easily be verified geometrically. We illustrate this for the fourth statement. Put $b$ and $c$ on the equator of a $2$-sphere. The two great circles that meet the equator perpendicularly in $b$ and $c$ only meet at the poles, which have distance $\pi/2$ from the equator.
\end{proof}

The following statements are less obvious:

\begin{prop}
\label{prop:angles_bound_length}
If in a spherical triangle the angles are at most $\pi/2$, then the edges have length at most $\pi/2$.

If in addition two of the edges have length $< \pi/2$ then so has the third.
\end{prop}

\begin{proof}
Let $a$, $b$ and $c$ be the vertices of the triangle and set $x \defeq \cos d(b,c)$, $y \defeq \cos d(a,c)$ and $z \defeq \cos d(a,b)$. Then the Spherical Law of Cosines implies that
\[
z \ge xy, \quad x \ge yz, \quad \text{and} \quad y \ge xz \text{ .}
\]
Substituting $y$ in the first inequality gives $z \ge x^2z$, i.e., $z(1-x^2) \ge 0$. Since $x \ne 1$ by our non-degeneracy assumption for spherical triangles, this implies that $z \ge 0$. Permuting the points yields the statement for the other edges.

For the second statement assume that there is an edge, say $[a,b]$, that has length $\pi/2$. Then $\cos d(a,b) = 0$.  By what we have just seen, all terms in the Spherical Law of Cosines are non-negative so one factor in each summand has to be zero. This implies that at least one of $d(a,c)$ and $d(b,c)$ is $\pi/2$.
\end{proof}

\subsection*{Decomposing Spherical Simplices}

Now we want to study higher dimensional simplices. We first study simplicial cones in Euclidean space.

Let $V$ be a Euclidean vector space of dimension $n+1$ and let $H_0^+,\ldots,H_n^+$ be linear halfspaces with bounding hyperplanes $H_0,\ldots,H_n$. We assume that the $H_i$ are in general position, i.e., that any $k$ of them meet in a subspace of dimension $n+1-k$. For $0 \le i \le n$ we set $L_i  \defeq \Intersect_{j \ne i} H_j$ and $L_i^+ \defeq L_i \intersect H_i^+$ and call the latter a \emph{bounding ray}. In this situation $S \defeq \Intersect_i H_i^+$ is a simplicial cone that is the convex hull of the bounding rays.

For every $i$ let $v_i$ be the unit vector in $L_i^+$. We define the \emph{angle} between $L_i^+$ and $L_j^+$ to be the angle between $v_i$ and $v_j$. Similarly, for two halfspaces $H_i^+$ and $H_j^+$ let $N$ be the orthogonal complement of $H_i \intersect H_j$. The \emph{angle} between $H_i^+$ and $H_j^+$ is defined to be the angle between $H_i^+ \intersect N$ and $H_j^+ \intersect N$. We are particularly interested in when two halfspaces or bounding rays are \emph{perpendicular}, i.e., include an angle of $\pi/2$.

So assume that there are index sets $I$ and $J$ that partition $\{0,\ldots,n\}$ such that $H_i^+$ is perpendicular to $H_j^+$ for every $i \in I$ and $j \in J$. Then $V$ decomposes as an orthogonal sum
\[
V = V_I \oplus V_J \quad \text{with} \quad V_I \defeq \Intersect_{i \in I} H_i \text{ ,} \quad V_J \defeq \Intersect_{j \in J} H_j
\]
where the $v_j, j \in J$ form a basis for $V_I$ and vice versa. In particular, $L_i^+$ is perpendicular to $L_j^+$ for $i \in I$ and $j \in J$.

By duality we see that conversely if $I$ and $J$ partition $\{0,\ldots,n\}$ such that $L_i^+$ is perpendicular to $L_j^+$ for every $i \in I$ and $j \in J$, then also $H_i^+$ is perpendicular to every $H_j^+$ for $i \in I$ and $j \in J$.

This shows:

\begin{obs}
\label{obs:direct_product_decomposition}
Let $S$ be a simplicial cone in $\E^{n+1}$ and let $S_1$, $S_2$ be faces of $S$ that span complementary subspaces of $V$. These are equivalent:
\begin{enumerate}
\item $S_1$ and $S_2$ span orthogonal subspaces.
\item every bounding ray of $S_1$ is perpendicular to every bounding ray of $S_2$.
\item every facet of $S$ that contains $S_1$ is perpendicular to every facet of $S$ that contains $S_2$.\qed
\end{enumerate}
\end{obs}

Now we translate the above to spherical geometry. We start with the angles. The definition is perfectly analogous to that made above in Euclidean space. Passage to the link plays the role of intersecting with the orthogonal complement.

Let $\BigCell$ be a spherical polyhedron. Let $\Cell_1$ and $\Cell_2$ be faces of $\BigCell$ of same dimension $k$ such that $\Cell \defeq \Cell_1 \intersect \Cell_2$ has codimension $1$ in both. Then $\Cell_1$ and $\Cell_2$ span a sphere $S$ of dimension $k+1$. We look at the $1$-sphere $\Link_{S} \Cell$. The subset $\Link_{\BigCell \intersect S} \Cell$ is a $1$-dimensional polyhedron with vertices $\Cell_1 \direction \Cell$ and $\Cell_2 \direction \Cell$. The diameter of this polyhedron is called the \emph{angle $\angle(\Cell_1,\Cell_2)$ between} $\Cell_1$ and $\Cell_2$.

\begin{rem}
\label{rem:vertex_angle}
Note that, in particular, if $\Cell_1$ and $\Cell_2$ are two vertices (faces of dimension $0$ that meet in their face $\emptyset$ of dimension $-1$), then the angle between them is just the length of the edge that joins them, i.e., their distance.
\end{rem}

A \emph{spherical simplex} of dimension $n$ is a spherical polytope of dimension $n$ that is the intersection of $n+1$ hemispheres (and of the $n$-dimensional sphere that it spans). Faces of spherical simplices are again spherical simplices. Spherical simplices of dimension $2$ are spherical triangles. If $\Cell$ is a face of a simplex $\BigCell$, its \emph{complement (in $\BigCell$)} is the face $\Cell'$ whose vertices are precisely the vertices that are not vertices of $\Cell$. In that case $\Cell$ and $\Cell'$ are also said to be \emph{complementary faces of $\BigCell$}.

We can now restate Observation~\ref{obs:direct_product_decomposition} as

\begin{obs}
\label{obs:spherical_join_decomposition}
Let $\BigCell$ be a spherical simplex and let $\Cell_1$, $\Cell_2$ be two complementary faces of $\BigCell$. These are equivalent:
\begin{enumerate}
\item $\BigCell = \Cell_1 * \Cell_2$.
\item $d(\Cell_1,\Cell_2) = \pi/2$.
\item $d(v,w) = \pi/2$ for any two vertices $v$ of $\Cell_1$ and $w$ of $\Cell_2$.
\item $\angle(\BigCell_1,\BigCell_2) = \pi/2$ for any two facets $\BigCell_1$ and $\BigCell_2$ that contain $\Cell_1$ respectively $\Cell_2$.\qed
\end{enumerate}
\end{obs}

Combinatorially this can be expressed as follows: Let $\Vertices \BigCell$ denote the set of vertices and $\Facets \BigCell$ denote the set of facets of a spherical simplex $\BigCell$. On $\Vertices \BigCell$ we define the relation $\sim_v$ of having distance $\ne \pi/2$ and on $\Facets \BigCell$ the relation $\sim_f$ of having angle $\ne \pi/2$. Both relations are obviously reflexive and symmetric. There is a map $\Vertices \BigCell \to \Facets \BigCell$ that takes a vertex to its complement. It is not generally clear how this map behaves with respect to the relations but Observation~\ref{obs:spherical_join_decomposition} states that it preserves their transitive hulls. More precisely:

\begin{obs}
Let $\BigCell$ be a spherical simplex and let $\sim_v$ be the relation on $\Vertices \BigCell$ and $\sim_f$ the relation on $\Facets \BigCell$ defined above. Let $\approx_v$ and $\approx_f$ be their transitive hulls. If $v_1$ and $v_2$ are vertices with complements $\BigCell_1$ and $\BigCell_2$, then
\[
v_1 \approx_v v_2 \quad \text{if and only if} \quad \BigCell_1 \approx_f \BigCell_2 \text{ .}\tag*{\qed}
\]
\end{obs}

The reason why we dwell on this is that if $\BigCell$ is the fundamental simplex of a finite reflection group, the relation $\sim_f$ will give rise to the Coxeter diagram while the relation $\sim_v$ will be seen to be an equivalence relation. We just chose to present the statements in more generality.

\subsection*{Spherical Polytopes with Non-Obtuse Angles}

Let $\BigCell$ be a spherical polytope. We have defined angles $\angle(\Cell_1,\Cell_2)$ for any two faces $\Cell_1$ and $\Cell_2$ of $\BigCell$ that have same dimension and meet in a common codimension-$1$ face. In what follows, we are interested in polytopes where all of these angles are at most $\pi/2$. We say that such a polytope has \emph{non-obtuse angles}. Our first aim is to show that it suffices to restrict the angles between facets, all other angles will then automatically be non-obtuse. Second we observe that if $\BigCell$ has non-obtuse angles, then the relation $\sim_v$ on the vertices of having distance $\ne \pi/2$ is an equivalence relation.

First however we need to note another phenomenon:

\begin{obs}
Let $\BigCell$ be a polyhedron and let $\Cell_1\le \Cell_2$ be faces. Then there is a canonical isometry
\[
\Link \Cell_2 \to \Link \Cell_2 \direction \Cell_1
\]
that takes $\Link_{\BigCell} \Cell_2$ to $\Link_{\BigCell \direction \Cell_1} \Cell_2 \direction \Cell_1$.\qed
\end{obs}
To describe this isometry let $p_2$ be an interior point of $\Cell_2$ (that has distance $< \pi/2$ to $\Cell_1$ and) that projects onto an interior point $p_1$ of $\Cell_1$. If $\gamma$ is a point of $\Link \Cell_2$, then there is a geodesic segment $[p_2,x]$ representing it. For every $y \in [p_2,x]$ the geodesic segment $[p_1,y]$ represents a direction at $p_1$ that is a point of $\Link \Cell_1$. All these points form a segment in $\Link \Cell_1$ that defines a point $\rho$ of $\Link \Cell_2 \direction \Cell_1$. The isometry takes $\gamma$ to $\rho$.
Formally (using the notation from Section~\ref{sec:metric_spaces}) this can be written as:
\[
[p_2,x]_{p_2} \mapsto [[p_1,p_2]_{p_1},[p_1,x]_{p_1}]_{[p_1,p_2]_{p_1}}
\]
(see Figure~\ref{fig:link_of_link}).

\begin{figure}[!ht]
\begin{center}
\includegraphics{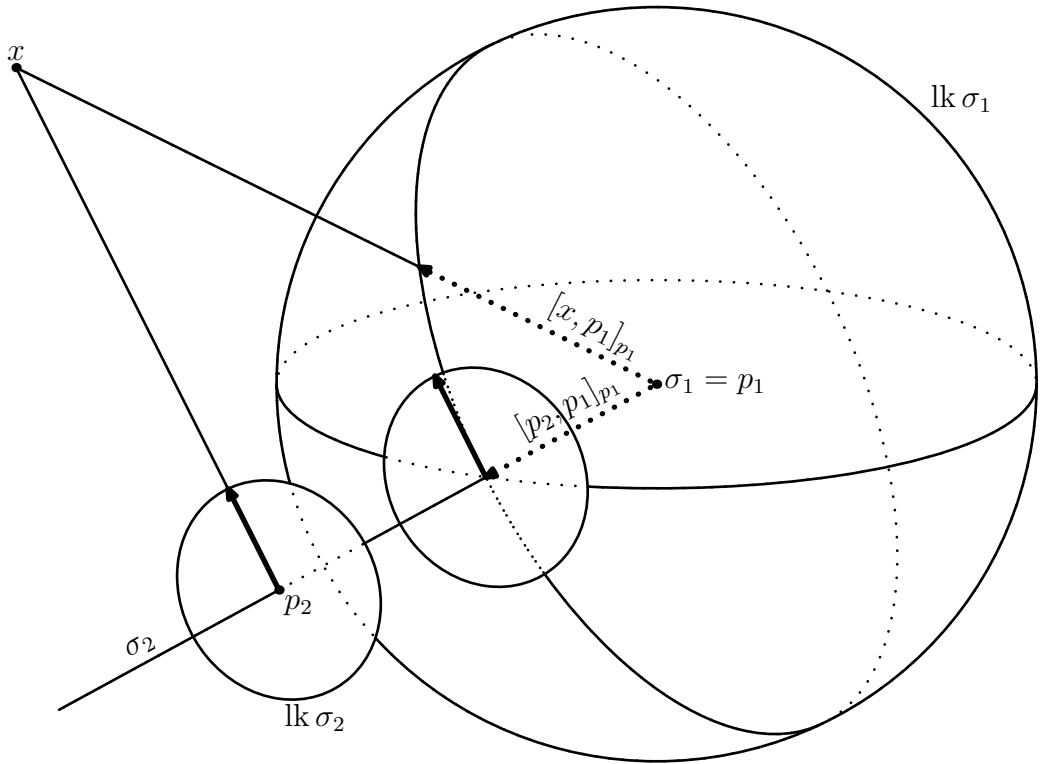}
\end{center}
\caption{The picture illustrates the identification of links. The cell $\Cell_2$ is an edge and $\Cell_1$ is one of its vertices. The link of the vertex $\Cell_2 \direction \Cell_1 = [p_2,p_1]_{p_1}$ is identified with the link of $\Cell_2$. The direction from $p_2$ toward $x$, which is an element of $\Link \Cell_2$, is identified with the direction from $[p_2,p_1]_{p_1}$ toward $[x,p_1]_{p_1}$.}
\label{fig:link_of_link}
\end{figure}

\begin{prop}
\label{prop:simplex_angles_bound_angles}
If a spherical polytope $\BigCell$ has the property that the angle between any two facets is at most $\pi/2$, then it has non-obtuse angles.
\end{prop}

\begin{proof}
Proceeding by induction, it suffices to show that if $\Cell_1$ and $\Cell_2$ are faces of codimension $2$ in $\BigCell$ that meet in a face $\Cell\defeq \Cell_1 \intersect \Cell_2$ of codimension $3$, then $\angle(\Cell_1,\Cell_2) \le \pi/2$. In that situation $\Link_{\BigCell} \Cell$ is a spherical polygon in the $2$-sphere $\Link \Cell$. As described above $\Link \Cell_i$ can be identified with $\Link \Cell_i \direction \Cell$ in such a way that directions into $\BigCell$ are identified with each other.

Under this identification angles between facets of $\BigCell$ that contain $\Cell$ are identified with the angles between edges of the polygon described above. Since the sum of angles of a spherical $n$-gon is $> (n-2)\pi$ but the sum of angles of our polygon is $\le n(\pi/2)$ we see that $n < 4$ hence the polygon is a triangle.

Since the angle $\angle(\Cell_1,\Cell_2)$ is the distance of the vertices $\Cell_1 \direction \Cell$ and $\Cell_2 \direction \Cell$, the statement follows from Proposition~\ref{prop:angles_bound_length}.
\end{proof}

Along the way we have seen that if $\BigCell$ has non-obtuse angles then it is simple (links of vertices are simplices). In fact more is true:

\begin{lem}[{\cite[Lemma~6.3.3]{davis08}}]
A spherical polytope that has non-obtuse angles is a simplex.
\end{lem}

Note that if $\BigCell$ has non-obtuse angles, then any two vertices have distance $\le \pi/2$, cf.\ Remark~\ref{rem:vertex_angle}.

\begin{obs}
\label{obs:acute_edge_equivalence_relation}
If $\BigCell$ is a spherical simplex that has non-obtuse angles, then the relation $\sim_v$ on the vertices of having distance $< \pi/2$ is an equivalence relation.
\end{obs}

\begin{proof}
Let $a$, $b$ and $c$ be vertices of $\BigCell$. We have to show that if $d(a,c) < \pi/2$ and $d(c,b) < \pi/2$, then $d(a,b) < \pi/2$. Consider the triangle spanned by $a$, $b$ and $c$. Since $\BigCell$ has non-obtuse angles, the angles in this triangle are at most $\pi/2$. Now the statement is the second statement of Proposition~\ref{prop:angles_bound_length}.
\end{proof}

This suggests to call a spherical simplex with non-obtuse angles \emph{irreducible} if it has diameter $< \pi/2$. Then Observation~\ref{obs:acute_edge_equivalence_relation} and Observation~\ref{obs:spherical_join_decomposition} show that such a simplex is irreducible if and only if it can not be decomposed as the join of two proper faces.

The following is easy to see and allows us to include polyhedra in our discussion:

\begin{obs}
A spherical polyhedron $\BigCell$ decomposes as a join $S * \Cell$ of its maximal subsphere $S$ and a polytope $\Cell$. If the angle between any two facets of $\BigCell$ is non-obtuse, the same is true of $\Cell$.\qed
\end{obs}

To sum up we have shown the following:

\begin{thm}
\label{thm:non_obtuse_angle_decomposition}
Let $\BigCell$ be a spherical polyhedron that has the property that any two of its facets include an angle of at most $\pi/2$. Then $\BigCell = S * \Cell$ where $S$ is the maximal sphere contained in $\BigCell$ and $\Cell$ is a spherical simplex that has non-obtuse angles.

Moreover, $\Cell$ decomposes as a join $\Cell_1 * \cdots * \Cell_k$ of irreducible faces and two vertices of $\Cell$ lie in the same join factor if and only if they have distance $< \pi/2$.
\end{thm}

\footerlevel{3}
\headerlevel{3}

\section{Finiteness Properties}
\label{sec:finiteness_properties}

In this section we collect the main facts about the topological finiteness properties $F_n$. Topological finiteness properties of groups were introduced by C.T.C. Wall in \cite{wall65,wall66}. A good reference on the topic is \cite{geoghegan}, where also other properties such as finite geometric dimension are introduced. At the end of the section we briefly describe the relation between topological and homological finiteness properties. The standard book on homology of groups is \cite{brown82}.

Let $\Disk[n]$ denote the closed unit-ball in $\R^n$ as a topological space and let $\Sphere[n-1] \subseteq \Disk[n]$ denote the unit sphere also regarded as a topological space, in particular, $\Sphere[-1] = \emptyset$. An \emph{$n$-cell} is a space homeomorphic to $\Disk[n]$ and its \emph{boundary} is the subspace that is identified with $\Sphere[n-1]$.

Recall that a \emph{CW-complex} $\Space$ is a topological space that is obtained from the empty set by inductively gluing in cells of increasing dimension along their boundary, see \cite[Chapter~0]{hatcher} for a proper definition. Under the gluing process, the cells need not be embedded in $\Space$ but nonetheless we call their images \emph{cells}. A \emph{subcomplex} of $\Space$ is the union of some of its cells. The union of all cells up to dimension $n$ is called the \emph{$n$-skeleton of $\Space$} and denoted $\Space^{(n)}$.

When we speak of a CW-complex we always mean the topological space together with its decomposition into cells. Furthermore by an action of a group on a CW-complex we mean an action that preserves the cell structure.

A topological space $\Space$ is \emph{$n$-connected} if for $-1 \le i \le n$ every map $\Sphere[i] \to \Space$ extends to a map $\Disk[i+1] \to \Space$. In other words a space is $n$-connected if it is non-empty and $\pi_i(\Space)$ is trivial for $0 \le i \le n$. We say that $\Space$ is \emph{$n$-aspherical} if it satisfies the same condition except possibly for $i=1$. A CW-complex is \emph{$n$-spherical} if it is $n$-dimensional and $(n-1)$-connected and it is \emph{properly $n$-spherical} if in addition it is not $n$-connected (equivalently if it is not contractible).

A connected CW-complex $\Space$ is called a \emph{classifying space} for a group $\Group$ or a \emph{$K(\Group,1)$ complex} if the fundamental group of $\Space$ is (isomorphic to) $\Group$ and all higher homotopy groups are trivial (cf.\ \cite[Sections~VII.11,12]{bredon}, \cite[Chapter~7]{geoghegan}, \cite[Section~1.B]{hatcher}). The latter condition means that every map $\Sphere[n] \to \Space$ extends to a map $\Disk[n+1] \to \Space$ for $n \ge 2$. Yet another way to formulate it is to require the universal cover $\tilde{\Space}$ to be contractible. Classifying spaces exist for every group and are unique up to homotopy equivalence. If $\Space$ is a classifying space for $\Group$ we can identify $\Group$ with the fundamental group of $\Space$ and obtain an action of $\Group$ on $\Space$ (which can be made to preserve the cell structure); we may sometimes do this implicitly.

We can now define the topological finiteness properties that we are interested in. A group $\Group$ is \emph{of type $F_n$} if there is a $K(\Group,1)$ complex that has finite $n$-skeleton (here ``finite'' means ``having a finite number of cells'', topologically this is equivalent to the complex being compact). A group that is of type $F_n$ for every $n \in \N$ is said to be \emph{of type $F_\infty$}.

There are a few obvious reformulations of this definition:

\begin{lem}
Let $\Group$ be a group and let $n \ge 2$. These are equivalent:
\begin{enumerate}
\item $\Group$ is of type $F_n$.
\item $\Group$ acts freely on a contractible CW-complex $\Space_2$ that has finite $n$-skeleton modulo the action of $\Group$.
\item there is a finite, $(n-1)$-aspherical CW-complex $\Space_3$ with fundamental group $\Group$.
\item $\Group$ acts freely on an $(n-1)$-connected CW-complex $\Space_4$ that is finite modulo the action of $\Group$.
\end{enumerate}
\end{lem}

\begin{proof}
Assume that $\Group$ is of type $F_n$ and let $\Space_1$ be a $K(\Group,1)$ complex with finite $n$-skeleton.

We may take $\Space_2$ to be the universal cover of $\Space_1$. Indeed $\tilde{\Space}_1$ is contractible and $\Group$ acts on it freely by deck transformations. Since $\tilde{\Space}_1/\Group = \Space_1$ we see that it also has finite $n$-skeleton modulo the action of $\Group$.

The space $\Space_3$ may be taken to be the $n$-skeleton of $\Space_1$: By assumption $\Space_1^{(n)}$ is finite and has fundamental group $\Group$. Since the $(i-1)$th homotopy group only depends on the $i$-skeleton, we see that it is also $(n-1)$-aspherical.

Finally the space $\Space_4$ may be taken to be the $n$-skeleton of the universal cover of $\Space_1$ as one sees by combining the arguments above.

Conversely if $\Space_2$ is given, one obtains a $K(\Group,1)$ complex with finite $n$-skeleton by taking the quotient modulo the action of $\Group$.

If $\Space_3$ is given, one may kill the higher homotopy groups by gluing in cells from dimension $n+1$ on. The homotopy groups up to $\pi_{n-1}(\Space_3)$ are unaffected by this because they only depend on the $n$-skeleton.

If $\Space_4$ is given, one may $\Group$-equivariantly glue in cells from dimension $n+1$ on to get a contractible space on which $\Group$ acts freely and then take the quotient modulo this action.
\end{proof}

The maximal $n$ in $\N \union \{\infty\}$ such that $\Group$ is of type $F_n$ is called the \emph{finiteness length} of $\Group$.

Until now the properties $F_n$ may seem fairly arbitrary so the following should serve as a motivation:

\begin{prop}
Every group is of type $F_0$. A group is of type $F_1$ if and only if it is finitely generated and is of type $F_2$ if and only if it is finitely presented.
\end{prop}

\begin{proof}
Given a group presentation $\Group = \gen{S \mid R}$ a $K(\Group,1)$ complex can be constructed as follows: Start with one vertex. Glue in a $1$-cell for every element of $S$ (at this stage the fundamental group is the free group generated by $S$) and pick an orientation for each of them. Then glue in $2$-cells for every element $r$ of $R$, along the boundary as prescribed by the $S$-word $r$ (cf.\ \cite[Chapter~6]{seithr80}). Finally kill all higher homotopy by gluing in cells from dimension $3$ on. This gives rise to a $K(\Group,1)$ complex and it is clear that it has finite $1$-skeleton if $S$ is finite and finite $2$-skeleton if $R$ is also finite.

Conversely assume we are given a $K(\Group,1)$ complex. Its $1$-skeleton is a graph so it contains a maximal tree $T$. Factoring this tree to a point is a homotopy equivalence (\cite[Corollary~3.2.5]{spanier}), so we obtain a $K(\Group,1)$ complex that has only one vertex, which shows the first statement. Moreover, we can read off a presentation of $\Group$ as follows: Take one generator for each $1$-cell. Again, after an orientation has been chosen for each $1$-cell, a relation for each $2$-cell can be read off the way the $2$-cell is glued in. If the $1$-skeleton was finite, the obtained presentation is finitely generated, if the $2$-skeleton was finite, the obtained presentation is finite.
\end{proof}

There is another, stronger, finiteness property: a group $\Group$ is of type $F$ if there is a finite $K(\Group,1)$ complex. Clearly if a group is of type $F$, then it is of type $F_\infty$, but the converse is false:

\begin{fact}[{\cite[Corollary~7.2.5, Proposition~7.2.12]{geoghegan}}]
\label{fact:finite_groups}
Every finite group is of type $F_\infty$ but is not of type $F$ unless it is the trivial group. In fact every group of type $F$ is torsion-free.
\end{fact}

To give examples of groups of type $F$ by definition means to give examples of finite classifying spaces:

\begin{exmpl}
\begin{enumerate}
\item For every $n \in \N$ the free group $F_n$ generated by $n$ elements is of type $F$: it is the fundamental group of a wedge of $n$ circles which is a classifying space because it is $1$-dimensional.
\item For every $n \in \N$ the free abelian group $\Z^n$ generated by $n$ elements is of type $F$: it is the fundamental group of an $n$-torus, which is a classifying space because its universal cover is $\R^n$.
\item For every $g \le 0$ the closed oriented surface $S_g$ of genus $g$ is a classifying space because it is $2$-dimensional and contains no embedded $2$-sphere. Hence its fundamental group $\pi_1(S_g)$ is of type $F$.
\end{enumerate}
\end{exmpl}

The properties $F_n$ have an important feature that the property $F$ has not:

\begin{fact}[{\cite[Corollary~7.2.4]{geoghegan}}]
\label{fact:fn_is_virtual}
For every $n$, if $\Group$ is a group and $\AltGroup$ is a subgroup of finite index, then $\AltGroup$ is of type $F_n$ if and only if $\Group$ is of type $F_n$.
\end{fact}

To see that the analogous statement is false for $F$, note that every finite group has a subgroup of finite index that is of type $F$: the trivial group. But in general it is not of type $F$ itself, see Fact~\ref{fact:finite_groups}.

A group is said to \emph{virtually} have some property if it has a subgroup of finite index that has that property. So one implication of Fact~\ref{fact:fn_is_virtual} can be restated by saying that a group that is virtually of type $F_n$ is of type $F_n$. Note in particular, that a group that is virtually of type $F$ is itself of type $F_\infty$.

The definition of the properties $F_n$ is not easy to work with mainly for two reasons: for a given group one often
knows the ``right'' space to act on, but the action is not free but only ``almost free'' for example in the sense that
cell stabilizers are finite. Sometimes the group has a torsion-free subgroup of finite index which then acts freely.
But, for example, the groups we want to study in these notes are not virtually torsion-free. Another problem that is
not so obvious to deal with from the definition is how to prove that a group is not of type $F_n$.

Fortunately Ken Brown has given a criterion which allows one to determine the precise finiteness length of a given group. Below we state Brown's Criterion in the full generality, even though we only need a special case.

We need some notation (see \cite{brown87}). Let $\Space$ be a CW-complex on which a group $\Group$ acts. By a $\Group$-invariant filtration we mean a family of $\Group$-invariant subcomplexes $(\Space_\alpha)_{\alpha \in I}$, where $I$ is some partially ordered index set, such that $\Space_\alpha \subseteq \Space_\beta$ whenever $\alpha \le \beta$, and such that $\Union_{\alpha \in I} \Space_\alpha = \Space$.

A directed system of groups is a family of groups $(\Group_\alpha)_{\alpha \in I}$, indexed by some partially ordered set $I$, together with morphisms $f_\alpha^\beta \colon \Group_\alpha \to \Group_\beta$ for $\alpha \le \beta$, such that $f_\beta^\gamma f_\alpha^\beta = f_\alpha^\gamma$ whenever $\alpha \le \beta \le \gamma$. A directed system of groups is said to be \emph{essentially trivial} if for every $\alpha$ there is a $\beta \ge \alpha$ such that $f_\alpha^\beta$ is the trivial morphism.

Clearly for every homotopy functor $\pi_i$, a filtration $(\Space_\alpha)_{\alpha \in I}$ induces a directed system of groups $(\pi_i(\Space_\alpha))_{\alpha \in I}$. We can now state Brown's Criterion:

\begin{thm}[{{\cite[Theorem~2.2, Theorem~3.2]{brown87}}}]
\label{thm:brown_criterion}
Let $\Group$ be a group that acts on an $(n-1)$-connected CW-complex $\Space$. Assume that for $0 \le k \le n$, the stabilizer of every $k$-cell of $\Space$ is of type $F_{n-k}$. Let $(\Space_\alpha)_{\alpha \in I}$ be a filtration of $\Group$-invariant subcomplexes of $\Space$ that are compact modulo the action of $\Group$. Then $\Group$ is of type $F_n$ if and only if the directed system $(\pi_i(\Space_\alpha))_{\alpha \in I}$ is essentially trivial for $0 \le i < n$.
\end{thm}

The weaker version we will be using is:

\begin{cor}
\label{cor:adapted_browns_criterion}
Let $\Group$ be a group that acts on a contractible CW-complex $\Space$. Assume that the stabilizer of every cell is finite. Let $(\Space_k)_{k \in \N}$ be a filtration of $\Group$-invariant subcomplexes of $\Space$ that are compact modulo the action of $\Group$. Assume that the maps $\pi_i(\Space_k) \to \pi_i(\Space_{k+1})$ are isomorphisms for $0 \le i < n - 1$ and that the maps $\pi_{n-1}(\Space_k) \to \pi_{n-1}(\Space_{k+1})$ are surjective and infinitely often not injective. Then $\Group$ is of type $F_{n-1}$ but not of type $F_n$.
\end{cor}

\begin{proof}
Since $\Space$ is contractible it is, in particular, $(n-1)$-connected. The finite cell stabilizers are of type $F_\infty$ by Fact~\ref{fact:finite_groups}. The directed systems $(\pi_i(\Space_k))_{k \in \N},0 \le i < n-1$ of isomorphisms have trivial limit and therefore must be trivial. It remains to look at the directed system $(\pi_{n-1}(\Space_k))_{k \in \N}$. Let $\alpha,\beta \in \N$ be such that $\beta \ge \alpha$. Let $\gamma \ge \beta$ be such that $\pi_{n-1}(\Space_\gamma) \to \pi_{n-1}(\Space_{\gamma+1})$ is not injective. Then $\pi_{n-1}(\Space_\gamma)$ is non-trivial. Thus, since $\pi_{n-1}(\Space_\alpha) \to \pi_{n-1}(\Space_\gamma)$ is surjective and factors through $\pi_{n-1}(\Space_\alpha) \to \pi_{n-1}(\Space_\beta)$, the latter cannot be trivial.
\end{proof}

Brown's original proof is algebraic using the relation between topological and homological finiteness properties (see below). A topological proof based on rebuilding a CW-complex within its homotopy type is sketched in \cite{geoghegan}.

The homological finiteness properties we want to introduce now are closely related to, but slightly weaker than, the topological finiteness properties discussed above -- as is homology compared to homotopy. We will not actually use them and therefore give a rather brief description. The interested reader is referred to \cite{brown82} and \cite{bieri76}.

Let $\Group$ be a group. The ring $\Z\Group$ consists of formal sums of the form $\sum_{g \in G} n_g g$ where the $n_g$ are elements of $\Z$ and all but a finite number of them is $0$. Addition and multiplication are defined in the obvious way. The ring $\Z$ becomes a $\Z\Group$-module by letting $\Group$ act trivially, i.e., via $(\sum_{g \in G} n_g g) \cdot m = \sum_{g \in G} n_g m$. A \emph{partial resolution} of length $n$ of the $\Z\Group$-module $\Z$ is an exact sequence
\begin{equation}
\label{eq:partial_resolution}
F_n \to \cdots \to F_1 \to F_0 \to \Z \to 0
\end{equation}
of $\Z\Group$-modules (this is not to be confused with a \emph{resolution} of length $n$ which would have a leading $0$). The partial resolution is said to be \emph{free}, \emph{projective}, or \emph{of finite type} if the modules are free, projective, or finitely generated respectively.

The group $\Group$ is said to be of type $\FP_n$ if there is a partial free resolution of length $n$ of finite type of the $\Z\Group$-module $\Z$. This is equivalent to the existence of a partial projective resolution of length $n$ of finite type (\cite[Proposition~VIII.4.3]{brown82}).

The following is not hard to see from the way the homology of $\Group$ is defined:

\begin{obs}
If $\Group$ is of type $\FP_n$, then $H_i\Group$ is finitely generated for $i \le n$.\qed
\end{obs}

Now we describe the relation between the properties $F_n$ and $\FP_n$ we mentioned earlier:

\begin{fact}
If a group is of type $F_n$, then it is of type $\FP_n$. It is of type $F_1$ if and only if it is of type $\FP_1$. For $n \ge 2$ it is of type $F_n$ if and only if it is of type $F_2$ and of type $\FP_n$. There are groups that are of type $\FP_2$ but not of type $F_2$.
\end{fact}

\begin{proof}
The first two statements are elementary: Let $\Group$ be a group. Let $\Space$ be a $K(\Group,1)$ complex with finite $n$-skeleton. Let $\tilde{\Space}$ be its universal cover. Then $\Group$ acts on $\tilde{\Space}$ so its augmented chain complex has a $\Z\Group$-structure (cf.\ \cite[Section~I.4]{brown82}). Since $\tilde{\Space}$ is contractible and thus has trivial homology, the augmented chain complex is a resolution of the $\Z\Group$-module $\Z$. That $\tilde{\Space}$ has finite $n$-skeleton modulo $\Group$ implies that the resolution is finitely generated up to the $n$-th term.

For the second statement consider first the resolution
\[
0 \to I \to \Z\Group \stackrel{\varepsilon}{\to} \Z \to 0
\]
where $\varepsilon(\sum_{g \in G} n_gg) = \sum_{g \in G} n_g$ and $I = \ker \varepsilon$ is the ideal generated by elements $g-1$ with $g \in \Group$. If $I$ is finitely generated, then there is a finite set $S \subseteq \Group$ such that $S = S^{-1}$ and $I$ is generated by elements of the form $s-1$. So if $g \in \Group$ is arbitrary, then we can write
\begin{equation}
g-1 = \sum_{s \in S}\sum_{h \in \Group} n_{s,h}h(s - 1)
\label{eq:g_minus_1}
\end{equation}
with $n_{s,h} \in \N$. We show that $g$ lies in the span of $S$ by induction on $\sum n_{s,h}$. If $\sum n_{s,h} = 1$ then $g = s$ for the unique pair $(s,h)$ for which $n_{s,h} = 1$ and we are done. Otherwise let $(s',g')$ be such that $n_{s',g'} > 0$ and $g's' = g$; this exists by \eqref{eq:g_minus_1}. Then
\[
g' - 1 = (g-1) - g'(s'-1) = \sum_{s \in S} \sum_{h \in \Group} n'_{s,h}h(s-1)
\]
with $n'_{s',g'} = n_{s',g'} - 1$ and $n'_{s,h} = n_{s,h}$ for $(s,h) \ne (s',g')$. So $g$ lies in the span of $S$ by induction. This shows that $\Group$ is generated by $S$.

Now let
\[
F_1 \to F_0 \stackrel{\varepsilon'}{\to} \Z \to 0
\]
be a partial resolution by finitely generated free $\Z\Group$-modules. There is a basis $f_1,\ldots,f_k$ for $F_0$ with $\varepsilon'(f_1)=1$. Choose elements $m_2,\ldots,m_k \in \Z\Group$ such that $\varepsilon'(f_i) = \varepsilon(m_i)$ and take $m_1 \defeq 1 \in \Z\Group$. The map $\psi_0 \colon F_0 \to \Z\Group$ that takes $f_i$ to $m_i$ is surjective and makes the right square of
\begin{diagram}
&& F_1 & \rTo & F_0 & \rTo^{\varepsilon'} & \Z & \rTo & 0\\
&& \dTo^{\psi_1} && \dTo^{\psi_0} && \dTo && & \\
0 & \rTo & I & \rTo & \Z\Group & \rTo^\varepsilon & \Z & \rTo & 0
\end{diagram}
commute. The map $\psi_1$ exists because $F_1$ is (free and therefore) projective (see \cite[Lemma~7.3]{brown82}). The five-lemma (\cite[Lemma~IV.5.10]{bredon}) implies that $\psi_1$ is surjective and hence $I$ is finitely generated. So it follows from our previous discussion that $\Group$ is finitely generated.

The third statement follows from Corollary~2.1 and the remark following Theorem~4 in \cite{wall66}. The fourth statement is proven in \cite{besbra97}.
\end{proof}

\footerlevel{3}
\headerlevel{3}

\section{Number Theory}
\label{sec:number_theory}

The aim of this section is to motivate and define the ring of $S$-integers of a set $S$ of places over a global function field. Polynomial rings and Laurent Polynomial rings are special cases which explains the relationship between our Main Theorem and the Rank Theorem. The proof of the Main Theorem does not depend on the contents of this section. The exposition does not follow any particular book, but most references are to \cite{weil74}. Other relevant books include \cite{artin67}, \cite{cassels}, and \cite{serre79}.

\subsection*{Valuations}

Let $\Field$ be a field. A \emph{valuation} (or \emph{absolute value}) on $\Field$ is a function $\Valuation \colon \Field \to \R$ such that
\begin{enumerate}[label=(VAL\arabic{*}), ref=VAL\arabic{*},leftmargin=*]
\item $\Valuation(a) \ge 0$ for all $a \in \Field$ with equality only for $a = 0$,\label{item:valuation_positive_definite}
\item $\Valuation(ab) = \Valuation(a) \cdot \Valuation(b)$ for all $a,b \in \Field$, and\label{item:valuation_morphism}
\item $\Valuation(a + b) \le \Valuation(a) + \Valuation(b)$ for all $a,b \in \Field$.\label{item:valuation_triangle_inequality}
\end{enumerate}
If it satisfies the stronger \emph{ultrametric inequality}
\begin{enumerate}[label=(VAL\arabic{*}'), ref=VAL\arabic{*}',leftmargin=*]
\setcounter{enumi}{2}
\item $\Valuation(a+b) \le \max \{\Valuation(a),\Valuation(b)\}$ for all $a,b \in \Field$,
\end{enumerate}
then it is said to be \emph{non-Archimedean}, otherwise \emph{Archimedean}.

The valuation with $\Valuation(0) = 0$ and $\Valuation(a) = 1$ for $a \ne 0$ is called the \emph{trivial valuation}.

Two valuations $\Valuation_1$ and $\Valuation_2$ are \emph{equivalent} if $\Valuation_1(a) \le 1$ if and only if $\Valuation_2(a) \le 1$ for every $a \in \Field$. If this is the case, then there is a constant $c > 0$ such that $\Valuation_1 = \Valuation_2^c$. The equivalence class $\PlaceOf\Valuation$ of a valuation $\Valuation$ is called a \emph{place}. Note that it makes sense to speak of a (non-)Archimedean place.

\begin{rem}
Usually only a weaker version of \eqref{item:valuation_triangle_inequality} is required. But since we are only interested in places, our definition suffices (see \cite[Theorem~3]{artin67}).
\end{rem}

\begin{exmpl}
\begin{enumerate}
\item The usual absolute value $\Valuation(a) \defeq \abs{a}$ is a valuation on $\Q$, it is Archimedean.
\item Let $p$ be a prime. Every $a \in \Q$ can be written in a unique way as $p^m(b/c)$ with $b$, $c$ integers not divisible by $p$, $c$ positive, and $m$ an integer. Setting $\Valuation_p(a) \defeq p^{-m}$ defines a valuation on $\Q$ that is non-Archimedean. It is called the \emph{$p$-adic valuation}.
\end{enumerate}
\end{exmpl}

Given a valuation $\Valuation$ on a field $\Field$, we can define a metric $d \colon \Field \times \Field \to \R$ by $d(a,b) = \Valuation(a - b)$ and have metric and topological concepts that come with it. In particular, $\Field$ may be \emph{complete} or \emph{locally compact with respect to $\Valuation$}. Note that neither the topology nor whether a sequence is a Cauchy-sequence depends on the particular valuation of a place so we may say that a field is for example locally compact or complete with respect to a \emph{place}.

The \emph{completion $\Field_\Valuation$ of $\Field$ with respect to $\Valuation$} is a field that is complete with respect to $\Valuation'$ and contains $\Field$ as a dense subfield such that $\Valuation'|_\Field = \Valuation$. The extension $\Valuation'$ of the valuation $\Valuation$ is usually also denoted $\Valuation$. Completions exist, are unique up to $\Field$-isomorphism, and can be constructed as $\R$ is constructed from $\Q$.

\begin{exmpl}
\label{exmpl:q_valuations}
The completion of $\Q$ with respect to the absolute value is $\R$. The completion of $\Q$ with respect to the $p$-adic valuation $\Valuation_p$ is $\Q_p$, the field of $p$-adic numbers.
\end{exmpl}

\subsection*{Discrete Valuations}

Unlike one might expect, a discrete valuation is not just a valuation that is discrete but rather it is the logarithm of a non-Archimedean valuation that is discrete. It is clear that this notion cannot produce anything essentially new compared to that of a valuation, but we mention it because it is commonly used in the algebraic theory of local fields (and rings).

A \emph{discrete valuation} on a field $\Field$ is a
\begin{enumerate}[label=(DVAL\arabic{*}), ref=DVAL\arabic{*},leftmargin=*]
\item homomorphism $\DiscreteValuation \colon \Units\Field \to \R$ that has discrete image, and
\item satisfies $\DiscreteValuation(a+b) \ge \min\{\DiscreteValuation(a),\DiscreteValuation(b)\}$.
\end{enumerate}

Two discrete valuations $\DiscreteValuation_1$ and $\DiscreteValuation_2$ are called \emph{equivalent} if $\DiscreteValuation_1 = c \cdot \DiscreteValuation_2$ for some $c \in \Units\R$.

One often makes the convention that $\DiscreteValuation(0) = \infty$. Note that if $\Valuation$ is a non-Archimedean valuation on $\Field$, then the map that takes $a$ to $-\log\Valuation(a)$ is a homomorphism. Its image is a subgroup of $\R$, thus either discrete or dense. In the first case it is a discrete valuation. Conversely if $\DiscreteValuation$ is a discrete valuation and $0 < r < 1$, then the map $a \mapsto r^{\DiscreteValuation(a)}$ is a non-Archimedean valuation. Both constructions are clearly inverse to each other up to equivalence. In particular, a discrete valuation $\DiscreteValuation$ defines a place and gives rise to a metric and we also denote the completion of $\Field$ with respect to this metric by $\Field_\DiscreteValuation$.

\begin{exmpl}
\label{exmpl:f_q_t_valuations}
\begin{enumerate}
\item Every $a \in \F_q(t)$ can be written in a unique way as $a = b/c$ with $b,c \in \F_q[t]$ and $c$ having leading coefficient $1$. Setting $\DiscreteValuation_{\infty}(a) \defeq \deg(c) - \deg(b)$ defines a discrete valuation on $\F_q(t)$. The completion of $\F_q(t)$ with respect to $\DiscreteValuation_{\infty}$ is $\F_q((t^{-1}))$.
\item Let $p \in \F_q[t]$ be irreducible. Every element in $a \in \F_q(t)$ can be written in a unique way as $p^m(b/c)$ with $m$ an integer, $b,c \in \F_q[t]$ such that the leading coefficient of $c$ is $1$ and $b$ and $c$ are not divisible by $p$. Setting $\DiscreteValuation_{p}(a) \defeq m$ defines a discrete valuation on $\F_q(t)$. The completion of $\F_q(t)$ with respect to $\DiscreteValuation_{p}$ is $\F_q((p(t)))$.
\end{enumerate}
\end{exmpl}

Let $\DiscreteValuation$ be a non-trivial discrete valuation on a field $\Field$. Since its image is infinite cyclic, $\DiscreteValuation$ can be considered as a surjective homomorphism $\Units\Field \to \Z$ (obscuring the distinction between equivalent discrete valuations). In what follows we adopt this point of view.

The topology defined by $\DiscreteValuation$ can be understood algebraically: The ring $\Ring \defeq \{a \in \Field \mid \DiscreteValuation(a) \ge 0\}$  is a \emph{discrete valuation ring}, i.e., an integral domain that has a unique maximal ideal and this ideal is non-zero and principal. Its maximal ideal is $\MaxIdeal \defeq \{a \in \Ring \mid \DiscreteValuation(a) \ge 1\}$. For $n \in \N$ the ideals $\MaxIdeal^n$ are open and closed in $\Ring$ and in fact they form a basis for the neighborhood filter of $0$. Since $\Ring$ is open and closed in $\Field$ this also describes the topology of $\Field$.

The completion of $\Ring$ is the inverse limit $\lim_{\leftarrow} \Ring/\MaxIdeal^n$ and the completion of $\Field$ is the field of fractions of the completion of $\Ring$ (see \cite[Section~7]{eisenbud94}).

The field $\Ring/\MaxIdeal$ is the \emph{residue field} of $\Field$ with respect to $\DiscreteValuation$.

\begin{rem}
The term ``discrete valuation ring'' reflects the following fact: Let $\Ring$ be a discrete valuation ring with maximal ideal $\MaxIdeal$ and field of fractions $\Field$. Let $\pi$ be an element that generates $\MaxIdeal$. For every $a \in \Units\Field$ there is a $u \in \Units\Ring$ and an $l \in \Z$ such that $a = u\pi^l$ (see \cite[Proposition~11.1]{eisenbud94}). The number $l$ is uniquely determined and the map $a \mapsto l$ is a discrete valuation.
\end{rem}

\subsection*{Local Fields}

A \emph{local field} is a non-discrete locally compact field.

Let $\LocalField$ be a local field and let $\Haar$ denote a Haar measure on $(\LocalField,+)$ (which is unimodular since the group is abelian). For $a \in \Units\LocalField$ the map $b \mapsto ab$ is an automorphism of $(\LocalField,+)$, so the measure $\Haar_a$ defined by $\Haar_a(\Set) = \Haar(a\Set)$ is again a Haar measure. By the uniqueness of the Haar measure, there is a constant $\HaarModule(a)$, called the \emph{module of $a$} such that $\Haar_a = \HaarModule(a) \Haar$. Setting $\HaarModule(0) = 0$ we obtain a map $\HaarModule \colon \LocalField \to \R$ which is easily seen to satisfy \eqref{item:valuation_positive_definite} and \eqref{item:valuation_morphism}. In fact it is a valuation (\cite[Theorem~I.3.4]{weil74}) and the topology on $\LocalField$ is the topology defined by $\HaarModule$ (\cite[Corollary~I.2.1]{weil74}). Thus:

\begin{prop}
A field is a local field if and only if it is equipped with a valuation with respect to which it is locally compact.
\end{prop}

So one can distinguish local fields by their valuations. In the Archimedean case we obtain:

\begin{thm}{{\cite[Theorem~I.3.5]{weil74}}}
\label{thm:archimedean_local_fields}
If $\LocalField$ is locally compact with respect to an Archimedean valuation $\Valuation$, then $\LocalField$ is isomorphic to either $\R$ or $\C$ and $\Valuation$ is equivalent to the usual absolute value.
\end{thm}

The non-Archimedean case offers more examples:

\begin{thm}{{\cite[Theorem~I.3.5, Theorem~I.4.8]{weil74}}}
\label{thm:non-archimedean_local_fields}
If $\LocalField$ is locally compact with respect to a non-Archimedean valuation, then either
\begin{enumerate}
\item $\LocalField$ is a completion of a finite extension of $\Q$ and isomorphic to a finite extension of some $\Q_p$, or
\item $\LocalField$ is a completion of a finite extension of $\F_q(t)$ and isomorphic (as a field) to $\F_{q^k}((t))$ for some $k$.
\end{enumerate}
\end{thm}

This suggests to introduce the following notion: a \emph{global field} is either a \emph{global number field}, that is, a finite extension of $\Q$, or a \emph{global function field}, that is, a finite extension of some $\F_q(t)$. Then Theorems~\ref{thm:archimedean_local_fields} and \ref{thm:non-archimedean_local_fields} can be restated to say that every local field is the completion of a global field with respect to some place. A partial converse is:

\begin{thm}[{{\cite[Theorem~I.3.5, Theorem~II.1.2]{weil74}}}]
\label{thm:classification_of_valuations}
Every non-trivial place of $\Q$ is one of those described in Example~\ref{exmpl:q_valuations}. Every non-trivial place of $\F_q(t)$ is one of those described in Example~\ref{exmpl:f_q_t_valuations}.
\end{thm}

Let $\GlobalField$ be one of $\Q$ and $\F_q(t)$. If $\GlobalField'$ is a finite extension of $\GlobalField$ and $\Valuation'$ is a valuation on $\GlobalField'$, then obviously $\Valuation \defeq \Valuation'|_{\GlobalField}$ is a valuation on $\GlobalField$. What $\Valuation'$ can look like if one knows $\Valuation$ can be understood by studying how $\GlobalField'$ embeds into the algebraic closure of $\GlobalField_\Valuation$, see \cite[Theorem~II.1.1]{weil74} and (for number fields) \cite[page~4]{plarap94}.

\subsection*{$S$-Integers}

Let $\GlobalField$ be a global field and let $S$ be a finite subset of the set of places of $\GlobalField$. If $\GlobalField$ is a number field, assume that $S$ contains all Archimedean places, if it is a function field, assume that $S$ is non-empty. The subring
\[
\Integers[S] \defeq \{a \in \GlobalField \mid \Valuation(a) \le 1 \text{ for all } \PlaceOf{\Valuation} \nin S\}
\]
is called the \emph{ring of $S$-integers of $\GlobalField$}. Informally one may think of it as the ring of elements of $\GlobalField$ that are integer except possibly at places in $S$. Indeed:

\begin{thm}[{{\cite[Theorem~17.6]{vanderwaerden91b}}}]
If $\GlobalField$ is a number field and $S$ is the set of Archimedean places, then $\Integers[S]$ is the ring of algebraic integers of $\GlobalField$.
\end{thm}

\begin{exmpl}
\begin{enumerate}
\item Let $\GlobalField = \Q$, let $\Valuation_\infty$ be the absolute value and let $\Valuation_p$ be the $p$-adic valuation for some prime $p$. If $S=\{\PlaceOf{\Valuation_\infty}, \PlaceOf{\Valuation_p}\}$, then $\Integers[S] = \Z[1/p]$.
\item Let $\GlobalField = \F_q(t)$, let $\Valuation_\infty$ be the valuation at infinity and for $a \in \F_q$ let $\Valuation_a$ be the valuation corresponding to the irreducible polynomial $t-a$. If $S = \{\PlaceOf{\Valuation_\infty}\}$, then $\Integers[S] = \F_q[t]$. If $S = \{\PlaceOf{\Valuation_a}\}$, then $\Integers[S] = \F_q[(t-a)^{-1}]$, in particular if $a=0$, then $\Integers[S] = \F_q[t^{-1}]$. Finally, if $S = \{\PlaceOf{\Valuation_\infty},\PlaceOf{\Valuation_0}\}$, then $\Integers[S] = \F_q[t,t^{-1}]$.
\end{enumerate}
\end{exmpl}

\footerlevel{3}
\headerlevel{3}

\section{Affine Varieties and Linear Algebraic Groups}
\label{sec:algebraic_groups}

In this section we try to introduce linear algebraic groups with as little theory as possible. In particular, we only consider subvarieties of affine space without giving an intrinsic definition. There are three standard books on linear algebraic groups, \cite{borel91}, \cite{humphreys81}, \cite{springer98}, which are recommended to the reader who looks for a proper introduction. In \cite{demgab} algebraic groups are developed in the terms of schemes which is the appropriate language for rationality questions and algebraic groups over rings (see also \cite{sga3_1,sga3_2,sga3_3}).

\subsection*{Affine Varieties}

Let $\Field$ be a field and let $\ClosedField$ be an algebraically closed field that contains it. \emph{Affine $n$-space} over $\ClosedField$ is defined to be $\A^n\ClosedField \defeq \ClosedField^n$. Let $\Ring \defeq \ClosedField[t_1, \ldots, t_n]$ be the ring of polynomials in $n$ variables over $\ClosedField$ and let $\Ring_\Field \defeq \Field[t_1,\ldots,t_n]$ be the ring of polynomials in $n$ variables over $\Field$. For a subset $M \subseteq \Ring$, we define the set
\[
V(M) \defeq \{(x_1,\ldots,x_n) \mid f(x_1,\ldots,x_n) = 0 \text{ for all }f \in M\} \text{ .}
\]
Clearly if $I$ is the ideal generated by $M$, then $V(I) = V(M)$.

We see at once that $V(0) = \A^n\ClosedField$ and $V(\Ring) = \emptyset$. If $I_1$ and $I_2$ are two ideals of $\Ring$, then $V(I_1 \intersect I_2) = V(I_1) \union V(I_2)$ and if $(I_i)_i$ is a family of ideals, then $V(\sum_i I_i) = \Intersect_i V(I_i)$. This shows that the sets of the form $V(I)$ are the closed sets of a topology on $\A^n\ClosedField$, called the \emph{Zariski topology}.

If $\Space$ is a closed subset of $\A^n\ClosedField$, we denote by $J(\Space)$ the ideal of polynomials in $\Ring$ vanishing on $\Space$ and by $J_\Field(\Space)$ the ideal of polynomials in $\Ring_\Field$ vanishing on $\Space$. We call $\Ring[\Space] \defeq \Ring/J(\Space)$ the \emph{affine algebra} of $\Space$ and analogously define $\Ring_\Field[\Space] \defeq \Ring_\Field/J_\Field(\Space)$.

If in the definition of the Zariski-topology above, we replace $\Ring$ by $\Ring_\Field$, we obtain a coarser topology, called the \emph{$\Field$-Zariski topology}. Subsets that are closed respectively open with respect to this topology are called \emph{$k$-closed} respectively \emph{$k$-open}. A $\Field$-closed subset $\Space$ is said to be \emph{defined over $\Field$} if the homomorphism $\ClosedField \otimes_\Field \Ring_\Field[\Space] \to \Ring[\Space]$ is an isomorphism. This is always the case if $\Field$ is perfect (in particular, if $\Field$ is finite or of characteristic $0$).

If $\Space \subseteq \A^n\ClosedField$ is a closed subset, we can equip it with the topology induced by the Zariski topology which we also call Zariski topology. If $\Space$ is $\Field$-closed, we may similarly define the $\Field$-Zariski topology on $\Space$ and accordingly say that a subset of $\Space$ is $\Field$-closed or $\Field$-open.

A closed subset of $\A^n\ClosedField$ is called an \emph{affine variety}. It is said to be \emph{irreducible} if it is not empty and is not the union of two distinct proper non-empty closed subsets.

If $\Space$ is a closed subset of $\A^n\ClosedField$ and $\AltSpace$ is a closed subset of $\A^m\ClosedField$, then $\Space \times \AltSpace$ is a closed subset of $\A^{n+m}\ClosedField$. Moreover, if $\Space$ and $\AltSpace$ are irreducible, then so is $\Space \times \AltSpace$.

The elements of the affine algebra $\Ring[\Space]$ of an affine variety $\Space$ can be regarded as $\ClosedField$-valued functions on $\Space$. These functions are called \emph{regular}. Let $\Space$ and $\AltSpace$ be affine varieties. A map $\alpha \colon \Space \to \AltSpace$ is a \emph{morphism} if its components are regular functions, that is, $\alpha(\Point_1,\ldots,\Point_n) = (f_1(\Point_1,\ldots,\Point_n),\ldots,f_m(\Point_1,\ldots,\Point_n))$ with $f_1,\ldots,f_m \in \Ring[\Space]$. The morphism is said to be \emph{defined over $\Field$} or to be a \emph{$\Field$-morphism}, if $f_1,\ldots,f_m \in \Ring_\Field[\Space]$. A ($\Field-$)morphism is a \emph{($\Field$-)iso\-mor\-phism} if there is a ($\Field$-)morphism that is its inverse.

If $\Integers$ is a subring of $\ClosedField$, then $\Integers^n$ is an $\Integers$-submodule of $\A^n\ClosedField$. If $\Space \subseteq \A^n\ClosedField$ is an affine variety, then we denote by $\Space(\Integers)$ the intersection $\Space \intersect \Integers^n$.

\subsection*{Linear Algebraic Groups}

\begin{exmpl}
The set
\[
\GL_n \defeq
 \left\{
\left(
\begin{array}{ccc}
x_{1,1}& \cdots & x_{1,n+1} \\
\vdots & \ddots & \vdots \\
x_{n+1,1} & \cdots & x_{n+1,n+1}
\end{array}
\right)
\in \A^{(n+1)^2}\ClosedField
\mathrel{\bigg|}
\begin{array}{c}
\det (x_{i,j}) = 1,\\
x_{i,n+1} = 0, 1 \le i \le n\\
x_{n+1,j} = 0, 1 \le j \le n
\end{array}
\right\}
\]
is an affine variety defined over $\Field$. In addition it is a group isomorphic to $\GL_n(\ClosedField)$ with multiplication being matrix multiplication.
\end{exmpl}

For our purpose, a linear algebraic group is a closed subgroup of $\GL_n$. If it is defined over $\Field$, we also say briefly that it is a $\Field$-group. A morphism of linear algebraic groups is a map that is at the same time a homomorphism and a morphism of affine varieties. An isomorphism of linear algebraic groups is a map that is an isomorphism of groups as well as of affine varieties. An (iso-)morphism of $\Field$-groups is a morphism of linear algebraic groups that is defined over $\Field$ (and whose inverse exists and is defined over $\Field$).

Let $\GroupScheme$ be a linear algebraic group. A subgroup $\Torus$ of $\GroupScheme$ is a \emph{torus} if it is isomorphic to $\GL_1 \times \cdots \times \GL_1$. The number of factors is the \emph{rank} of $\Torus$. The torus $\Torus$ is $\Field$-split, if it is defined over $\Field$ and $\Field$-isomorphic to $\GL_1 \times \cdots \times \GL_1$, where the number of factors is the rank of $\Torus$.

The \emph{rank} of $\GroupScheme$ is the rank of a maximal torus that it contains. The \emph{$\Field$-rank} of $\GroupScheme$ is the rank of a maximal $\Field$-split torus that it contains. If the ($\Field$-)rank of $\GroupScheme$ is $0$, then $\GroupScheme$ is said to be ($\Field$-)\emph{anisotropic}, otherwise ($\Field$-)\emph{isotropic}. If $\GroupScheme$ contains no non-trivial proper connected closed normal subgroup then it is said to be \emph{almost simple}.

\subsection*{$S$-Arithmetic Groups}

Let $\Field$ be a global field and $\ClosedField$ its algebraic closure. Let $\GroupScheme \le \GL_n$ be a linear algebraic group. Let $S$ be a set of places of $\Field$ that contains all Archimedean places if $\Field$ is a number field and is non-empty if $\Field$ is a function field, and let $\Integers[S]$ denote the ring of $S$-integers of $\Field$. A group of the form $\GroupScheme(\Integers[S])$ is called an \emph{$S$-arithmetic group}.

\footerlevel{3}
\headerlevel{3}

\section{Buildings}
\label{sec:buildings}

The possible points of view on buildings are quite various. They can be regarded combinatorially as edge-colored graphs or geometrically as metric spaces. The concept that mediates between the two is the building as a simplicial complex.

Buildings were developed by Jacques Tits, who wrote \cite{tits74} on spherical buildings. The standard reference on buildings today is \cite{abrbro}, it develops the different definitions of buildings and also the theory of twin buildings. Its predecessor \cite{brown89} is a beautiful introduction to buildings as simplicial complexes and is probably the best book with which to start learning the topic (also it is available online). The books \cite{weiss04} and \cite{weiss09} develop the theory of spherical respectively affine buildings in terms of edge-colored graphs and, in particular, contain (together with \cite{titwei02}) a revision of the classification of buildings of these types. The same language is used in \cite{ronan}. For twin buildings \cite{abramenko96} has long been the standard reference.

We consider buildings as cell complexes that are equipped with a metric, to be more precise as $\ModelSpace{\kappa}$-polyhedral complexes in the terminology of Section~\ref{sec:metric_spaces}. Our exposition is motivated by \cite{klelee97} but changed so as to keep the terminology and results in \cite{abrbro} within reach.

\subsection*{Spherical Coxeter Complexes}

We start by introducing spherical Coxeter complexes, see \cite[Section~1]{abrbro}.

Let $\SApartment \defeq \S^n$ be a sphere. A \emph{reflection} of $\SApartment$ is an involutory isometry that fixes a hyperplane. A finite subgroup $\Weyl$ of $\Isom\SApartment$ that is generated by reflections is a (\emph{finite} or) \emph{spherical reflection group}. A hyperplane $\Wall$ that is the fixed point set of some $\weyl \in \Weyl$ is called a \emph{wall}. The closure of a connected component of the complement of all walls is a polyhedron that is called a \emph{chamber}, its facets are \emph{panels}. Every closed hemisphere defined by a wall is a \emph{root}. Two points or cells of $\SApartment$ are called \emph{opposite} if they are mapped onto each other by the antipodal map.

The action of $\Weyl$ on $\SApartment$ is simply transitive on chambers, see \cite[Theorem~1.69]{abrbro}. Therefore the restriction of the projection $\SApartment \to \Weyl \LeftMod \SApartment$ to chambers is an isometry. We call $\ModelChamber \defeq \Weyl \LeftMod \SApartment$ the \emph{model chamber} of $\SApartment$.

The chambers induce a cell structure on $\SApartment$ so that it becomes an $\ModelSpace{1}$-polyhedral complex. We call $\SApartment$ equipped with this structure a \emph{spherical Coxeter complex}. Combinatorially the Coxeter complex is a simplicial complex, that is, its face lattice is that of an abstract simplicial complex. To be more precise, $\ModelChamber$ is clearly a polyhedron whose facets include angles at most $\pi/2$. Thus it decomposes as in Theorem~\ref{thm:non_obtuse_angle_decomposition} as a join of a sphere and a simplex with non-obtuse angles. The simplex decomposes further into irreducible simplices. By \cite[Section~3.3]{klelee97} this decomposition induces a decomposition of $\SApartment$. So $\SApartment$ decomposes as a join of a sphere and a spherical Coxeter complex whose cells are simplices which have non-obtuse angles (in particular, diameter $\le \pi/2$). The simplicial complex is called the \emph{essential part} and $\SApartment$ \emph{essential} if it equals its essential part. The essential part decomposes further as a join of \emph{irreducible} spherical Coxeter complexes, that is, of Coxeter complexes whose cells have diameter $< \pi/2$.

From now on all spherical Coxeter complexes are assumed to be essential.

There is a structure theory (including classification) of reflection groups that puts the following into a broader context. We refer the reader to \cite{bourbaki_lie_4-6, groben85, humphreys90}.

Let $\Types$ be the set of vertices of $\ModelChamber$. For a cell $\Cell$ of $\SApartment$ we define $\typ \Cell$ to be the image of the vertex set of $\Cell$ under the projection $\SApartment \to \ModelChamber$ and call it the \emph{type of $\Cell$}. The \emph{cotype} of cell is the complement of $\typ \Cell$ in $\Types$. Given two walls $\Wall_1,\Wall_2$, the fact that the group generated by the reflections at these walls is finite implies that the angle between them can only be $\pi/n$ with $n \ge 2$ an integer. The \emph{Coxeter diagram} $\typ \SApartment$ of $\SApartment$ is a graph whose vertex set is $\Types$, and where there is an edge between $\Type$ and $\AltType$ if the complements of $i$ and $j$ in $\ModelChamber$ are not perpendicular. In that case they include an angle of $\pi/n$ for $n \ge 3$ and the edge is labelled by $n$. By Observation~\ref{obs:spherical_join_decomposition} the irreducible join factors of $\SApartment$ correspond to connected components of $\typ \SApartment$; more explicitly: if $\AltTypes \subseteq \Types$ is the vertex set of a connected component of $\typ \SApartment$, then the cells $\Cell$ with $\typ \Cell \subseteq \AltTypes$ form an irreducible join factor of $\SApartment$.

\begin{figure}[p]
\begin{center}
\includegraphics{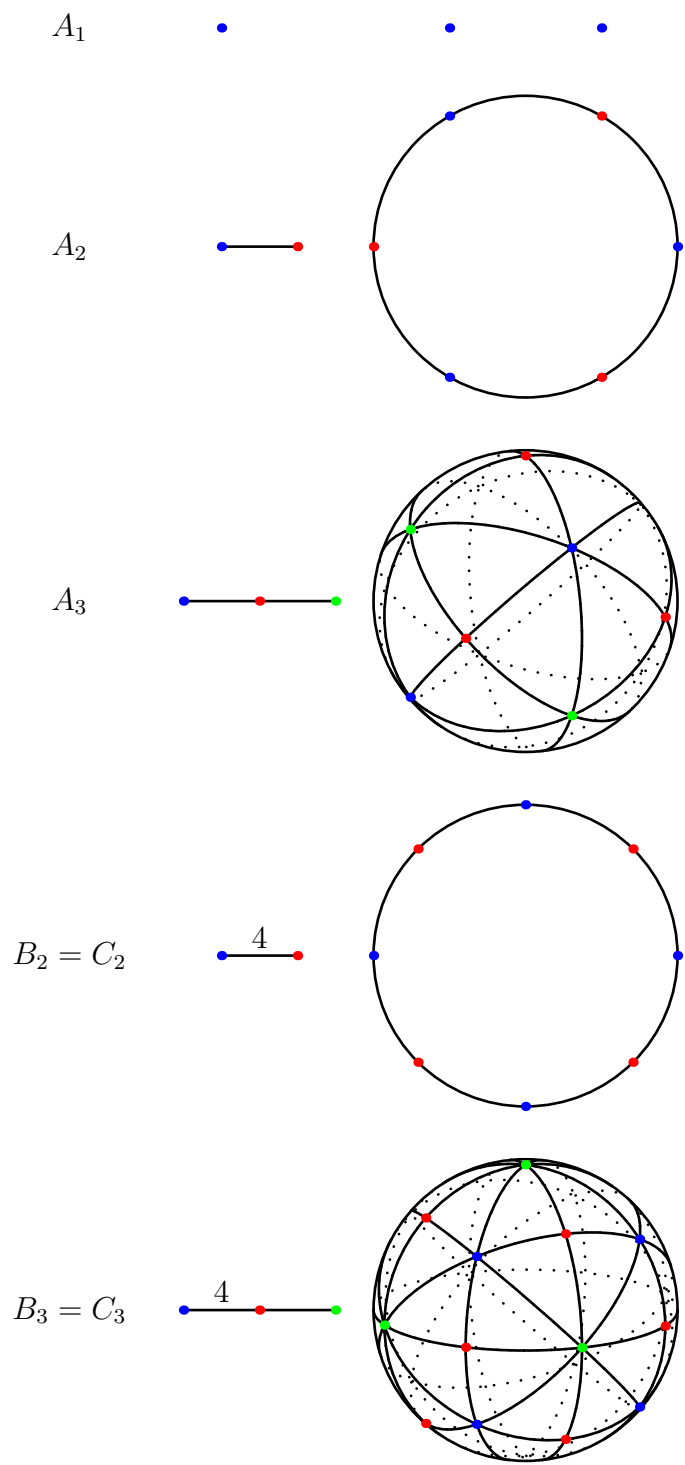}
\end{center}
\caption{Some spherical Coxeter complexes. To the left of each complex is its name and its diagram.}
\label{fig:spherical_coxeter_complexes}
\end{figure}

Fix a chamber $\Chamber_0 \subseteq \SApartment$. Let $S$ be the set of reflections at walls that bound $\Chamber_0$. Note that every $s \in S$ corresponds to a panel of $\Chamber_0$ and any two of these panels have different cotype. Let $\WeylDistance(\Chamber_0,\Chamber)$ be the element of $\Weyl$ that takes $\Chamber_0$ to $\Chamber$. Using simple transitivity, this can be extended to define a \emph{Weyl-distance}: if $\AltChamber_1$ and $\AltChamber_2$ are arbitrary chambers, we can write $\AltChamber_1 = \weyl' \Chamber_0$ and $\AltChamber_2 = \weyl'\weyl \Chamber_0$. Then $\WeylDistance(\AltChamber_1,\AltChamber_2) = \weyl$. Every element $\weyl \in \Weyl$ can be assigned a \emph{length}, namely the number of walls that separate $\Chamber_0$ from $\weyl \Chamber_0$. Replacing $\Chamber_0$ by a different chamber corresponds to conjugating the Weyl-distance by an element of $\Weyl$. This conjugation takes $S$ to a set $S'$ in a type-preserving way. So if we regard the pair $(\Weyl,S)$ as an abstract Coxeter system and identify $\Types$ with $S$, we get a Weyl-distance that is independent of a fixed chamber.

\subsection*{Euclidean Coxeter Complexes}

Now we turn to Euclidean Coxeter complexes, see \cite[Section~10]{abrbro}.

Let $\Apartment \defeq \E^n$ be a Euclidean space. A reflection of $\Apartment$ is an involutory isometry that fixes a hyperplane. A subgroup $\EWeyl$ of $\Isom \Apartment$ that is discrete, generated by reflections, and has no proper invariant subspace (in particular, no fixed point) is called a \emph{Euclidean reflection group} (cf.\ \cite{coxeter35}, \cite[Theorem~10.9]{abrbro}). A hyperplane that is the fixed point set of some $\weyl \in \EWeyl$ is called a \emph{wall}. As before the walls define a cell structure on $\Apartment$ and we call $\Apartment$ with this cell structure a \emph{Euclidean Coxeter complex}. The maximal cells are called \emph{chambers}, the codimension-$1$-cells \emph{panels}. A \emph{root} is a closed halfspace defined by some wall.

A Euclidean Coxeter complex is called \emph{irreducible} if it is a simplicial complex and an arbitrary Euclidean Coxeter complex decomposes as a direct product of its irreducible Coxeter subcomplexes.

The action of $\EWeyl$ is simply transitive on chambers. To define the type of cells let us first assume that $\Apartment$ is irreducible. The restriction of the projection $\Apartment \to \ModelChamber \defeq \EWeyl \LeftMod \Apartment$ to chambers is an isometry. This allows us as before to assign a \emph{type} $\typ{\Cell} \subseteq \Types$ to every cell $\Cell$ of $\Apartment$ where $\Types$ is the set of vertices of $\ModelChamber$. The walls of $\Apartment$ still include angles $\pi/n$ with $n \ge 3$ (\cite[Lemma~4.2]{coxeter35}) but now they can in addition be parallel. The Coxeter diagram $\typ \Apartment$ is defined to be the graph with vertices $\Types$ where there is an edge from $\Type$ to $\AltType$ if $\conv(\Types \setminus \Type)$ and $\conv(\Types \setminus \AltType)$ are either disjoint or meet in an angle $< \pi/2$. If the angle is $\pi/n$ in the latter case, the edge is labeled by $n$, and in the former case it is labeled by $\infty$.

\begin{figure}[!ht]
\begin{center}
\includegraphics{figs/affine_coxeter_complexes\screenornot}
\end{center}
\caption{Some (excerpts of) affine Coxeter complexes. To the left of each complex is its name and its diagram.}
\label{fig:affine_coxeter_complexes}
\end{figure}

If $\Apartment$ is not irreducible, the Coxeter diagram $\typ \Apartment$ is the disjoint union of the Coxeter diagrams of the individual factors and the type of a cell $\Cell_1 \times \cdots \times \Cell_n$ is the union $\typ \Cell_1 \union \cdots \union \typ \Cell_n$.

The action of $\EWeyl$ on $\Apartment$ induces an action on the visual boundary $\Infty\Apartment$ which is a sphere. We denote the image of $\EWeyl$ in $\Isom \Infty\Apartment$ by $\InftyEWeyl$. The group $\InftyEWeyl$ is a finite reflection group that turns $\Infty\Apartment$ into a spherical Coxeter complex. Let $\Vertex$ be a vertex of $\Apartment$ and let $\EWeyl_\Vertex$ be its stabilizer in $\EWeyl$. Then $\EWeyl_\Vertex$ acts on $\Infty\Apartment$ as a subgroup of $\InftyEWeyl$ and we call $\Vertex$ \emph{special} if it acts as all of $\InftyEWeyl$. In that case, since $\EWeyl_\Vertex$ acts simply transitively on chambers that contain $\Vertex$ and acts simply transitively on chambers of $\Infty\Apartment$, the link of $\Vertex$ is isomorphic to $\Infty\Apartment$. A \emph{sector} is the convex hull of a special vertex $\Vertex$ and a chamber $\InftyChamber$ of $\Infty\Apartment$, i.e., the union of geodesic rays $[\Vertex,\InftyPoint)$ with $\InftyPoint \in \InftyChamber$.

Let $\Chamber_0$ be a chamber and let $S$ be the set of reflections at walls bounding $\Chamber_0$. Considering $(\EWeyl,S)$ as an abstract Coxeter system, we obtain as before \emph{Weyl-distance} $\WeylDistance$ for $\Apartment$ with values in $\EWeyl$. The elements $\weyl \in \EWeyl$ have a \emph{length} which is, as before, defined to be the number of walls that separate $\Chamber_0$ from $\weyl \Chamber_0$.

\subsection*{Hyperbolic Coxeter Complexes}

We do not introduce hyperbolic Coxeter complexes. We just mention briefly that hyperbolic reflection groups may contain parabolic isometries which lead to the existence of ``vertices at infinity''. To obtain a Coxeter complex with compact cells one therefore has to cut out a horoball around each of these vertices. This leads to the so-called Davis-realization, see \cite[Chapter~7]{davis08}.

\subsection*{Buildings}

We now define spherical and Euclidean buildings based on our previous definition of spherical and Euclidean Coxeter complexes, cf.\ \cite[Section~4]{abrbro}.

A \emph{building} $\Building$ is an $\ModelSpace{\kappa}$-polytopal complex that can be covered by subcomplexes $\Apartment \in \ApartmentSystem$, called \emph{apartments}, subject to the following conditions:
\begin{enumerate}[label=(B\arabic{*}), ref=B\arabic{*},leftmargin=*]
\setcounter{enumi}{-1}
\item All apartments are Coxeter complexes. \label{item:apartments_are_coxeter}
\item For any two points of $\Building$ there is an apartment that contains them both.\label{item:plenty_of_apartments}
\item Whenever two apartments $\SApartment_1$ and $\SApartment_2$ contain a common chamber, there is an isometry $\SApartment_1 \to \SApartment_2$ that takes cells onto cells of the same type and restricts to the identity on $\SApartment_1 \intersect \SApartment_2$.\label{item:compatible_apartments}
\end{enumerate}

A set $\ApartmentSystem$ of apartments, i.e., of subcomplexes for which \eqref{item:apartments_are_coxeter} to \eqref{item:compatible_apartments} are satisfied, is called an \emph{apartment system} for $\Building$. The axioms imply that all apartments are of the same type.

A building is \emph{spherical} if its apartments are spherical Coxeter complexes (so that the building is an $\ModelSpace{1}$-polytopal complex) and is \emph{Euclidean} if its apartments are Euclidean Coxeter complexes (so that the building is an $\ModelSpace{0}$-polytopal complex). We usually denote spherical buildings by $\SBuilding$ and Euclidean buildings by $\EBuilding$.

Let $\SBuilding$ be a spherical building. Two points or cells are \emph{opposite} in $\SBuilding$ if there is an apartment that contains them and in which they are opposite. The apartment that contains two given opposite chambers is unique. This can be used to show that spherical buildings have a unique apartment system.

In general, a union of apartment systems is again an apartment system (see \cite[Theorem~4.54]{abrbro}), so there is a maximal apartment system, called the \emph{complete system of apartments}. It is characterized by the fact that it contains every subcomplex that is isomorphic to an apartment. If we talk about apartments without specifying the apartment system, we mean the complete system of apartments.

\emph{Chambers}, \emph{panels}, \emph{walls}, \emph{roots} of a building are chambers, panels, walls, roots of any of its apartments. If $\EBuilding$ is a Euclidean building, then a vertex is \emph{special} if it is a special vertex of an apartment that contains it and a \emph{sector} of $\EBuilding$ is a sector of one of its apartments. Note that if $\EBuilding$ is not spherical, then the notions of walls, roots and sectors depend on the apartment system.

For every apartment $\Apartment$ there is a quotient map onto the model chamber $\ModelChamber$. By \eqref{item:compatible_apartments} these fit together to define a projection $\ModelProjection \colon \SBuilding \to \ModelChamber$. In particular, every cell $\Cell$ of $\Building$ can be given a well defined type $\typ \Cell$ and the building has a Coxeter diagram $\typ \Building$.

A building is \emph{thick} if every panel is contained in at least three chambers. A building is \emph{thin} if every panel is contained in precisely two chambers, i.e., it is a Coxeter complex. A building is \emph{irreducible} if its apartments are irreducible.

Throughout, actions on buildings are assumed to be type preserving, i.e., the induced action on $\ModelChamber$ is trivial. The action of a group on a building is said to be \emph{strongly transitive} if it is transitive on pairs $(\Chamber,\Apartment)$ where $\Chamber$ is a chamber of an apartment $\Apartment$. For spherical buildings, this is the same as to say that the action is transitive on pairs of opposite chambers.

\begin{fact}
Every spherical building $\SBuilding$ decomposes as a spherical join $\SBuilding = \SBuilding_1 * \cdots * \SBuilding_n$ of irreducible spherical buildings $\SBuilding_i$. Every Euclidean building $\EBuilding$ decomposes as a direct product $\EBuilding = \EBuilding_1 \times \cdots \times \EBuilding_n$ of irreducible Euclidean buildings $\EBuilding_i$. In both cases the irreducible factors are the subcomplexes of the form $\typ^{-1} \CoxeterDiagram$ where $\CoxeterDiagram$ is a connected component of the Coxeter diagram.
\end{fact}

In the case of a spherical building $\SBuilding$ this will be important later, so we make the statement a bit more explicit. Note first that a join decomposition of $\SBuilding$ gives rise to a join decomposition of $\ModelChamber$. The converse is also true:

\begin{prop}[{\cite[Proposition~3.3.1]{klelee97}}]
\label{prop:chamber_decomposition_induces_building_decomposition}
Let $\SBuilding$ be a spherical building and let $\ModelProjection \colon \SBuilding \to \ModelChamber$ be the projection onto the model chamber. If the model chamber decomposes as $\ModelChamber = \Chamber_1 * \cdots * \Chamber_n$ then the building decomposes as $\SBuilding = \SBuilding_1 * \cdots * \SBuilding_n$ where $\SBuilding_i = \ModelProjection^{-1}(\Chamber_i)$.
\end{prop}

Together with Observation~\ref{obs:spherical_join_decomposition} and Observation~\ref{obs:acute_edge_equivalence_relation} this gives us two ways to determine whether two adjacent vertices lie in the same join factor:

\begin{obs}
\label{obs:vertices_not_in_same_join_factor}
Let $\SBuilding$ be a spherical building. Two adjacent vertices $\Vertex$ and $\AltVertex$ lie in the same irreducible join factor of $\SBuilding$ if the following equivalent conditions are satisfied:
\begin{enumerate}
\item there is an edge path in $\typ \SBuilding$ that connects $\typ \Vertex$ to $\typ \AltVertex$.
\item $d(\Vertex,\AltVertex)<\pi/2$.\qed
\end{enumerate}
\end{obs}

Let $\Building$ be either a spherical or a Euclidean building. In Section~\ref{sec:metric_spaces} we described a natural cell structure on the link $\Link \Cell$ of a cell $\Cell$ consisting of $\BigCell \direction \Cell$ where $\BigCell$ is a coface of $\Cell$. Moreover, for every apartment $\Apartment$ of $\Building$ the subspace $\Link_\Apartment \Cell$ of directions that point into $\Apartment$ form a subcomplex. The following is fundamental, cf.\ Proposition~4.9 in \cite{abrbro}:

\begin{fact}
\label{fact:links_are_spherical_buildings}
Let $\Building$ be a spherical or Euclidean building and let $\Cell \subseteq \Building$ be a cell. Then $\Link \Cell$ is a spherical building with apartment system $\Link_\Apartment \Cell$ where $\Apartment$ ranges over the apartments that contain $\Cell$. Its Coxeter diagram is obtained from $\typ \Building$ by removing $\typ \Cell$.
\end{fact}

Since spherical buildings are flag complexes, it follows from \cite[Theorem~II.5.4]{brihae} that Euclidean buildings are \CAT{0}-spaces and spherical buildings are \CAT{1}-spaces.

A statement similar to Fact~\ref{fact:links_are_spherical_buildings} holds for the asymptotic structure of Euclidean buildings. To describe it, we need a further notion. Let $\EBuilding$ be a Euclidean building. An apartment system $\ApartmentSystem$ of $\EBuilding$ is called a \emph{system of apartments} if given any two sectors $\Sector_1$ and $\Sector_2$ there exist subsectors $\Sector_i'$ of $\Sector_i$ and an apartment $\Apartment$ that contains $\Sector_1'$ and $\Sector_2'$. Note that asymptotically this implies that $\Infty\Sector_i = \Infty{\Sector_i'}$ and that $\Infty\Apartment$ contains $\Infty{\Sector_i'}$. Thus $\Union_{\Apartment \in \ApartmentSystem} \Infty\Apartment \subseteq \Infty\EBuilding$ is covered by the spherical Coxeter complexes $\Infty\Apartment, \Apartment \in \ApartmentSystem$. The Coxeter complexes allow to define a cell structure on $\Union_{\Apartment \in \ApartmentSystem} \Infty\Apartment$ which turns out to be a spherical building. We only state this for the complete system of apartments which covers all of $\Infty\EBuilding$ (see Theorem~11.79 in \cite{abrbro}):

\begin{fact}
The visual boundary of a Euclidean building is a spherical building. More precisely if $\EBuilding$ is a Euclidean building equipped with the complete system of apartments, then $\Infty\EBuilding$ is a spherical building whose chambers are the visual boundaries of sectors and whose apartment system consists of visual boundaries of apartments.
\end{fact}

We collect some general facts about buildings:

\begin{fact}
\label{fact:building}
Let $\Building$ be a spherical or Euclidean building.
\begin{enumerate}
\item The Weyl-distances on the apartments fit together to define a well-defined Weyl-distance $\WeylDistance$ on the chambers of $\Building$. That is, if $\WeylDistance_\Apartment$ denotes the Weyl-distance on an apartment $\Apartment$, then $\WeylDistance_\Apartment(\Chamber,\AltChamber)$ is the same for every apartment $\Apartment$ that contains $\Chamber$ and $\AltChamber$.\label{item:weyl_distance}
\item If $\Chamber$ is a chamber of $\Building$ and $\Cell$ is an arbitrary cell, then there is a unique chamber $\AltChamber \ge \Cell$ such that $\WeylDistance(\Chamber,\AltChamber)$ has minimal length. This element is called the \emph{projection of $\Chamber$ to $\Cell$} and denoted $\WeylProjection_\Cell \Chamber$. It has the property that every apartment that contains $\Chamber$ and $\Cell$ also contains $\AltChamber$. If $\AltCell$ is an arbitrary cell, then the projection of $\AltCell$ to $\Cell$ is $\WeylProjection_\Cell \AltCell \defeq \Intersect_{\Chamber \ge \AltCell} \WeylProjection_\Cell \Chamber$.
\item If $\Apartment$ is an apartment of $\Building$ and $\Chamber$ is a chamber of $\Apartment$, the \emph{retraction onto $\Apartment$ centered at $\Chamber$}, denoted $\Retraction{\Apartment}{\Chamber}$, is the map that (isometrically and in a type preserving way) takes a chamber $\AltChamber$ to the chamber $\AltChamber'$ of $\Apartment$ that is characterized by $\WeylDistance(\Chamber,\AltChamber) = \WeylDistance(\Chamber,\AltChamber')$. It is an isometry on apartments that contain $\Chamber$ and is generally contracting (both, in the usual sense and in terms of Weyl-distance).
\end{enumerate}
\end{fact}

The existence of the objects is shown in Proposition~4.81, Proposition~4.95, and Proposition~4.39 of \cite{abrbro} respectively.

In the remainder of this paragraph we will be concerned with the relation between the asymptotic and the local structure of Euclidean buildings. The results are either general facts about \CAT{0}-spaces or can be reduced to the study of a single apartment using:

\begin{lem}
\label{lem:ray_in_apartment}
Let $\EBuilding$ be a Euclidean building and let $\Ray$ be a geodesic ray in $\EBuilding$. There is an apartment in the complete system of apartments of $\EBuilding$ that contains $\Ray$.
\end{lem}

\begin{proof}
Let $\Infty\Root$ be a root of $\Infty\EBuilding$ that contains $\Infty\Ray$. There is a corresponding root $\Root$ of $\EBuilding$ that contains a subray of $\Ray$. Moving the wall that bounds $\Root$ backwards along $\Ray$ produces a subcomplex of $\EBuilding$ that is itself isomorphic to a root and thus a root in the sense of \cite[Definition~5.80]{abrbro}. It is therefore contained in an apartment by \cite[Proposition~5.81~(2)]{abrbro}, which makes it again a root in our sense. Iterating this procedure one obtains a root that fully contains $\Ray$.
\end{proof}

\begin{obs}
\label{obs:euclidean_building_decomposition}
Let $\EBuilding$ be a Euclidean building. A decomposition as a direct product $\EBuilding = \EBuilding_1 \times \cdots \times \EBuilding_n$ induces
\begin{enumerate}
\item for every point $\Point \in \EBuilding$ a decomposition $\Link \Point = \Link_{\EBuilding_1} \Point * \cdots \Link_{\EBuilding_n} \Point$.
\item for every cell $\Cell = \Cell_1 \times \cdots \times \Cell_n$ a decomposition $\Link \Cell = \Link_{\EBuilding_1} \Cell_1 * \cdots * \Link_{\EBuilding_n} \Cell_n$.
\item a decomposition $\Infty\EBuilding = \Infty\EBuilding_1 * \cdots * \Infty\EBuilding_n$.\qed
\end{enumerate}
\end{obs}

Let $\Point \in \EBuilding$ be a point. By Proposition~\ref{prop:ray_from_point_to_point_at_infty} there is a projection from the building at infinity onto $\Link \Point$. Namely if $\PointAtInfty$ is a point of $\Infty\EBuilding$, there is a unique geodesic $\Ray$ that issues at $\Point$ and tends to $\PointAtInfty$. The direction $\Ray_\Point$ defined by this ray will also be called the direction defined by $\PointAtInfty$ and denoted $\PointAtInfty_\Point$.

\begin{obs}
Let $\EBuilding$ be a Euclidean building and $\Point \in \EBuilding$. The projection $\Infty\EBuilding \to \Link \Point$ that takes $\PointAtInfty$ to $\PointAtInfty_\Point$ maps cells to (but generally not onto) cells.
\end{obs}

\begin{proof}
Using Lemma~\ref{lem:ray_in_apartment} we may consider an apartment $\Apartment$ that contains $\Point$ and $\PointAtInfty$. So what remains to be seen is that the cell structure of $\Apartment$ is at least as fine as that of $\Link_{\Apartment} \Point$ but that is clear from the definition.
\end{proof}

This projection is compatible with the join decompositions in Observation~\ref{obs:euclidean_building_decomposition}. In particular:

\begin{obs}
\label{obs:asymptotic_to_local_join_compatible}
Let $\EBuilding = \EBuilding_1 \times \cdots \times \EBuilding_n$ be a Euclidean building and let $\Point \in \EBuilding$. A point at infinity $\PointAtInfty \in \Infty\EBuilding$ has distance $< \pi/2$ to $\Infty\EBuilding_i$ if and only if the direction $\PointAtInfty_\Point$ it defines at $\Point$ has distance $< \pi/2$ to $\Link_{\EBuilding_i} \Point$. In that case the direction defined by the projection of $\PointAtInfty$ to $\Infty\EBuilding_i$ is the same as the projection of $\PointAtInfty_\Point$ to $\Link_{\EBuilding_i} \Point$.\qed
\end{obs}

\begin{obs}
\label{obs:flat_equivalent_perpendicular}
Let $\EBuilding$ be a Euclidean building and let $\Cell \subseteq \EBuilding$ be a cell. Let $\PointAtInfty$ be a point at infinity of $\EBuilding$ and let $\Busemann$ be a Busemann function centered at $\PointAtInfty$. The restriction of $\Busemann$ to $\Cell$ is constant if and only if $\PointAtInfty_\Point$ is perpendicular to $\Cell$ for every interior point $\Point$ of $\Cell$. In particular, in that case $\PointAtInfty_\Point$ is a direction of $\Link \Cell$.
\end{obs}

\begin{proof}
We use again Lemma~\ref{lem:ray_in_apartment} to obtain an apartment $\Apartment$ that contains $\PointAtInfty$ and $\Point$ (and thus $\Cell$). On that apartment $\Busemann$ is just an affine form whose level sets are perpendicular to the direction towards $\PointAtInfty$.
\end{proof}

\subsection*{Twin Buildings}

Twin buildings generalize spherical buildings. The crucial feature of spherical buildings is the opposition relation. An approach to twin buildings founded on the existence of an opposition relation has been described in \cite{abrvanmal01}. We will use this approach but we will not be very economical with our axioms.

Twin buildings will be defined to be pairs of polyhedral complexes and we fix some shorthand notation concerning such pairs: by a point, cell, etc.\ of a pair $(A,B)$ of polyhedral complexes we mean a point, cell, etc.\ of either $A$ or $B$. We also write $\Point \in (A,B)$, $\Cell \subseteq (A,B)$ and the like. A map $(A,B) \to (A',B')$ between pairs of polyhedral complexes is a pair of maps $A \to A'$ and $B \to B'$. The letter $\varepsilon$ refers to either $+$ or $-$ and, in each statement, $-\varepsilon$ refers to the other of the two.

For us a \emph{twin building} is a pair $(\Building_+,\Building_-)$ of (disjoint) buildings of same type together with an \emph{opposition relation} $\opm{} \subseteq \Building_+ \times \Building_-$ subject to the following conditions:
there exists a set $\ApartmentSystem$ of \emph{twin apartments} $(\Apartment_+,\Apartment_-)$, which are pairs of subcomplexes $\Apartment_\varepsilon$ of $\Building_\varepsilon$, satisfying
\begin{enumerate}[label=(TB\arabic{*}), ref=TB\arabic{*},leftmargin=*]
\setcounter{enumi}{-1}
\item every $\Apartment_\varepsilon$ with $(\Apartment_+,\Apartment_-) \in \ApartmentSystem$ is a Coxeter complex of the same type as $\Building_\varepsilon$.
\item any two points $\Point,\AltPoint \in (\Building_+,\Building_-)$ are contained in a common twin apartment\label{item:twin_building_common_apartment}
\item the relation $\op$ restricted to a twin apartment $(\Apartment_+,\Apartment_-)$ induces a type-preserving isomorphism of polyhedral complexes $\Apartment_+ \leftrightarrow \Apartment_-$.\label{item:twin_building_apartment_opposition}
\item if $\Cell_+$ and $\Cell_-$ are opposite panels, then being non-opposite is a bijective correspondence between the chambers that contain $\Cell_+$ and the chambers that contain $\Cell_-$.
\label{item:twin_building_panel_isomorphism}
\end{enumerate}

Two points $\Point_+ \in \Building_+$ and $\Point_- \in \Building_-$ are \emph{opposite} if $\Point_+ \op \Point_-$. To give a meaning to the last axiom, we have to observe that the opposition relation naturally induces an opposition relation on the cells: namely if $\Cell_+ \subseteq \Building_+$ and $\Cell_- \subseteq \Building_-$ are cells, we say that $\Cell_+$ is \emph{opposite} $\Cell_-$ if $\op$ induces a bijection $\Cell_+ \leftrightarrow \Cell_-$. By \eqref{item:twin_building_common_apartment} and \eqref{item:twin_building_apartment_opposition} this is equivalent to the condition that $\Cell_+$ and $\Cell_-$ contain interior points that are opposite. If this is the case, we also write $\Cell_+ \op \Cell_-$.

The buildings $\Building_+$ and $\Building_-$ are called the \emph{positive} and the \emph{negative half} of $(\Building_+,\Building_-)$. The type $\typ (\Building_+,\Building_-)$ of the twin building is the type of its halves. We denote the Weyl-distance (Fact~\ref{fact:building}~\eqref{item:weyl_distance}) on $\PosBuilding$, respectively $\NegBuilding$, by $\PosWeylDistance$, respectively $\NegWeylDistance$.

A group acts on a twin building if it acts on each of the buildings and preserves the opposition relation. The action is said to be \emph{strongly transitive} if it is transitive on pairs $(\Chamber,\PNApartments)$ where $\Chamber$ is a chamber of a twin apartment $\PNApartments$. As for spherical buildings this is equivalent to the action being transitive on pairs of opposite chambers.

Let $(\Apartment_+,\Apartment_-)$ be a twin apartment of a twin building $(\Building_+,\Building_-)$. Let $\Chamber_+ \subseteq \Apartment_+$ and $\Chamber_- \subseteq \Apartment_-$ be chambers. By \eqref{item:twin_building_apartment_opposition} there is a unique chamber $\AltChamber_-$ in $\Apartment_-$ that is opposite $\Chamber_+$. The Weyl-codistance $\WeylCoDistance(\Chamber_+,\Chamber_-)$ in $(\Apartment_+,\Apartment_-)$ between $\Chamber_+$ and $\Chamber_-$ is defined to be the Weyl-distance from $\AltChamber_-$ to $\Chamber_-$. Note that this is the same as the Weyl-distance from $\Chamber_+$ to the unique chamber in $(\Apartment_+,\Apartment_-)$ opposite $\Chamber_-$. The Weyl-codistance $\WeylCoDistance(\Chamber_-,\Chamber_+)$ is the inverse of $\WeylCoDistance(\Chamber_+,\Chamber_-)$.

It is clear that every twin building according to the definition in \cite[Section~5.8]{abrbro} gives rise to a twin building according to our definition and the converse follows from \cite{abrvanmal01}. Thus we may use results about twin buildings from \cite{abrbro}. From these we need the following:

\begin{fact}
\label{fact:twin_building}
Let $(\Building_+,\Building_-)$ be a twin building.
\begin{enumerate}
\item The Weyl-codistances on the twin apartments fit together to define a well-defined Weyl-codistance $\WeylCoDistance$ on the chambers of $(\Building_+,\Building_-)$. That is, if $\WeylCoDistance_{(\Apartment_+,\Apartment_-)}$ denotes the Weyl-distance on a twin apartment $(\Apartment_+,\Apartment_-)$, then $\WeylCoDistance_{(\Apartment_+,\Apartment_-)}(\Chamber_\varepsilon,\Chamber_{-\varepsilon})$ is the same for every twin apartment $(\Apartment_+,\Apartment_-)$ that contains two given chambers $\Chamber_\varepsilon$ and $\Chamber_{-\varepsilon}$.
\item If $\Chamber \subseteq \Building_\varepsilon$ is a chamber and $\Cell \subseteq \Building_{-\varepsilon}$ is an arbitrary cell, then there is a unique chamber $\AltChamber \ge \Cell$ such that $\WeylCoDistance(\Chamber,\AltChamber)$ has maximal length. This element is called the \emph{projection of $\Chamber$ to $\Cell$} and denoted $\WeylProjection_\Cell \Chamber$. It has the property that every twin apartment that contains $\Chamber$ and $\Cell$ also contains $\AltChamber$. If $\AltCell \subseteq \Building_\varepsilon$ is an arbitrary cell, then the projection of $\AltCell$ to $\Cell$ is $\WeylProjection_\Cell \AltCell \defeq \Intersect_{\Chamber \ge \AltCell} \WeylProjection_\Cell \Chamber$.\label{item:twin_projection}
\item If $(\Apartment_+,\Apartment_-)$ is a twin apartment of $\Building$ and $\Chamber$ is a chamber of $(\Apartment_+,\Apartment_-)$, the \emph{retraction onto $(\Apartment_+,\Apartment_-)$ centered at $\Chamber$}, denoted $\Retraction{(\Apartment_+,\Apartment_-)}{\Chamber}$, is the map that (isometrically and in a type preserving way) takes a chamber $\AltChamber$ to the chamber $\AltChamber'$ of $(\Apartment_+,\Apartment_-)$ that is characterized by $\WeylDistance_\varepsilon(\Chamber,\AltChamber) = \WeylDistance_\varepsilon(\Chamber,\AltChamber')$, respectively $\WeylCoDistance(\Chamber,\AltChamber) = \WeylCoDistance(\Chamber,\AltChamber')$, depending on whether $\Chamber$ and $\AltChamber$ lie in the same half of the twin building. It is an opposition-preserving isometry on twin apartments that contain $\Chamber$ and generally contracting (both, in the usual sense and in terms of Weyl-distance).
\end{enumerate}
\end{fact}

The first statement is implied by \cite{abrvanmal01}. The existence of the projection is shown in Lemma~5.149 and the statement about the containment in a twin apartment in Lemma~5.173 of \cite{abrbro}.

\footerlevel{3}
\headerlevel{3}

\section{Buildings and Groups}
\label{sec:buildings_and_groups}

Buildings are a tool to better understand groups. The link is via strongly transitive actions as introduced in the last section. In this section we give an overview of how one obtains for a given group a building and a strongly transitive action thereon.
The definitions are taken from the Chapters~6 and 7 of \cite{abrbro} which provide a thorough introduction.

\subsection*{BN-Pairs}

Let $\Group$ be a group. A tuple $(\Group,B,N,S)$ is said to be a \emph{Tits system} and $(B,N)$ is said to be a \emph{BN-pair} if $\Group$ is generated by $B$ and $N$, the intersection $T \defeq B \intersect N$ is normal in $N$, $S$ is a generating set for $W \defeq N/T$ and the following conditions hold:
\begin{enumerate}[label=(BN\arabic{*}), ref=BN\arabic{*},leftmargin=*]
\item For $s \in S$ and $w \in W$,
\[
sBw \subseteq BswB \union BwB\text{ .}
\]
\item For $s \in S$,
\[
sBs^{-1} \not\leq B \text{ .}
\]
\end{enumerate}

\begin{fact}[{\cite[Theorem~6.56]{abrbro}}]
Let $(\Group,B,N,S)$ be a Tits system. Let $T \defeq B \intersect N$ and $W \defeq N/T$. Then the pair $(W,S)$ is a Coxeter system and there is a thick building $\SBuilding$ of type $(W,S)$ on which $\Group$ acts strongly transitively. The group $B$ is the stabilizer of a chamber and the group $N$ stabilizes an apartment, which contains this chamber, and acts transitively on its chambers.
\end{fact}

If $\GroupScheme$ is a semisimple algebraic group defined over a field $\Field$, then $\GroupScheme(\Field)$ admits a BN-pair of spherical type, see \cite[Section~5]{tits74} (see any of the books \cite{borel91,humphreys81,springer98} for the notions from the theory of algebraic groups). Assume for simplicity that $\GroupScheme$ is $\Field$-split, i.e., there is a maximal torus $T$ that is $\Field$-split. Let $N$ be its normalizer and $B$ a Borel group that contains $T$. Then $(B(\Field),N(\Field))$ is a BN-pair for $\GroupScheme(\Field)$. Its type is that of the root system associated to $\GroupScheme$.

\subsection*{Twin BN-Pairs}

Let $B_+$, $B_-$, and $N$ be subgroups of a group $\Group$ such that $B_+ \intersect N = B_- \intersect N \eqdef T$. Assume that $N$ normalizes $T$ and set $W \defeq N/T$. A tuple $(\Group,B_+,B_-,N,S)$ is a \emph{twin Tits system} and $(B_+,B_-,N)$ is a \emph{twin BN-pair} if $S \subseteq W$ is such that $(W,S)$ is a Coxeter system and the following hold for $\varepsilon \in \{+,-\}$:
\begin{enumerate}[label=(TBN\arabic{*}), ref=TBN\arabic{*},leftmargin=*]
\setcounter{enumi}{-1}
\item $(G,B_\varepsilon,N,S)$ is a Tits system.
\item If $l(sw) < l(w)$, then $B_\varepsilon s B_{\varepsilon} w B_{-\varepsilon} = B_\varepsilon s w B_{-\varepsilon}$.\label{tbn:codistance}
\item $B_+ s \intersect B_- = \emptyset$.
\end{enumerate}
Here $l(w)$ denotes the length of an element $w \in W$ when written as a product of elements of $S$.

\begin{fact}[{\cite[Theorem~6.87]{abrbro}}]
Let $(\Group,B_+,B_-,N,S)$ be a twin Tits system of type $(W,S)$. There is a thick twin building $\PNBuildings$ of type $(W,S)$ on which $\Group$ acts strongly transitively. The groups $B_+$ and $B_-$ are stabilizers of opposite chambers. The group $N$ stabilizes an apartment and acts transitively on the chambers of each half.
\end{fact}

We said above that every semisimple algebraic group admits a spherical BN-pair. In a similar way every Kac--Moody group admits a twin BN-pair, see \cite{remy02}.

In Appendix~\ref{chap:affine_kac-moody_groups} it is shown that if $\GroupScheme$ is a connected, simply connected, almost simple $\F_q$ group, then $\GroupScheme(\F_q[t,t^{-1}])$ is a Kac--Moody group of affine type. This is how twin buildings enter the scene.

\subsection*{BN-pairs of Groups over Local Fields}

Let $\GroupScheme$ be a group defined over a field $\GlobalField$ and assume for simplicity that $\GroupScheme$ is semisimple, connected and simply connected. If $\GlobalField$ is equipped with a valuation, then $\GroupScheme(\GlobalField)$ carries another BN-pair besides the one discussed above. It is of Euclidean type and the theory around it was developed by Nagayoshi Iwahori and Hideya Matsumoto \cite{iwamat65} in the split case and widely generalized by François Bruhat and Jacques Tits in \cite{brutit72, brutit84}, see also \cite{rousseau77}.

If $\EBuilding$ is the Euclidean building associated to $\GroupScheme(\GlobalField)$ with the valuation on $\GlobalField$ and $\SBuilding$ is the spherical building associated to $\GroupScheme(\GlobalField)$, then $\SBuilding$ can be identified with a subspace of the building at infinity of $\EBuilding$, see \cite{buxgrawit10b}.

\footerlevel{3}
\headerlevel{3}

\section{Simplicial Morse Theory}
\label{sec:morse_theory}

In this section we reformulate Bestvina--Brady Morse theory as introduced in \cite{besbra97} (see also \cite{bestvina08}) to make it more easily applicable later on.

Let $P$ be a totally ordered set and let $\Space$ be an $\ModelSpace{0}$-simplicial complex. A map $f \colon \Vertices \Space \to P$ is a \emph{Morse function on $\Space$} if
\begin{enumerate}[label=(Mor\arabic{*}), ref=Mor\arabic{*},leftmargin=*]
\item $f(v) \ne f(w)$ for two adjacent vertices $v$ and $w$ and\label{item:non_constant_on_edges}
\item the image of $f$ is order-equivalent to a subset of $\Z$.\label{item:discrete_image}
\end{enumerate}
We sometimes speak of $f(v)$ as the \emph{height} of $v$.

If $f$ is a Morse function on $\Space$, then every simplex $\Cell$ has a unique vertex $\Vertex$ on which $f$ is maximal. The \emph{descending link} $\Link\Descending \Vertex$ of a vertex $\Vertex$ is the subcomplex of simplices $\Cell \direction \Vertex$ such that $\Vertex$ is the vertex of maximal height of $\Cell$. By condition \eqref{item:non_constant_on_edges} this is the full subcomplex of vertices $\AltVertex$ adjacent to $\Vertex$ such that $f(\AltVertex) < f(\Vertex)$ (speaking in terms of the combinatorial link).

For $J \subseteq P$ we let $\Space_J$ denote the full subcomplex of $f^{-1}(J)$.

The corestriction to its image of a Morse function $f$ as above may by \eqref{item:discrete_image} be regarded as a map $\Vertices \Space \to \R$ with discrete image. Since $\Space$ is a simplicial complex, this map induces a map $f_\R \colon \Space \to \R$ that is cell-wise affine. Moreover, by \eqref{item:non_constant_on_edges} $f_\R$ is non-constant on cells of dimension $\ge 1$. Hence it is a Morse function in the sense of \cite{besbra97}.

The following two statements are at the core of Morse theory. Using our construction of $f_\R$ above, they are immediate consequences of Lemma~2.5 respectively Corollary~2.6 in \cite{besbra97}.

\begin{lem}[Morse Lemma]
Let $f \colon \Vertices \Space \to P$ be a Morse function. Let $r,s \in P$ be such that $r<s$ and $f(\Vertices \Space) \intersect (r,s) = \emptyset$. Then $X_{(-\infty,s]}$ is homotopy equivalent to $X_{(-\infty,r]}$ with copies of $\Link\Descending \Vertex$ coned off for $\Vertex \in \Space_{\{s\}}$.
\end{lem}

\begin{cor}
\label{cor:morse_theory}
Let $f \colon \Vertices \Space \to P$ be a Morse function. Assume that there is an $R \in P$ such that $\Link\Descending \Vertex$ is $(n-1)$-connected for every $\Vertex$ with $f(\Vertex) > R$.

Let $s,r \in P \union \{\infty\}$ be such that $s \ge r \ge R$. Then the inclusion $\Space_{(-\infty,r]} \into \Space_{(-\infty,s]}$ induces an isomorphism in $\pi_i$ for $0 \le i \le n-1$ and an epimorphism in $\pi_n$.
\end{cor}

Finally we state an elementary fact that will be useful for verifying that a function is a Morse function.

\begin{obs}
\label{obs:product_order_in_z}
Let $P = \R \times \cdots \times \R$ with the lexicographic order. Let $Q \subseteq P$ be such that $\CoordinateProjection_1 Q$ is discrete and $\CoordinateProjection_i Q$ is finite for $i > 1$. Then $Q$ is order-isomorphic to a subset of $\Z$.\qed
\end{obs}

\footerlevel{3}

\footerlevel{2}
\headerlevel{2}

\chapter{Finiteness Properties of $\GroupScheme(\F_q[t])$}
\label{chap:one_place}

In geometric group theory, it is a common situation to have a group $\Group$ that acts on a polyhedral complex $\Space$ with the properties that $\Space$ is contractible and the stabilizers of cells are finite but $\Space$ is not compact modulo the action of $\Group$. One is then interested in a $\Group$-invariant subspace $\Space_0$ of $\Space$ that is compact modulo $\Group$ and still highly connected.

Since this is a common problem, there is a standard procedure to solve it. Namely one has to construct a $\Group$-invariant Morse-function whose sublevel sets are $\Group$-co\-com\-pact and whose descending links are highly connected. Then Bestvina--Brady Morse theory shows that the sublevel sets are highly connected and the problem is solved. Obviously there has to remain some work to be done and so constructing an appropriate Morse-function and analyzing the descending links is usually not a trivial task.

Let $\GroupScheme$ be an $\F_q$-isotropic, connected, noncommutative, almost simple $\F_q$-group. In this chapter we want to determine the finiteness length of $\GroupScheme(\F_q[t])$. We will see that $\GroupScheme(\F_q[t,t^{-1}])$ acts strongly transitively on a locally finite irreducible Euclidean twin building and that $\GroupScheme(\F_q[t])$ is the stabilizer in $\GroupScheme(\F_q[t,t^{-1}])$ of a point of the twin building. Postponing the verification of this statement for the moment our goal is therefore to prove:

\begin{xrefthm}{thm:one_place_geometric}
Let $\PNBuildings$ be an irreducible, thick, locally finite Euclidean twin building of dimension $\Dimension$. Let $\BigGroup$ be a group that acts strongly transitively on $\PNBuildings$ and assume that the kernel of the action is finite. Let $\TheNegPoint \in \NegBuilding$ be a point and let $\Group \defeq \BigGroup_\TheNegPoint$ be the stabilizer of $\TheNegPoint$. Then $\Group$ is of type $F_{\Dimension-1}$ but not of type $F_{\Dimension}$.
\end{xrefthm}

Throughout the chapter we fix an irreducible, thick, locally finite Euclidean twin building $\PNBuildings$ of dimension $\Dimension$ and a point $\TheNegPoint \in \NegBuilding$. We consider the action of the stabilizer $\Group$ of $\TheNegPoint$ in the automorphism group of $\PNBuildings$ on $\OneSpace \defeq \PosBuilding$. Our task is to define a $\Group$-invariant Morse function on $\OneSpace$ that has $\Group$-cocompact sublevel sets and whose descending links are $(\Dimension-2)$-connected.

In Section~\ref{sec:schulz} we describe an important result that indicates the preferable structure of descending links. In Sections~\ref{sec:codistance} and \ref{sec:approximate_height} we construct a function that almost works and sketch the further course of action. The Sections~\ref{sec:zonotopes} to \ref{sec:morse_function} are devoted to rectifying the flaws of the first function. In Sections~\ref{sec:spherical_subcomplexes} and \ref{sec:descending_links} the descending links are analyzed and in Section~\ref{sec:main_theorem} the theorem is proved.

This chapter is based on \cite{buxgrawit10} (except for Section~\ref{sec:moves} which is taken from \cite{witzel10}). However, the proof differs from the proof given there: we define the Morse function on the barycentric subdivision of $\OneSpace$ instead of the coarser subdivision in \cite{buxgrawit10} where only cells of constant height are subdivided. This makes it necessary to define the Morse function also on cells of non-constant height. The additional technicalities needed for this pay off by noticeably simplifying the analysis of the descending links.

\headerlevel{3}

\section{Hemisphere Complexes}
\label{sec:schulz}

In his Ph.D.\ thesis \cite{schulz} (see also \cite{schulz10}) Bernd Schulz investigated subcomplexes of spherical buildings that he expected to occur as descending links of Morse-functions in Euclidean buildings. As we will see, these \emph{hemisphere complexes} are indeed just the right class of subcomplexes and we will make heavy use of Schulz' results. Here we only describe his main result. Partial results that need slight generalizations will be discussed in Section~\ref{sec:spherical_subcomplexes}.

Let $\SBuilding$ be a thick spherical building. If $\Set$ is a subset of $\SBuilding$ we write $\SBuilding(\Set)$\index[xsyms]{deltaa@$\SBuilding(\Set)$} for the subcomplex supported by $\Set$. Recall that this is the subcomplex of all cells of $\SBuilding$ that are fully contained in $\Set$.

We fix a point $n \in \SBuilding$ which we call the \emph{north pole} of $\SBuilding$. The \emph{closed hemisphere} $S\ClosedHemi$ is the set of all points of $\SBuilding$ that have distance $\ge \pi/2$ from $n$. The \emph{open hemisphere} $S\OpenHemi$ is defined analogously. In other words $S\ClosedHemi$ is $\SBuilding$ with the open ball of radius $\pi/2$ around $n$ removed and $S\OpenHemi$ is $\SBuilding$ with the closed ball of radius $\pi/2$ around $n$ removed. The \emph{equator} $S\Equator$ is the set of all points that have distance precisely $\pi/2$ from $n$, i.e., $S\Equator = S\ClosedHemi \setminus S\OpenHemi$.

The \emph{closed hemisphere complex} is the subcomplex $\SBuilding\ClosedHemi \defeq \SBuilding(S\ClosedHemi)$ supported by the closed hemisphere. The \emph{open hemisphere complex} is the subcomplex $\SBuilding\OpenHemi \defeq \SBuilding(S\OpenHemi)$\index[xsyms]{deltapi@$\SBuilding\OpenHemi$} supported by the open hemisphere. The \emph{equator complex} is the subcomplex $\SBuilding\Equator \defeq \SBuilding(S\Equator)$ supported by the equator.

\begin{obs}
The open hemisphere complex, the closed hemisphere complex and the equator complex each is a full subcomplex of $\SBuilding$.\qed
\end{obs}

\begin{proof}
For every simplex $\Cell$ there is an apartment $\Apartment$ that contains $n$ and $\Cell$. The result follows from the fact that $S^{\sim \pi/2} \intersect \Apartment$ is $\pi$-convex and $\Cell$ is the convex hull of its vertices, where $\sim$ is either of $\ge$, $>$, and $=$.
\end{proof}

Recall that $\SBuilding$ decomposes as a spherical join $\SBuilding = \SBuilding_1 * \cdots * \SBuilding_k$ of irreducible subbuildings. The \emph{horizontal part} $\SBuilding\Hor$\index[xsyms]{deltahor@$\SBuilding\Hor$} is defined to be the join of all join factors that are contained in the equator complex. The \emph{vertical part} $\SBuilding\Ver$\index[xsyms]{deltaver@$\SBuilding\Ver$} is the join of all remaining join factors. So there is an obvious decomposition
\begin{equation}
\label{eq:building_decomposes_as_hor_join_ver}
\SBuilding = \SBuilding\Hor * \SBuilding\Ver \text{ .}
\end{equation}

We can now state the main result of Schulz' thesis, see \cite[Satz,~p.27]{schulz} and \cite[Theorem~B]{schulz10}:

\begin{thm}
\label{thm:schulz_main}
Let $\SBuilding$ be a thick spherical building with north pole $n \in \SBuilding$. The closed hemisphere complex $\SBuilding\ClosedHemi$ is properly $(\dim \SBuilding)$-spherical. The open hemisphere complex $\SBuilding\OpenHemi$ is properly $(\dim \SBuilding\Ver)$-spherical.
\end{thm}

Recall that a CW-complex is properly $k$-spherical if it is $k$-dimensional, $(k-1)$-connected and not $k$-connected.

To determine whether a simplex lies in the horizontal link or not, we have the following criterion (cf.\ \cite[Lemma~4.2]{buxwor08}):

\begin{lem}
\label{lem:horizontal_criterion}
Let $\SBuilding$ be a spherical building with north pole $n$. Let $\Vertex \in \SBuilding$ be a vertex. These are equivalent:
\begin{enumerate}
\item $\Vertex \in \SBuilding\Hor$.\label{item:vertex_in_horizontal_link}
\item $\SpherDistance(\Vertex,\AltVertex) = \pi/2$ for every non-equatorial vertex $\AltVertex$ adjacent to $\Vertex$.\label{item:non_equatorial_pi/2_away}
\item $\typ \Vertex$ and $\typ \AltVertex$ lie in different connected components of $\typ \SBuilding$ for every non-equatorial vertex $\AltVertex$ adjacent to $\Vertex$.\label{item:non_equatorial_not_in_same_component}
\end{enumerate}
The statement remains true, if in the second and third statement $\AltVertex$ ranges over the non-equatorial vertices of a fixed chamber that contains $\Vertex$.
\end{lem}

\begin{proof}
The implications $\eqref{item:vertex_in_horizontal_link} \implies \eqref{item:non_equatorial_pi/2_away} \equiv \eqref{item:non_equatorial_not_in_same_component}$ follow from Observation~\ref{obs:vertices_not_in_same_join_factor}.

For $\eqref{item:non_equatorial_pi/2_away} \implies \eqref{item:vertex_in_horizontal_link}$ it remains to see that if $\Chamber$ is a chamber and $\SBuilding_1$ is a join factor of $\SBuilding$ that contains $n$, then $\Chamber \intersect \SBuilding_1$ contains a non-equatorial vertex. This follows from the fact that $\Chamber \intersect \SBuilding_1$ has the same dimension as $\SBuilding_1$ while $\SBuilding\Equator \intersect \SBuilding_1$ has strictly lower dimension.
\end{proof}

\begin{lem}
\label{lem:horizontal_link_decomposition}
Let $\SBuilding$ be a spherical building with north pole $n$. Assume that the building decomposes as a spherical join $\SBuilding = \Join_i \SBuilding_i$ of (not necessarily irreducible) subbuildings $\SBuilding_i$. Let $I$ be the set of indices $i$ such that $\SBuilding_i$ is not entirely contained in $\SBuilding\Equator$. Then
\[
\SBuilding\Hor = \Join_{i \in I} \SBuilding_i\Hor * \Join_{i \nin I} \SBuilding_i
\]
where the north pole of $\SBuilding_i$ is the point $n_i$ closest to $n$.
\end{lem}

\begin{proof}
First note that the subbuildings $\SBuilding_i$ are $\pi$-convex and if $i \in I$ then $\SpherDistance(n,\SBuilding_i) < \pi/2$, so $n_i \defeq \ClosestPointProjection_{\SBuilding_i} n$ exists by Lemma~\ref{lem:spherical_projection}. Note further that it suffices to show that
\[
\SBuilding\Equator = \Join_{i \in I} \SBuilding_i\Equator * \Join_{i \nin I} \SBuilding_i
\]
because the decomposition of $\SBuilding$ into irreducible factors is clearly a refinement of the decomposition $\Join_i \SBuilding_i$.

The north pole $n$ can be written as the spherical join of the $n_i, i \in I$ and none of the coefficients is zero. It thus follows from the definition of the spherical join \eqref{eq:spherical_join}, that a vertex $\Vertex$ in a join factor $\SBuilding_i$ has distance $\pi/2$ from $n$ if and only if it has distance $\pi/2$ from $n_i$. Clearly every vertex of $\SBuilding$ is contained in some $\SBuilding_i$. The result therefore follows from the fact that $\SBuilding\Equator$ is a full subcomplex.
\end{proof}

\footerlevel{3}
\headerlevel{3}

\section{Metric Codistance}
\label{sec:codistance}

We want to define a metric codistance on the twin building $\PNBuildings$, i.e., a metric analogue of the Weyl-codistance.

Let $\PosPoint \in \PosBuilding$ and $\NegPoint \in \NegBuilding$ be points. Let $\TwinApartment = \PNApartments$ be a twin apartment that contains both. We define $d^*_\TwinApartment(\PosPoint,\NegPoint)$ to be the distance from $\PosPoint$ to the unique point in $\Apartment$ that is opposite $\NegPoint$. It is clear that this is the same as the distance from $\NegPoint$ to the unique point in $\TwinApartment$ that is opposite $\PosPoint$.

\begin{obs}
\label{obs:restricted_retractions_are_isomorphisms}
Let $\Chamber$ be a chamber and let $\TwinApartment$ and $\TwinApartment'$ be twin apartments that contain $\Chamber$. Let $\Retraction{\TwinApartment}{\Chamber}$ and $\Retraction{\TwinApartment'}{\Chamber}$ be the retractions centered at $\Chamber$ onto $\TwinApartment$ respectively $\TwinApartment'$. Then $\Retraction{\TwinApartment}{\Chamber}|_{\TwinApartment'}$ and $\Retraction{\TwinApartment'}{\Chamber}|_{\TwinApartment}$ are isomorphisms of thin twin apartments that are inverse to each other. In particular, they preserve Weyl- and metric distance and opposition.
\end{obs}

\begin{lem}
\label{lem:codistance_well-defined}
Let $\TwinApartment$ and $\TwinApartment'$ be two twin apartments that contain $\PosPoint$ and $\NegPoint$. Then $d^*_\TwinApartment(\PosPoint,\NegPoint) = d^*_{\TwinApartment'}(\PosPoint,\NegPoint)$.
\end{lem}

\begin{proof}
Let $\PosChamber \subseteq \TwinApartment$ be a chamber that contains $\PosPoint$ and let $\NegChamber \subseteq \TwinApartment'$ be a chamber that contains $\NegPoint$. Let $\TwinApartment''$ be a twin apartment that contains $\PosChamber$ and $\NegChamber$.

By Observation~\ref{obs:restricted_retractions_are_isomorphisms} the map $\Retraction{\TwinApartment''}{\NegChamber}|_{\Apartment}$ is an isometry that takes the point opposite $\NegPoint$ in $\TwinApartment$ to the point opposite $\NegPoint$ in $\TwinApartment''$. Thus $d^*_\TwinApartment(\PosPoint,\NegPoint) = d^*_{\TwinApartment''}(\PosPoint,\NegPoint)$. Applying the same argument to $\Retraction{\TwinApartment''}{\PosChamber}|_{\Apartment'}$ yields $d^*_{\TwinApartment'}(\PosPoint,\NegPoint) = d^*_{\TwinApartment''}(\PosPoint,\NegPoint)$.
\end{proof}

Thus we obtain a well-defined \emph{metric codistance} $\EuclCoDistance$\index[xsyms]{dstar@$\EuclCoDistance$} by taking $\EuclCoDistance(\PosPoint,\NegPoint)$ to be $\EuclCoDistance_{\TwinApartment}(\PosPoint,\NegPoint)$ for any twin apartment $\Apartment$ that contains $\PosPoint$ and $\NegPoint$.

An important feature of the metric codistance is that it gives rise to a unique direction toward infinity that we describe now. We consider as before points $\PosPoint \in \PosBuilding$ and $\NegPoint \in \NegBuilding$ and a twin apartment $\PNApartments$ that contains them. We assume that the two points are not opposite, i.e., that $d^*(\PosPoint,\NegPoint) \ne 0$.

We define the geodesic ray in $\TwinApartment$ from $\PosPoint$ to $\NegPoint$ to be the geodesic ray in $\PosApartment$ that issues at $\PosPoint$ and moves away from the point opposite $\NegPoint$. As a set we denote it by $\TwinRay{\PosPoint}{\NegPoint}[][\TwinApartment]$.

\begin{lem}
\label{lem:twin_ray_well-defined}
Let $\TwinApartment$ and $\TwinApartment'$ be two twin apartments that contain $\PosPoint$ and $\NegPoint$. Then $\TwinRay{\PosPoint}{\NegPoint}[][\TwinApartment] \subseteq \TwinApartment'$. That is, $\TwinRay{\PosPoint}{\NegPoint}[][\TwinApartment] = \TwinRay{\PosPoint}{\NegPoint}[][\TwinApartment']$.
\end{lem}

\begin{proof}
Let $\AltPoint$ be a point of $\Set \defeq \TwinRay{\PosPoint}{\NegPoint}[][\TwinApartment] \intersect \TwinApartment'$. We will show that a neighborhood of $\AltPoint$ in $\TwinRay{\PosPoint}{\NegPoint}[][\TwinApartment]$ is also contained in $\Set$, which is therefore open. On the other hand it is clearly closed and since $\TwinRay{\PosPoint}{\NegPoint}[][\TwinApartment]$ is connected we deduce that $\Set = \TwinRay{\PosPoint}{\NegPoint}[][\TwinApartment]$.

First note that $[\PosPoint,\AltPoint] \subseteq \Set$ because the positive half of $\TwinApartment'$ is convex. Let $\NegChamber$ be a chamber that contains $\NegPoint$ and let $\Cell$ be the carrier of $\AltPoint$. Let $\AltChamber$ be the projection of $\NegChamber$ to $\Cell$. The chamber $\Chamber_0$ opposite $\NegChamber$ contains the point $\Point_0$ opposite $\NegPoint$ in $\TwinApartment$. Since $\TwinRay{\PosPoint}{\NegPoint}[][\TwinApartment]$ moves away from $\Point_0$, an initial part of it is contained in the chamber over $\PosPoint$ furthest away from $\Chamber_0$, but this is just $\AltChamber$. The result now follows from the fact that $\AltChamber \subseteq \TwinApartment'$ by Fact~\ref{fact:twin_building}~\eqref{item:twin_projection}.
\end{proof}

By the lemma setting $\TwinRay{\PosPoint}{\NegPoint} \defeq \TwinRay{\PosPoint}{\NegPoint}[][\TwinApartment]$ for any twin apartment $\TwinApartment$ that contains $\PosPoint$ and $\NegPoint$ defines a well-defined ray in the Euclidean building. The ray $\TwinRay{\NegPoint}{\PosPoint}$ is defined in the same way.

\footerlevel{3}
\headerlevel{3}

\section{Height: a First Attempt}
\label{sec:approximate_height}

After the introduction of the metric codistance in Section~\ref{sec:codistance} an obvious height function on $\PosBuilding$ imposes itself, namely
\[
\ApproxHeight(\Point) \defeq \EuclCoDistance(\Point,\TheNegPoint) \text{ .}
\]
This function has a \emph{gradient} $\Gradient\ApproxHeight$ that is defined by letting $\Gradient_\Point\ApproxHeight$ be the direction of $\TwinRay{\Point}{\TheNegPoint}$ for every $\Point \in \PosBuilding$ with $\ApproxHeight(\Point) > 0$. It is a gradient in the following sense:

\begin{obs}
\label{obs:weak_gradient_criterion}
Let $\Point$ be such that $\ApproxHeight(\Point) > 0$. Let $\Path$ be a path that issues at $\Point$. The direction $\Path_\Point$ points into $\ApproxHeight^{-1}([0,\ApproxHeight(\Point)))$, i.e., $\ApproxHeight\circ \Path$ is descending on an initial interval, if and only if $\angle(\Gradient_\Point\ApproxHeight,\Path_\Point) > \pi/2$. In other words, the set of directions of $\Link \Point$ that are infinitesimally descending is an open hemisphere complex with north pole $\Gradient_\Point\ApproxHeight$.
\end{obs}

\begin{proof}
We may assume $\Path$ to be sufficiently short such that its image is contained in a chamber $\Chamber$ that contains $\Point$. Let $\TwinApartment = \PNApartments$ be an apartment that contains $\Chamber$ and $\TheNegPoint$ and let $\ThePosPoint$ be the point opposite $\TheNegPoint$ in $\TwinApartment$. The level set of $\Point$ in $\PosApartment$ is a round sphere around $\ThePosPoint$. The gradient $\Gradient_\Point \ApproxHeight$ is the direction away from $\ThePosPoint$. So it is clear that $\Path_\Point$ points into the sphere if and only if it includes an obtuse angle with $\Gradient_\Point \ApproxHeight$.
\end{proof}

\begin{figure}[!ht]
\begin{center}
\includegraphics{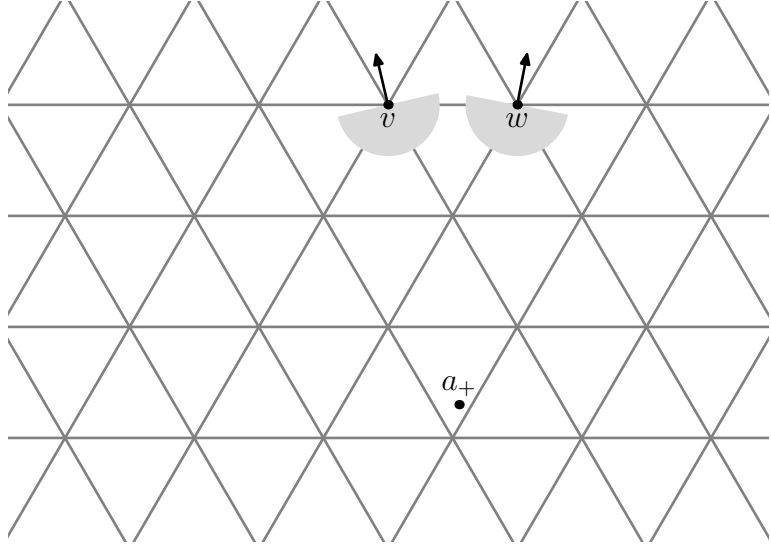}
\end{center}
\caption{The picture shows part of an apartment $\PosApartment$ where $\PNApartments$ is a twin apartment that contains $\TheNegPoint$. The point $\ThePosPoint$ is opposite $\TheNegPoint$. The arrows at the vertices $\Vertex$ and $\AltVertex$ indicate the gradients. The shaded regions show the infinitesimal descending links. One can see that the direction from $\Vertex$ toward $\AltVertex$ is infinitesimally descending, as is the direction from $\AltVertex$ toward $\Vertex$.}
\label{fig:infinitesimal_descending_links}
\end{figure}

The height function $\ApproxHeight$ is almost enough to make the strategy sketched at the beginning of the chapter work: It is $\Group$-invariant and its sublevel sets are compact modulo $\Group$. Moreover, Observation~\ref{obs:weak_gradient_criterion} shows that a direction $\Path_\Point$ issuing at some point $\Point$ is descending if and only if it includes an obtuse angle with the gradient at $\Point$. Let us call this the \emph{infinitesimal angle criterion}. So the space of directions that are infinitesimally descending is an open hemisphere complex and therefore spherical by Theorem~\ref{thm:schulz_main}. However this is not the same as the descending link. The difference is indicated in Figure~\ref{fig:infinitesimal_descending_links}: There are adjacent vertices such that for both vertices the direction toward the other vertex is infinitesimally descending and yet at most one of them can actually be descending for the other. So what we need instead of Observation~\ref{obs:weak_gradient_criterion} is a criterion stating that if $\Vertex$ and $\AltVertex$ are adjacent vertices then $\Height(\AltVertex) < \Height(\Vertex)$ if and only if the angle in $\Vertex$ between the gradient and $\AltVertex$ is obtuse. We call this the \emph{angle criterion}. The macroscopic condition that $\Height(\AltVertex) < \Height(\Vertex)$ replaces the infinitesimal condition that the direction from $\Vertex$ to $\AltVertex$ be descending by demanding that the direction remain descending all the way from $\Vertex$ to $\AltVertex$.

There would be no difference between being infinitesimally descending and being macroscopically descending if the level sets in every apartment $\PosApartment$ were a hyperplane. The hope that $\ApproxHeight$ works after some modifications is nourished by the observation that the actual level sets, which are spheres, become more and more flat with increasing height and thus locally look more and more like a hyperplane. In fact, if we fix a point and consider spheres through that point whose radii tend to infinity, in the limit we get a horosphere which in Euclidean space is the same as a hyperplane.

\begin{figure}[!ht]
\begin{center}
\includegraphics{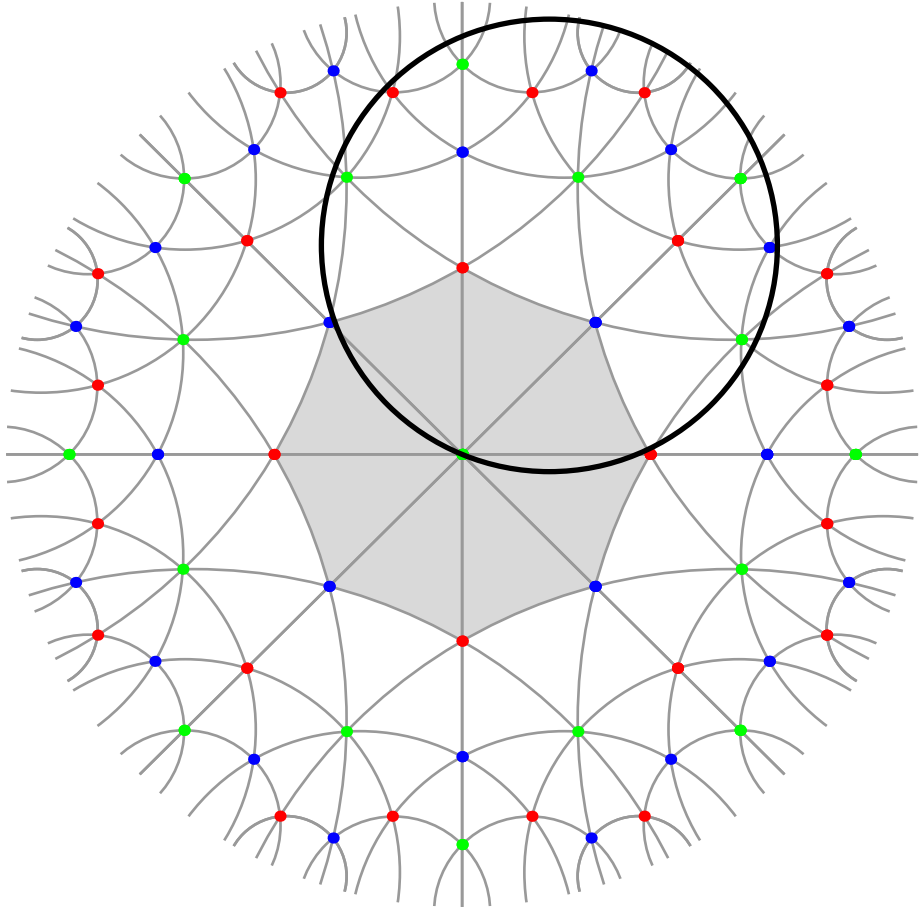}
\end{center}
\caption{A hyperbolic Coxeter complex. The shaded region is the star of the central vertex. The circle is a horosphere, i.e., the level set of a Busemann function. The picture shows that the angle criterion does not hold for Busemann functions in hyperbolic space: there are descending vertices that include acute angles with the gradient.}
\label{fig:hyperbolic_plane}
\end{figure}

Before we descend from this philosophical level to deal with the actual problem at hand, let us look how far our hope goes in the hyperbolic case (that would be interesting in studying finiteness properties of hyperbolic Kac--Moody groups). If $(\PosBuilding',\NegBuilding')$ is a twin building of compact hyperbolic type, we can define a metric codistance just as we have done for Euclidean twin buildings. So if $(\PosApartment',\NegApartment')$ is a twin apartment that contains $\TheNegPoint$, then the level sets of codistance from $\TheNegPoint$ are still spheres in $\PosApartment'$. It is also true that as a limit of spheres we get a horosphere. What is not true is that a horosphere in hyperbolic space is the same as a subspace, see Figure~\ref{fig:hyperbolic_plane}. So even if the level set were a horosphere, the angle criterion would be false. In other words, metric codistance looks much less promising as a Morse function for hyperbolic twin buildings. This matches examples by Abramenko of cell stabilizers in hyperbolic twin buildings that have finiteness length less than dimension minus one.

So we return to our Euclidean twin building and have a closer look at where the problems occur. Let $\Vertex$ and $\AltVertex$ be adjacent vertices of $\PosBuilding$ so that $[\Vertex,\AltVertex]$ is an edge (look again at Figure~\ref{fig:infinitesimal_descending_links}). Let $\PNApartments$ be a twin apartment that contains $[\Vertex,\AltVertex]$ and $\TheNegPoint$ and let $\ThePosPoint$ be the point opposite $\TheNegPoint$ in $\PNApartments$. Let $L$ be the line in $\PosApartment$ spanned by $\Vertex$ and $\AltVertex$. We distinguish two cases:

In the first case the projection of $\ThePosPoint$ onto $L$ does not lie in the interior of $[\Vertex,\AltVertex]$. In that case precisely one of the gradients $\Gradient_\Vertex\ApproxHeight$ and $\Gradient_\AltVertex\ApproxHeight$ includes an obtuse angle with $[\Vertex,\AltVertex]$ and the edge is indeed descending for that vertex.

In the second case $\ThePosPoint$ projects into the interior of $[\Vertex,\AltVertex]$ and both gradients include an obtuse angle with the edge but the edge can only be descending for at most one of them. This is the problematic case.

Consider a hyperplane $H$ perpendicular to $L$ that contains $\ThePosPoint$. The fact that the projection of $\ThePosPoint$ to $L$ lies in the interior of $[\Vertex,\AltVertex]$ can be rephrased to say that $H$ meets the interior of $[\Vertex,\AltVertex]$.

Now there is a finite number of parallelity classes of edges in $\PosApartment$. Let $H_1, \ldots, H_m$ be the family of hyperplanes through $\ThePosPoint$ that are perpendicular to one of these classes (see Figure~\ref{fig:problematic_regions}). The cells in which the gradient criterion fails are precisely those that meet one of the $H_i$ perpendicularly in an interior point.

\begin{figure}[!ht]
\begin{center}
\includegraphics{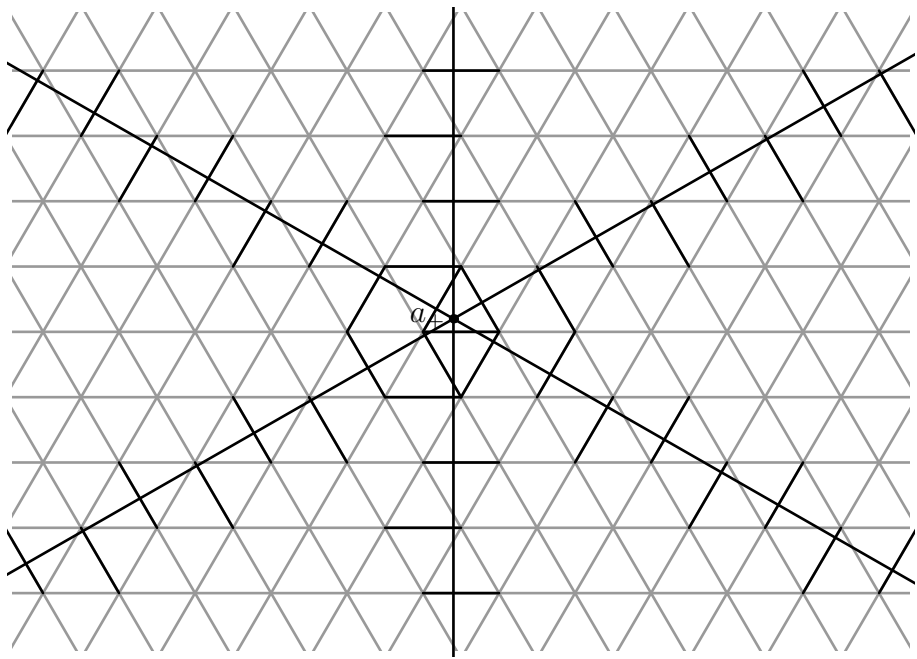}
\end{center}
\caption{The setting is as in Figure~\ref{fig:infinitesimal_descending_links}. Drawn are hyperplanes through $\ThePosPoint$ that are perpendicular to a parallelity class of edges. Every edge that meets such a hyperplane perpendicularly is problematic.}
\label{fig:problematic_regions}
\end{figure}

We will resolve this problem by introducing a new height function $\Height$ that artificially flattens the problematic regions. So adjacent vertices between which the gradient of $\ApproxHeight$ does not decide correctly will have same height with respect to $\Height$. We then introduce a secondary height function to decide between points of same height.

\footerlevel{3}
\headerlevel{3}

\section{Zonotopes}
\label{sec:zonotopes}

In the last section a hyperplane arrangement of problematic regions turned up. Corresponding to a hyperplane arrangement there is always a zonotope $\Zonotope$ (zonotopes will be defined below, see also \cite{mcmullen71,ziegler}). To each of the individual hyperplanes $H$ corresponds a \emph{zone} of the zonotope, which is the set of faces of $\Zonotope$ that contain an edge perpendicular to $H$ as a summand. This suggests that zonotopes can be helpful in flattening the height function in the problematic regions. Indeed they will turn out to be a very robust tool for solving a diversity of problems concerning the height function.

Let $\E$ be a Euclidean vector space with scalar product $\scp{\DummyArg}{\DummyArg}$ and metric $\EuclDistance$. Recall that the relative interior $\relint F$ of a polyhedron $F$ in $\E$ is the interior of $F$ in its affine span. It is obtained from $F$ by removing all proper faces.

Let $\Zonotope \subseteq \E$ be a convex polytope. We denote by $\EuclProjection[\Zonotope]$ the closest point-projection onto $\Zonotope$, i.e., $\EuclProjection[\Zonotope] \Point = \AltPoint$ if $y$ is the point in $\Zonotope$ closest to $\Point$. The \emph{normal cone} of a non-empty face $F$ of $\Zonotope$ is the set
\[
\NormalCone{F} \defeq \{\Vector \in \E \mid \scp{\Vector}{\Point} = \max_{\AltPoint \in \Zonotope} \scp{\Vector}{\AltPoint} \text{ for every }\Point \in F \} \text{ .}
\]
The significance of this notion for us is (see Figure~\ref{fig:projection_decomposition}):
\begin{obs}
\label{obs:space_decomposition}
The space $\E$ decomposes as a disjoint union
\[
\E = \Union_{\emptyset \ne F \le \Zonotope} \relint F + \NormalCone{F}
\]
with $(F - F) \intersect(\NormalCone{F} - \NormalCone{F}) = \{0\}$
and if $\Point$ is written in the unique way as $f + n$ according to this decomposition, then $\EuclProjection[\Zonotope] \Point = f$.\qed
\end{obs}

\begin{figure}[!ht]
\begin{center}
\includegraphics{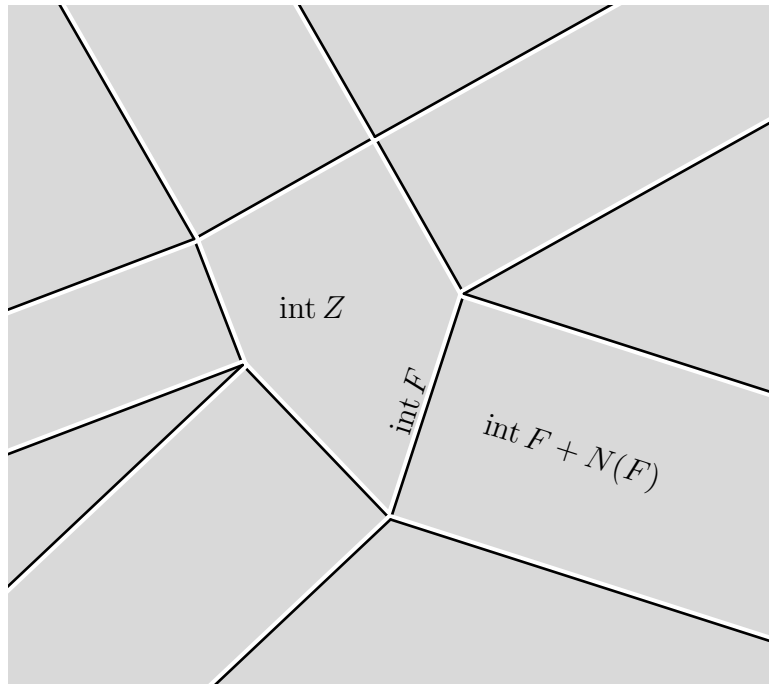}
\end{center}
\caption{The decomposition given by Observation~\ref{obs:space_decomposition}: The shaded regions are the classes of the partition. The boundary points are drawn in black and belong to the the shaded region they touch.}
\label{fig:projection_decomposition}
\end{figure}

We are interested in the situation where $\Zonotope$ is a zonotope.
For our purpose a \emph{zonotope} is described by a finite set $\Directions \subseteq \E$ and defined to be\index[xsyms]{zd@$\Zonotope[\Directions]$}
\[
\Zonotope[\Directions] = \sum_{\Direction \in \Directions} [0,\Direction]
\]
where the sum is the Minkowski sum ($\ConvexSet_1 + \ConvexSet_2 = \{\Vector_1 + \Vector_2 \mid \Vector_1 \in \ConvexSet_1, \Vector_2 \in \ConvexSet_2\}$). The faces of zonotopes are themselves translates of zonotopes, they have the following nice description:

\begin{lem}
If $F$ is a face of $\Zonotope[\Directions]$ and $\Vector \in \relint \NormalCone{F}$, then
\[
F = \Zonotope[\Directions_\Vector] + \sum_{\substack{\Direction \in \Directions \\ \scp{\Vector}{\Direction} > 0}} \Direction \text{ ,}
\]
where $\Directions_\Vector \defeq \{\Direction \in \Directions \mid \scp{\Vector}{\Direction} = 0\}$. 
\end{lem}

\begin{proof}
For a convex set $\ConvexSet$ let $\ConvexSet^v$ be the set of points of $\ConvexSet$ on which the linear form $\scp{v}{\DummyArg}$ attains its maximum. By linearity
\[
\bigg(\sum_{\Direction \in \Directions} [0,\Direction]\bigg)^v = \sum_{\Direction \in \Directions} [0,\Direction]^v \text{ .}
\]
The result now follows from the fact that $[0,\Direction]^v$ respectively equals $\{0\}$, $[0,\Direction]$, or $\{\Direction\}$ depending on whether $\scp{v}{\Direction}$ is negative, zero, or positive.
\end{proof}

It is a basic fact from linear optimization that the relative interiors of normal cones of non-empty faces of $\Zonotope$ partition $\E$, so for every non-empty face $F$ a vector $\Vector$ as in the lemma exists and vice versa.

The zonotopes we will be concerned with have the following interesting property:

\begin{prop}
\label{prop:zonotope_contains_parallel_translate}
Let $\Cell$ be a polytope and let $\Directions$ be a finite set of vectors that has the property that $\AltVertex - \Vertex \in \Directions$ for any two vertices $\Vertex, \AltVertex$ of $\Cell$. Then for every point $\Point$ of $\Zonotope[\Directions]$ there is a parallel translate of $\Cell$ through $\Point$ contained in $\Zonotope[\Directions]$.

More precisely for a vertex $\Vertex$ of $\Cell$ let $E_\Vertex$ be the set of vectors $\AltVertex - \Vertex$ for vertices $\AltVertex \ne \Vertex$ of $\Cell$. Then there is a vertex $\Vertex$ of $\Cell$ such that $\Point + \Zonotope[E_v] \subseteq \Zonotope[\Directions]$.
\end{prop}

This is illustrated in Figure~\ref{fig:almost_rich}.

\begin{proof}
We first show the second statement. It is not hard to see that we may assume that $\Directions$ contains precisely the vectors $\AltVertex - \Vertex$ with $\Vertex, \AltVertex$ vertices of $\Cell$ and we do so. Write
\begin{equation}
\label{eq:point_sum}
\Point = \sum_{z \in D} \alpha_z z
\end{equation}
with $0 \le \alpha_z \le 1$. We consider the complete directed graph whose vertices are the vertices of $\Cell$ and label the edge from $v$ to $w$ by $\alpha_{w-v}$.

If there is a cycle in this graph, all edges of which have a strictly positive label, we set
\[
C \defeq \{w-v \mid \text{the cycle contains an edge from }v\text{ to }w\}\text{ ,}
\]
which is a subset of $\Directions$. For $z \in C$ let $k_z$ be the number of edges in the cycle from a vertex $v$ to a vertex $w$ with $w-v = z$. Let $m$ be the minimum over the $\alpha_z/k_z$ with $z \in C$. We may then subtract $k_z \cdot m$ from $\alpha_z$ for every $z \in C$ and \eqref{eq:point_sum} remains true. Moreover, at least one edge in the cycle is now labeled by $0$. Iterating this procedure we eventually obtain a graph that does not contain any cycles with strictly positive labels. In particular, there is a vertex whose outgoing edges are all labeled by $0$ because there are only finitely many vertices.

Let $v$ be such a vertex. Then $\alpha_z = 0$ for $z \in E_v$. Thus $\Point = \sum_{z \in D \setminus E_v} \alpha_{z} z$ and $\Point + \Zonotope(E_v) \subseteq \Zonotope(D)$.

For the first statement note that $x \in x + (\Cell - v) \subseteq x + \Zonotope(E_v)$ because $x + \Zonotope(E_v)$ is convex and contains all vertices of $x + (\Cell - v)$.
\end{proof}

We say that a finite set of vectors $\Directions$ is \emph{sufficiently rich} for a polytope $\Cell$ if it satisfies the hypothesis of Proposition~\ref{prop:zonotope_contains_parallel_translate}, i.e., if for any two distinct vertices $\Vertex$ and $\AltVertex$ of $\Cell$ the vector $\AltVertex - \Vertex$ is in $\Directions$. Trivially if $\Directions$ is sufficiently rich for $\Cell$ then it is sufficiently rich for the convex hull of any set of vertices of $\Cell$. Note that the property in the conclusion of Proposition~\ref{prop:zonotope_contains_parallel_translate} is not hereditary in that way: for example, a square contains a parallel translate of itself through each of its points, but it does not contain a parallel translate of a diagonal through each of its points.

\begin{prop}
\label{prop:sufficiently_rich_min_vertex}
If $\Directions$ is sufficiently rich for a polytope $\Cell$ then among the points of $\Cell$ closest to $\Zonotope[\Directions]$ there is a vertex. Moreover, the points farthest from $\Zonotope[\Directions]$ form a face of $\Cell$.
\end{prop}

\begin{proof}
Let $\Point \in \Cell$ be a point that minimizes distance to $\Zonotope[\Directions]$. Proceeding inductively it suffices to find a point in a proper face of $\Cell$ that has the same distance. Let $\bar{\Point} = \EuclProjection[{\Zonotope[\Directions]}] \Point$. By Proposition~\ref{prop:zonotope_contains_parallel_translate} $\Zonotope[\Directions]$ contains a translate $\Cell'$ of $\Cell$ through $\bar{\Point}$. All points in $\Cell \intersect (\Point - \bar{\Point} + \Cell')$ have the same distance to $\Zonotope[\Directions]$ as $\Point$. And since this set is the non-empty intersection of $\Cell$ with a translate of itself, it contains a boundary point of $\Cell$.

For the second statement note that if $\EuclDistance(\Zonotope[\Directions],\DummyArg)$ attains its maximum over $\conv V$ in a relatively interior point then it is in fact constant on $\conv V$ by convexity. Now if $V$ is a set of vertices of $\Cell$ on which $\EuclDistance(\Zonotope[\Directions],\DummyArg)$ is maximal, we apply the first statement to $\conv V$ and see that an element of $V$ is in fact a minimum and thus $\EuclDistance(\Zonotope[\Directions],\DummyArg)$ is constant on $\conv V$. Since $\conv V$ contains an interior point of the minimal face $\BigCell$ of $\Cell$ that contains $V$, this shows that $\EuclDistance(\Zonotope[\Directions],\DummyArg)$ is constant on $\BigCell$.
\end{proof}

Now let $\Weyl$ be a finite linear reflection group of $\E$. The action of $\Weyl$ induces a decomposition of $\E$ into cones, the maximal of which we call $\Weyl$-chambers. Clearly if $\Directions$ is $\Weyl$-invariant, then so is $\Zonotope[\Directions]$. To this situation we will apply:

\begin{lem}
\label{lem:n_f_lemma}
Let $\Zonotope$ be a $\Weyl$-invariant polytope. Let $\Vector \in \E$ be arbitrary and let $n = \Vector - \EuclProjection[\Zonotope](\Vector)$. Every $\Weyl$-chamber that contains $\Vector$ also contains $n$.
\end{lem}

\begin{proof}
Let $f = \EuclProjection[\Zonotope](\Vector)$ so that $\Vector = n + f$. It suffices to show that there is no $\Weyl$-wall $\Wall$ that separates $f$ from $n$, i.e., is such that $f$ and $n$ lie in different components of $\E \setminus \Wall$. Assume to the contrary that there is such a wall $\Wall$.

Let $\sigma_\Wall \in \Weyl$ denote the reflection at $\Wall$. Since $f$ is a point of $\Zonotope$ and $\Zonotope$ is $\Weyl$-invariant, $\sigma_\Wall(f)$ is a point of $\Zonotope$ as well. The vector $\sigma_\Wall(f) - f$ is orthogonal to $\Wall$ and lies on the same side as $n$ (the side on which $f$ does not lie). Thus $\scp{n}{\sigma_\Wall(f) - f}> 0$ which can be rewritten as $\scp{n}{\sigma_\Wall(f)} > \scp{n}{f}$. This is a contradiction because $\scp{n}{\DummyArg}$ attains its maximum over $\Zonotope$ in $f$.
\end{proof}

Note that the lemma allows the case where $n$ is contained in a wall and $\Vector$ is not, but not the other way round. The precise statement will be important.

\footerlevel{3}
\headerlevel{3}

\section{Height}
\label{sec:height}

With the tools from Section~\ref{sec:zonotopes} we can in this section define the actual height function we will be working with. Recall that we fixed a Euclidean twin building $\PNBuildings$ and a point $\TheNegPoint \in \NegBuilding$ and that the space we are interested in is $\OneSpace \defeq \PosBuilding$.

Let $\Weyl$ be the spherical Coxeter group associated to $\Infty\OneSpace$. Let $\E$ be a Euclidean vector space of the same dimension as $\OneSpace$ and let $\Weyl$ act on $\E$ as a linear reflection group. The action of $\Weyl$ turns $\Infty\E$ into a spherical Coxeter complex. Every apartment $\PosApartment$ of $\PosBuilding$ (or $\NegApartment$ of $\NegBuilding$) can be isometrically identified with $\E$ in a way that respects the asymptotic structure, i.e., such that the induced map $\Infty\PosApartment \to \Infty\E$ is a type preserving isomorphism. This identification is only unique up to the action of $\Weyl$ and the choice of the base point of $\E$ so we have to take care that nothing we construct depends on the concrete identification.

Let $\Directions$ be a finite subset of $\E$. We will make increasingly stronger assumptions on $\Directions$ culminating in the assumption that it be rich as defined in Section~\ref{sec:descending_links} (page \pageref{page:rich}) but for the moment we only assume that $\Directions$ is $\Weyl$-invariant and centrally symmetric ($\Directions = - \Directions$). In the last section we have seen how $\Directions$ defines a zonotope $\Zonotope \defeq \Zonotope[\Directions]$. It follows from the assumptions on $\Directions$ that $\Zonotope$ is $\Weyl$-invariant and centrally symmetric.

Let $\TwinApartment = \PNApartments$ be a twin apartment and let $\PosPoint \in \PosApartment$ and $\ANegPoint \in \NegApartment$ be points. We identify $\E$ with $\PosApartment$ which allows us to define the polytope $\PosPoint + \Zonotope$. This is well-defined because $\Zonotope$ is $\Weyl$-invariant.

Let $\APosPoint$ be the point opposite $\ANegPoint$ in $\TwinApartment$. We define the $\Zonotope$-perturbed codistance between $\PosPoint$ and $\ANegPoint$ in $\TwinApartment$ to be
\begin{equation}
\label{eq:apartment_perturbed_codistance}
\EuclCoDistance[\Zonotope][\TwinApartment](\PosPoint,\ANegPoint) \defeq \EuclDistance(\PosPoint + \Zonotope,\APosPoint) \text{ ,}
\end{equation}
i.e., the minimal distance from a point in $\PosPoint + \Zonotope$ to $\APosPoint$. This is again independent of the chosen twin apartment:

\begin{lem}
\label{lem:perturbed_codistance_well-defined}
If $\TwinApartment$ and $\TwinApartment'$ are twin apartments that contain points $\PosPoint$ and $\ANegPoint$, then $\EuclCoDistance[\Zonotope][\TwinApartment](\PosPoint,\ANegPoint) = \EuclCoDistance[\Zonotope][\TwinApartment'](\PosPoint,\ANegPoint)$.
\end{lem}

\begin{proof}
Let $\APosPoint'$ be the point opposite $\ANegPoint$ in $\TwinApartment'$. As in Lemma~\ref{lem:codistance_well-defined} one sees that there is a map from $\TwinApartment$ to $\TwinApartment'$ that takes $\ANegPoint$ as well as $\PosPoint$ to themselves and preserves distance and opposition. Thus it also takes $\APosPoint$ to $\APosPoint'$ and $\PosPoint + \Zonotope$ to itself. This shows that the configuration in $\TwinApartment'$ is an isometric image of the configuration in $\TwinApartment$, hence the distances agree.
\end{proof}

We may therefore define the \emph{$\Zonotope$-perturbed codistance} of two points $\PosPoint \in \PosBuilding$ and $\ANegPoint \in \NegBuilding$ to be $\EuclCoDistance[\Zonotope](\PosPoint,\ANegPoint) \defeq \EuclCoDistance[\Zonotope][\TwinApartment](\PosPoint,\ANegPoint)$\index[xsyms]{dstarz@$\EuclCoDistance[\Zonotope]$} for any twin apartment $\TwinApartment$ that contains $\PosPoint$ and $\ANegPoint$.

\begin{figure}[!ht]
\begin{center}
\includegraphics{figs/perturbed_codistance}
\end{center}
\caption{The figure shows two points $\PosPoint$ and $\NegPoint$ that lie in a twin apartment $\PNApartments$. The halves $\PosApartment$ and $\NegApartment$ are identified with each other via $\op$ and with $\E$. Each of the dashed lines represents the $\Zonotope$-perturbed codistance of $\PosPoint$ and $\NegPoint$.}
\label{fig:perturbed_codistance}
\end{figure}

\begin{obs}
If $\PosPoint$ and $\APosPoint$ are two points of $\E$, then
\[
\EuclDistance(\PosPoint + \Zonotope,\APosPoint) = \EuclDistance(\PosPoint,\APosPoint + \Zonotope) \text{ .}
\]
In particular, $\EuclCoDistance[\Zonotope](\PosPoint,\ANegPoint) = \EuclDistance(\PosPoint,\APosPoint + \Zonotope)$ in the situation of \eqref{eq:apartment_perturbed_codistance}.
\end{obs}

This is illustrated in Figure~\ref{fig:perturbed_codistance}

\begin{proof}
If $\Vector$ is a vector in $\Zonotope$, then $\EuclDistance(\PosPoint + \Vector,\APosPoint) = \EuclDistance(\PosPoint,\APosPoint - \Vector)$. The statement now follows from the fact that $\Zonotope$ is centrally symmetric.
\end{proof}

It is clear that we might as well have identified $\E$ with the negative half of a twin apartment and taken the distance there.

We can now define the height function. The \emph{height} of a point $\Point \in \OneSpace$ is defined to be the $\Zonotope$-perturbed codistance from the fixed point $\TheNegPoint$\index[xsyms]{height@$\Height$}:
\[
\Height(\Point) \defeq \EuclCoDistance[\Zonotope](\Point,\TheNegPoint) \text{ .}
\]

\begin{obs}
\label{obs:height_discrete_image}
The set $\Height(\Vertices \OneSpace)$ is discrete.
\end{obs}

\begin{proof}
Let $\NegChamber$ be a chamber that contains $\TheNegPoint$ and let $\PNApartments$ be a twin apartment that contains $\NegChamber$. Let $\rho \defeq \Retraction{\PNApartments}{\NegChamber}$ be the retraction onto $\PNApartments$ centered at $\NegChamber$. Then $\Height = \Height|_{\PosApartment} \circ \rho|_{\OneSpace}$. And $\Height|_{\PosApartment}$ is clearly a proper map. So every compact subset of $\R$ meets $\Height(\Vertices \OneSpace)$ in a finite set.
\end{proof}

Proceeding as in Section~\ref{sec:codistance} we next want to define a gradient for $\Height$.

Consider again a twin apartment $\TwinApartment \defeq \PNApartments$ and let $\PosPoint \in \PosApartment$ and $\ANegPoint \in \NegApartment$ be points. Let $\APosPoint$ be the point opposite $\ANegPoint$ in $\TwinApartment$. Assume that $\EuclCoDistance[\Zonotope](\PosPoint,\ANegPoint) > 0$, i.e., that $\PosPoint \nin \APosPoint + \Zonotope$. The ray $\TwinRay{\PosPoint}{\ANegPoint}[\Zonotope][\TwinApartment]$ is defined to be the ray in $\PosApartment$ that issues at $\PosPoint$ and moves away from (the projection point of $\PosPoint$ onto) $\APosPoint + \Zonotope$.

\begin{prop}
Let $\TwinApartment$ and $\TwinApartment'$ be twin apartments that contain points $\PosPoint$ and $\ANegPoint$. If $\EuclCoDistance[\Zonotope](\PosPoint,\ANegPoint) > 0$, then $\TwinRay{\PosPoint}{\ANegPoint}[\Zonotope][\TwinApartment] = \TwinRay{\PosPoint}{\ANegPoint}[\Zonotope][\TwinApartment']$.
\end{prop}

\begin{proof}
To simplify notation we identify $\E$ with $\PosApartment$ in such a way that the origin of $\E$ gets identified with $\PosPoint$. Let $\Vector$ be the vector that points from $\APosPoint$ to $\PosPoint$, let $f$ be the vector that points from $\APosPoint$ to the projection point of $\PosPoint$ onto $\APosPoint + \Zonotope$, and let $n = \Vector - f$.

Then $\TwinRay{\PosPoint}{\ANegPoint}$ is the geodesic ray spanned by $\Vector$ and $\TwinRay{\PosPoint}{\ANegPoint}[\Zonotope][\TwinApartment]$ is the geodesic ray spanned by $n$. By Lemma~\ref{lem:twin_ray_well-defined} the ray $\TwinRay{\PosPoint}{\ANegPoint}$ is a well-defined ray in the building, in particular, it defines a point at infinity $\Infty{\TwinRay{\PosPoint}{\ANegPoint}}$ that is contained in (the visual boundary of) every twin apartment that contains $\PosPoint$ and $\ANegPoint$. Hence also the carrier $\Cell$ of $\Infty{\TwinRay{\PosPoint}{\ANegPoint}}$ in $\Infty\PosBuilding$ is contained in every such twin apartment. By Lemma~\ref{lem:n_f_lemma} the ray spanned by $n$ lies in every chamber in which $\Vector$ lies, so $\Infty{\TwinRay{\PosPoint}{\ANegPoint}[\Zonotope][\TwinApartment]}$ lies in $\Cell$.

This shows that if $\TwinApartment'$ contains $\PosPoint$ and $\ANegPoint$, then it also contains $\TwinRay{\PosPoint}{\ANegPoint}[\Zonotope][\TwinApartment]$. Thus $\TwinRay{\PosPoint}{\ANegPoint}[\Zonotope][\TwinApartment] = \TwinRay{\PosPoint}{\ANegPoint}[\Zonotope][\TwinApartment']$.
\end{proof}

The \emph{$\Zonotope$-perturbed ray from $\PosPoint$ to $\ANegPoint$} is defined to be $\TwinRay{\PosPoint}{\ANegPoint}[\Zonotope] \defeq \TwinRay{\PosPoint}{\ANegPoint}[\Zonotope][\TwinApartment]$ for any twin apartment $\TwinApartment$ that contains $\PosPoint$ and $\ANegPoint$. It is well-defined by the proposition. The $\Zonotope$-perturbed ray $\TwinRay{\ANegPoint}{\PosPoint}[\Zonotope]$ from $\ANegPoint$ to $\PosPoint$ is defined analogously.

The \emph{gradient} $\Gradient\Height$\index[xsyms]{gradientnoninfty@$\Gradient\Height$} of $\Height$ is given by letting $\Gradient_{\Point} \Height$\index[xsyms]{gradientnoninftypoint@$\Gradient_\Point\Height$} be the direction in $\Link \Point$ defined by $\TwinRay{\Point}{\TheNegPoint}[\Zonotope]$. The \emph{asymptotic gradient} $\Infty\Gradient \Height$\index[xsyms]{gradientinfty@$\Infty\Gradient\Height$} is given by letting $\Infty\Gradient_{\Point} \Height$\index[xsyms]{gradientinftypoint@$\Infty\Gradient_\Point\Height$} be the limit point of $\TwinRay{\Point}{\TheNegPoint}[\Zonotope]$.

\footerlevel{3}
\headerlevel{3}

\section{Flat Cells and the Angle Criterion}
\label{sec:angle_criterion}

In the last section we introduced a height function by perturbing the metric codistance. In this section we describe in which way the perturbation influences the resulting height function.

We start with a property that is preserved by the perturbation.

\begin{obs}
\label{obs:height_convex}
Let $\PNApartments$ be a twin apartment that contains $\TheNegPoint$. The restriction of $\Height$ to $\PosApartment$ is a convex function. In particular, if $\Cell \subseteq \OneSpace$ is a cell, then among the $\Height$-maximal points of $\Cell$ there is a vertex.
\end{obs}

\begin{proof}
On $\PosApartment$ the function $\Height$ is distance from a convex set. The second statement follows by choosing a twin apartment $\PNApartments$ that contains $\Cell$ and $\TheNegPoint$.
\end{proof}

Another property that is preserved is the infinitesimal angle criterion that we know from Observation~\ref{obs:weak_gradient_criterion}:

\begin{obs}
\label{obs:perturbed_weak_gradient_criterion}
Let $\Path$ be a geodesic that issues at a point $\Point \in \OneSpace$ with $\Height(\Point) > 0$. The function $\Height \circ \Path$ is strictly decreasing on an initial interval if and only if $\angle_{\Point}(\Gradient_{\Point}\Height,\Path) > \pi/2$.
\end{obs}

\begin{proof}
Let $\TwinApartment = \PNApartments$ be a twin apartment that contains $\TheNegPoint$ and an initial segment of the image of $\Path$. The statement follows from the fact that $\Height$ restricted to $\PosApartment$ measures distance from a convex set and $\Gradient\Height$ is the direction away from that set.
\end{proof}

We now come to a phenomenon that arises from the perturbation: the existence of higher-dimensional cells of constant height. A cell $\Cell$ is called \emph{flat} if $\Height|_{\Cell}$ is constant.

\begin{obs}
If $\Cell$ is flat, then the (asymptotic) gradient of $\Height$ is the same for all points $\Point$ of $\Cell$. It is perpendicular to $\Cell$.
\end{obs}

\begin{proof}
Let $\TwinApartment = \PNApartments$ be a twin apartment that contains $\TheNegPoint$ and $\Cell$. Let $\ThePosPoint$ be the point opposite $\TheNegPoint$ in $\TwinApartment$. If $\Height$ is constant on $\Cell$ then the projection of $\Cell$ to $\ThePosPoint + \Zonotope$ is a parallel translate and the flow lines away from it are parallel to each other and perpendicular to $\Cell$.
\end{proof}

The observation allows us to define the (\emph{asymptotic}) \emph{gradient} ($\Infty\Gradient_\Cell \Height$\index[xsyms]{gradientinftycell@$\Infty\Gradient_\Cell\Height$}) $\Gradient_\Cell \Height$\index[xsyms]{gradientnoninftycell@$\Gradient_\Cell\Height$} at a flat cell $\Cell$ to be the (asymptotic) gradient of either of its interior points. Note that since $\Gradient_\Cell \Height$ is perpendicular to $\Cell$ it defines a direction in $\Link \Cell$. We take this direction to be the \emph{north pole} of $\Link\Cell$ and define the \emph{horizontal link} $\Link\Hor \Cell$\index[xsyms]{lkhor@$\Link\Hor\Cell$}, the \emph{vertical link} $\Link\Ver \Cell$\index[xsyms]{lkver@$\Link\Ver\Cell$} and the \emph{open hemisphere link} $\Link\OpenHemi\Cell$ according to Section~\ref{sec:schulz}. The decomposition \eqref{eq:building_decomposes_as_hor_join_ver} then reads:
\begin{equation}
\label{eq:link_decomposes_as_hor_join_ver}
\Link \Cell = \Link\Hor\Cell * \Link\Ver\Cell \text{ .}
\end{equation}

\begin{figure}[!ht]
\begin{center}
\includegraphics{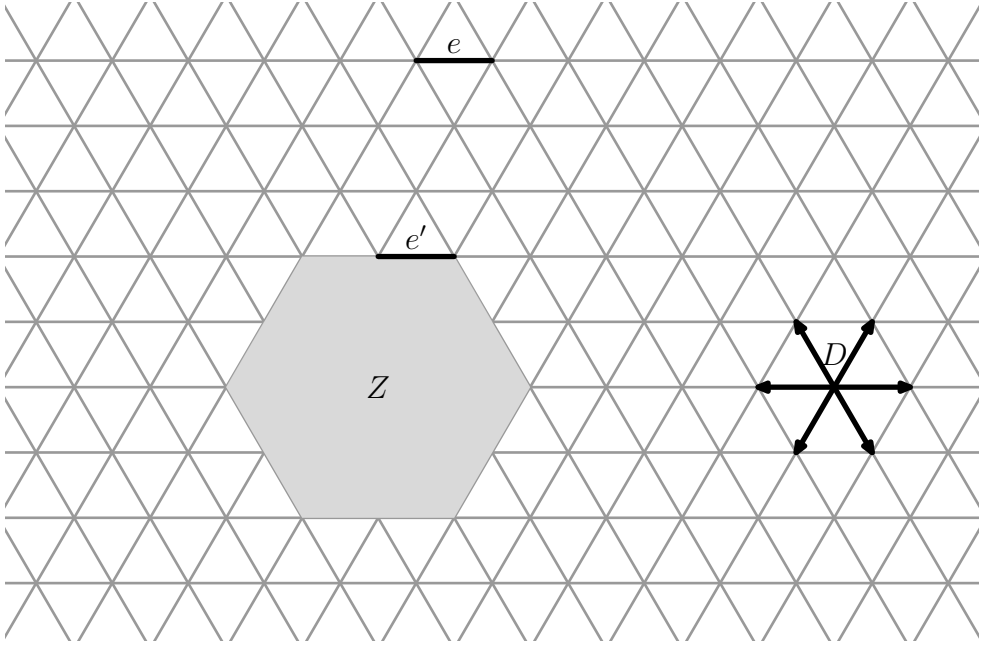}
\end{center}
\caption{The zonotope $\Zonotope$ of an almost rich set $\Directions$. Since $\Directions$ is almost rich, it is sufficiently rich for the edge $e$. Hence $\Zonotope$ contains a parallel translate of $e$ through each of its points. For example, $e'$ is a parallel translate through every point of the projection of $e$ onto $\Zonotope$.}
\label{fig:almost_rich}
\end{figure}

Finally we turn our attention to what we have gained by introducing the zonotope. Let $\PosApartment$ be an apartment of $\PosBuilding$ and identify $\E$ with $\PosApartment$ as before. We say that $\Directions$ is \emph{almost rich} if it contains the vector $\Vertex - \AltVertex$ for any two adjacent vertices $\Vertex$ and $\AltVertex$ of $\PosApartment$ (see Figure~\ref{fig:almost_rich}). Note that this condition is independent of the apartment as well as of the identification.

\begin{prop}
\label{prop:almost_rich_min_vertex}
Assume that $\Directions$ is almost rich and let $\Cell$ be a cell. Among the $\Height$-minima of $\Cell$ there is a vertex and the set of $\Height$-maxima of $\Cell$ is a face.

The statement remains true if $\Cell$ is replaced by the convex hull of some of its vertices.
\end{prop}

\begin{proof}
Let $\TwinApartment = \PNApartments$ be a twin apartment that contains $\Cell$ and $\TheNegPoint$. Let $\ThePosPoint$ be the point opposite $\TheNegPoint$ in $\TwinApartment$. The restriction of $\Height$ to $\PosApartment$ is distance from $\ThePosPoint + \Zonotope[\Directions]$. Since $\Directions$ is almost rich, it is sufficiently rich for $\Cell$. The statement now follows from Proposition~\ref{prop:sufficiently_rich_min_vertex}.
\end{proof}

This implies the angle criterion:

\begin{cor}
\label{cor:angle_criterion}
Assume that $\Directions$ is almost rich. Let $\Vertex$ and $\AltVertex$ be adjacent vertices. The restriction of $\Height$ to $[\Vertex,\AltVertex]$ is monotone. In particular, $\Height(v) > \Height(w)$ if and only if $\angle_\Vertex(\Gradient_\Vertex \Height,\AltVertex) > \pi/2$.
\end{cor}

\begin{proof}
The restriction of $\Height$ to $[\Vertex,\AltVertex]$ is monotone because it is convex (Observation~\ref{obs:height_convex}) and attains its minimum in a vertex (Proposition~\ref{prop:almost_rich_min_vertex}). Let $\Path$ be the geodesic path from $\Vertex$ to $\AltVertex$. Since $\Height \circ \Path$ is monotone and convex, it is descending if and only if it is descending on an initial interval. The second statement therefore follows from Observation~\ref{obs:perturbed_weak_gradient_criterion}.
\end{proof}

A more convenient version is:

\begin{cor}
\label{cor:higher_dimensional_angle_criterion}
Assume that $\Directions$ is almost rich. Let $\Cell$ be a flat cell and $\BigCell \ge \Cell$. Then $\Cell$ is the set of $\Height$-maxima of $\BigCell$ if and only if $\BigCell \direction \Cell \subseteq \Link\OpenHemi\Cell$.
\end{cor}

\begin{proof}
The implication $\Rightarrow$ is clear from Observation~\ref{obs:perturbed_weak_gradient_criterion}.
Conversely let $\BigCell \ge \Cell$ be such that $\BigCell \direction \Cell \subseteq \Link\OpenHemi\Cell$. Then $\angle_{\Point}(\Gradient_\Point\Height,\AltPoint) > \pi/2$ for every point $\Point$ of $\Cell$ and every point $\AltPoint$ of $\BigCell$ not in $\Cell$. To see this consider a twin apartment $\PNApartments$ that contains $\TheNegPoint$ and use that $\PosApartment$ is Euclidean. In particular, if $\Vertex$ is a vertex of $\Cell$ and $\AltVertex$ is a vertex of $\BigCell$ not in $\Cell$, then $\angle_{\Vertex}(\Gradient_\Vertex\Height,\AltVertex) > \pi/2$. Thus Corollary~\ref{cor:angle_criterion} implies $\Height(\Vertex) > \Height(\AltVertex)$.
\end{proof}

Flat cells obviously prevent $\Height$ from being a Morse function. To obtain a Morse function, we need a secondary height function that decides for any two vertices of a flat cell $\Cell$ which one should come first. In fact we will actually define the Morse function on the barycentric subdivision of $\OneSpace$ so the secondary height function will have to decide which of $\Cell$ and its faces should come first.

One has to keep in mind however that the descending link of $\Cell$ with respect to the Morse function should be a hemisphere complex with north pole $\Gradient_\Cell \Height$. What is more, according to Theorem~\ref{thm:schulz_main} the full horizontal part of $\Link \Vertex$ has to be descending to obtain maximal connectedness.

In the next section we will provide the means to define a secondary height function that takes care of this in the case where the primary height function is a Busemann function. The connection to our height function $\Height$ is that, informally speaking, around a flat cell $\Cell$ it looks like a Busemann function centered at $\Infty\Gradient_\Cell \Height$.

\footerlevel{3}
\headerlevel{3}

\section{Secondary Height: the Game of Moves}
\label{sec:moves}

Kai-Uwe Bux and Kevin Wortman in Section~5 of \cite{buxwor08} have devised a machinery that for a Busemann function on a Euclidean building produces a secondary Morse function such that the descending links are either contractible or closed hemisphere complexes. It would be possible at this point to refer to their article. However since in Chapter~\ref{chap:two_places} we will need a generalization of their method to arbitrary Euclidean buildings, we directly prove this generalization here.

Much of the argument in \cite{buxwor08} is carried out in the Euclidean building even though most of the statements are actually statements about links, which are spherical buildings. Here we take a local approach, arguing as much as possible inside the links.

This section is fairly independent from our considerations so far and can be used separately, as has been done in \cite{buxgrawit10b}.

Throughout the section let $\Space = \prod_i\EBuilding_i$ be a finite product of irreducible Euclidean buildings. Since $\Space$ is in general not simplicial, it cannot be a flag complex, however it has the following property reminiscent of flag complexes:

\begin{obs}
\label{obs:polyflag}
If $\Cell^1,\ldots,\Cell^k$ are cells in a product of flag-complexes and for $1 \le l < m \le k$ the cell $\Cell^l \vee \Cell^m$ exists, then $\Cell^1 \vee \cdots \vee \Cell^k$ exists.
\end{obs}

\begin{proof}
Write $\Cell^l = \prod_i \Cell^l_i$ and $\Cell^m = \prod_i \Cell^m_i$. Then $\Cell^l \vee \Cell^m$ exists if and only if $\Cell^l_i \vee \Cell^m_i$ exist for every $i$. The statement is thus translated to a family of statements, one for each factor, that hold because the factors are flag-complexes.
\end{proof}

Let $\Busemann$ be a Busemann function on $\Space$ centered at $\PointAtInfty \in \Infty\OneSpace$. A cell on which $\Busemann$ is constant is called \emph{flat}. If $\Cell$ is flat, then the direction from any point of $\Cell$ toward $\PointAtInfty$ is perpendicular to $\Cell$ (see Observation~\ref{obs:flat_equivalent_perpendicular}) so that it defines a point $n$ in $\Link \Cell$. This point shall be our north pole and the notions from Section~\ref{sec:schulz} carry over accordingly. In particular, the \emph{horizontal link} $\Link\Hor\Cell$\index[xsyms]{lkhor@$\Link\Hor\Cell$} is the join of all join factors of $\Link \Cell$ that are perpendicular to $n$. Note that the north pole does not actually depend on $\Busemann$ but only on $\PointAtInfty$.

We write\index[xsyms]{tauhorizontalsigma@$\BigCell\horizontal\Cell$}
\[
\BigCell \horizontal \Cell \quad \text{if} \quad \BigCell \direction \Cell \subseteq \Link\Hor \Cell
\]
and say for short that ``$\BigCell$ lies in the horizontal link of $\Cell$''. Note that this, in particular, requires $\BigCell$ to be flat but is a stronger condition. If we want to emphasize the point at infinity $\PointAtInfty$ with respect to which $\BigCell$ lies in the horizontal link of $\Cell$ then we write $\BigCell \horizontal_\PointAtInfty \Cell$.

The next observation deals with the interaction of $\PointAtInfty$ and its projections onto the factors of $\Infty\Space$.

\begin{obs}
\label{obs:horizontal_iff_horizontal_on_factors}
Let $\BigCell = \prod_i \BigCell_i$ and $\Cell = \prod_i \Cell_i$ be non-empty flat cells. Let $I$ be the set of indices $i$ such that $d(\PointAtInfty,\Infty\EBuilding_i) \ne \pi/2$.
Then 
\[
\BigCell \horizontal_{\PointAtInfty} \Cell \quad \text{if and only if} \quad \BigCell_i \horizontal_{\PointAtInfty_i} \Cell_i \text{ for every } i \in I\text{ ,}
\]
where $\PointAtInfty_i \defeq \ClosestPointProjection[\Infty\EBuilding_i]\PointAtInfty$.
In other words
\[
\Link\Hor\Cell = (\Join_{i \in I} \Link\Hor \Cell_i) * (\Join_{i \nin I} \Link \Cell_i)
\]
where the north pole of $\Link \Cell_i$ is the direction toward $\PointAtInfty_i$.
\end{obs}

\begin{proof}
By Observation~\ref{obs:asymptotic_to_local_join_compatible} it makes no difference whether we first take the direction toward $\PointAtInfty$ and then project it to a join factor or we first project $\PointAtInfty$ to a join factor and then take the direction toward that point. The result therefore follows from its local analogue, Lemma~\ref{lem:horizontal_link_decomposition}.
\end{proof}

For each factor we have:

\begin{lem}
\label{lem:cohorizontal_faces_meet}
Let $\EBuilding_i$ be an irreducible Euclidean building and $\PointAtInfty \in \Infty\EBuilding_i$. If cells $\Cell_1$, $\Cell_2$, and $\BigCell$ of $\EBuilding_i$ satisfy $\BigCell \horizontal_\PointAtInfty \Cell_1$ and $\BigCell \horizontal_\PointAtInfty \Cell_2$ then $\Cell_1 \intersect \Cell_2 \ne \emptyset$.
\end{lem}

\begin{proof}
Let $\Busemann$ be a Busemann function that defines $\PointAtInfty$. Let $\Chamber$ be a chamber that contains $\BigCell$ and let $\Vertex$ be a vertex of $\Chamber$ with $\Busemann(v) \ne \Busemann(\BigCell)$. We take the quotient of $\Chamber$ modulo directions in $\Cell_1$ and $\Cell_2$ (i.e., we factor out the linear span of $\Cell_1 - \Cell_1$ and $\Cell_2 - \Cell_2$). The images of $\BigCell$, $\Vertex$, $\Cell_1$, and $\Cell_2$ under this projection are denoted $\overline\BigCell$, $\overline\Vertex$, $\overline{\Cell_1}$, and $\overline{\Cell_2}$ respectively. If $\Cell_1$ and $\Cell_2$ did not meet, then $\overline{\Cell_1}$ and $\overline{\Cell_2}$ would be distinct points. In any case $\overline{\Vertex}$ is distinct from both. By Lemma~\ref{lem:horizontal_criterion} we would have $\angle_{\overline{\Cell_1}}(\overline{\Cell_2},\overline{\Vertex}) = \angle_{\overline{\Cell_2}}(\overline{\Cell_1},\overline{\Vertex}) = \pi/2$ which is impossible.
\end{proof}

For the rest of the section all horizontal links are taken with respect to a fixed Busemann function $\Busemann$ centered at a point $\PointAtInfty$.
We say that $\Busemann$ is \emph{in general position} if it is not constant on any (non-trivial) factor of $\Space$. This is equivalent to the condition that $\PointAtInfty$ is not contained in any (proper) join factor of $\Infty\Space$ and in that case we also call $\PointAtInfty$ \emph{in general position}. Combining Observation~\ref{obs:horizontal_iff_horizontal_on_factors} and Lemma~\ref{lem:cohorizontal_faces_meet} we see:

\begin{obs}
\label{obs:general_position_cells_meet}
Assume that $\PointAtInfty$ is in general position. If $\BigCell \horizontal \Cell_1$ and $\BigCell \horizontal \Cell_2$ then $\Cell_1 \intersect \Cell_2 \ne \emptyset$.\qed
\end{obs}

The assumption that $\PointAtInfty$ be in general position is crucial as can be seen in the most elementary case:

\begin{exmpl}
\label{exmpl:cohorizontal_faces_dont_meet}
Consider the product $\EBuilding_1 \times \EBuilding_2$ of two buildings of type $\tilde{A}_1$. Let $\Busemann$ be such that $\Infty\Busemann \in \Infty\EBuilding_1$. Let $\Vertex_1 \in \EBuilding_1$ be a vertex and $\Chamber_2 \subseteq \EBuilding_2$ be a chamber with vertices $\Vertex_2$ and $\AltVertex_2$. The links of $(\Vertex_1,\Vertex_2)$ and of $(\Vertex_1, \AltVertex_2)$ are of type $A_1 * A_1$ and $\{\Vertex_1\} \times \Chamber_2$ lies in the horizontal link of both.
\end{exmpl}

We are now ready to state a technical tool that we will use throughout the section. We will give two proofs at the end.

\begin{prop}
\label{prop:horizontal_properties}
The relation $\horizontal$ (that is, $\horizontal_{\PointAtInfty}$) has the following properties:
\begin{enumerate}
\item If $\BigCell \horizontal \Cell$ and $\BigCell \ge \Another\BigCell \ge \Cell$ then $\Another\BigCell \horizontal \Cell$.\label{item:faces}
\item If $\BigCell \horizontal \Cell$ and $\BigCell \vee \Another\Cell$ exists and is flat then $\BigCell \vee \Another\Cell \horizontal \Cell \vee \Another\Cell$. In particular, if $\BigCell \horizontal \Cell$ and $\BigCell \ge \Another\Cell \ge \Cell$ then $\BigCell \horizontal \Another\Cell$. \label{item:join}
\item If $\BigCell \horizontal \Another\Cell$ and $\Another\Cell \horizontal \Cell$ then $\BigCell \horizontal \Cell$, i.e., $\horizontal$ is transitive.\label{item:transitivity}
\item If $\BigCell \horizontal \Cell_1$ and $\BigCell \horizontal \Cell_2$ and $\Cell_1 \intersect \Cell_2 \ne \emptyset$ then $\BigCell \horizontal \Cell_1 \intersect \Cell_2$.\label{item:non_empty_meet}
\end{enumerate}
\end{prop}

A key observation in \cite{buxwor08} is that for every flat cell $\BigCell$, among its faces $\Cell$ with $\BigCell \horizontal \Cell$ there is a minimal one provided $\Space$ is irreducible. Observation~\ref{obs:general_position_cells_meet} allows us to replace the irreducibility assumption by the assumption that $\PointAtInfty$ be in general position:

\begin{lem}
\label{lem:tau_min}
Assume that $\PointAtInfty$ is in general position. Let $\BigCell$ be a flat cell of $\Space$. The set of $\Cell \le \BigCell$ such that $\BigCell \horizontal \Cell$ is an interval, i.e., it contains a minimal element $\BigCell\Min$\index[xsyms]{sigmamin@$\Cell\Min$} and
\[
\BigCell \horizontal \Cell \quad \text{if and only if} \quad \BigCell\Min \le \Cell \le \BigCell \text{ .}
\]
\end{lem}

In particular, $\BigCell \horizontal \BigCell\Min$.

\begin{proof}
Let $T \defeq \{\Cell \le \BigCell \mid \BigCell \horizontal \Cell\}$, which is finite. If $\Cell_1$ and $\Cell_2$ are in $T$, then since $\Busemann$ is in general position, Observation~\ref{obs:general_position_cells_meet} implies that $\Cell_1 \intersect \Cell_2 \ne \emptyset$. So by Proposition~\ref{prop:horizontal_properties}~\eqref{item:non_empty_meet} $\Cell_1 \intersect \Cell_2 \in T$. Hence there is a minimal element $\BigCell\Min$, namely the intersection of all elements of $T$. If $\Another\Cell$ satisfies $\BigCell\Min \le \Another\Cell \le \BigCell$, then $\Another\Cell \in T$ by Proposition~\ref{prop:horizontal_properties}~\eqref{item:join}.
\end{proof}

To see what can go amiss if $\PointAtInfty$ is not in general position we need a slightly bigger example than Example~\ref{exmpl:cohorizontal_faces_dont_meet}:
\begin{exmpl}
Let $\EBuilding = \EBuilding_1 \times \EBuilding_2$ where the first factor is of type $\tilde{A}_1$ and the second is of type $\tilde{A_2}$. The factor $X_2$ has three parallelity classes of edges. Let $\Busemann$ be a Busemann function that is constant on $\EBuilding_1$ and on one class of edges in $\EBuilding_2$. Consider a square $\BigCell$ that is flat. Its two edges in the $\EBuilding_2$-factor have a link of type $A_1 * A_1$ and $\BigCell$ lies in the horizontal link of each of them. Hence if there were to be a $\BigCell\Min$ it would have to be the empty simplex. However every vertex of $\BigCell$ has a link of type $A_1 * A_2$ and $\BigCell$ does not lie the horizontal link of any of them.
\end{exmpl}

To understand this example note that if a Busemann function is constant on some factor, then in this factor $\BigCell\Min = \emptyset$ for all cells $\BigCell$. But being empty does not behave well with respect to taking products: a product is empty if one of the factors is empty, not if all of the factors are empty. In other words the face lattice of a product of simplices is not the product of the face lattices of the simplices. But the face lattice of a product of simplices without the bottom element is the product of the face lattices of the simplices without the bottom elements: $\FaceLattice(\prod_i \Cell_i)_{> \emptyset} = \prod_i \FaceLattice(\Cell_i)_{> \emptyset}$.

Lemma~\ref{lem:tau_min} generalizes \cite[Lemma~5.2]{buxwor08} except for the explicit description in terms of orthogonal projections. We will see that transitivity of $\horizontal$ suffices to replace the explicit description. For the rest of the section we assume that $\PointAtInfty$ is in general position.

We define \emph{going up} by\index[xsyms]{sigmauptau@$\Cell\Up\BigCell$}
\[
\Cell \Up \BigCell \quad \text{if} \quad \BigCell\Min = \Cell \ne \BigCell
\]
and \emph{going down} by\index[xsyms]{sigmadowntau@$\Cell\Down\BigCell$}
\[
\BigCell \Down \Cell \quad \text{if} \quad \Cell \lneq \BigCell \text{ but not }\BigCell \horizontal \Cell\text{ .}
\]
A \emph{move} is either going up or going down. The main result of this section is:

\begin{prop}
\label{prop:bound_on_moves}
There is a bound on the lengths of sequences of moves that only depends on the dimensions of the $\EBuilding_i$. In particular, no sequence of moves enters a cycle.
\end{prop}

The results in \cite{buxwor08} for which the arguments do not apply analogously are Observation~5.3 and the Lemmas~5.10 and 5.13. They correspond to Observation~\ref{obs:min_min}, Lemma~\ref{lem:min_in_vertical_link}, and Lemma~\ref{lem:cell1_bigcell2_min} below. For the convenience of the reader we also give proofs of the statements that can be easily adapted from those in \cite{buxwor08}. Observation~\ref{obs:significant_either_or} is new and simplifies some arguments.

A good starting point is of course:

\begin{obs}
\label{obs:no_up_down_cycle}
There do not exist cells $\Cell$ and $\BigCell$ such that $\Cell \Up \BigCell$ and $\BigCell \Down \Cell$.
\end{obs}

\begin{proof}
If $\Cell \Up \BigCell$, then in particular $\BigCell \horizontal \Cell$ which contradicts $\BigCell \Down \Cell$.
\end{proof}

We come to the first example of how transitivity of $\horizontal$ replaces the explicit description of $\BigCell\Min$:

\begin{obs}
\label{obs:min_min}
If $\BigCell \horizontal \Cell$, then $\Cell\Min = \BigCell\Min$. In particular, $(\BigCell\Min)\Min = \BigCell\Min$.
\end{obs}

\begin{proof}
We have $\BigCell \horizontal \BigCell\Min$ and $\BigCell \ge \Cell \ge \BigCell\Min$ so by Proposition~\ref{prop:horizontal_properties}~\eqref{item:faces} $\Cell \horizontal \BigCell\Min$, i.e., $\Cell\Min \le \BigCell\Min$. Conversely $\BigCell \horizontal \Cell \horizontal \Cell\Min$ so by Proposition~\ref{prop:horizontal_properties}~\eqref{item:transitivity} $\BigCell \horizontal \Cell\Min$, i.e., $\BigCell\Min \le \Cell\Min$.
\end{proof}

We call a cell $\Cell$ \emph{significant} if $\Cell\Min = \Cell$.

\begin{obs}
\label{obs:significant_either_or}
If $\Cell$ is significant and $\BigCell \gneq \Cell$ is a proper flat coface, then either $\Cell \Up \BigCell$ or $\BigCell \Down \Cell$.
\end{obs}

\begin{proof}
If $\BigCell \horizontal \Cell$ then $\BigCell\Min = \Cell\Min = \Cell$ by Observation~\ref{obs:min_min} so $\Cell \Up \BigCell$. Otherwise $\BigCell \Down \Cell$.
\end{proof}

The next two lemmas show transitivity of $\Up$ and $\Down$ so that we can restrict our attention to alternating sequences of moves.

\begin{lem}
\label{lem:up_transitive}
It never happens that $\Cell_1 \Up \Cell_2 \Up \Cell_3$. In particular, $\Up$ is transitive.
\end{lem}

\begin{proof}
Suppose $\Cell_1 \Up \Cell_2 \Up \Cell_3$. Then by Observation~\ref{obs:min_min} $\Cell_1 = \Cell_2\Min = (\Cell_3\Min)\Min = \Cell_3\Min = \Cell_2$ contradicting $\Cell_1 \ne \Cell_2$.
\end{proof}

\begin{lem}
\label{lem:down_transitive}
The relation $\Down$ is transitive.
\end{lem}

\begin{proof}
Assume $\Cell_1 \Down \Cell_2 \Down \Cell_3$ but not $\Cell_1 \Down \Cell_3$. Clearly $\Cell_1 \gneq \Cell_3$. So $\Cell_1 \horizontal \Cell_3$ and by Proposition~\ref{prop:horizontal_properties}~\eqref{item:faces} $\Cell_2 \horizontal \Cell_3$, contradicting $\Cell_2 \Down \Cell_3$.
\end{proof}

Now we approach the proof that the length of an alternating sequences of moves is bounded.

\begin{lem}
If
\[
\Cell_1 \Up \BigCell_1 \Down \Cell_2
\]
then
\[
\Cell_1 = (\Cell_1 \vee \Cell_2)\Min \quad \text{and} \quad \Cell_1 \vee \Cell_2 \Down \Cell_2 \text{ .}
\]
In particular, $\Cell_1 \Up \Cell_1 \vee \Cell_2 \Down \Cell_2$ unless $\Cell_1 \Down \Cell_2$.
\end{lem}

\begin{proof}
By Proposition~\ref{prop:horizontal_properties}~\eqref{item:faces} $\Cell_1 \vee \Cell_2 \horizontal \Cell_1$ so $\Cell_1 = (\Cell_1 \vee \Cell_2)\Min$ by Observation~\ref{obs:min_min}. And $\Cell_1 = \BigCell_1\Min \not\le \Cell_2$ whence $\Cell_2 \lneq \Cell_1 \vee \Cell_2$ so that $\Cell_1 \vee \Cell_2 \Down \Cell_2$.
\end{proof}

\begin{lem}
\label{lem:min_in_vertical_link}
If $\BigCell \ge \Cell$ are flat cells then $(\BigCell\Min \vee \Cell) \direction \Cell \subseteq \Link\Ver \Cell$.
\end{lem}

\begin{proof}
Write $(\BigCell\Min \vee \Cell) \direction \Cell = (\Cell_v \direction \Cell) \vee (\Cell_h \direction \Cell)$ with $\Cell_v \direction \Cell \subseteq \Link\Ver \Cell$ and $\Cell_h \direction \Cell \subseteq \Link\Hor \Cell$. We want to show that $\Cell_v = \BigCell\Min \vee \Cell$.

Since $\Cell_h \horizontal \Cell$, Proposition~\ref{prop:horizontal_properties}~\eqref{item:join} implies
\[
\BigCell\Min \vee \Cell = \Cell_h \vee \Cell_v \horizontal \Cell \vee \Cell_v = \Cell_v \text{ .}
\]
And since $\BigCell \horizontal \BigCell\Min \vee \Cell$, Proposition~\ref{prop:horizontal_properties}~\eqref{item:transitivity} implies $\BigCell \horizontal \Cell_v$. Thus $\BigCell\Min \le \Cell_v$ as desired.
\end{proof}

\begin{cor}
If
\[
\Cell_1 \Up \BigCell_1 \Down \Cell_2 \Up \BigCell_2
\]
then $\BigCell_2 \vee \Cell_1$ exists.
\end{cor}

\begin{proof}
By assumption $\BigCell_2 \direction \Cell_2 \subseteq \Link\Hor \Cell_2$. Since $\Cell_1 = \BigCell_1\Min$, Lemma~\ref{lem:min_in_vertical_link} shows that $(\Cell_1 \vee \Cell_2) \direction \Cell_2 \subseteq \Link\Ver \Cell_2$. Hence $(\BigCell_2 \direction \Cell_2) \vee ((\Cell_1 \vee \Cell_2) \direction \Cell_2)$ exists and so does $\Cell_1 \vee \BigCell_2$.
\end{proof}

\begin{lem}
\label{lem:cell1_bigcell2_min}
If
\[
\Cell_1 \Up \BigCell_1 \Down \Cell_2 \Up \BigCell_2
\]
then $(\BigCell_2 \vee \Cell_1)\Min = \Cell_1$.
\end{lem}

\begin{proof}
By assumption $\BigCell_2 \horizontal \Cell_2$ so Proposition~\ref{prop:horizontal_properties}~\eqref{item:join} implies that ${\BigCell_2 \vee \Cell_1} \horizontal {\Cell_2 \vee \Cell_1}$. Moreover, $\BigCell_1 \horizontal \Cell_2 \vee \Cell_1$ so $(\Cell_2 \vee \Cell_1)\Min = \BigCell\Min = \Cell_1$ by Observation~\ref{obs:min_min}. Thus $(\BigCell_2 \vee \Cell_1)\Min = \Cell_1$ again by Observation~\ref{obs:min_min}.
\end{proof}

\begin{lem}
\label{lem:shortening}
An alternating chain
\[
\Cell_1 \Up \BigCell_1 \Down \Cell_2 \Up \BigCell_2
\]
can be shortened to either
\[
\Cell_1 \Up \Cell_1 \vee \BigCell_2 \Down \BigCell_2 \quad \text{or} \quad \Cell_1 \Up \BigCell_1 \Down \BigCell_2 \text{ .}
\]
\end{lem}

\begin{proof}
We know by assumption that $\BigCell_2\Min = \Cell_2$ and Lemma~\ref{lem:cell1_bigcell2_min} implies $(\Cell_1 \vee \BigCell_2)\Min = \Cell_1$. Since by Observation~\ref{obs:no_up_down_cycle} $\Cell_1 \ne \Cell_2$ this implies $\BigCell_2 \ne \Cell_1 \vee \BigCell_2$ so that $\Cell_1 \vee \BigCell_2 \Down \BigCell_2$.

If $\Cell_1 \ne \Cell_1 \vee \BigCell_2$, then $\Cell_1 \Up \Cell_1 \vee \BigCell_2$. If $\Cell_1 = \Cell_1 \vee \BigCell_2$, then $\BigCell_1 \gneq \Cell_1 \gneq \BigCell_2$ so that $\BigCell_1 \Down \BigCell_2$.
\end{proof}

\begin{cor}
\label{cor:no_cycles}
No sequence of moves enters a cycle.
\end{cor}

\begin{proof}
Since $\Up$ and $\Down$ are transitive by Lemma~\ref{lem:up_transitive} respectively \ref{lem:down_transitive}, a cycle of minimal length must be alternating. Thus by Lemma~\ref{lem:shortening} it can go up at most once. But then it would have to be of the form ruled out by Observation~\ref{obs:no_up_down_cycle}.
\end{proof}

\begin{lem}
\label{lem:vee_of_lower_terms_exists}
If
\[
\Cell_1 \Up \BigCell_1 \Down \cdots \Down \Cell_{k-1} \Up \BigCell_{k-1} \Down \Cell_k
\]
then $\Cell_1 \vee \cdots \vee \Cell_k$ exists.
\end{lem}

\begin{proof}
First we show by induction that $\Cell_1 \vee \Cell_k$ exists. If $k=2$ this is obvious. Longer chains can be shortened using Lemma~\ref{lem:shortening}.

Applying this argument to subsequences we see that $\Cell_i \vee \Cell_j$ exists for any two indices $i$ and $j$. So by Observation~\ref{obs:polyflag} $\Cell_1 \vee \cdots \vee \Cell_k$ exists.
\end{proof}

\begin{proof}[Proof of Proposition \ref{prop:bound_on_moves}]
We first consider the case of alternating sequences. By Lemma~\ref{lem:vee_of_lower_terms_exists} for any alternating sequence of moves, the join of the lower elements (to which there is a move going down or from which there is a move going up) exists. This cell has at most $N \defeq \prod_i (2^{\dim(\EBuilding_i)+1}-1)$ non-empty faces. Since by Corollary~\ref{cor:no_cycles} the alternating sequence cannot contain any cycles $N$, is also the maximal number of moves.

The maximal number of successive moves down is $\dim(\Space) = \sum_{i} \dim(\EBuilding_i)$ while there are no two consecutive moves up by Lemma~\ref{lem:up_transitive}.

So an arbitrary sequence of moves can go up at most $N$ times and can go down at most $\dim(\Space) \cdot N$ times making $(\dim(\Space)+1) \cdot N$ a bound on its length (counting moves, not cells).
\end{proof}

\subsection*{Proof of Proposition \ref{prop:horizontal_properties} Using Spherical Geometry}

In this paragraph we prove Proposition~\ref{prop:horizontal_properties} using spherical geometry. This proof has the advantage of being elementary but on the other hand it is quite technical. The main tool will be the equivalence $\eqref{item:vertex_in_horizontal_link} \equiv \eqref{item:non_equatorial_pi/2_away}$ of Lemma~\ref{lem:horizontal_criterion}:

\begin{reminder}
\label{reminder:metric_horizontal_criterion}
Let $\SBuilding$ be a spherical building with north pole $n$. Let $\Vertex \in \SBuilding$ be a vertex. These are equivalent:
\begin{enumerate}
\item $\Vertex \in \SBuilding\Hor$.\label{item:mrem_vertex_in_horizontal_link}
\item $\SpherDistance(\Vertex,\AltVertex) = \pi/2$ for every non-equatorial vertex $\AltVertex$ adjacent to $\Vertex$.\label{item:mrem_non_equatorial_pi/2_away}
\end{enumerate}
\end{reminder}

We will repeatedly be in the situation where we have a cell $\Cell$ and two adjacent vertices $\Vertex$ and $\AltVertex$ and want to compare $\angle_{\Cell}(\Cell \vee \Vertex,\Cell \vee \AltVertex)$ to $d(\Vertex,\AltVertex)$. In the case where $\Cell$ is a vertex, this is essentially a $2$-dimensional problem and comes down to a statement about spherical triangles. For higher dimensional cells it is possible to consider the projections of $\Vertex$ and $\AltVertex$ onto $\Cell$ and obtain a statement about spherical $3$-simplices. The argument then becomes a little less transparent. For that reason, we prefer to carry out an induction on cells of the Euclidean space that allows us to keep the spherical arguments $2$-dimensional.

Before we start with the actual proof, we need to record one more fact. Recall that if $\Building$ is a spherical or Euclidean building and $\BigCell \ge \Cell$ are cells, then the link of $\BigCell$ in $\Building$ can be identified with the link of $\BigCell \direction \Cell$ in $\Link \Cell$.

\begin{obs}
\label{obs:non_equatorial_in_link_and_link_of_link}
Let $\Busemann$ be a Busemann function on $\Space$ and let $\Cell \le \BigCell$ be cells that are flat with respect to $\Busemann$. Let $n_\Cell$ and $n_\BigCell$ be the north poles determined by $\Busemann$ in $\Link \Cell$ and $\Link \BigCell$ respectively. The identification of $\Link \BigCell$ with $\Link (\BigCell \direction \Cell)$ identifies $n_\BigCell$ with the direction toward $n_\Cell$.

Moreover, a vertex $\Vertex$ adjacent to $\BigCell \direction \Cell$ in $\Link \Cell$ is equatorial in $\Link \Cell$ if and only if it is equatorial as a vertex of $\Link (\BigCell \direction \Cell)$.
\end{obs}

\begin{proof}
The first statement is clear. For the second note that $\BigCell \direction \Cell$ is equatorial in $\Link \Cell$, i.e., that $\SpherDistance(n_\Cell,\BigCell \direction \Cell) = \pi/2$. Let $\SpherPoint$ be a point in $\BigCell \direction \Cell$ that is closest to $\Vertex$ (if $\SpherDistance(\Vertex,\BigCell) < \pi/2$ this is the unique projection point, otherwise any point). Then the distance in $\Link \Cell$ between the direction toward $v$ and the direction toward $n_\Cell$ is $\angle_\SpherPoint(\Vertex,n_\Cell)$. Now Observation~\ref{obs:spherical_triangles}~\eqref{item:edge_and_angle_give_edge_and_angle} and \eqref{item:edges_give_angles} imply that $\Vertex$ is equatorial if and only if $\angle_\SpherPoint(\Vertex,n_\Cell) = \pi/2$.
\end{proof}

Now we start with the proof of Proposition~\ref{prop:horizontal_properties}. The first statement is trivial.

\begin{proof}[Proof of Proposition~\ref{prop:horizontal_properties}~\eqref{item:join}]
Let $\SBuilding$ be the link of $\Cell$ with the north pole given by $\PointAtInfty$. Assume first that $\Cell$ is a facet of $\Another\Cell \vee \Cell$. Let $\overline{\Another\Cell} = (\Cell \vee \Another\Cell) \direction \Cell$ which is a vertex.

We identify $\Link (\Cell \vee \Another\Cell)$ with $\Link \overline{\Another\Cell}$. By Observation~\ref{obs:non_equatorial_in_link_and_link_of_link} this identification preserves being equatorial. So we may take a vertex $\AltVertex$ of $\BigCell \direction \Cell$ other than $\overline{\Another\Cell}$ (that corresponds to a vertex of $(\BigCell \vee \Another\Cell) \direction (\Cell \vee \Another\Cell)$) and a non-equatorial vertex $\Vertex$ adjacent to it (corresponding to a non-equatorial vertex of $\Link \Cell \vee \Another\Cell$) and by Reminder~\ref{reminder:metric_horizontal_criterion} our task is to show that $\angle_{\overline{\Another\Cell}}(\Vertex,\AltVertex) = \pi/2$ (meaning that the distance of the corresponding vertices in $\Link \Cell \vee \Another\Cell$ is $\pi/2$).

Since by assumption $\BigCell \horizontal \Cell$, we know by Reminder~\ref{reminder:metric_horizontal_criterion}, that $\SpherDistance(\Vertex,\AltVertex) = \pi/2$.

Thus the triangle with vertices $\Vertex$, $\AltVertex$, and $\overline{\Another\Cell}$ satisfies $\SpherDistance(\Vertex,\AltVertex) = \pi/2$ and the angle at $\overline{\Another\Cell}$ can be at most $\pi/2$ because we are considering cells in a Coxeter complex. Hence by Observation~\ref{obs:spherical_triangles}~\eqref{item:edge_gives_angle} it has to be precisely $\pi/2$ as desired.

For the general case set $\Another\Cell_0 \defeq \Another\Cell \vee \Cell$ and inductively take $\Another\Cell_{i+1}$ to be a facet of $\Another\Cell_i$ that contains $\Cell$ until $\Another\Cell_n = \Cell$ for some $n$.

By assumption $\BigCell = \BigCell \vee \Another\Cell_n \horizontal \Cell \vee \Another\Cell_n = \Cell$ and the above argument applied to $\Another\Cell_{n-1}$ shows that $\BigCell \vee \Another\Cell_{n-1} \horizontal \Cell \vee \Another\Cell_{n-1}$. Proceeding inductively we eventually obtain $\BigCell \vee \Another\Cell_0 \horizontal \Cell \vee \Another\Cell_0$ which is what we want.
\end{proof}

\begin{proof}[Proof of Proposition~\ref{prop:horizontal_properties}~\eqref{item:transitivity}]
Let $\SBuilding$ be the link of $\Cell$ and let $\overline{\Another\Cell} = \Another\Cell \direction \Cell$.

We identify $\Link \Another\Cell$ with $\Link \overline{\Another\Cell}$. Let $\AltVertex$ be a vertex of $\BigCell \direction \Cell$ that is not contained in $\overline{\Another\Cell}$ (and corresponds to a vertex of $\BigCell \direction \Another\Cell$) and let $\Vertex$ be a non-equatorial vertex adjacent to it (corresponding to a non-equatorial vertex of $\Link \Another\Cell$). From the fact that $\BigCell \horizontal \Another\Cell$ we deduce using Reminder~\ref{reminder:metric_horizontal_criterion} that $\angle(\overline{\Another\Cell} \vee \Vertex,\overline{\Another\Cell} \vee \AltVertex) = \pi/2$. Similarly, $\SpherDistance(\overline{\Another\Cell},\Vertex) = \pi/2$ because $\Another\Cell \horizontal \Cell$. We want to show that $\SpherDistance(\Vertex,\AltVertex) = \pi/2$.

Let $\SpherPoint$ be the projection of $\AltVertex$ to $\overline{\Another\Cell}$ if it exists or otherwise take any point of $\overline{\Another\Cell}$. Then $[\SpherPoint,\AltVertex]$ is perpendicular to $\overline{\Another\Cell}$ by the choice of $\SpherPoint$ and $[\SpherPoint,\Vertex]$ is perpendicular to $\overline{\Another\Cell}$ because $\SpherDistance(\Vertex,\overline{\Another\Cell}) = \pi/2$. Thus $\angle(\overline{\Another\Cell} \vee \Vertex,\overline{\Another\Cell} \vee \AltVertex) = \angle_\SpherPoint(\Vertex,\AltVertex)$.

We consider the triangle with vertices $\Vertex$, $\AltVertex$, and $\SpherPoint$. We know that $\SpherDistance(\Vertex,\SpherPoint) = \pi/2$ and that $\angle_\SpherPoint(\Vertex,\AltVertex) = \pi/2$. From this we deduce that $\SpherDistance(\Vertex,\AltVertex) = \pi/2$ using Observation~\ref{obs:spherical_triangles}~\eqref{item:edge_and_angle_give_edge_and_angle}.
\end{proof}

\begin{proof}[Proof of Proposition~\ref{prop:horizontal_properties}~\eqref{item:non_empty_meet}]
We assume first that $\Cell_1 \intersect \Cell_2$ is a facet of both, $\Cell_1$ and $\Cell_2$. Because we have already proven transitivity of $\horizontal$, it suffices to show that $\Cell_1 \horizontal \Cell_1 \intersect \Cell_2$.

Let $\SBuilding$ be the link of $\Cell_1 \intersect \Cell_2$ and let $\overline{\Cell_i} = \Cell_i \direction (\Cell_1 \intersect \Cell_2)$ for $i \in \{1,2\}$. Note that $\overline{\Cell_1}$ and $\overline{\Cell_2}$ are distinct vertices. Let $\Vertex$ be a non-equatorial vertex in $\SBuilding$ adjacent to $\overline{\Cell_1}$.

Using our criterion Reminder~\ref{reminder:metric_horizontal_criterion}, we have to show that $\SpherDistance(\Vertex,\overline{\Cell_1}) = \pi/2$. We identify $\Link \Cell_1$ with $\Link \overline{\Cell_1}$ and $\Link \Cell_2$ with $\Link \overline{\Cell_2}$. Since $\Vertex$ corresponds to a non-equatorial vertex in both of $\Link \Cell_1$ and $\Link \Cell_2$, the criterion implies that $\angle_{\overline{\Cell_2}}(\Vertex,\overline{\Cell_1}) = \pi/2 = \angle_{\overline{\Cell_1}}(\Vertex,\overline{\Cell_2})$. So the statement follows at once from Observation~\ref{obs:spherical_triangles}~\eqref{item:angles_give_edges}.

Now we consider the more general case where $\Cell_1 \intersect \Cell_2$ is a facet of $\Cell_2$ but need not be a facet of $\Cell_1$. Set $\Cell_2^0 \defeq \Cell_1 \vee \Cell_2$ and let $\Cell_2^{i+1}$ be a facet of $\Cell_2^i$ that contains $\Cell_2$ until $\Cell_2^n = \Cell_2$ for some $n$. Also let $\Cell_1^{i} = \Cell_1 \intersect \Cell_2^{i-1}$ for $1 \le i \le n$ and note that $\Cell_1^i = \Cell_1^{i-1} \intersect \Cell_2^{i-1}$. Then $\Cell_1^i \intersect \Cell_2^i$ is a facet of both, $\Cell_1^i$ and $\Cell_2^i$, for $1 \le i \le n$. By assumption $\BigCell \horizontal \Cell_1$ and Proposition~\ref{prop:horizontal_properties}~\eqref{item:join} shows that $\BigCell \horizontal \Cell_2^i$ for $0 \le i \le n$. Thus the argument above applied inductively shows that $\Cell_1 \horizontal \Cell_1 \intersect \Cell_2$.

An analogous induction allows to drop the assumption that $\Cell_1 \intersect \Cell_2$ be a facet of $\Cell_2$.
\end{proof}

\subsection*{Proof of Proposition \ref{prop:horizontal_properties} Using Coxeter Diagrams}

In this paragraph we prove Proposition \ref{prop:horizontal_properties} using Coxeter diagrams. We use the equivalence $\eqref{item:vertex_in_horizontal_link} \equiv \eqref{item:non_equatorial_not_in_same_component}$ of Lemma~\ref{lem:horizontal_criterion}:

\begin{reminder}
\label{reminder:combinatorial_horizontal_criterion}
Let $\SBuilding$ be a spherical building with north pole $n$. Let $\Vertex \in \SBuilding$ be a vertex and let $\Chamber$ be a chamber that contains $\Vertex$. These are equivalent:
\begin{enumerate}
\item $\Vertex \in \SBuilding\Hor$.\label{item:crem_vertex_in_horizontal_link}
\setcounter{enumi}{2}
\item $\typ \Vertex$ and $\typ \AltVertex$ lie in different connected components of $\typ \SBuilding$ for every non-equatorial vertex $\AltVertex$ of $\Chamber$.\label{item:crem_non_equatorial_not_in_same_component}
\end{enumerate}
\end{reminder}

As in the last paragraph, we need a statement about the compatibility of north-poles of links of cells that are contained in each other. Let $\Busemann$ be a Busemann function on $\Space$, let $\Another\Cell \ge \Cell$ be flat cells (with respect to $\Busemann$). Then $\Busemann$ defines a north pole in $\Link \Another\Cell$ as well as in $\Link\Cell$. If $\BigCell$ is a coface of $\Another\Cell$, then $\BigCell \direction \Cell$ is equatorial if and only if $\BigCell$ is flat if and only if $\BigCell \direction \Another\Cell$ is equatorial.

In the above situation $\typ \Link \Another\Cell$ can be considered as the sub-diagram obtained from $\typ \Link \Cell$ by removing $\typ \Another\Cell$ (or more precisely $\typ (\Another\Cell \direction \Cell)$). What we have just seen is:
\begin{obs}
Let $\Chamber$ be a chamber that contains flat cells $\Another\Cell \ge \Cell$. If $\typ \Vertex = \typ \Another\Vertex$ for vertices $\Vertex$ of $\Chamber \direction \Cell$ and $\Another\Vertex$ of $\Chamber \direction \Another\Cell$, then $\Vertex$ is equatorial if and only if $\Another\Vertex$ is.\qed
\end{obs}
So once we have chosen a chamber $\Chamber$ and a non-empty flat face $\Cell$, we may think of being equatorial as a property of the nodes of $\typ\Link\Cell$.

We come to the proof of Proposition \ref{prop:horizontal_properties}. The first statement is again trivial.

\begin{proof}[Proof of Proposition~\ref{prop:horizontal_properties}~\eqref{item:join}]
Let $\Chamber$ be a chamber that contains $\BigCell \vee \Another\Cell$. Let $\Vertex$ be a non-equatorial vertex of $\Chamber \direction \Cell \vee \Another\Cell$ and let $\AltVertex$ be a vertex of $(\BigCell \vee \Another\Cell) \direction (\Cell \vee \Another\Cell)$. Note that $\typ (\Link \Cell \vee \Another\Cell)$ is obtained from $\typ \Link\Cell$ by removing $\typ \Another\Cell$.

Assume that there were a path in $\typ \Link (\Cell \vee \Another\Cell)$ that connects $\typ\AltVertex$ to $\typ\Vertex$. Then this would, in particular, be a path in $\typ \Link \Cell$ from a vertex of $\typ (\BigCell \direction \Cell)$ to the type of a non-equatorial vertex. But by Reminder~\ref{reminder:combinatorial_horizontal_criterion} there cannot be such a path because $\BigCell \horizontal \Cell$. Hence Reminder~\ref{reminder:combinatorial_horizontal_criterion} implies that $\BigCell \vee \Another\Cell \horizontal \Cell \vee \Another\Cell$.
\end{proof}

\begin{proof}[Proof of Proposition~\ref{prop:horizontal_properties}~\eqref{item:transitivity}]
Let $\Chamber$ be a chamber that contains $\BigCell$, let $\Vertex$ be a non-equatorial vertex of $\Chamber \direction \Cell$ and let $\AltVertex$ be a vertex of $\BigCell \direction \Cell$. Note that $\typ \Link \Another\Cell$ is obtained from $\typ \Link\Cell$ by removing $\typ \Another\Cell$.

Assume that there were a path in $\typ \Link \Cell$ from $\typ \Vertex$ to $\typ \AltVertex$. Then either this path does not meet $\typ \Another\Cell$, thus lying entirely in $\typ \Link \Another\Cell$ and therefore contradicting $\BigCell \horizontal \Another\Cell$ by Reminder~\ref{reminder:combinatorial_horizontal_criterion}. Or there would be a first vertex in the path that lies in $\typ \Another\Cell$, say $\typ \Another\AltVertex$.  Then the subpath from $\typ\Vertex$ to $\typ\Another\AltVertex$ would lie in $\typ\Link\Cell$ and thus contradict $\Another\Cell \horizontal \Cell$ by Reminder~\ref{reminder:combinatorial_horizontal_criterion}. Since no such path exists, Reminder~\ref{reminder:combinatorial_horizontal_criterion} implies that $\BigCell \horizontal \Cell$.
\end{proof}

\begin{proof}[Proof of Proposition~\ref{prop:horizontal_properties}~\eqref{item:non_empty_meet}]
Let $\Chamber$ be a chamber that contains $\BigCell$, let $\Vertex$ be a non-equatorial vertex of $\Chamber \direction (\Cell_1 \intersect \Cell_2)$ and let $\AltVertex$ be a vertex of $\BigCell \direction (\Cell_1 \intersect \Cell_2)$. Again $\typ \Link \Cell_i$ is obtained from $\typ \Link (\Cell_1 \intersect \Cell_2)$ by removing $\typ \Cell_i$. Note also that $\typ (\Cell_1 \direction (\Cell_1 \intersect \Cell_2))$ and $\typ (\Cell_2 \direction (\Cell_1 \intersect \Cell_2))$ are disjoint.

Assume that there were a path in $\typ \Link (\Cell_1 \intersect \Cell_2)$ from $\typ \Vertex$ to $\typ \AltVertex$. Let $\typ \Another\AltVertex$ be the first vertex in $\typ (\Cell_1 \direction (\Cell_1 \intersect \Cell_2)) \union \typ (\Cell_2 \direction (\Cell_1 \intersect \Cell_2))$ that the path meets and assume without loss of generality that $\typ \Another\AltVertex \in \typ (\Cell_1 \direction (\Cell_1 \intersect \Cell_2))$. Then the subpath from $\typ \Vertex$ to $\typ \Another\AltVertex$ lies entirely in $\typ \Link \Cell_2$ contradicting $\BigCell \horizontal \Cell_2$ by Reminder~\ref{reminder:combinatorial_horizontal_criterion}. Hence no such path exists and Reminder~\ref{reminder:combinatorial_horizontal_criterion} implies that $\BigCell \horizontal (\Cell_1 \intersect \Cell_2)$.
\end{proof}

\footerlevel{3}
\headerlevel{3}

\section{The Morse Function}
\label{sec:morse_function}

In this section we define the Morse function we will be using. Recall from Section~\ref{sec:height} that the definition of $\Height$ involves a set $\Directions$ of vectors in $\E$. We assume from now on that this set is almost rich.

Then Proposition~\ref{prop:almost_rich_min_vertex} implies that for every cell $\Cell$, the set of points of maximal height form a cell. We denote this cell by $\Roof\Cell$\index[xsyms]{roof@$\Roof\Cell$} and call it the \emph{roof} of $\Cell$. The roof of any cell is clearly a flat cell and the roof of a flat cell is the cell itself.

If $\Cell$ is flat, we can apply the results of the last section with respect to the point at infinity $\Infty\Gradient_\Cell \Height$. Note that the condition that $\Infty\Gradient_\Cell \Height$ be in general position is void because $\OneSpace$ is irreducible. Thus by Lemma~\ref{lem:tau_min} a flat cell $\Cell$ has a unique face $\Cell\Min$\index[xsyms]{sigmamin@$\Cell\Min$} that is minimal with the property that $\Cell$ lies in its horizontal link.

We define the \emph{depth} $\Depth\Cell$\index[xsyms]{dp@$\Depth\Cell$} of a cell $\Cell$ as follows: if $\Cell$ is flat, then $\Depth\Cell$ is the maximal length of a sequence of moves (with respect to $\Infty\Gradient_\Cell\Height$) that starts with $\Cell$, which makes sense by Proposition~\ref{prop:bound_on_moves}. If $\Cell$ is not flat, then $\Depth\Cell \defeq \Depth\Roof\Cell - 1/2$.

Note that if $\BigCell$ is a coface that is flat with respect to a Busemann function centered at $\Infty\Gradient_\Cell \Height$, then it still need not be flat with respect to $\Height$. But the important thing is that if $\BigCell$ is flat with respect to $\Height$, then $\Infty\Gradient_\Cell \Height = \Infty\Gradient_\BigCell \Height$ so, in particular, $\BigCell$ is flat with respect to the Busemann function centered at that point. Therefore $\Cell$ and $\BigCell$ share the same notion of moves. In particular:

\begin{obs}
If $\Cell \le \BigCell$ are flat and there is a move $\Cell \Up \BigCell$ then $\Depth\Cell > \Depth\BigCell$. If there is a move $\BigCell \Down \Cell$, then $\Depth\BigCell > \Depth\Cell$.\qed
\end{obs}

Let $\Subdiv\OneSpace$ be the flag complex of $\OneSpace$. Vertices of $\Subdiv\OneSpace$ are cells of $\OneSpace$. The Morse function $\Morse$ on $\Subdiv\OneSpace$ is defined to be
\begin{align*}
\Morse \colon \Vertices \Subdiv\OneSpace &\to \R \times \R \times \R\\
\Cell &\mapsto (\max \Height|_{\Cell}, \Depth\Cell, \dim\Cell)
\end{align*}
where the range is ordered lexicographically.

Cells of $\Subdiv\OneSpace$ are flags of $\OneSpace$ so, in particular, if $\Cell$ and $\Another\Cell$ are adjacent vertices in $\Subdiv\OneSpace$ then either $\Cell \lneq \Another\Cell$ or $\Another\Cell \lneq \Cell$. So $\dim\Cell \ne \dim\Another\Cell$, which shows that $\Morse$ takes different values on any two adjacent vertices of $\Morse$. Note further that $\max \Height|_\Cell$ is the height of some vertex of $\Cell$ (Observation~\ref{obs:height_convex}). So the first component of the image of $\Morse$ is $\Height(\Vertices \OneSpace)$ which is discrete by Observation~\ref{obs:height_discrete_image}. Dimension and, by Proposition~\ref{prop:bound_on_moves}, depth can take only finitely many values. Thus Observation~\ref{obs:product_order_in_z} shows that the image of $\Morse$ is order-isomorphic to $\Z$ and thus $\Morse$ is indeed a Morse function in the sense of Section~\ref{sec:morse_theory}.

We identify the flag complex $\Subdiv\OneSpace$ with the barycentric subdivision of $\OneSpace$ by identifying $\Cell$ with its barycenter $\Subdiv\Cell$. Nevertheless, as far as $\Subdiv\OneSpace$ is concerned, we are only interested in combinatorial, not in geometric, properties so instead of $\Subdiv\Cell$ we might as well take any other interior point of $\Cell$. The asset of this identification is essentially to make the following distinction: if we write $\Link \Cell$, we mean the link of the cell $\Cell$ in $\OneSpace$. If we write $\Link \Subdiv\Cell$, we mean the link of the vertex $\Subdiv\Cell$ in $\Subdiv\OneSpace$. Here $\Link\Subdiv\Cell$ is the combinatorial link, i.e., the poset of cofaces of $\Subdiv\Cell$ (which can be identified to the full subcomplex of $\Subdiv\OneSpace$ of vertices adjacent to $\Subdiv\Cell$). A join decomposition of $\Link\Subdiv\Cell$ is understood to be a join decomposition of simplicial complexes.

Let $\Cell \subseteq \OneSpace$ be a cell. The link of its barycenter in $\OneSpace$ decomposes as $\Link_{\OneSpace} \Subdiv\Cell = \partial \Cell * \Link \Cell$ by \eqref{eq:point_link_decomposition}. Passage to the barycentric subdivision $\Subdiv{\OneSpace}$ induces a barycentric subdivision on each of the join factors:
\begin{equation}
\label{eq:barycenter_link_decomposition}
\Link_{\Subdiv\OneSpace} \Subdiv\Cell = \Link\FacePart \Subdiv\Cell * \Link\CofacePart \Subdiv\Cell
\end{equation}
where $\Link\FacePart \Subdiv\Cell$ is the barycentric subdivision of $\partial \Cell$ and called the \emph{face part} and $\Link\CofacePart \Subdiv\Cell$ is the barycentric subdivision of $\Link \Cell$ and called the \emph{coface part} of $\Link \Subdiv\Cell$ and the join is a simplicial join.

The \emph{descending link} $\Link\Descending \Subdiv\Cell$ of a vertex $\Subdiv\Cell$ is the full subcomplex of $\Link \Subdiv\Cell$ of vertices $\Subdiv\Cell'$ with $\Morse(\Subdiv\Cell') < \Morse(\Subdiv\Cell)$ (see Section~\ref{sec:morse_theory}). Since the descending link is a full subcomplex, the decomposition \eqref{eq:barycenter_link_decomposition} immediately induces a decomposition
\begin{equation}
\label{eq:barycenter_descending_link_decomposition}
\Link\Descending \Subdiv\Cell = \Link\Descending\FacePart \Subdiv\Cell * \Link\Descending\CofacePart \Subdiv\Cell
\end{equation}
where of course $\Link\Descending\FacePart \Subdiv\Cell = \Link\Descending \Subdiv\Cell \intersect \Link\FacePart \Subdiv\Cell$ and $\Link\Descending\CofacePart \Subdiv\Cell = \Link\Descending \Subdiv\Cell \intersect \Link\CofacePart \Subdiv\Cell$.

\footerlevel{3}
\headerlevel{3}

\section{More Spherical Subcomplexes of Spherical Buildings}
\label{sec:spherical_subcomplexes}

Before we analyze the descending links of our Morse function and finish the proof of Theorem~\ref{thm:one_place_geometric}, we have to extend the class of subcomplexes of spherical buildings which we know to be highly connected slightly beyond hemisphere complexes.

\begin{obs}
\label{obs:cell_retraction}
Let $M \defeq \ModelSpace{\kappa}^m$ be some model space. Let $P \subseteq M$ be a compact polyhedron that is not all of $M$. Let $U \subseteq M$ be a proper open and convex subset. If $P \intersect U \ne \emptyset$, then $P \setminus U$ strongly deformation retracts onto $(\partial P) \setminus U$.
\end{obs}

\begin{rem}
The statement about compact polyhedra may seem a little strange. It applies to polytopes if $\kappa \le 0$ and to arbitrary polyhedra that are not the whole space if $\kappa >0$. This ensures that the boundary of $P$ actually bounds $P$.
\end{rem}

\begin{proof}
Since $U$ is open, the intersection $U \intersect P$ contains a (relatively) interior point of $P$. Let $\Point$ be such a point. The geodesic projection $P \to \partial P$ away from $\Point$ takes $P \setminus U$ onto $(\partial P) \setminus U$ because $U$ is convex.
\end{proof}

\begin{prop}
\label{prop:complex_retraction}
Let $\Lambda$ be an $\ModelSpace{\kappa}$-polytopal complex. Let $U \subseteq \Lambda$ be an open subset of $\Lambda$ that intersects each cell in a convex set. Then there is a strong deformation retraction
\[
\rho \colon \Lambda \setminus U \to \Lambda(\Lambda \setminus U)
\]
from the complement of $U$ onto the subcomplex supported by that complement.
\end{prop}

\begin{proof}
The proof is inductively over the skeleta of $\Lambda$: For $i \in \N$ we show that $\Lambda^{(i)} \setminus U$ strongly deformation retracts onto $\Lambda^{(i)}(\Lambda^{(i)} \setminus U) \union (\Lambda^{(i-1)} \setminus U)$. For $i=0$ there is nothing to show. For $i>0$ we apply Observation~\ref{obs:cell_retraction} to each $i$-cell that meets $U$ but is not contained in it to obtain the desired strong deformation retraction.

To deduce the statement is now routine, cf.\ for example \cite[Proposition~0.16]{hatcher}: If $H_i$ is the strong deformation retraction from $\Lambda^{(i)}(\Lambda^{(i)} \setminus U)$ to $\Lambda^{(i)}(\Lambda^{(i)} \setminus U) \union (\Lambda^{(i-1)} \setminus U)$, we obtain a deformation retraction $H$ from $\Lambda \setminus U$ onto $\Lambda(\Lambda \setminus U)$ by performing $H_i$ in time $[1/2^{i}, 1/2^{i-1}]$. Continuity in $0$ follows from the fact that for a cell $\Cell \subseteq \Lambda^{(i)}$ the retraction $H$ is constant on $\Cell \times [0,1/2^i]$.
\end{proof}

\begin{prop}
\label{prop:coconvex_complexes_spherical}
Let $\SBuilding$ be a spherical building and let $\Chamber \subseteq \SBuilding$ be a chamber. Let $U \subseteq \SBuilding$ be an open subset such that for every apartment $\Apartment$ that contains $\Chamber$ the intersection $U \intersect \Apartment$ is a proper convex subset of $\Apartment$. Then the set $E \defeq \SBuilding \setminus U$ as well as the subcomplex $\SBuilding(E)$ supported by it are $(\dim \SBuilding - 1)$-connected.
\end{prop}

\begin{proof}
First note that $E$ and $\SBuilding(E)$ are homotopy equivalent by Proposition~\ref{prop:complex_retraction}, so it suffices to prove the statement for $E$.

We have to contract spheres of dimensions up to $\dim\SBuilding-1$. Let $S \subseteq E$ be such a sphere. Since $S$ is compact in $\SBuilding$, it is covered by a finite family of apartments that contain $\Chamber$. We apply \cite[Lemma~3.5]{vheydebreck03} to obtain a finite sequence $\Apartment_1, \ldots, \Apartment_k$ of apartments that satisfies the following three properties: each $\Apartment_i$ contains $\Chamber$, the sphere $S$ is contained in the union $\Union_i\SApartment_i$, and for $i \ge 2$ the intersection $\Apartment_i \intersect (\Apartment_1 \union \cdots \union\Apartment_{i-1})$ is a union of roots, each of which contains $\Chamber$.

For $1 \le i \le k$ set $\Lambda_i \defeq \Apartment_1 \union \cdots \union \Apartment_i$ so that $S \subseteq \Lambda_k \setminus U$. Then $\Lambda_i$ is obtained from $\Lambda_{i-1}$ by gluing in the set $A_i \defeq \overline{\Apartment_i \setminus (\Apartment_1 \union \cdots \union \Apartment_{i-1})}$ along its boundary. Note that $A_i$ is an $n$-dimensional polyhedron.

Now we study how the inductive construction above behaves when $U$ is cut out. To start with, $\Lambda_1 \setminus U = \Apartment_1 \setminus U$ is contractible or a $(\dim \SBuilding)$-sphere. The space   $\Lambda_i \setminus U$ is obtained from $\Lambda_{i-1}\setminus U$ by gluing in $A \setminus U$ along $(\partial A) \setminus U$. If $A$ and $U$ are disjoint, then this is gluing in an $n$-cell along its boundary. Otherwise Observation~\ref{obs:cell_retraction} implies that $A \setminus U$ deformation retracts onto $(\partial A) \setminus U$, so that $\Lambda_{i} \setminus U$ is a deformation retract of $\Lambda_{i-1} \setminus U$. In the end, the sphere $S$ can be contracted inside $\Lambda_k \setminus U$.
\end{proof}

\begin{rem}
Proposition~\ref{prop:coconvex_complexes_spherical} has some interesting special cases:
\begin{enumerate}
\item In the case where $U = \emptyset$, the proposition becomes the Solomon--Tits theorem that a spherical building is spherical (in the topological sense).
\item In the case where $U$ is the open $\pi/2$-ball around a point of $\Chamber$, it becomes Schulz' statement that closed hemisphere complexes are spherical (Theorem~\ref{thm:schulz_main}).
\item In fact Schulz proved the proposition in the case where $U$ is convex (see \cite[Theorem~A]{schulz10}) and our proof is an extension of his.
\end{enumerate}
\end{rem}

\footerlevel{3}
\headerlevel{3}

\section{Descending Links}
\label{sec:descending_links}

It remains to show that the descending link of every vertex of $\Subdiv\OneSpace$ is $(\Dimension-1)$-connected. To do so we have to put all the bits that we have amassed in the last sections together. Using \eqref{eq:barycenter_descending_link_decomposition}, we can study the face part and the coface part of $\Subdiv\OneSpace$ separately.

Recall that a cell on which $\Height$ is constant is called \emph{flat}. A flat cell $\BigCell$ has a face $\BigCell\Min$ and we say that $\BigCell$ is \emph{significant} if $\BigCell = \BigCell\Min$. A cell $\BigCell$ that is not significant, i.e., either not flat or flat but not equal to $\BigCell\Min$ is called \emph{insignificant}. Using that $\BigCell$ coincides with its roof $\Roof\BigCell$ if and only if it is flat we can say more concisely that $\BigCell$ is insignificant if $\BigCell \ne \Roof\BigCell\Min$. These cells are called so because:

\begin{lem}
\label{lem:insignificant_descending_link}
If $\BigCell$ is insignificant, then the descending link of $\Subdiv\BigCell$ is contractible. More precisely $\Link\FacePart\Descending \Subdiv\BigCell$ is already contractible.
\end{lem}

\begin{proof}
Consider the full subcomplex $\Lambda$ of $\Link\FacePart \Subdiv\BigCell$ of vertices $\Subdiv\Cell$ with $\Roof{\BigCell}\Min \not\le \Cell \lneq \BigCell$: this is the barycentric subdivision of $\partial \BigCell$ with the open star of $\Roof{\BigCell}\Min$ removed. Therefore it is a punctured sphere and, in particular, contractible. We claim that $\Link\FacePart\Descending \Subdiv\BigCell$ deformation retracts on $\Lambda$.

So let $\Subdiv\Cell$ be a vertex of $\Lambda$. Then
\[
\max \Height|_{\Cell} \le \max \Height|_{\BigCell} = \max \Height|_{\Roof\BigCell\Min}
\]
so $\Height$ either makes $\Subdiv\Cell$ descending or is indifferent. As for depth, the fact that $\Roof\BigCell\Min \not\le \Cell$ implies of course that $\Roof\BigCell\Min \not\le \Roof\Cell$. So there is a move $\Roof\BigCell \Down \Roof\Cell$ which implies $\Depth \Roof\Cell < \Depth\Roof\BigCell - 1/2$. Therefore
\[
\Depth \Cell \le \Depth \Roof\Cell < \Depth\Roof\BigCell - 1/2 \le \Depth\BigCell
\]
so $\Subdiv\Cell$ is descending. This shows that $\Lambda \subseteq \Link\FacePart\Descending \Subdiv\BigCell$.

On the other hand $(\Roof\BigCell\Min)^\circ$ is not descending: Height does not decide because $\max \Height|_{\Roof\BigCell\Min} = \max \Height|_{\Roof\BigCell} = \max\Height|_{\BigCell}$. As for depth we have
\[
\Depth \BigCell \le \Depth \Roof\BigCell \le \Depth \Roof\BigCell\Min \text{ .}
\]
If $\BigCell$ is not flat, then the first inequality is strict. If $\BigCell$ is flat, then there is a move $\Roof\BigCell\Min \Up \Roof\BigCell = \BigCell$ so the second inequality is strict. In either case $(\Roof\BigCell\Min)^\circ$ is ascending.

So geodesic projection away from $(\Roof\BigCell\Min)^\circ$ defines a deformation retraction of $\Link\FacePart\Descending \Subdiv\BigCell$ onto $\Lambda$.
\end{proof}

\begin{lem}
\label{lem:significant_face_part_sphere}
If $\BigCell$ is significant, then all of $\Link\FacePart \Subdiv\BigCell$ is descending. So $\Link\Descending\FacePart \Subdiv\BigCell$ is a $(\dim \BigCell - 1)$-sphere.
\end{lem}

\begin{proof}
Let $\Cell \lneq \BigCell$ be arbitrary. We have $\max \Height|_{\Cell} = \max\Height|_{\BigCell}$ because $\BigCell$ is flat. Moreover $\Cell \lneq \BigCell =\BigCell\Min$ so that, in particular, $\BigCell\Min \not\le \Cell$. Hence there is a move $\BigCell \Down \Cell$ which implies $\Depth \BigCell > \Depth \Cell$ so that $\Subdiv\Cell$ is descending.
\end{proof}

Let $\Cell$ be a significant cell. To study the coface part $\Link\CofacePart \Subdiv\Cell$ it is tempting to argue that $\Link\Cell$ decomposes as $\Link\Ver \Cell * \Link\Hor \Cell$ by \eqref{eq:link_decomposes_as_hor_join_ver} and that this decomposition induces a decomposition of $\Link\CofacePart \Subdiv\Cell$. However this is not the case because $\Link\CofacePart \Subdiv\Cell$ contains barycenters of cells in $\Link\Cell$ that have vertical as well as horizontal vertices. But we will see that the descending coface part does decompose as a join of its horizontal and vertical part. Even better: the set $\Link\CofacePart\Descending \Subdiv\Cell$ is a subcomplex of $\Link \Cell$ and that subcomplex decomposes into its horizontal and vertical part.

Recall Observation~\ref{obs:min_min} and Observation~\ref{obs:significant_either_or}:

\begin{reminder}
\label{reminder:move_reminder}
\begin{enumerate}
\item If $\BigCell$ is flat and $\BigCell\Min \le \Cell \le \BigCell$, then $\Cell\Min = \BigCell\Min$.\label{item:height_min_min}
\item If $\Cell$ is significant and $\BigCell \ge \Cell$ is flat, then there is either a move $\Cell \Up \BigCell$ or a move $\BigCell \Down \Cell$.\label{item:height_either_or}
\end{enumerate}
\end{reminder}

\begin{prop}
\label{prop:descending_coface_link_is_descending_link}
Let $\Cell$ be significant. The descending coface part $\Link\Descending\CofacePart \Subdiv\Cell$ is a subcomplex of $\Link \Cell$. That is, for cofaces $\BigCell \gneq \Another\Cell \gneq \Cell$, if $\Subdiv\BigCell$ is descending then $\Subdiv{\Cell}'$ is descending.
\end{prop}

\begin{proof}
Let $\BigCell \gneq \Another\Cell \gneq \Cell$ and assume that $\Morse(\Subdiv\BigCell) < \Morse(\Subdiv\Cell)$. By inclusion of cells we have
\[
\max \Height|_{\BigCell} \ge \max \Height|_{\Cell'} \ge \max \Height|_{\Cell}
\]
and since $\Subdiv\BigCell$ is descending $\max \Height|_\BigCell \le \max \Height|_\Cell$ so equality holds. Clearly $\dim \BigCell > \dim \Cell$ so since $\Subdiv\BigCell$ is descending we conclude $\Depth\BigCell < \Depth\Cell$.
We have inclusions of flat cells
\[
\Roof\BigCell \ge \Roof\Cell' \ge \Cell \text{ .}
\]
If the second inclusion is equality, then $\Another\Cell \ne \Cell = \Roof\Cell'$ so $\Depth \Another\Cell < \Depth \Roof\Cell' = \Depth \Cell$ and $\Subdiv\Cell'$ is descending. Otherwise $\Roof\BigCell$ is a proper coface of $\Cell$ so by Reminder~\ref{reminder:move_reminder}~\eqref{item:height_either_or} there is a move $\Cell \Up \Roof\BigCell$ or a move $\Roof\BigCell \Down \Cell$. In the latter case we would have $\Depth \BigCell \ge \Depth \Roof\BigCell - 1/2 > \Depth \Cell$ contradicting the assumption that $\Subdiv\BigCell$ is descending. Hence the move is $\Cell \Up \Roof\BigCell$, that is, $\Cell = \Roof\BigCell\Min$. It then follows from Reminder~\ref{reminder:move_reminder}~\eqref{item:height_min_min} that also $\Roof\Cell'{}\Min = \Cell$ so that there is a move $\Cell \Up \Roof\Cell'$. Thus $\Depth \Cell' \le \Depth \Roof\Cell' < \Depth \Roof\Cell$.
\end{proof}

Instead of studying the descending part of the subdivision $\Link\Descending\CofacePart \Subdiv\Cell$ of the link of a significant cell $\Cell$ we may now study the \emph{descending link} $\Link\Descending \Cell$ of $\Cell$ of all cells $\BigCell \direction \Cell$ with $\Morse(\BigCell) < \Morse(\Cell)$.

We define the \emph{horizontal descending link} $\Link\Hor{}\Descending \Cell = \Link\Hor \Cell \intersect \Link\Descending \Cell$ and the \emph{vertical descending link} $\Link\Ver{}\Descending \Cell = \Link\Ver \Cell \intersect \Link\Descending \Cell$. Beware that we do not know yet whether $\Link\Descending \Cell$ decomposes as a join of these two subcomplexes. One inclusion however is clear: $\Link\Descending \Cell \subseteq \Link\Hor{}\Descending \Cell * \Link\Ver{}\Descending \Cell$.

\begin{lem}
\label{lem:significant_vertical_open_hemisphere}
If $\Cell$ is significant, then $\Link\Ver{}\Descending \Cell$ is an open hemisphere complex with north pole $\Gradient_\Cell \Height$.
\end{lem}

\begin{proof}
Let $\Link\OpenHemi \Cell$ denote the open hemisphere complex with north pole $\Gradient_\Cell \Height$. By Corollary~\ref{cor:higher_dimensional_angle_criterion} $\Link\OpenHemi \Cell \subseteq \Link\Ver{}\Descending \Cell$.

Conversely assume that $\BigCell \ge \Cell$ is such that $\BigCell \direction \Cell$ contains a vertex that includes a non-obtuse angle with $\Gradient_\Cell\Height$. Then either
\[
\max \Height|_{\BigCell} = \max \Height|_{\Roof\BigCell} > \max \Height|_{\Cell} 
\]
or $\Roof\BigCell$ is a proper flat coface of $\Cell$. In the latter case since $\Roof\BigCell$ does not lie in the horizontal link of $\Cell$ there is a move $\Roof\BigCell \Down \Cell$ so that
\[
\Depth \BigCell \ge \Depth \Roof\BigCell - \frac{1}{2} > \Depth \Cell \text{ .}
\]
In both cases $\BigCell$ is not descending.
\end{proof}

\begin{obs}
\label{obs:descending_iff_flat}
If $\Cell$ is significant and $\BigCell \ge \Cell$ is such that $\BigCell \direction \Cell \subseteq \Link\Hor \Cell$, then these are equivalent:
\begin{enumerate}
\item $\BigCell$ is flat.
\item $\BigCell$ is descending.
\item $\Height|_{\BigCell} \le \Height(\Cell)$.
\end{enumerate}
\end{obs}

\begin{proof}
If $\BigCell$ is flat then clearly $\max\Height|_\BigCell = \Height(\Cell)$. Moreover $\BigCell\Min = \Cell\Min$ by Reminder~\ref{reminder:move_reminder}~\eqref{item:height_min_min}. Thus there is a move $\Cell = \Cell\Min = \BigCell\Min \Up \BigCell$ so that $\Depth \Cell > \Depth \BigCell$ and $\BigCell$ is descending.

If $\BigCell$ is not flat, then it contains vertices of different heights. Since $\BigCell \direction \Cell$ lies in the horizontal link it in particular includes a right angle with $\Gradient_\Cell \Height$. So by the angle criterion Corollary~\ref{cor:angle_criterion} no vertex has lower height than $\Cell$. Hence $\max \Height|_\BigCell > \max \Height|_\Cell$ and $\BigCell$ is not descending.
\end{proof}

\begin{prop}
\label{prop:descending_coface_link_decomposes}
If $\Cell$ is significant, then the descending link decomposes as a join
\[
\Link\Descending\Cell = \Link\Hor{}\Descending \Cell * \Link\Ver{}\Descending \Cell
\]
of the horizontal descending link and the vertical descending link.
\end{prop}

\begin{proof}
Let $\BigCell_h$ and $\BigCell_v$ be proper cofaces of $\Cell$ such that $\BigCell_h$ lies in the horizontal descending link, $\BigCell_v$ lies in the vertical descending link and $\BigCell \defeq \BigCell_h \vee \BigCell_v$ exists. We have to show that $\BigCell$ is descending.

By Lemma~\ref{lem:significant_vertical_open_hemisphere} $\BigCell_v$ includes an obtuse angle with $\Gradient_\Cell \Height$ so by Proposition~\ref{cor:higher_dimensional_angle_criterion} $\Roof\BigCell_v = \Cell$. On the other hand $\BigCell_h$ is flat by Observation~\ref{obs:descending_iff_flat}. Thus $\Roof\BigCell = \BigCell_h$ so that $\Depth\BigCell = \Depth\BigCell_h - 1/2$ and $\BigCell$ is descending because $\BigCell_h$ is.
\end{proof}

Before we analyze the descending horizontal part, we strengthen our assumption on $\Directions$: we say that $\Directions$ is \emph{rich}\label{page:rich} if $\Vertex - \AltVertex \in \Directions$ for any two vertices $\Vertex$ and $\AltVertex$ whose closed stars meet.

We fix a significant cell $\Cell \subseteq \OneSpace$ and a twin apartment $\PNApartments$ that contains $\Cell$ and $\TheNegPoint$. We set
\[
\UpSet \defeq \{\Vertex \in \Vertices \PosApartment \mid \Vertex\text{ is adjacent to }\Cell \text{ and }\Height(\Vertex) > \Height(\Cell)\}
\]
and let $\UpConvex$ be the convex hull of $\UpSet$.

\begin{obs}
Assume that $\Directions$ is rich. Then $\min \Height|_\UpConvex > \Height(\Cell)$.
\end{obs}

\begin{proof}
Since $\Directions$ is rich it is sufficiently rich for $\UpConvex$. So by Proposition~\ref{prop:sufficiently_rich_min_vertex} $\Height$ attains its minimum over $\UpConvex$ in a vertex, i.e.\ in a point of $\UpSet$. But the elements of $\UpSet$ all have height strictly larger than $\Height(\Cell)$.
\end{proof}

We assume from now on that $\Directions$ is rich. Since $\UpConvex$ is closed, there is an $\varepsilon >0$ such that the $\varepsilon$-neighborhood of $\UpConvex$ in $\PosApartment$ still contains no point of height $\Height(\Cell)$. Fix such an $\varepsilon$ and denote the $\varepsilon$-neighborhood of $\UpConvex$ by $\UpOpen$. Let $\UpLink$ be the set of directions of $\Link_{\PosApartment}\Cell$ toward $\UpOpen$.

\begin{obs}
\label{obs:apartment_flat_iff_disjoint_from_up}
The set $\UpLink$ is open, convex and is such that a coface $\BigCell$ of $\Cell$ that is contained in $\PosApartment$ contains a point of height strictly above $\Height(\Cell)$ if and only if $\BigCell \direction \Cell$ meets $\UpLink$.\qed
\end{obs}

We want to extend this statement to the whole horizontal link of $\Cell$. To do so, we fix a chamber $\NegChamber \subseteq \NegApartment$ that contains $\TheNegPoint$ and set $\PosChamber \defeq \WeylProjection_\Cell \NegChamber$. We set $\Apartment \defeq \Link_{\PosApartment} \Cell$ and $\Chamber \defeq \PosChamber \direction \Cell$.

\begin{obs}
\label{obs:every_apartment_induced}
Let $\AltPNApartments$ be a twin apartment that contains $\Cell$ and $\NegChamber$. Then $\Link_{\PosApartment'}\Cell$ is an apartment of $\Link\Cell$ and every apartment of $\Link\Cell$ that contains $\Chamber$ is of this form.
\end{obs}

\begin{proof}
For the first statement we observe that $\PosChamber$ is contained in $\AltPNApartments$ by Fact~\ref{fact:twin_building}~\eqref{item:twin_projection}.

Let $\Apartment'$ be an apartment of $\Link\Cell$ that contains $\Chamber$. Let $\AltChamber \ge \Cell$ be the chamber such that $\AltChamber \direction \Cell$ is opposite $\Chamber$. Let $\AltPNApartments$ be a twin apartment that contains $\AltChamber$ and $\NegChamber$. Then $\Link_{\PosApartment'} \Cell$ contains the opposite chambers $\AltChamber \direction \Cell$ and $\PosChamber \direction \Cell$ and therefore equals $\Apartment$.
\end{proof}

\begin{obs}
\label{obs:isometry_preserves_height}
Let $\AltPNApartments$ be a twin apartment that contains $\TheNegPoint$. Every isomorphism $\AltPNApartments \to \PNApartments$ that takes $\TheNegPoint$ to itself preserves height.\qed
\end{obs}

This observation is of course of particular interest to us in the situation where $\AltPNApartments$ contains $\Cell$ and the map takes $\Cell$ to itself. The restriction of the retraction $\Retraction\PNApartments\NegChamber$ to $\AltPNApartments$ is such a map.

Let $\rho \defeq \Retraction{\Apartment}{\Chamber}$ be the retraction of $\Link\Cell$ onto $\Apartment$ centered at $\Chamber$. 

\begin{obs}
\label{obs:retraction_and_projection_commute}
Let $\AltPNApartments$ be a twin apartment that contains $\NegChamber$ and $\Cell$ and let $\Apartment' = \Link_{\PosApartment'}\Cell$. The diagram
\begin{diagram}
\AltPNApartments & \rTo^{\Retraction{\PNApartments}{\NegChamber}} & \PNApartments \\
 \dTo & & \dTo \\
\Apartment' & \rTo^{\Retraction{\Apartment}{\Chamber}} & \Apartment\text{ ,}
\end{diagram}
where the vertical maps are the projections onto the link, commutes.\qed
\end{obs}

Let $\UpFat \defeq \rho^{-1}(\UpLink)$.

\begin{lem}
\label{lem:flat_iff_disjoint_from_up}
The set $\UpFat$ is open and meets every apartment of $\Link\Cell$ that contains $\Chamber$ in a convex set. Moreover it  has the property that if $\BigCell$ is a coface of $\Cell$ such that $\BigCell \direction \Cell \subseteq \Link\Hor\Cell$, then $\BigCell$ is flat if and only if $\BigCell \direction \Cell$ is disjoint from $\UpFat$.
\end{lem}

\begin{proof}
That $\UpFat$ is open is clear from continuity of $\rho$. If $\Apartment'$ is an apartment of $\Link\Cell$ that contains $\Chamber$, then $\UpFat \intersect \Apartment' = \rho|_{\Apartment'}^{-1}(\UpLink)$ is the isometric image of $\UpLink$ which is convex.

Let $\BigCell \ge \Cell$ be such that $\BigCell \direction \Cell \subseteq \Link\Hor\Cell$. Let $\Apartment'$ be an apartment that contains $\BigCell$ and $\Chamber$. By Observation~\ref{obs:every_apartment_induced} there is a twin apartment $\AltPNApartments$ that contains $\NegChamber$ and $\Cell$ such that $\Apartment' = \Link_{\PosApartment'}\Cell$. Moreover by Observation~\ref{obs:retraction_and_projection_commute} $\rho$ is induced by the retraction $\Retraction{\PNApartments}{\NegChamber}$ which is height preserving by Observation~\ref{obs:isometry_preserves_height}. Hence $\BigCell$ is flat if and only if $\Retraction{\PNApartments}{\NegChamber}(\BigCell)$ is. And $\BigCell \direction \Cell$ meets $\UpFat$ if and only if $\rho(\BigCell \direction \Cell)$ meets $\UpLink$. Thus the statement follows from Observation~\ref{obs:apartment_flat_iff_disjoint_from_up}.
\end{proof}

\begin{lem}
\label{lem:significant_horizontal_spherical}
If $\Cell$ is significant, then $\Link\Hor{}\Descending \Cell$ is spherical.
\end{lem}

\begin{proof}
Let $\BigCell \ge \Cell$ be such that $\BigCell \direction \Cell \subseteq \Link\Hor\Cell$. By Observation~\ref{obs:descending_iff_flat} $\BigCell$ is descending if and only if it is flat. And by Lemma~\ref{lem:flat_iff_disjoint_from_up} this is the case if and only if $\BigCell$ is disjoint from $\UpFat$. In other words, the horizontal descending link is the full subcomplex of the horizontal link supported by the complement of $\UpFat$. The statement now follows from Proposition~\ref{prop:coconvex_complexes_spherical} where the building is taken to be $\Link\Hor\Cell$, the chamber is $\Chamber \intersect \Link\Hor\Cell$, and the subset is $\UpFat \intersect \Link\Hor\Cell$.
\end{proof}

\begin{prop}
\label{prop:significant_descending_link}
Assume that $\Directions$ is rich. If $\Cell$ is significant, then the descending link $\Link\Descending \Subdiv\Cell$ is spherical. If the horizontal link is empty, it is properly spherical.
\end{prop}

\begin{proof}
The descending link decomposes as a join
\[
\Link\Descending \Subdiv\Cell = \Link\Descending\FacePart \Subdiv\Cell * \Link\Ver{}\Descending\Cell * \Link\Hor{}\Descending\Cell
\]
of the descending face part, the vertical descending link, and the horizontal descending link by \eqref{eq:barycenter_descending_link_decomposition}, Proposition~\ref{prop:descending_coface_link_is_descending_link}, and Proposition~\ref{prop:descending_coface_link_decomposes}. The descending face part is a sphere by Lemma~\ref{lem:significant_face_part_sphere}. The descending vertical link is an open hemisphere complex by Lemma~\ref{lem:significant_vertical_open_hemisphere} which is properly spherical by Theorem~\ref{thm:schulz_main}. The horizontal descending link is spherical by Lemma~\ref{lem:significant_horizontal_spherical}.
\end{proof}

\footerlevel{3}
\headerlevel{3}

\section{Proof of the Main Theorem for $\GroupScheme(\F_q[t])$}
\label{sec:main_theorem}

After we have established the sphericity of the descending links, the finiteness length of $\Group$ follows by standard arguments. We want to apply Brown's criterion. For the finiteness of cell stabilizers the following will be useful:

\begin{lem}
\label{lem:double_stabilizer_finite}
Let $\PNBuildings$ be a locally finite twin building and let $\PosCell \subseteq \PosBuilding$ and $\NegCell\subseteq\NegBuilding$ be cells. The pointwise stabilizer of $\PosCell \union \NegCell$ in the full automorphism group of $\PNBuildings$ is finite.
\end{lem}

\begin{proof}
If $\PosChamber$ and $\NegChamber$ are opposite chambers then the pointwise stabilizer of $\PosChamber$, $\NegChamber$ and all chambers adjacent to $\NegChamber$ is trivial by Theorem~5.205 of \cite{abrbro} which also applies to twin buildings by Remark~5.208. Since the building is locally finite, this implies that the stabilizer of two opposite chambers is finite. Local finiteness then also implies that the stabilizer of any two cells in distinct halves of the twin building is finite.
\end{proof}

\begin{thm}
\label{thm:one_place_geometric}
Let $\PNBuildings$ be an irreducible, thick, locally finite Euclidean twin building of dimension $\Dimension$. Let $\BigGroup$ be a group that acts strongly transitively on $\PNBuildings$ and assume that the kernel of the action is finite. Let $\TheNegPoint \in \NegBuilding$ be a point and let $\Group \defeq \BigGroup_\TheNegPoint$ be the stabilizer of $\TheNegPoint$. Then $\Group$ is of type $F_{\Dimension-1}$ but not of type $F_{\Dimension}$.
\end{thm}

\begin{proof}
Set $\OneSpace \defeq \PosBuilding$ and consider the action of $\Group$ on the barycentric subdivision $\Subdiv\OneSpace$. We want to apply Corollary~\ref{cor:adapted_browns_criterion} and check the premises. The space $\Subdiv\OneSpace$ is \CAT{0} hence contractible.

The stabilizer of a cell $\Cell$ of $\OneSpace$ in $\Group$ is the simultaneous stabilizer of $\Cell$ and the carrier of $\TheNegPoint$ in $\BigGroup$. Since the stabilizer of these two cells in the full automorphism group of the twin building is finite (by the lemma above) and the action of $\BigGroup$ has finite center, this stabilizer is finite. That the stabilizer of a cell of $\Subdiv\OneSpace$ is then also finite is immediate.

Let $\Morse$ be the Morse function on $\Subdiv\OneSpace$ as defined in Section~\ref{sec:morse_function} based on a rich set of directions $\Directions$. Its sublevel sets are $\Group$-invariant subcomplexes. The group $\Group$ acts transitively on points opposite $\TheNegPoint$ by strong transitivity of $\BigGroup$. Since $\OneSpace$ is locally finite, this implies that $\Group$ acts cocompactly on any sublevel set of $\Morse$.

The descending links of $\Morse$ are $(\dim \Dimension - 1)$-spherical by Lemma~\ref{lem:insignificant_descending_link} and Proposition~\ref{prop:significant_descending_link}. If $\Cell$ is significant then the descending link of $\Subdiv\Cell$ is properly $(\dim \Dimension - 1)$-spherical provided the horizontal part is empty. This is the generic case and happens infinitely often.

Applying Corollary~\ref{cor:morse_theory} we see that the induced maps $\pi_i(\Space_k) \to \pi_i(\Space_{k+1})$ are isomorphisms for $0 \le i < n-2$ and are surjective and infinitely often not injective for $i = n-1$. So it follows from Corollary~\ref{cor:adapted_browns_criterion} that $\Group$ is of type $F_{n-1}$ but not $F_n$.
\end{proof}

Now we make the transition to $S$-arithmetic groups based on Appendix~\ref{chap:affine_kac-moody_groups}:

\begin{thm}
\label{thm:one_place_arithmetic}
Let $\GroupScheme$ be a connected, noncommutative, almost simple $\F_q$-group of $\F_q$-rank $n \ge 1$. The group $\GroupScheme(\F_q[t])$ is of type $F_{n-1}$ but not of type $F_n$.
\end{thm}

\begin{proof}
Let $\tilde{\GroupScheme}$ be the universal cover of $\GroupScheme$ (see Proposition~2.24 and Dé\-fi\-ni\-tion~2.25 of \cite{bortit72}). Let $\PNBuildings$ be the thick locally finite irreducible $n$-dimensional Euclidean twin building associated to $\tilde{\GroupScheme}(\F_q[t,t^{-1}])$ by Proposition~\ref{prop:twin_building_of_kac-moody_group}. Since the isogeny $\tilde{\GroupScheme}(\F_q[t,t^{-1}]) \to \GroupScheme(\F_q[t,t^{-1}])$ is central, the action of $\tilde{\GroupScheme}(\F_q[t,t^{-1}])$ on the twin building factors through it. Let $\Group$ be the image of $\tilde{\GroupScheme}(\F_q[t])$ under the map $\tilde{\GroupScheme}(\F_q[t]) \to \GroupScheme(\F_q[t])$. By \cite[Satz~1]{behr68} $\Group$ has finite index in $\GroupScheme(\F_q[t])$, hence both have the same finiteness length.

Fact~\ref{fact:twin_halves_identification} shows that $\NegBuilding$ may be regarded as the Bruhat-Tits building associated to $\tilde{\GroupScheme}(\F_q((t)))$. The compact subring of integers of $\F_q((t))$ is $\F_q[[t]]$. Thus $\tilde{\GroupScheme}(\F_q[[t]])$ is a maximal compact subgroup of $\tilde{\GroupScheme}(\F_q((t)))$ hence the stabilizer of a vertex $\Vertex \in \NegBuilding$ in $\tilde{\GroupScheme}(\F_q((t)))$. Consequently, $\tilde{\GroupScheme}(\F_q[t]) = \tilde{\GroupScheme}(\F_q[t,t^{-1}]) \intersect \tilde{\GroupScheme}(F_q[[t]])$ is the stabilizer of $\Vertex$ in $\tilde{\GroupScheme}(\F_q[t,t^{-1}])$. The statement now follows from Theorem~\ref{thm:one_place_geometric}.
\end{proof}

\footerlevel{3}

\footerlevel{2}
\headerlevel{2}

\chapter{Finiteness Properties of $\GroupScheme(\F_q[t,t^{-1}])$}
\label{chap:two_places}

Let $\GroupScheme$ be an $\F_q$-isotropic, connected, noncommutative, almost simple $\F_q$-group. In this chapter we want to determine the finiteness length of $\GroupScheme(\F_q[t,t^{-1}])$. We have already seen in the last chapter that there is a locally finite irreducible Euclidean twin building on which the group acts strongly transitively so in geo\-metric language we have to show:

\begin{xrefthm}{thm:two_places_geometric}
Let $\PNBuildings$ be an irreducible, thick, locally finite Euclidean twin building of dimension $\Dimension$. Let $\Group$ be a group that acts strongly transitively on $\PNBuildings$ and assume that the kernel of the action is finite. Then $\Group$ is of type $F_{2\Dimension-1}$ but not of type $F_{2\Dimension}$.
\end{xrefthm}

We fix an irreducible, thick, locally finite Euclidean twin building $\PNBuildings$ on which a group $\Group$ acts strongly transitively and let $\Dimension$ denote its dimension. We consider the action of $\Group$ on $\TwoSpace \defeq \PosBuilding \times \NegBuilding$. Again we have to construct a $\Group$-invariant Morse-function on $\Space$ with highly connected descending links and cocompact sublevel sets. The construction is very similar to that in the last chapter: essentially the point $\TheNegPoint$ that was fixed there is allowed to vary now.

One difference is that the Euclidean building $\Space$ is not irreducible any more so this time we actually use the greater generality of Section~\ref{sec:moves} compared to \cite[Section~5]{buxwor08}.

A technical complication concerns the analysis of the horizontal descending links. To describe it we note first:

\begin{obs}
\label{obs:two_places_twin_apartments}
If $\PNApartments$ is a twin apartment of $\PNBuildings$, then $\PosApartment \times \NegApartment$ is an apartment of $\TwoSpace$. For every chamber of $\TwoSpace$ there is an apartment of this form that contains it.
\end{obs}

\begin{proof}
Let $\Chamber \subseteq \TwoSpace$ be a chamber. Write $\Chamber = \PosChamber \times \NegChamber$ with $\PosChamber \subseteq \PosBuilding$ and $\NegChamber \subseteq \NegBuilding$. If $\PNApartments$ is a twin apartment that contains $\PosChamber$ and $\NegChamber$ then $\PosApartment \times \NegApartment$ contains $\Chamber$.
\end{proof}

However the set of apartments of $\PNBuildings$ that arise in the above form from twin apartments is far from being an apartment system for $\TwoSpace$: in fact if $\Chamber$ is $\PosChamber \times \NegChamber$ with $\PosChamber \op \NegChamber$ then the apartment $\PosApartment \times \NegApartment$ that contains it and comes from a twin apartment is unique. In particular, if $\Cell$ is a cell and $\Chamber \ge \Cell$ is a chamber, the link of $\Cell$ is generally not covered by apartments that are induced from twin apartments that contain $\Chamber$. Before we can translate the argument for the horizontal descending link from Section~\ref{sec:descending_links} we will therefore have to extend the class of apartments to study. For the time being however, the apartments coming from twin apartments will suffice.

This chapter is based on \cite{witzel10} and reproduces the proof given there. The vocabulary has been adjusted to match that of \cite{buxgrawit10b}.

\headerlevel{3}

\section{Height}
\label{sec:two_places_height}

As before let $\Weyl$ be the spherical Coxeter group associated to $\Infty\PosBuilding$ which is the same as that of $\Infty\NegBuilding$ and let $\E$ be a Euclidean vector space of dimension $n = \dim \PosBuilding = \dim \NegBuilding$ on which $\Weyl$ acts faithfully as a linear reflection group. Let $\PNApartments$ be a twin apartment of $\PNBuildings$. We may as before identify $\E$ with $\PosApartment$ as well as with $\NegApartment$ in such a way that the $\Weyl$-structure at infinity is respected. In fact the opposition relation induces a bijection $\PosApartment \stackrel{\op}{\longleftrightarrow} \NegApartment$ so there is a natural way to make both identifications at the same time. To prevent confusion we will this time make the identifications explicit by choosing maps $\EIdent_\varepsilon \colon \Apartment_\varepsilon \to \E$ such that the following diagram commutes:

\begin{diagram}[LaTeXeqno]
\label{dia:twin_identification}
\PosApartment & & \lToFro^{\op} & & \NegApartment\\
 & \rdTo_{\PosEIdent} & & \ldTo_{\NegEIdent} & \\
 && \E &&
\end{diagram}

With these identifications the metric codistance of two points $\PosPoint \in \PosApartment$ and $\NegPoint \in \NegApartment$ is $\EuclCoDistance(\PosPoint,\NegPoint) = \EuclDistance(\PosEIdent(\PosPoint),\NegEIdent(\NegPoint))$. In other words it is the length of the vector $\PosEIdent(\PosPoint) - \NegEIdent(\NegPoint)$.

This is the first occurrence of the projection
\begin{align*}
\Difference \colon \E \times \E  &\to \E\\
(\Point,\AltPoint) &\mapsto \Point - \AltPoint \text{ .}
\end{align*}
It will turn out that even though $\TwoSpace$ is $2\Dimension$-dimensional, most problems are essentially $\Dimension$-dimensional because the height function apartment-wise factors through $\pi$.

For two finite subsets $\Directions_1$ and $\Directions_2$ of $\E$ we define\index[xsyms]{plus@$\Plus$}
\begin{equation*}
\Directions_1 \Plus \Directions_2 \defeq (\Directions_1 + \Directions_2) \union \Directions_1 \union \Directions_2 \text{ .}
\end{equation*}
Note that $\Directions_1 \Plus \Directions_2 = ((\Directions_1 \union \{0\}) + (\Directions_2 \union \{0\})) \setminus \{0\}$ if $\Directions_1$ and $\Directions_2$ do not contain $0$. Recall from Section~\ref{sec:zonotopes} that a set $\Directions$ is sufficiently rich for a polytope $\Cell$ if $\Vertex - \AltVertex \in \Directions$ for any two vertices of $\Cell$. With the above notation we get:

\begin{obs}
\label{obs:fold_sufficiently_rich}
Let $\Cell_1$ and $\Cell_2$ be polytopes in $\E$ and let $\Cell$ be the convex hull of some of the vertices of $\Cell_1 \times \Cell_2$.
If $\Directions_1$ is sufficiently rich for $\Cell_1$ and $\Directions_2$ is sufficiently rich for $\Cell_2$ then $\Directions_1 \Plus \Directions_2$ is sufficiently rich for $\Difference(\Cell)$. In particular, it is sufficiently rich for $\Cell_1 - \Cell_2$.
\end{obs}

\begin{proof}
Every vertex of $\Difference(\Cell)$ is of the form $\Vertex_1 - \Vertex_2$ for vertices $\Vertex_i$ of $\Cell_i$. So if $\Vertex$ and $\AltVertex$ are distinct vertices of $\Cell$, then $\Vertex - \AltVertex = (\Vertex_1 - \AltVertex_1) + (\AltVertex_2 - \Vertex_2)$ where $\Vertex_i$ and $\AltVertex_i$ may or may not be distinct. If they are distinct for $i=1,2$, then $\Vertex - \AltVertex \in \Directions_1 + \Directions_2$. If $\AltVertex_i = \Vertex_i$ for some $i$, then $\Vertex - \AltVertex \in \Directions_{3-i}$. In any case $\Vertex - \AltVertex \in \Directions_1 \Plus \Directions_2$. The last statement is obtained by taking $\Cell = \Cell_1 \times \Cell_2$.
\end{proof}

Let $\Directions \subseteq \E$ be finite, $\Weyl$-invariant and centrally symmetric. As before we will eventually require $\Directions$ to be rich but for the moment no such assumption is made.

Let $\Zonotope \defeq \Zonotope(\Directions \Plus \Directions)$ be the zonotope as defined in Section~\ref{sec:zonotopes}. The height function $\Height$ that we consider on $\TwoSpace$ is just $\Zonotope$-perturbed codistance (see Section~\ref{sec:codistance})\index[xsyms]{height@$\Height$}:
\[
\Height \defeq \EuclCoDistance[\Zonotope] \text{ .}
\]

\begin{obs}
\label{obs:two_places_height_factors_through_difference}
Let $\Point = (\PosPoint,\NegPoint) \in \TwoSpace$ and let $\PNApartments$ be a twin apartment such that $\PosApartment \times \NegApartment$ contains $\Point$. Then $\Height(\Point) = \EuclDistance(\Difference(\PosEIdent(\PosPoint),\NegEIdent(\NegPoint)),\Zonotope)$ with identifications as in \eqref{dia:twin_identification}.\qed
\end{obs}

Note that Observation~\ref{obs:two_places_twin_apartments} implies that if $\Path$ is a path in $\TwoSpace$ then there is an apartment $\PosApartment\times\NegApartment$ with $\PNApartments$ a twin apartment that contains an initial segment of $\Path$. If $\Path$ issues at $\Point$ we may interpret this as saying that $\PosApartment\times\NegApartment$ contains $\Point$ and the direction $\Path_\Point$.

\begin{figure}[!ht]
\begin{center}
\includegraphics{figs/zero_level}
\end{center}
\caption{The set $\ZeroLevel$ in an apartment $\Apartment \defeq \PosApartment \times \NegApartment$ where $\PNApartments$ is a twin apartment.}
\label{fig:zero_level}
\end{figure}

Let $\ZeroLevel \defeq \Height^{-1}(0)$ be the set of points of height $0$. If $\PNApartments$ is a twin apartment and identifications as in \eqref{dia:twin_identification} are made, the set $\ZeroLevel \intersect (\PosApartment \times \NegApartment)$ is the set of points $(\APosPoint,\ANegPoint)$ with $\PosEIdent(\APosPoint)-\NegEIdent(\ANegPoint) \in \Zonotope$ which is a strip along the ``diagonal'' $\{(\PosPoint,\NegPoint) \mid \PosPoint \op \NegPoint\}$ (see Figure~\ref{fig:zero_level}).

Let $\Point = (\PosPoint,\NegPoint)$ be a point of $\PosApartment\times\NegApartment$. Let
\[
\Zonotope_+ \defeq (\opm{\PNApartments} \NegPoint) + \PosEIdent^{-1}(\Zonotope) \quad \text{and} \quad \Zonotope_- \defeq (\opm{\PNApartments} \PosPoint) + \NegEIdent^{-1}(\Zonotope) \text{ ,}
\]
where $\opm{\PNApartments}$ denotes the map that assigns to a point of $\PNApartments$ its opposite point in $\PNApartments$.

The sets $\Zonotope_+ \times \{\NegPoint\}$ and $\Zonotope_- \times \{\PosPoint\}$ are slices of $\ZeroLevel \intersect (\PosApartment \times \NegApartment)$ and by definition, $\Height(\PosPoint,\NegPoint)$ is the distance to either one of them, i.e., the distance to $(\ClosestPointProjection[\Zonotope_+] \PosPoint,\NegPoint)$ and to $(\PosPoint,\ClosestPointProjection[\Zonotope_-]\NegPoint)$. The point in $\ZeroLevel \intersect (\PosApartment \times \NegApartment)$ closest to $\PosPoint,\NegPoint$ is the midpoint
\[
\ClosestPointProjection[\ZeroLevel \intersect (\PosApartment \times \NegApartment)] (\PosPoint,\NegPoint) = \frac{1}{2} (\ClosestPointProjection[\Zonotope_+] \PosPoint,\NegPoint) + \frac{1}{2}(\PosPoint,\ClosestPointProjection[\Zonotope_-]\NegPoint)
\]
of these two projection points. This shows:

\begin{obs}
\label{obs:height_is_distance}
If $\Point \in \TwoSpace$ is a point and $\PNApartments$ is a twin apartment such that $\Apartment \defeq \PosApartment \times \NegApartment$ contains $\Point$, then $\Height(\Point) = \sqrt{2} \cdot \EuclDistance(\Point,\ZeroLevel \intersect \Apartment)$.\qed
\end{obs}

That $\Height$ on $\PosApartment \times \NegApartment$ looks like distance to $\ZeroLevel$ (up to a constant factor) immediately suggests that the gradient should be the direction away from $\ZeroLevel$.

Assume that $\Height(\Point) > 0$. Recall that we defined in Section~\ref{sec:height} the $\Zonotope$-perturbed ray from $\PosPoint$ to $\NegPoint$ which is the ray that issues at $\PosPoint$ and moves away from $\Zonotope_+$. It is a well defined ray inside $\PosBuilding$. Let $\TwinRayMap{\PosPoint}{\NegPoint}$ be this ray as a map, i.e., the image of $\TwinRayMap{\PosPoint}{\NegPoint}$ is $\TwinRay{\PosPoint}{\NegPoint}[\Zonotope]$. Analogously let $\TwinRayMap{\NegPoint}{\PosPoint}$ be the $\Zonotope$-perturbed ray from $\NegPoint$ to $\PosPoint$. Then the ray $\Ray^{(\PosPoint,\NegPoint)}$ in $\PosApartment \times \NegApartment$ that issues at $(\PosPoint,\NegPoint)$ and moves away from $\ZeroLevel \intersect (\PosApartment \times \NegApartment)$ is given by
\[
\Ray^{(\PosPoint,\NegPoint)}(t) = \bigg(\TwinRayMap{\PosPoint}{\NegPoint}\Big(\frac{1}{\sqrt{2}}t\Big), \TwinRayMap{\NegPoint}{\PosPoint}\Big(\frac{1}{\sqrt{2}}t\Big)\bigg) \text{ .}
\]
Since $\TwinRayMap{\PosPoint}{\NegPoint}$ and $\TwinRayMap{\NegPoint}{\PosPoint}$ are well-defined rays in $\PosBuilding$ respectively $\NegBuilding$ we get immediately:

\begin{obs}
For every point $\Point \in \TwoSpace$ with $\Height(\Point) > 0$ the ray $\Ray^{\Point}$ is a well-defined ray in $\TwoSpace$ (i.e., independent of the chosen twin apartment). If $\PNApartments$ is a twin apartment such that $\Apartment \defeq \PosApartment \times \NegApartment$ contains $\Point$, then $\Ray^{\Point}$ lies in $\Apartment$ and moves away from $\ZeroLevel \intersect \Apartment$.\qed
\end{obs}

The \emph{asymptotic gradient} $\Infty\Gradient\Height$\index[xsyms]{gradientinfty@$\Infty\Gradient\Height$} of $\Height$ is defined by letting $\Infty\Gradient_\Point\Height$\index[xsyms]{gradientinftypoint@$\Infty\Gradient_\Point\Height$} be the limit of $\Ray^\Point$. Similarly, the \emph{gradient} $\Gradient\Height$\index[xsyms]{gradientnoninfty@$\Gradient\Height$} of $\Height$ is defined by letting $\Gradient_\Point\Height$\index[xsyms]{gradientnoninftypoint@$\Gradient_\Point\Height$} be the direction $(\Ray^\Point)_\Point$ in $\Link\Point$ defined by $\Ray^\Point$.

Recall that the link decomposes as a spherical join $\Link_{\TwoSpace} \Point = \Link_{\PosBuilding} \PosPoint * \Link_{\NegBuilding} \NegPoint$. In this decomposition $\Gradient_\Point\Height$ is the midpoint of the two points $(\TwinRayMap{\PosPoint}{\NegPoint})_{\PosPoint}$ and $(\TwinRayMap{\NegPoint}{\PosPoint})_{\NegPoint}$.

Similarly, the visual boundary of $\TwoSpace$ decomposes as a spherical join $\Infty\TwoSpace = \Infty\PosBuilding * \Infty\NegBuilding$ of irreducible join factors. The asymptotic gradient $\Infty\Gradient_\Point\Height$ is the midpoint of $\Infty{(\TwinRayMap{\PosPoint}{\NegPoint})}$ and $\Infty{(\TwinRayMap{\NegPoint}{\PosPoint})}$. Recall from Section~\ref{sec:moves} that a point at infinity $\PointAtInfty \in \Infty\TwoSpace$ is in general position if it is not contained in any proper join factor. So we have just seen:

\begin{obs}
\label{obs:two_places_gradient_in_general_position}
Let $\Point \in \TwoSpace$ with $\Height(\Point) > 0$. The asymptotic gradient $\Infty\Gradient_\Point\Height$ is in general position.\qed
\end{obs}

\footerlevel{3}
\headerlevel{3}

\section{Flat Cells and the Angle Criterion}
\label{sec:two_places_angle_criterion}

In this section we show how the condition that $\Directions$ is almost rich implies the angle criterion. The argument is entirely parallel to that in Section~\ref{sec:angle_criterion}.

We begin with properties of $\Height$ that hold irrespective of richness such as convexity:

\begin{obs}
\label{obs:two_places_height_is_convex}
Let $\PNApartments$ be a twin apartment. The restriction of $\Height$ to $\PosApartment \times \NegApartment$ is convex. In particular, if $\Cell\subseteq \TwoSpace$ is a cell, then among the $\Height$-maximal points of $\Cell$ there is a vertex.
\end{obs}

\begin{proof}
By Observation~\ref{obs:height_is_distance} the restriction of $\Height$ to $\PosApartment \times \NegApartment$ is, up to a constant, distance from the convex set $\ZeroLevel \intersect (\PosApartment \times \NegApartment)$. The second statement follows by choosing a twin apartment $\PNApartments$ such that $\Cell \subseteq \PosApartment \times \NegApartment$.
\end{proof}

Moreover we have the infinitesimal angle criterion:

\begin{obs}
Let $\Path$ be a path in $\TwoSpace$ that issues at a point $\Point$ with $\Height(\Point)>0$. The function $\Height \circ \Path$ is strictly decreasing on an initial interval if and only if $\angle_{\Point}(\Gradient_\Point\Height,\Path_\Point) > \pi/2$.
\end{obs}

\begin{proof}
Let $\PNApartments$ be a twin apartment such that $\PosApartment \times \NegApartment$ contains an initial interval of $\Path$. The statement follows from the fact that on $\PosApartment \times \NegApartment$ the function $\Height$ up to a constant measures distance from $\ZeroLevel$ and $\Gradient_\Point\Height$ is the direction that points away from $\ZeroLevel$.
\end{proof}

As before we call a cell $\Cell \subseteq \TwoSpace$ \emph{flat} if $\Height|_{\Cell}$ is constant.

\begin{obs}
If $\Cell$ is flat, then the (asymptotic) gradient is the same for all points $\Point$ of $\Cell$. It is perpendicular to $\Cell$.
\end{obs}

\begin{proof}
Let again $\PNApartments$ be a twin apartment such that $\PosApartment \times \NegApartment$ contains $\Cell$. Since the restriction of $\Height$ to $\PosApartment \times \NegApartment$ essentially measures distance to $\ZeroLevel$, the cell $\Cell$ can be of constant height only if its projection onto $\ZeroLevel \intersect (\PosApartment \times \NegApartment)$ is a parallel translate by a vector perpendicular to $\Cell$.
\end{proof}

Let $\Cell$ be a flat cell. We define the \emph{asymptotic gradient} $\Infty\Gradient_\Cell \Height$\index[xsyms]{gradientinftycell@$\Infty\Gradient_\Cell\Height$} of $\Height$ at $\Cell$ to be the asymptotic gradient of any of its interior points. The observation implies that the gradient $\Gradient_\Point \Height$ of an interior point $\Point$ of $\Cell$ is a direction in $\Link \Cell$ and independent of $\Point$. We define the \emph{gradient} $\Gradient_\Cell \Height$\index[xsyms]{gradientnoninftycell@$\Gradient_\Cell\Height$} of $\Height$ at $\Cell$ to be that direction. We take the gradient at $\Cell$ to be the north pole of $\Link \Cell$ and obtain accordingly a \emph{horizontal link} $\Link\Hor \Cell$\index[xsyms]{lkhor@$\Link\Hor\Cell$} and a \emph{vertical link} $\Link\Ver\Cell$\index[xsyms]{lkver@$\Link\Ver\Cell$} and an \emph{open hemisphere link} $\Link\OpenHemi\Cell$.

Let $\PNApartments$ be a twin apartment of $\PNBuildings$ and make identifications as in \eqref{dia:twin_identification}. We say that $\Directions$ is \emph{almost rich} if $\PosEIdent(\Vertex) - \PosEIdent(\AltVertex) \in \Directions$ whenever $\Vertex$ and $\AltVertex$ are vertices of $\PosApartment$ that are contained in a common cell. We say that $\Directions$ is \emph{rich} if $\PosEIdent(\Vertex) - \PosEIdent(\AltVertex) \in \Directions$ whenever $\Vertex$ and $\AltVertex$ are vertices of $\PosApartment$ whose closed stars meet. Note that we might as well have taken vertices in $\NegApartment$ or any other twin apartment instead, since only the Coxeter complex structure matters.

\begin{prop}
\label{prop:two_places_almost_rich_min_vertex}
Assume that $\Directions$ is almost rich and let $\Cell$ be a cell of $\TwoSpace$. Among the $\Height$-minima of $\Cell$ there is a vertex and the set of $\Height$-maxima of $\Cell$ is a face.
The statement remains true if $\Cell$ is replaced by the convex hull of some of its vertices.
\end{prop}

\begin{proof}
Write $\Cell = \PosCell \times \NegCell$ and let $\PNApartments$ be a twin apartment that contains $\PosCell$ and $\NegCell$. Make identifications as in \eqref{dia:twin_identification}. Consider $\bar{\Cell} \defeq \Difference(\PosEIdent(\PosCell),\NegEIdent(\NegCell))$. Since $\Directions$ is sufficiently rich for $\PosEIdent(\PosCell)$ as well as for $\NegEIdent(\NegCell)$ Observation~\ref{obs:fold_sufficiently_rich} shows that $\Directions \Plus \Directions$ is sufficiently rich for $\bar{\Cell}$. Hence we may apply Proposition~\ref{prop:sufficiently_rich_min_vertex} to conclude that among the points of $\bar{\Cell}$ closest to $\Zonotope$ there is a vertex and that the set of points of $\bar{\Cell}$ farthest from $\Zonotope$ form a face.

Observation~\ref{obs:two_places_height_factors_through_difference} implies that $\Height$ on $\Cell$ is nothing but the composition of the affine map $\Difference\circ(\PosEIdent \times \NegEIdent)$ and distance from $\Zonotope$. The result now follows from the fact that the preimage of a face of $\bar{\Cell}$ under $\Difference\circ(\PosEIdent \times \NegEIdent)$ is a face of $\Cell$ (see \cite[Lemma~7.10]{ziegler}).

The second statement is proved analogously.
\end{proof}

The angle criterion follows as before:

\begin{cor}
\label{cor:two_places_angle_criterion}
Assume that $\Directions$ is almost rich. Let $\Vertex$ and $\AltVertex$ be two vertices that are contained in a common cell. The restriction of $\Height$ to $[\Vertex,\AltVertex]$ is monotone. In particular, $\Height(v) > \Height(w)$ if and only if $\angle_\Vertex(\Gradient_\Vertex \Height,\AltVertex) > \pi/2$.
\end{cor}

\begin{cor}
\label{cor:two_places_higher_dimensional_angle_criterion}
Assume that $\Directions$ is almost rich. Let $\Cell$ be a flat cell and $\BigCell \ge \Cell$. Then $\Cell$ is the set of $\Height$-maxima of $\BigCell$ if and only if $\BigCell \direction \Cell \subseteq \Link\OpenHemi\Cell$.
\end{cor}

\footerlevel{3}
\headerlevel{3}

\section{The Morse Function}
\label{sec:two_places_morse_function}

From now on we assume that $\Directions$ is almost rich. Then by Proposition~\ref{prop:two_places_almost_rich_min_vertex} the set of $\Height$-maxima of any cell $\Cell$ is a face which we call the \emph{roof} of $\Cell$ and denote by $\Roof{\Cell}$\index[xsyms]{roof@$\Roof\Cell$}.

Let $\Cell$ be a flat cell. By Observation~\ref{obs:two_places_gradient_in_general_position} $\Infty\Gradient_\Cell\Height$ is in general position. So we may apply the results of Section~\ref{sec:moves} with respect to $\Infty\Gradient_\Cell\Height$. In particular, by Lemma~\ref{lem:tau_min} $\Cell\Min$\index[xsyms]{sigmamin@$\Cell\Min$} exists: the unique minimal face of $\Cell$ in the horizontal link of which it lies.

Also, we can define the \emph{depth} $\Depth\Cell$\index[xsyms]{dp@$\Depth\Cell$} of $\Cell$ to be the maximal length of a sequence of moves (with respect to $\Infty\Gradient_\Cell\Height$) that starts with $\Cell$. It exists by Proposition~\ref{prop:bound_on_moves}. If $\Cell$ is not flat, we define $\Depth \Cell \defeq \Depth \Roof\Cell - 1/2$.

If $\BigCell$ is flat and $\Cell$ is a face, then $\Infty\Gradient_\BigCell = \Infty\Gradient_\Cell$, so:

\begin{obs}
If $\Cell \le \BigCell$ are flat and there is a move $\Cell \Up \BigCell$ then $\Depth\Cell > \Depth\BigCell$. If there is a move $\BigCell \Down \Cell$, then $\Depth\BigCell > \Depth\Cell$.\qed
\end{obs}

Let $\Subdiv\TwoSpace$ be the flag complex of $\TwoSpace$. Note that $\Subdiv\TwoSpace$ is a simplicial complex (as flag complexes always are), even though $\TwoSpace$ is not. We define the Morse function $\Morse$ on $\Subdiv\TwoSpace$ by
\begin{align*}
\Morse \colon \Vertices\Subdiv\TwoSpace & \to \R \times \R \times \R \\
\Cell & \mapsto (\max \Height|_{\Cell},\Depth\Cell,\dim\Cell)
\end{align*}
and order the range lexicographically.

As in Section~\ref{sec:morse_function} one verifies that $\Morse$ is indeed a Morse function in the sense of Section~\ref{sec:morse_theory}.

We identify $\Subdiv\TwoSpace$ with the barycentric subdivision of $\TwoSpace$ and write $\Link \Subdiv\Cell$ to mean the link in $\Subdiv\TwoSpace$ as opposed to $\Link \Cell$ which is the link of the cell $\Cell$ in $\TwoSpace$.

The link of a vertex $\Subdiv\Cell$ of $\Subdiv\TwoSpace$ decomposes as a (simplicial) join
\[
\Link \Subdiv\Cell = \Link\FacePart \Subdiv\Cell * \Link\CofacePart \Subdiv\Cell
\]
of the \emph{face part} and the \emph{coface part}.

The \emph{descending link} $\Link\Descending \Subdiv\Cell$ is the full subcomplex of vertices $\Subdiv\Cell'$ with $\Morse(\Subdiv\Cell') < \Morse(\Subdiv\Cell)$. As a full subcomplex the descending link decomposes as a simplicial join
\begin{equation}
\label{eq:two_places_barycenter_descending_link_decomposition}
\Link\Descending \Subdiv\Cell = \Link\FacePart\Descending \Subdiv\Cell * \Link\CofacePart\Descending \Subdiv\Cell
\end{equation}
of the \emph{descending face part} and the \emph{descending coface part}.

\footerlevel{3}
\headerlevel{3}

\section{Beyond Twin Apartments}
\label{sec:two_places_more_apartments}

Before we proceed to the analysis of the descending links we have to address the problem mentioned in the introduction of the chapter, namely that it does not suffice to understand twin apartments.

To make this more precise consider cells $\PosCell \subseteq \PosApartment$ and $\NegCell \subseteq \NegApartment$ in a twin apartment $\PNApartments$. By a \emph{twin wall} $\Wall$ we mean a pair of walls $\PosWall$ of $\PosApartment$ and $\NegWall$ of $\NegApartment$ such that $\PosWall$ is opposite $\NegWall$. Assume that $\PosCell$ and $\NegCell$ do not lie in a common twin wall. Then for every twin wall $\Wall$ that contains $\PosCell$, every chamber $\AltChamber \ge \NegCell$ lies on the same side of $\Wall$. Hence $\PosChamber \defeq \WeylProjection_{\PosCell} \NegCell$ is a chamber. Similarly $\NegChamber \defeq \WeylProjection_{\NegCell} \PosCell$ is a chamber. So every twin apartment that contains $\PosCell$ and $\NegCell$ contains $\PosChamber$ and $\NegChamber$. In other words every apartment of $\Link (\PosCell \times \NegCell)$ that comes from a twin apartment contains the chamber $(\PosChamber \direction \PosCell) * (\NegChamber \direction \NegCell)$. An immediate consequence is:

\begin{obs}
\label{obs:unproblematic_case}
Let $\PosCell \subseteq \PosBuilding$ and $\NegCell \subseteq \NegBuilding$ be cells such that $\PosChamber \defeq \WeylProjection_{\PosCell} \NegCell$ and $\NegChamber \defeq \WeylProjection_{\NegCell} \PosCell$ are chambers. Then every apartment of $\Link (\PosCell \times \NegCell)$ that contains $\Chamber \defeq (\PosChamber \direction \PosCell)*(\NegChamber \direction \NegCell)$ is of the form $\Link_{\PosApartment \times \NegApartment} (\PosCell \times \NegCell)$ for some twin apartment $\PNApartments$.
\end{obs}

\begin{proof}
Let $\Apartment$ be an apartment that contains $\Chamber$. Let $(\AltChamber_+ \direction \PosCell) * (\AltChamber_- \direction \NegCell)$ be the chamber in $\Apartment$ opposite $\Chamber$. Let $\PNApartments$ be a twin apartment that contains $\AltChamber_+$ and $\AltChamber_-$. Since $\PNApartments$ also contains $\PosChamber$ and $\NegChamber$, necessarily $\Apartment = \Link_{\PosApartment \times \NegApartment} (\PosCell \times \NegCell)$.
\end{proof}

If $\PosCell$ and $\NegCell$ do lie in a common twin wall, there is no chamber in $\Link (\PosCell \times \NegCell)$ such that every apartment containing this chamber comes from a twin apartment. Thus we have to extend the class of apartments to consider.
To do so, we have to break the twin structure, that is, we have to consider symmetries of the individual buildings that are not symmetries of the twin building. The aim is to show that the height function is to some extent preserved under such symmetries.

The first statement, which contains all technicalities, deals with the ar\-che\-type of a symmetry, reflection at a wall:

\begin{lem}
\label{lem:one_reflection_stable}
Assume that $\Directions$ is rich. Let $\PNApartments$ be a twin apartment. Let $\Wall = (\PosWall,\NegWall)$ be a twin wall of $\PNApartments$, i.e., $\PosWall \subseteq \PosApartment$ and $\NegWall \subseteq \NegApartment$ are walls such that $\PosWall$ is opposite $\NegWall$. Let $r_\Wall$ denote the reflection at $\Wall$. If $\PosVertex \in \PosApartment$ and $\NegVertex \in \NegApartment$ are vertices each adjacent to a cell of $\Wall$ (or contained in $\Wall$), then
\[
\Height(r_\Wall(\PosVertex),\NegVertex) = \Height(\PosVertex,\NegVertex) = \Height(\PosVertex,r_\Wall(\NegVertex)) \text{ .}
\]
\end{lem}

To prove this we want to say that $\Directions \Plus \Directions$ is sufficiently rich for the geodesic segment $e \defeq \PosEIdent([\PosVertex,r_\Wall(\PosVertex)]) - \NegEIdent([\NegVertex,r_\Wall(\NegVertex)])$. This time however, it is not enough that $\Zonotope$ contains a parallel translate through every projection point of $e$. We want that all of $e$ linearly projects onto $\Zonotope$. The reason for this to be true is of course that $\NegVertex$ and $\PosVertex$ are both close to the wall $\Wall$. Before we can make this precise, we need some elementary statements about the arithmetic of zonotopes:

\begin{obs}
\label{obs:zonotope_arithmetic}
Let $\E$ be a Euclidean vector space and let $\Directions$, $E$, $\Directions_1$, and $\Directions_2$ be finite subsets. Then
\begin{enumerate}
\item $\Directions \subseteq E$ implies $\Zonotope(\Directions) \subseteq \Zonotope(E)$.
\item $\Zonotope(\Directions_1 \union \Directions_2) \subseteq \Zonotope(\Directions_1) + \Zonotope(\Directions_2)$ with equality if $\Directions_1 \intersect \Directions_2 = \emptyset$.
\item $\Zonotope(\Directions_1) + \Zonotope(\Directions_2) \subseteq \Zonotope(\Directions_1 \Plus \Directions_2)$.
\end{enumerate}
\end{obs}

\begin{proof}
The first and second statement are clear from the definition. The third is a case distinction similar to Observation~\ref{obs:fold_sufficiently_rich}.
\end{proof}

\begin{proof}[Proof of Lemma~\ref{lem:one_reflection_stable}.]
We make identifications as in \eqref{dia:twin_identification}. Note that $\PosEIdent$ and $\NegEIdent$ induce the same Coxeter structure on $\E$ and that $\PosEIdent(\PosWall) = \NegEIdent(\NegWall)$ is a wall which we also denote by $\Wall$. We reduce notation by taking the origin of $\E$ to lie in $\Wall$. Also we make the identifications via $\PosEIdent$ and $\NegEIdent$ implicit so that $\PosVertex, \NegVertex \in \E$. Let $\Wall^\perp$ denote the orthogonal complement of $\Wall$. Our goal is to show that $[\PosVertex,r_\Wall(\PosVertex)] - [\NegVertex,r_\Wall(\NegVertex)]$ linearly projects onto $\Zonotope$.

We write $E \defeq \Directions \Plus \Directions$ and consider the subsets
\[
E^0 \defeq \{\Direction \in E \mid \Direction \nin \Wall^\perp\} \quad \text{and} \quad \Directions^\perp \defeq \{\Direction \in \Directions \mid \Direction \in \Wall^\perp\} \text{ .}
\]
Note that $\Directions^\perp \Plus \Directions^\perp$ and $E^0$ are disjoint subsets of $E$. Therefore Observation~\ref{obs:zonotope_arithmetic} implies
\begin{align}
\Zonotope(E) \ \supseteq \ \Zonotope((\Directions^\perp \Plus \Directions^\perp) \union E^0) \ &=\ \nonumber\\
\Zonotope(\Directions^\perp \Plus \Directions^\perp) + \Zonotope(E^0)\ 
&\supseteq \ \Zonotope(\Directions^\perp) + \Zonotope(\Directions^\perp) + \Zonotope(E^0)\text{ .}
\label{eq:zonotope_inclusion}
\end{align}

Let $\bar{\Vertex}_+ \defeq 1/2\Vertex_+ + 1/2r_\Wall(\Vertex_+)$ be the projection of $\Vertex_+$ onto $\Wall$ and let $\bar{\Vertex}_-$ be the projection of $\Vertex_-$ onto $\Wall$. Note that since $E$ is $\Weyl$-invariant it is, in particular, invariant by $r_\Wall$. Thus $\Zonotope$ is also $r_\Wall$-invariant. This together with the fact that $\bar{\Vertex}_+ - \bar{\Vertex}_-$ lies in $\Wall$ implies that the projection $\Point \defeq \ClosestPointProjection[\Zonotope] \bar{\Vertex}_+ - \bar{\Vertex}_-$ also lies in $\Wall$.

We claim that $\Point$ already has to lie in $\Zonotope[E^0]$. Indeed write
\[
\Point = \sum_{\Direction \in E^0} \alpha_\Direction \Direction + \sum_{\Direction \in E \intersect H^\perp} \alpha_\Direction \Direction \text{ .}
\]
Invariance under $r_\Wall$ implies that we can also write
\[
\Point = \sum_{\Direction \in E^0} \alpha_{r_\Wall(\Direction)} \Direction + \sum_{\Direction \in E \intersect H^\perp} - \alpha_\Direction \Direction \text{ .}
\]
Taking the mean of both expressions gives $\Point = \sum_{\Direction \in E^0} 1/2(\alpha_\Direction + \alpha_{r_\Wall(\Direction)}) \Direction$.

Next note that $\PosVertex - r_\Wall(\PosVertex)$ lies in $\Wall^\perp$. Moreover, if $\PosCell \subseteq \Wall$ is a cell to which $\PosVertex$ is adjacent, which exists by assumption, then $r_\Wall(\PosVertex)$ is adjacent to $\PosCell$ as well. This shows that the closed stars of $\PosVertex$ and $r_\Wall(\PosVertex)$ meet. So richness of $\Directions$ implies that $\PosVertex - r_\Wall(\PosVertex)$ lies in $\Directions$ and thus in $\Directions^\perp$. In the same way one sees that $\NegVertex - r_\Wall(\NegVertex) \in \Directions^\perp$.

Thus $[\PosVertex,r_\Wall(\PosVertex)] \subseteq \bar{\Vertex}_+ + \Zonotope[\Directions^\perp]$ and $[\NegVertex,r_\Wall(\NegVertex)] \subseteq \bar{\Vertex}_- + \Zonotope[\Directions^\perp]$. Consequently 
\[
[\PosVertex,r_\Wall(\PosVertex)] - [\NegVertex,r_\Wall(\NegVertex)] \quad \subseteq \quad (\bar{\Vertex}_+ - \bar{\Vertex}_-) + \Zonotope(\Directions^\perp) + \Zonotope(\Directions^\perp) \text{ .}
\]
Now $\Point + \Zonotope[\Directions^\perp] + \Zonotope[\Directions^\perp]$ is fully contained in $\Zonotope = \Zonotope[E]$ by \eqref{eq:zonotope_inclusion}. So the closest point projection onto $\Zonotope$ takes $(\bar{\Vertex}_+ - \bar{\Vertex}_-) + \Zonotope(\Directions^\perp) + \Zonotope(\Directions^\perp)$ linearly onto $\Point + \Zonotope(\Directions^\perp) + \Zonotope(\Directions^\perp)$.
\end{proof}

\begin{cor}
\label{cor:weyl_preserves_height}
Assume that $\Directions$ is rich. Let $\PNApartments$ be a twin apartment and let $\EWeyl_+$ respectively $\EWeyl_-$ be the affine reflection groups of $\PosApartment$ respectively $\NegApartment$. Let $\PosCell \subseteq \PosApartment$ and $\NegCell \subseteq \NegApartment$ be cells and let $\PosBigCell \defeq \WeylProjection_{\PosCell} \NegCell$ and $\NegBigCell \defeq \WeylProjection_{\NegCell} \PosCell$ be the projections of one onto the other. Let $R_+$ respectively $R_-$ be the stabilizer of $\PosBigCell$ in $\EWeyl_+$ respectively of $\NegBigCell$ in $\EWeyl_-$. Then
\[
\Height(\PosVertex,\NegVertex) = \Height(\weyl_+\PosVertex,\weyl_-\NegVertex)
\]
for all group elements $\weyl_+ \in R_+$ and $\weyl_- \in R_-$ and all vertices $\PosVertex$ adjacent to $\PosCell$ and $\NegVertex$ adjacent to $\NegCell$.
\end{cor}

\begin{proof}
The affine span of $\PosBigCell$ in $\PosApartment$ is the intersection of the positive halves of twin walls that contain $\PosCell$ and $\NegCell$. Similarly, the affine span of $\NegBigCell$ in $\NegApartment$ is the intersection of negative halves of twin walls that contains $\PosCell$ and $\NegCell$. The group $R_+ \times R_-$ is therefore generated by the reflections described in Lemma~\ref{lem:one_reflection_stable}.
\end{proof}

Note that Corollary~\ref{cor:weyl_preserves_height} is essentially a statement about a \emph{spherical} reflection group: it says that the stabilizer of $\PosBigCell \times \NegBigCell$ in the group of symmetries of $\Star_{\PosApartment} \PosCell \times \Star_{\NegApartment} \NegCell$, which is the reflection group of $\Link_{\PosApartment} \PosCell * \Link_{\NegApartment} \NegCell$, preserves height (on vertices).

In the remainder of the section we want to use this result to show how height is preserved in the twin building $\PNBuildings$. First we look at symmetries that preserve the twin structure:

\begin{obs}
\label{obs:twin_iso_preserves_height}
Let $\PNApartments$ and $(\PosApartment',\NegApartment')$ be twin apartments. Any isomorphism $\kappa \colon \PNApartments \to (\PosApartment',\NegApartment')$ of thin twin buildings preserves height in the sense that $\Height(\PosPoint,\NegPoint) = \Height(\kappa(\PosPoint),\kappa(\NegPoint))$.
\end{obs}

\begin{proof}
The restriction of $\Height$ to $\PNApartments$ respectively $(\PosApartment',\NegApartment')$ are the intrinsically defined height functions of these thin twin buildings.
\end{proof}

Again we are particularly interested in retractions:

\begin{obs}
\label{obs:retraction_partially_preserves_height}
Let $\PosCell \subseteq \PosBuilding$ and $\NegCell \subseteq \NegBuilding$ be cells such that $\PosChamber \defeq \WeylProjection_{\PosCell} \NegCell$ and $\NegChamber \defeq \WeylProjection_{\NegCell} \PosCell$ are chambers. Let $\PNApartments$ be a twin apartment that contains $\PosCell$ and $\NegCell$ and therefore $\PosChamber$. Let $\rho \defeq \Retraction{\PNApartments}{\PosChamber}$ be the retraction onto $\PNApartments$ centered at $\PosChamber$. If $\PosPoint$ lies in the closed star of $\PosCell$ and $\NegPoint$ lies in the closed star of $\NegCell$ then
\[
\Height(\PosPoint,\NegPoint) = \Height(\rho(\PosPoint),\rho(\NegPoint)) \text{ .}
\]
\end{obs}

\begin{proof}
Let $(\PosApartment',\NegApartment')$ be a twin apartment that contains $\PosPoint$ and $\PosCell$ as well as $\NegPoint$ and $\NegCell$. Then it also contains $\PosChamber$. Hence $\rho|_{(\PosApartment',\NegApartment')}$ is an isomorphism of thin twin buildings.
\end{proof}

\begin{rem}
\label{rem:symmetry}
There is an apparent asymmetry in the last observation between $\PosChamber$ and $\NegChamber$. To explain why the statement is in fact symmetric we consider a more general setting. Let $\PosCell \subseteq \PosBuilding$ and $\NegCell \subseteq \NegBuilding$ be arbitrary cells and let $\PosBigCell \defeq \WeylProjection_{\PosCell} \NegCell$ and $\NegBigCell \defeq \WeylProjection_{\NegCell} \PosCell$ be the projections of one onto the other (by ``arbitrary'' we mean that these are not required to be chambers).

The first thing to note is that if $\PosChamber$ contains $\PosBigCell$, then not only does $\NegChamber \defeq \WeylProjection_{\NegCell}\PosChamber$ contain $\NegBigCell$ (which is clear from the definition of $\NegBigCell$), but also $\PosChamber = \WeylProjection_{\PosCell}\NegChamber$.

Secondly, if $\PosChamber$ and $\NegChamber$ are as above projections of each other, then for every chamber $\AltChamber \ge \NegCell$ we have $\WeylCoDistance(\PosChamber,\AltChamber) = \WeylCoDistance(\PosChamber,\NegChamber) \NegWeylDistance(\NegChamber,\AltChamber)$ and the same is true with the roles of $\PosChamber$ and $\NegChamber$ exchanged (recall that $\PosWeylDistance$ and $\NegWeylDistance$ denote the Weyl-distance on $\PosBuilding$ and $\NegBuilding$, respectively). This shows that the retractions centered at $\PosChamber$ and centered at $\NegChamber$ coincide on $\Star \PosCell$ and $\Star \NegCell$.

So had we replaced $\rho$ by the retraction centered at $\NegChamber$ in the observation, then the statement would not only have remained true, but would have been the same statement.
\end{rem}

Now we incorporate Corollary~\ref{cor:weyl_preserves_height} to get the result we were aiming for:

\begin{prop}
\label{prop:retraction_preserves_height}
Assume that $\Directions$ is rich. Let $\PosCell \subseteq \PosBuilding$ and $\NegCell \subseteq \NegBuilding$ be cells. Let $\PosChamber \ge \WeylProjection_{\PosCell} \NegCell$ be a chamber and let $\PNApartments$ be a twin apartment that contains $\PosChamber$ and $\NegCell$. Let $\rho \defeq \Retraction{\PNApartments}{\PosChamber}$ be the retraction onto $\PNApartments$ centered at $\PosChamber$. Then
\[
\Height(\PosVertex,\NegVertex) = \Height(\rho(\PosVertex),\rho(\NegVertex))
\]
for every vertex $\PosVertex$ adjacent to $\PosCell$ and every vertex $\NegVertex$ adjacent to $\NegCell$.
\end{prop}

\begin{proof}
Let $(\PosApartment'',\NegApartment'')$ be a twin apartment that contains $\PosVertex$ and $\PosCell$ as well as $\NegVertex$ and $\NegCell$. Let $\NegChamber \subseteq \NegApartment''$ be a chamber that contains $\WeylProjection_{\NegCell}\PosCell$. Let $(\PosApartment',\NegApartment')$ be a twin apartment that contains $\PosChamber$ and $\NegChamber$. Let $\rho' \defeq \Retraction{(\PosApartment',\NegApartment')}{\NegChamber}$. Applying Observation~\ref{obs:retraction_partially_preserves_height} first to $\rho'|_{(\PosApartment'',\NegApartment'')}$ and then to $\rho|_{(\PosApartment',\NegApartment')}$ we find that
\begin{equation}
\label{eq:double_retraction_preserves_height}
\Height(\PosVertex,\NegVertex) = \Height(\rho\circ\rho'(\PosVertex),\rho\circ\rho'(\NegVertex)) \text{ .}
\end{equation}
It remains to compare the heights of $(\rho\circ\rho'(\PosVertex),\rho\circ\rho'(\NegVertex))$ and $(\rho(\PosVertex),\rho(\NegVertex))$ in the twin apartment $\PNApartments$.

Let $\AltChamber$ be a chamber of $\PosApartment''$ that contains $\PosVertex$ and $\PosCell$ and let $\YetAltChamber \defeq \WeylProjection_{\PosBigCell} \AltChamber$. Let $\weyl_+$ be the element of the Coxeter group of $\PosApartment$ that takes $\rho\circ\rho'(\YetAltChamber)$ to $\rho(\YetAltChamber)$. Note that $\weyl_+$ fixes $\PosBigCell$. We claim that $\weyl_+$ takes $\rho\circ\rho'(\AltChamber)$ to $\rho(\AltChamber)$. More precisely we claim that
\[
\PosWeylDistance(\rho \circ \rho'(\AltChamber),\rho \circ \rho'(\YetAltChamber)) = \PosWeylDistance(\AltChamber,\YetAltChamber) = \PosWeylDistance(\rho(\AltChamber),\rho(\YetAltChamber)) \text{ .}
\]

The first equation follows from $\rho'|_{\PosApartment''}$ and $\rho|_{\PosApartment'}$ being isomorphisms. The second follows by an analogous argument for an apartment that contains $\AltChamber$ and $\PosChamber$.

This shows that $\weyl_+$ indeed takes $\rho\circ\rho'(\AltChamber)$ (the unique chamber in $\PosApartment$ that has distance $\PosWeylDistance(\rho \circ \rho'(\AltChamber),\rho \circ \rho'(\YetAltChamber))$ to $\rho \circ \rho'(\YetAltChamber)$) to $\rho(\AltChamber)$ (the unique chamber in $\PosApartment$ that has distance $\PosWeylDistance(\rho(\AltChamber),\rho(\YetAltChamber))$ to $\rho(\YetAltChamber)$). In particular, $\weyl_+$ takes $\rho \circ \rho'(\PosVertex)$ to $\rho(\PosVertex)$.

Arguing in the same way produces an element $\weyl_-$ that takes $\rho\circ\rho'(\NegVertex)$ to $\rho(\NegVertex)$.

Applying Corollary~\ref{cor:weyl_preserves_height} we get
\[
\Height(\rho\circ\rho'(\PosVertex),\rho\circ\rho'(\NegVertex)) = \Height(w_+\rho\circ\rho'(\PosVertex),w_-\rho\circ\rho'(\NegVertex)) = \Height(\rho(\PosVertex),\rho(\NegVertex))
\]
which together with \eqref{eq:double_retraction_preserves_height} proves the claim.
\end{proof}

\footerlevel{3}
\headerlevel{3}

\section{Descending Links}
\label{sec:two_places_descending_links}

With the tools from the last section the analysis of the descending links runs fairly parallel to that in Section~\ref{sec:descending_links}. In fact many proofs carry over in verbatim. We still reproduce them because we have to check that they apply even though $\TwoSpace$ is not simplicial.

Recall from \eqref{eq:two_places_barycenter_descending_link_decomposition} that the descending link of a vertex $\Subdiv\Cell$ decomposes as a join $\Link\Descending \Subdiv\Cell = \Link\FacePart\Descending \Subdiv\Cell * \Link\CofacePart\Descending \Subdiv\Cell$ of the descending face part and the descending coface part.

Recall also that $\Cell$ is flat if $\Height|_{\Cell}$ is constant. If $\Cell$ is flat, then it has a face $\Cell\Min$. The roof $\Roof\BigCell$ of any cell $\BigCell$ is flat. We say that $\BigCell$ is \emph{significant} if $\BigCell = \Roof\BigCell\Min$ and that it is \emph{insignificant} otherwise.

\begin{lem}
\label{lem:two_places_insignificant_descending_link}
If $\BigCell$ is insignificant, then the descending link of $\Subdiv\BigCell$ is contractible. More precisely $\Link\FacePart\Descending \Subdiv\BigCell$ is already contractible.
\end{lem}

\begin{proof}
Consider the full subcomplex $\Lambda$ of $\Link\FacePart \Subdiv\BigCell$ of vertices $\Subdiv\Cell$ with $\Roof{\BigCell}\Min \not\le \Cell \lneq \BigCell$: this is the barycentric subdivision of $\partial \BigCell$ with the open star of $\Roof{\BigCell}\Min$ removed. Therefore it is a punctured sphere and, in particular, contractible. We claim that $\Link\FacePart\Descending \Subdiv\BigCell$ deformation retracts on $\Lambda$.

So let $\Subdiv\Cell$ be a vertex of $\Lambda$. Then
\[
\max \Height|_{\Cell} \le \max \Height|_{\BigCell} = \max \Height|_{\Roof\BigCell\Min}
\]
so $\Height$ either makes $\Subdiv\Cell$ descending or is indifferent. As for depth, the fact that $\Roof\BigCell\Min \not\le \Cell$ implies of course that $\Roof\BigCell\Min \not\le \Roof\Cell$. So there is a move $\Roof\BigCell \Down \Roof\Cell$ which implies $\Depth \Roof\Cell < \Depth\Roof\BigCell - 1/2$. Therefore
\[
\Depth \Cell \le \Depth \Roof\Cell < \Depth\Roof\BigCell - 1/2 \le \Depth\BigCell
\]
so $\Subdiv\Cell$ is descending. This shows that $\Lambda \subseteq \Link\FacePart\Descending \Subdiv\BigCell$.

On the other hand $(\Roof\BigCell\Min)^\circ$ is not descending: Height does not decide because $\max \Height|_{\Roof\BigCell\Min} = \max \Height|_{\Roof\BigCell} = \max\Height|_{\BigCell}$. As for depth we have
\[
\Depth \BigCell \le \Depth \Roof\BigCell \le \Depth \Roof\BigCell\Min \text{ .}
\]
If $\BigCell$ is not flat, then the first inequality is strict. If $\BigCell$ is flat, then there is a move $\Roof\BigCell\Min \Up \Roof\BigCell = \BigCell$ so the second inequality is strict. In either case $(\Roof\BigCell\Min)^\circ$ is ascending.

So geodesic projection away from $(\Roof\BigCell\Min)^\circ$ defines a deformation retraction of $\Link\FacePart\Descending \Subdiv\BigCell$ onto $\Lambda$.
\end{proof}

\begin{lem}
\label{lem:two_places_significant_face_part_sphere}
If $\BigCell$ is significant, then all of $\Link\FacePart \Subdiv\BigCell$ is descending. So $\Link\Descending\FacePart \Subdiv\BigCell$ is a $(\dim \BigCell - 1)$-sphere.
\end{lem}

\begin{proof}
Let $\Cell \lneq \BigCell$ be arbitrary. We have $\max \Height|_{\Cell} = \max\Height|_{\BigCell}$ because $\BigCell$ is flat. Moreover, $\Cell \lneq \BigCell =\BigCell\Min$ so that, in particular, $\BigCell\Min \not\le \Cell$. Hence there is a move $\BigCell \Down \Cell$ which implies $\Depth \BigCell > \Depth \Cell$ so that $\Subdiv\Cell$ is descending.
\end{proof}

We recall Observation~\ref{obs:min_min} and Observation~\ref{obs:significant_either_or}:

\begin{reminder}
\label{reminder:two_places_move_reminder}
\begin{enumerate}
\item If $\BigCell$ is flat and $\BigCell\Min \le \Cell \le \BigCell$, then $\Cell\Min = \BigCell\Min$.\label{item:two_places_height_min_min}
\item If $\Cell$ is significant and $\BigCell \ge \Cell$ is flat, then there is either a move $\Cell \Up \BigCell$ or a move $\BigCell \Down \Cell$.\label{item:two_places_height_either_or}
\end{enumerate}
\end{reminder}

\begin{prop}
\label{prop:two_places_descending_coface_link_is_descending_link}
Let $\Cell$ be significant. The descending coface part $\Link\Descending\CofacePart \Subdiv\Cell$ is a subcomplex of $\Link \Cell$. That is, for cofaces $\BigCell \gneq \Another\Cell \gneq \Cell$, if $\Subdiv\BigCell$ is descending then $\Subdiv{\Cell}'$ is descending.
\end{prop}

\begin{proof}
Let $\BigCell \gneq \Another\Cell \gneq \Cell$ and assume that $\Morse(\Subdiv\BigCell) < \Morse(\Subdiv\Cell)$. By inclusion of cells we have
\[
\max \Height|_{\BigCell} \ge \max \Height|_{\Cell'} \ge \max \Height|_{\Cell}
\]
and since $\Subdiv\BigCell$ is descending $\max \Height|_\BigCell \le \max \Height|_\Cell$ so equality holds. Clearly $\dim \BigCell > \dim \Cell$ so since $\Subdiv\BigCell$ is descending we conclude $\Depth\BigCell < \Depth\Cell$.
We have inclusions of flat cells
\[
\Roof\BigCell \ge \Roof\Cell' \ge \Cell \text{ .}
\]
If the second inclusion is equality, then $\Another\Cell \ne \Cell = \Roof\Cell'$ so $\Depth \Another\Cell < \Depth \Roof\Cell' = \Depth \Cell$ and $\Subdiv\Cell'$ is descending. Otherwise $\Roof\BigCell$ is a proper coface of $\Cell$ so by Reminder~\ref{reminder:two_places_move_reminder}~\eqref{item:two_places_height_either_or} there is a move $\Cell \Up \Roof\BigCell$ or a move $\Roof\BigCell \Down \Cell$. In the latter case we would have $\Depth \BigCell \ge \Depth \Roof\BigCell - 1/2 > \Depth \Cell$ contradicting the assumption that $\Subdiv\BigCell$ is descending. Hence the move is $\Cell \Up \Roof\BigCell$, that is, $\Cell = \Roof\BigCell\Min$. It then follows from Reminder~\ref{reminder:two_places_move_reminder}~\eqref{item:two_places_height_min_min} that also $\Roof\Cell'{}\Min = \Cell$ so that there is a move $\Cell \Up \Roof\Cell'$. Thus $\Depth \Cell' \le \Depth \Roof\Cell' < \Depth \Roof\Cell$.
\end{proof}

Let $\Cell \subseteq \TwoSpace$ be a significant cell. We define the \emph{descending link} $\Link\Descending \Cell$ of $\Cell$ to be the subcomplex of cells $\BigCell \direction \Cell$ with $\Morse(\BigCell) < \Morse(\Cell)$. As a set, this is by Proposition~\ref{prop:two_places_descending_coface_link_is_descending_link} the same as $\Link\CofacePart\Descending \Subdiv\Cell$. We define the \emph{horizontal descending link} $\Link\Hor{}\Descending \Cell \defeq \Link\Hor \Cell \intersect \Link\Descending \Cell$ and the \emph{vertical descending link} $\Link\Ver{}\Descending \Cell \defeq \Link\Ver \Cell \intersect \Link\Descending \Cell$ in the obvious way. We see immediately that $\Link\Descending \Cell \subseteq \Link\Hor{}\Descending \Cell * \Link\Ver{}\Descending \Cell$ and will show the converse later.

\begin{lem}
\label{lem:two_places_significant_vertical_open_hemisphere}
If $\Cell$ is significant, then $\Link\Ver{}\Descending \Cell$ is an open hemisphere complex with north pole $\Gradient_\Cell \Height$.
\end{lem}

\begin{proof}
Let $\Link\OpenHemi \Cell$ denote the open hemisphere complex with north pole $\Gradient_\Cell \Height$. By Corollary~\ref{cor:two_places_higher_dimensional_angle_criterion} $\Link\OpenHemi \Cell \subseteq \Link\Ver{}\Descending \Cell$.

Conversely assume that $\BigCell \ge \Cell$ is such that $\BigCell \direction \Cell$ contains a vertex that includes a non-obtuse angle with $\Gradient_\Cell\Height$. Then either
\[
\max \Height|_{\BigCell} = \max \Height|_{\Roof\BigCell} > \max \Height|_{\Cell} 
\]
or $\Roof\BigCell$ is a proper flat coface of $\Cell$. In the latter case since $\Roof\BigCell$ does not lie in the horizontal link of $\Cell$ there is a move $\Roof\BigCell \Down \Cell$ so that
\[
\Depth \BigCell \ge \Depth \Roof\BigCell - \frac{1}{2} > \Depth \Cell \text{ .}
\]
In both cases $\BigCell$ is not descending.
\end{proof}

\begin{obs}
\label{obs:two_places_descending_iff_flat}
If $\Cell$ is significant and $\BigCell \ge \Cell$ is such that $\BigCell \direction \Cell \subseteq \Link\Hor \Cell$, then these are equivalent:
\begin{enumerate}
\item $\BigCell$ is flat.
\item $\BigCell$ is descending.
\item $\Height|_{\BigCell} \le \Height(\Cell)$.
\end{enumerate}
\end{obs}

\begin{proof}
If $\BigCell$ is flat then clearly $\max\Height|_\BigCell = \Height(\Cell)$. Moreover, $\BigCell\Min = \Cell\Min$ by Reminder~\ref{reminder:two_places_move_reminder}~\eqref{item:height_min_min}. Thus there is a move $\Cell = \Cell\Min = \BigCell\Min \Up \BigCell$ so that $\Depth \Cell > \Depth \BigCell$ and $\BigCell$ is descending.

If $\BigCell$ is not flat, then it contains vertices of different heights. Since $\BigCell \direction \Cell$ lies in the horizontal link it in particular includes a right angle with $\Gradient_\Cell \Height$. So by the angle criterion Corollary~\ref{cor:two_places_angle_criterion} no vertex has lower height than $\Cell$. Hence $\max \Height|_\BigCell > \max \Height|_\Cell$ and $\BigCell$ is not descending.
\end{proof}

\begin{prop}
\label{prop:two_places_descending_coface_link_decomposes}
If $\Cell$ is significant, then the descending link decomposes as a join
\[
\Link\Descending\Cell = \Link\Hor{}\Descending \Cell * \Link\Ver{}\Descending \Cell
\]
of the horizontal descending link and the vertical descending link.
\end{prop}

\begin{proof}
Let $\BigCell_h$ and $\BigCell_v$ be proper cofaces of $\Cell$ such that $\BigCell_h$ lies in the horizontal descending link, $\BigCell_v$ lies in the vertical descending link and $\BigCell \defeq \BigCell_h \vee \BigCell_v$ exists. We have to show that $\BigCell$ is descending.

By Lemma~\ref{lem:two_places_significant_vertical_open_hemisphere} $\BigCell_v$ includes an obtuse angle with $\Gradient_\Cell \Height$ so by Proposition~\ref{cor:two_places_higher_dimensional_angle_criterion} $\Roof\BigCell_v = \Cell$. On the other hand $\BigCell_h$ is flat by Observation~\ref{obs:two_places_descending_iff_flat}. Thus $\Roof\BigCell = \BigCell_h$ so that $\Depth\BigCell = \Depth\BigCell_h - 1/2$ and $\BigCell$ is descending because $\BigCell_h$ is.
\end{proof}

It remains to study the horizontal descending links of significant cells. As before we want to eventually apply Proposition~\ref{prop:coconvex_complexes_spherical}. We will be able to do so thanks to the results of the last section. So essentially we have to understand what happens inside one apartment.

We assume from now on that $\Directions$ is rich. We fix a significant cell $\Cell \subseteq \TwoSpace$ and write $\Cell = \PosCell \times \NegCell$ with $\PosCell \subseteq \PosBuilding$ and $\NegCell \subseteq \NegBuilding$. We also fix a twin apartment $\PNApartments$ that contains $\PosCell$ and $\NegCell$ and let $\EApartment = \PosApartment \times \NegApartment$. Let further $\Apartment \defeq \Link_{\EApartment} \Cell$ be the apartment of $\Link \Cell$ defined by $\EApartment$.

We set
\[
\UpSet \defeq \{\Vertex \in \Vertices \EApartment \mid \Vertex \vee \Cell \text{ exists and } \Height(\Vertex) > \Height(\Cell)\}
\]
and let $\UpConvex$ be the convex hull of $\UpSet$.

\begin{obs}
The minimum of $\Height$ over $\UpConvex$ is strictly higher than $\Height(\Cell)$.
\end{obs}

\begin{proof}
Make identifications for $\PNApartments$ as in \eqref{dia:twin_identification}. Since $\Directions$ is rich, it contains all vectors of the form $\PosEIdent(\PosVertex) - \PosEIdent(\PosVertex')$ where $\PosVertex$ and $\PosVertex'$ are vertices adjacent to $\PosCell$. It also contains all vectors of the form $\NegEIdent(\NegVertex) - \NegEIdent(\NegVertex')$ for $\NegVertex$ and $\NegVertex'$ vertices adjacent to $\NegVertex$. Therefore $\Directions \Plus \Directions$ contains all vectors of the form $\Vertex - \Vertex'$ where $\Vertex$ and $\Vertex'$ lie in $\Difference \circ (\PosEIdent \times \NegEIdent)(\UpSet)$. All this is just to say that $\Directions \Plus \Directions$ is sufficiently rich for $\Difference \circ (\PosEIdent \times \NegEIdent)(\UpConvex)$.

Thus by Proposition~\ref{prop:sufficiently_rich_min_vertex} distance from $\Zonotope$ attains its minimum over $\Difference \circ (\PosEIdent \times \NegEIdent)(\UpConvex)$ in a vertex. Consequently, $\Height$ attains its minimum over $\UpConvex$ in a vertex. That vertex is an element of $\UpSet$ and hence has height strictly higher than $\Height(\Cell)$.
\end{proof}

Since $\UpConvex$ is closed, there is an $\varepsilon > 0$ such that the $\varepsilon$-neighborhood of $\UpConvex$ in $\EApartment$ still has height strictly higher than $\Height(\Cell)$. We fix such an $\varepsilon$ and let $\UpOpen$ denote the corresponding neighborhood. We let $\UpLink$ denote the set of directions in $\Apartment$ that point toward points of $\UpOpen \intersect \Star \Cell$.

\begin{obs}
\label{obs:uplink_open_convex}
The set $\UpLink$ is a proper, open, convex subset of $\Apartment$ and has the property that a coface $\BigCell$ of $\Cell$ that is contained in $\EApartment$ contains a point of height strictly above $\Height(\Cell)$ if and only if $\BigCell \direction \Cell$ meets $\UpLink$.
\end{obs}

\begin{proof}
Note that $\UpOpen \intersect \Star \Cell$ is convex as an intersection of convex sets, and is disjoint from $\Cell$ by choice of $\varepsilon$. It follows that $\UpLink$ is convex. To see that $\UpLink$ is open in $\Link\Cell$ note that it can also be described as the set of directions toward $\UpOpen \intersect (\partial \Star \Cell)$ which is open in $\partial \Star\Cell$.

If $\BigCell$ contains a point of height strictly above $\Height(\Cell)$, then by Observation~\ref{obs:two_places_height_is_convex} it also contains a vertex with that property. That vertex therefore lies in $\UpSet$ and the direction toward it defines a direction in $\UpLink \intersect (\BigCell \direction \Cell)$.

Conversely assume that $\Point$ lies in $\UpOpen \intersect \Star \Cell$ and defines a direction in $\BigCell \direction \Cell$. Since $\Point \in \Star\Cell$, this implies that $\Point \in \BigCell$.
\end{proof}

The transition from $\Apartment$ to the full link of $\Cell$ is via retractions. So let $\PosChamber \ge \PosCell$ and $\NegChamber \ge \NegCell$ be chambers of $\PNApartments$ such that $\WeylProjection_{\PosCell} \NegChamber =\PosChamber$ and $\WeylProjection_{\NegCell} \PosChamber = \NegChamber$. Let $\Chamber \defeq (\PosChamber \direction \PosCell) * (\NegChamber \direction \NegCell)$ be the chamber of $\Apartment$ defined by $\PosChamber$ and $\NegChamber$.

Let $\tilde{\rho} \defeq \Retraction{\PNApartments}{\PosChamber}$ and recall from Remark~\ref{rem:symmetry} that $\tilde{\rho}$ restricts to the same map on $\Star \PosCell \union \Star\NegCell$ as the retraction centered at $\NegChamber$. Let $\rho \defeq \Retraction{\Apartment}{\Chamber}$ be the retraction onto $\Apartment$ centered at $\Chamber$.

\begin{obs}
\label{obs:two_places_retractions_compatible}
The diagram
\begin{diagram}
\Star\Cell & \rTo^{\tilde{\rho} \times \tilde{\rho}} & \EApartment \intersect \Star\Cell\\
\dTo && \dTo \\
\Link\Cell & \rTo^{\rho} & \Apartment
\end{diagram}
where the vertical maps are projection onto the link, commutes.\qed
\end{obs}

Let $\UpFat \defeq \rho^{-1}(\UpLink)$.

\begin{obs}
\label{obs:upfat_open_apartment_convex_determines_flat_cells}
The set $\UpFat$ is open and meets every apartment that contains $\Chamber$ in a proper convex subset. Moreover, it has the property that if $\BigCell \ge \Cell$ is such that $\BigCell \direction \Cell \subseteq \Link\Hor \Cell$, then $\BigCell$ is flat if and only if $\BigCell \direction \Cell$ is disjoint from $\UpFat$.
\end{obs}

\begin{proof}
We repeatedly apply Observation~\ref{obs:uplink_open_convex}. That $\UpFat$ is open follows from $\UpLink$ being open by continuity of $\rho$. If $\Apartment'$ is an apartment that contains $\Chamber$, then $\rho|_{\Apartment'}$ is an isometry, so $\Apartment' \intersect \UpFat$ is convex as an isometric image of $\UpLink$.

Let $\BigCell \ge \Cell$ be such that $\BigCell \direction \Cell$ lies in the horizontal link of $\Cell$. By Observation~\ref{obs:two_places_descending_iff_flat}, $\BigCell$ is flat if and only if it does not contain a point of height $> \Height(\Cell)$. Write $\BigCell = \PosBigCell \times \NegBigCell$. By Proposition~\ref{prop:retraction_preserves_height} $\BigCell$ is flat if and only if $\tilde{\rho}(\PosBigCell) \times \tilde{\rho}(\NegBigCell)$ is flat. By Observation~\ref{obs:two_places_retractions_compatible} this is precisely the cell that defines $\rho(\BigCell)$ and is therefore flat if and only if $\rho(\BigCell)$ is disjoint from $\UpLink$. This is clearly equivalent to $\BigCell$ being disjoint from $\UpFat$.
\end{proof}

\begin{lem}
\label{lem:two_places_significant_horizontal_spherical}
If $\Cell$ is significant, then $\Link\Hor{}\Descending \Cell$ is $(\dim \Link\Hor \Cell-1)$-connected.
\end{lem}

\begin{proof}
If $\BigCell \ge \Cell$ is such that $(\BigCell \direction \Cell) \subseteq \Link\Hor\Cell$, then by Observation~\ref{obs:two_places_descending_iff_flat} $\BigCell$ is descending if and only if it is flat. Let $\UpFat$ be as before. By Observation~\ref{obs:upfat_open_apartment_convex_determines_flat_cells} $\BigCell$ is flat if and only if it is disjoint from $\UpFat$. We may therefore apply Proposition~\ref{prop:coconvex_complexes_spherical} to $\Link\Hor \Cell$ and $\UpFat \intersect \Link\Hor\Cell$ from which the result follows.
\end{proof}

\begin{prop}
\label{prop:two_places_significant_descending_link}
Assume that $\Directions$ is rich. If $\Cell$ is significant, then the descending link $\Link\Descending \Subdiv\Cell$ is spherical. If the horizontal link is empty, it is properly spherical.
\end{prop}

\begin{proof}
The descending link decomposes as a join
\[
\Link\Descending \Subdiv\Cell = \Link\Descending\FacePart \Subdiv\Cell * \Link\Ver{}\Descending\Cell * \Link\Hor{}\Descending\Cell
\]
of the descending face part, the vertical descending link, and the horizontal descending link by \eqref{eq:two_places_barycenter_descending_link_decomposition}, Proposition~\ref{prop:two_places_descending_coface_link_is_descending_link}, and Proposition~\ref{prop:two_places_descending_coface_link_decomposes}. The descending face part is a sphere by Lemma~\ref{lem:two_places_significant_face_part_sphere}. The descending vertical link is an open hemisphere complex by Lemma~\ref{lem:two_places_significant_vertical_open_hemisphere} which is properly spherical by Theorem~\ref{thm:schulz_main}. The horizontal descending link is spherical by Lemma~\ref{lem:two_places_significant_horizontal_spherical}.
\end{proof}

\footerlevel{3}
\headerlevel{3}

\section{Proof of the Main Theorem for $\GroupScheme(\F_q[t,t^{-1}])$}
\label{sec:two_places_main_theorem}

\begin{thm}
\label{thm:two_places_geometric}
Let $\PNBuildings$ be an irreducible, thick, locally finite Euclidean twin building of dimension $\Dimension$. Let $\Group$ be a group that acts strongly transitively on $\PNBuildings$ and assume that the kernel of the action is finite. Then $\Group$ is of type $F_{2\Dimension-1}$ but not of type $F_{2\Dimension}$.
\end{thm}

\begin{proof}
Let $\TwoSpace \defeq \PosBuilding \times \NegBuilding$ and note that $\dim \TwoSpace = 2 \Dimension$. Consider the action of $\Group$ on the barycentric subdivision $\Subdiv\TwoSpace$. We want to apply Corollary~\ref{cor:adapted_browns_criterion} and check the premises. The space $\TwoSpace$ is contractible being the product of two contractible spaces.

If $\Cell \subseteq \TwoSpace$ is a cell, we can write $\Cell = \PosCell \times \NegCell$ with $\PosCell \subseteq \PosBuilding$ and $\NegCell \subseteq \NegBuilding$. The stabilizer of $\Cell$ in $\Group$ is the simultaneous stabilizer of $\PosCell$ and $\NegCell$ which is finite because the center of the action of $\Group$ is finite by assumption and the stabilizer in the full automorphism group is finite by Lemma~\ref{lem:double_stabilizer_finite}. The stabilizer of a cell of $\Subdiv\TwoSpace$ stabilizes any cell of $\TwoSpace$ that contains it and is thus also finite.

Let $\Morse$ be the Morse function on $\Subdiv\TwoSpace$ as defined in Section~\ref{sec:two_places_morse_function} based on a rich set of directions $\Directions$. Its sublevel sets are $\Group$-invariant subcomplexes. The group $\Group$ acts transitively on chambers $\PosChamber \times \NegChamber$ with $\PosChamber \op \NegChamber$ by strong transitivity. Since $\TwoSpace$ is locally finite, this implies that $\Group$ acts cocompactly on any sublevel set of $\Morse$.

The descending links of $\Morse$ are $(2 \Dimension - 1)$-spherical by Lemma~\ref{lem:two_places_insignificant_descending_link} and Proposition~\ref{prop:two_places_significant_descending_link}. If $\Cell$ is significant then the descending link of $\Subdiv\Cell$ is properly $(2 \Dimension - 1)$-spherical provided the horizontal part is empty. This is the generic case and happens infinitely often.

Applying Corollary~\ref{cor:morse_theory} we see that the induced maps $\pi_i(\Space_k) \to \pi_i(\Space_{k+1})$ are isomorphisms for $0 \le i < n-2$ and are surjective and infinitely often not injective for $i = n-1$. So it follows from Corollary~\ref{cor:adapted_browns_criterion} that $\Group$ is of type $F_{2n-1}$ but not $F_{2n}$.
\end{proof}

The statement about $S$-arithmetic groups is even easier to deduce this time. Before we do so we reinterpret Theorem~\ref{thm:two_places_geometric} group theoretically using the interaction between buildings and groups acting on them.

\begin{thm}
\label{thm:two_places_group_theoretic}
Let $(\Group,B_+,B_-,N,S)$ be a twin Tits system so that, in particular, $T \defeq B_+ \intersect N = B_- \intersect N$. Set as usual $W \defeq N / T$. Assume that $(W,S)$ is of irreducible affine type and rank $\abs{S} = \Dimension+1$, that $[B_\varepsilon s B_\varepsilon:B_\varepsilon]$ is finite for $s \in S$ and $\varepsilon \in \{+,-\}$, and that $\Intersect_{g \in \Group} gB_+g^{-1} \intersect \Intersect_{g \in \Group} gB_-g^{-1}$ is finite. Then $\Group$ is of type $F_{2\Dimension-1}$ but not of type $F_{2\Dimension}$.

This is in particular the case if there is an RGD system $(\Group,(U_\alpha)_{\alpha \in \Phi},T)$ where $\Phi$ is an irreducible affine root system, each $U_\alpha$ is finite and $\Group_+ \defeq \gen{U_\alpha \mid \alpha \in \Phi}$ has finite index in $\Group$.
\end{thm}

\begin{proof}
The twin Tits system gives rise to a thick twin building $\PNBuildings$ on which $\Group$ acts strongly transitively, see \cite[Theorem~6.87]{abrbro}. That $(W,S)$ is irreducible and of rank $\Dimension+1$ means that $\PNBuildings$ is irreducible and of dimension $\Dimension$. The condition that $[B_\varepsilon s B_\varepsilon:B_\varepsilon]$ is finite for $s \in S$ and $\varepsilon \in \{+,-\}$ implies that $\PNBuildings$ is locally finite, cf.\ \cite[Section~6.1.7]{abrbro}. The subgroup $\Intersect_{g \in \Group} gB_+g^{-1} \intersect \Intersect_{g \in \Group} gB_-g^{-1}$ is the kernel of the action of $\Group$ on $\PNBuildings$. So the first statement follows from Theorem~\ref{thm:two_places_geometric}.

An RGD system gives rise to a twin Tits system by \cite[Theorem~8.80]{abrbro}. Moreover, \cite[Theorem~8.81]{abrbro} implies that the associated twin building is locally finite if the $U_\alpha$ are finite. Finally by \cite[Proposition~8.82]{abrbro} the centralizer of $\Group_+$ in $\Group$ is the kernel of the action of $\Group$ on the twin building. It is finite because $\Group_+$ has trivial center.
\end{proof}

\begin{thm}
\label{thm:two_places_arithmetic}
Let $\GroupScheme$ be a connected, noncommutative, almost simple $\F_q$-group of rank $\Dimension \ge 1$. The group $\GroupScheme(\F_q[t,t^{-1}])$ is of type $F_{2n-1}$ but not of type $F_{2n}$.
\end{thm}

\begin{proof}
Let $\tilde{\GroupScheme}$ be the universal cover of $\GroupScheme$. By Proposition~\ref{prop:twin_building_of_kac-moody_group}  there is a thick locally finite irreducible $n$-dimensional Euclidean twin building $\PNBuildings$ associated to $\tilde{\GroupScheme}(\F_q[t,t^{-1}])$. The action on $\PNBuildings$ factors through $\tilde{\GroupScheme}(\F_q[t,t^{-1}]) \to \GroupScheme(\F_q[t,t^{-1}])$ and the image has finite index in $\GroupScheme(\F_q[t,t^{-1}])$. Thus the statement follows from Theorem~\ref{thm:two_places_geometric}.
\end{proof}

\footerlevel{3}

\footerlevel{2}

\appendix

\headerlevel{2}

\chapter{Affine Kac--Moody Groups}
\label{chap:affine_kac-moody_groups}

This paragraph is mostly due to Ralf Gramlich and Kai-Uwe Bux and taken from \cite{buxgrawit10}.

\begin{approp}
\label{prop:split_kac_moody_functor}
Let $\Field$ be a field and let $\GroupScheme$ be an isotropic, connected, simply connected, almost simple, split $\Field$-group. Then the functor $\GroupScheme(\DummyArg[t,t^{-1}])$ is a Kac--Moody functor.
\end{approp}

The functor in question is the functor that assigns to a field $\ExtensionField$ the group of $\ExtensionField[t,t^{-1}]$-points of $\GroupScheme$.

\begin{proof}
By \cite[Theorem~16.3.2]{springer98} and
\cite[\textsection II]{chevalley55}, an isotropic, connected,
simply connected, almost simple
$\Field$-group that splits over $\Field$ is a Chevalley group. It follows that the group scheme $\GroupScheme$ is defined over $\Z$. Hence the functor $\GroupScheme(\DummyArg[t,t^{-1}])$ can be defined for all fields.

A Kac-Moody functor is associated to a root datum $\Datum$, the main part of which is a generalized Cartan matrix $\Matrix$.
    Classically, this kind of datum classifies reductive groups
    over the complex numbers. There, the generalized
    Cartan matrix 
    is not really generalized and defines a
    finite Coxeter group. Kac-Moody functors were
    defined by Tits \cite{tits87} in the case where the
    generalized Cartan matrix defines an arbitrary Coxeter group.

    In order to recognize $\GroupScheme(\DummyArg[t,t^{-1}])$
    as a Kac-Moody functor, we have to correctly identify its
    defining datum $\Datum$. Since the group $\GroupScheme$ is
    simply connected, we only have to choose the generalized
    Cartan
    matrix $\Matrix$. Here, we use the unique generalized
    Cartan matrix
    given by a Euclidean Coxeter diagram extending the spherical
    diagram as defined by $\GroupScheme$.

    To show that
    $\GroupScheme(\DummyArg[t,t^{-1}])$
    is the Kac-Moody functor associated to $\Datum$, one needs
    to verify the axioms {\small (KMG~1)} through
    {\small (KMG~9)} in \cite{tits87}. All axioms are straight
    forward to check; however {\small (KMG~5)} and
    {\small (KMG~6)} involve the complex Kac-Moody algebra
    $\KMalgOf{\Matrix}$ associated to the given Cartan matrix.
    To verify these, one needs to know that
    $\KMalgOf{\Matrix}$ is the universal central extension
    of the Lie algebra $\LieAlg(\C[t,t^{-1}])$
    where $\LieAlg$ is the Lie algebra associated
    to $\GroupScheme$. See e.g., \cite[Theorem~9.11]{kac90} or
    \cite[Section~5.2]{preseg86}.
  \end{proof}
  
  In \cite{remy02}, B.\,R\'emy has extended the construction
  to non-split groups using the method of Galois descent.
\begin{approp}\label{prop:non-split_kac_moody_functor}
    Let $\GroupScheme$ be an isotropic, connected,
    simply connected, almost simple
    group defined over the finite field $\F_q$. Then
    the functor $\GroupScheme(\DummyArg[t,t^{-1}])$
    is an almost split $\F_q$-form of a Kac-Moody group
    defined over the algebraic
    closure $\bar{\F}_q$.
\end{approp}
  \begin{proof}
    First, $\GroupScheme$ splits
    over $\bar{\F}_q$. Hence,
    $\GroupScheme(\DummyArg[t,t^{-1}])$
    is a Kac-Moody functor over $\bar{\F}_q$
    by the preceding proposition. Let $\Datum$ be the
    associated root datum.
  
    Note that the conditions {\small(KMG 6)} through
    {\small(KMG 9)} ensure that the ``abstract'' and ``constructive''
    Kac-Moody functors associated to $\Datum$ coincide
    \cite[Theorem 1']{tits87},
    which holds in particular for
    $\GroupScheme(\DummyArg[t,t^{-1}])$.
    This is relevant as R\'emy discusses Galois descent
    for constructive Kac-Moody functors.

    The claim follows from \cite[Chapitre~11]{remy02}
    once a list of conditions scattered throughout that section
    have been verified. Checking individual axioms is easy,
    the hard part (left to the reader) is making sure that
    no condition is left out. Here is the list:
    \begin{description}
      \item[\normalfont{\small(PREALG 1)} {[p.~257]}]
        One needs to know that $U_\Datum$ is the
        $\Z$-form of the universal enveloping algebra
        of $\KMalgOf{\Matrix}$. Its $\F_q$-form
        is obtained by the Galois action.
      \item[\normalfont{\small(PREALG 2)} {[p.~257]}]
        Clear.
      \item[\normalfont{\small(SGR)} {[p.~266]}]
        Clear.
      \item[\normalfont{\small(ALG 1)} {[p.~267]}]
        Use Definition~11.2.1 on page~261.
      \item[\normalfont{\small(ALG 2)} {[p.~267]}]
        Clear.
      \item[\normalfont{\small(PRD)} {[p.~273]}]
        Observe that the Galois group acts trivially on
        $t$ and $t^{-1}$.\qedhere
    \end{description}
  \end{proof}

We are finally closing in on twin buildings.
\begin{approp}
Let $\GroupScheme$ be as in Proposition~\ref{prop:non-split_kac_moody_functor}.
The group $\GroupScheme(\F_q[t,t^{-1}])$ has an RGD~system with finite root groups.
\end{approp}
\begin{proof}
This follows from \cite[Theorem~12.4.3]{remy02}; but once again, we need to verify hypotheses. This time, we have to deal with only two:
\begin{description}
\item[\normalfont{{\small (DCS$_1$)} [p.~284]}]
This holds as $\GroupScheme$ splits already over a finite field extension of $\F_q$.
\item[\normalfont{{\small (DCS$_2$)} [p.~284]}]
This follows from $\F_q$ being a finite, and hence perfect field.\qedhere
\end{description}
\end{proof}

\begin{approp}
\label{prop:twin_building_of_kac-moody_group}
Let $\GroupScheme$ be an isotropic, connected, simply connected, almost simple group defined over the finite field $\F_q$
(i.e., $\GroupScheme$ is as in Proposition~\ref{prop:non-split_kac_moody_functor}). Then there is a thick, locally finite, irreducible Euclidean twin building $\PNBuildings$ on which $\GroupScheme(\F_q[t,t^{-1}])$ acts strongly transitively.
\end{approp}
\begin{proof}
By the preceding proposition, the group $\GroupScheme(\F_q[t,t^{-1}])$ has an RGD~system. By
\cite[Theorem~8.80 and Theorem~8.81]{abrbro}, we find an associated twin building upon which the group
acts strongly transitively. Theorem~8.81 also tells us that the root groups act simply transitively, which implies that the twin building is thick and locally finite. That it is irreducible and Euclidean is clear as we chose the generalized Cartan matrix $\Matrix$ back in the proof of Proposition~\ref{prop:split_kac_moody_functor} to match the spherical type of $\GroupScheme$, which is almost simple.
\end{proof}

\begin{aprem}
It also follows from \cite[Theorem~8.81]{abrbro}
that the building thus constructed is Moufang.
\end{aprem}

\begin{aprem}
\label{rem:abramenko_explicit}
For split groups, Abramenko gives the RGD~system explicitly in \cite[Example~3, page~18]{abramenko96}. He also derives RGD~systems for groups of the types ${}^2\widetilde{\text{\textsf{A}}}_{n}$ and ${}^2\widetilde{\text{\textsf{D}}}_{n}$ in \cite[Chapter~III.1]{abramenko96}. Hence, the only types not covered by his explicit computations are ${}^3\widetilde{\text{\textsf{D}}}_4$ and ${}^2\widetilde{\text{\textsf{E}}}_6$. The marginal gain also explains why we merely sketched the general argument.
\end{aprem}

The two buildings $\PosBuilding$ and $\NegBuilding$ in Proposition~\ref{prop:twin_building_of_kac-moody_group} are isomorphic to the Bruhat--Tits buildings associated to $\GroupScheme(\F_q((t^{-1})))$ and $\GroupScheme(\F_q((t)))$. In fact even more is true:

\begin{apfact}
\label{fact:twin_halves_identification}
The two halves $\PosBuilding$ and $\NegBuilding$ of the twin building $\PNBuildings$ in Proposition~\ref{prop:twin_building_of_kac-moody_group} can be identified with the Bruhat--Tits buildings associated to $\GroupScheme(\F_q((t^{-1})))$ and $\GroupScheme(\F_q((t)))$ in an $\GroupScheme(\F_q[t,t^{-1}])$-equivariant way. \end{apfact}

That the buildings associated to $\GroupScheme(\F_q(t))$ with respect to the valuations $\Place_\infty$ and $\Place_0$ are those associated to $\GroupScheme(\F_q((t^{-1})))$ and $\GroupScheme(\F_q((t)))$ follows from functoriality, see \cite[5.1.2]{rousseau77}. It remains to compare twin BN-pair of the Kac--Moody group $\GroupScheme(\F_q[t,t^{-1}])$ to the BN-pairs of $\GroupScheme(\F_q(t))$ with respect to the valuations $\Place_\infty$ and $\Place_0$. That has been done explicitly by Peter Abramenko in most cases, see Remark~\ref{rem:abramenko_explicit}, but no abstract argument is known to the author.

\footerlevel{2}
\headerlevel{2}

\chapter{Adding Places}
\label{chap:adding_places}

In this paragraph we show that augmenting the set of places can only increase the finiteness length of an almost simple $S$-arithmetic group. Since the proof of the Rank Conjecture in \cite{buxgrawit10}, the finiteness length of any such group is known, so one can verify the statement by just looking at the number there. Still it is interesting to observe that this fact is clear a priori for relatively elementary reasons. The proof works as in the special case considered in \cite{abramenko96}.

\begin{apthm}
\label{thm:increasing_places_preserves_topfin}
Let $\GlobalField$ be a global function field, $\GroupScheme$ a $\GlobalField$-isotropic, connected, almost simple $\GlobalField$-group, and $S$ a non-empty, finite set of places of $\GlobalField$. If $\GroupScheme(\Integers[S])$ is of type $F_n$ and $S' \supseteq S$ is a larger finite set of places, then $\GroupScheme(\Integers[S'])$ is also of type $F_n$.
\end{apthm}

\begin{proof}
Proceeding by induction it suffices to prove the case where only one place is added to $S$, i.e., $S' = S \union \{\Place\}$ for some place $\Place$. Also note that as far as finiteness properties are concerned, we may (and do) assume that $\GroupScheme$ is simply connected.

Let $\EBuilding_\Place$ be the Bruhat--Tits building that belongs to $\GroupScheme(\GlobalField_\Place)$ (see \cite{brutit72, brutit84}). The group $\GroupScheme(\Integers[S']) \subseteq \GroupScheme(\GlobalField_\Place)$ acts continuously on $\EBuilding_\Place$. We claim that this action is cocompact and that cell stabilizers are abstractly commensurable to $\GroupScheme(\Integers[S])$. With these two statements the result follows from Theorem~\ref{thm:brown_criterion}.

Note that the stabilizer of a cell is commensurable to the stabilizers of its faces and cofaces since the building is locally finite (because the residue field of $\GlobalField$ is finite). Also all cells of same type are conjugate by the action of $\GroupScheme(\GlobalField_\Place)$. Hence it remains to see that some cell-stabilizer is commensurable to $\GroupScheme(\Integers[S])$. To see this note that $\GroupScheme(\Integers[\Place])$ is a maximal compact subgroup of $\GroupScheme(\GlobalField_\Place)$. The Bruhat--Tits Fixed Point Theorem \cite[Lemme 3.2.3]{brutit72} (see also \cite[Corollary~II.2.8]{brihae}) implies that it has a fixed point and by maximality the fixed point is a vertex and $\GroupScheme(\Integers[\Place])$ is its full stabilizer. Now $\GroupScheme(\Integers[S]) = \GroupScheme(\Integers[S']) \intersect \GroupScheme(\Integers[\Place])$ so $\GroupScheme(\Integers[S])$ is the stabilizer in $\GroupScheme(\Integers[S'])$ of that vertex.

For cocompactness we use that $\GroupScheme(\Integers[S'])$ is dense in $\GroupScheme(\GlobalField_\Place)$, see Lemma~\ref{lem:density_of_s-arith_subgroup} below. Let $\Point$ be an interior point of some chamber of $\EBuilding_\Place$. The orbit $\GroupScheme(\GlobalField_\Place).\Point$ is a discrete space which, by strong transitivity, contains one point from every chamber of $\EBuilding_\Place$. The orbit map $\GroupScheme(\GlobalField_\Place) \to \GroupScheme(\GlobalField_\Place).x$ is continuous by continuity of the action, so the image of the dense subgroup $\GroupScheme(\Integers[S'])$ is dense in the discrete space $\GroupScheme(\GlobalField_\Place).x$. Hence $\GroupScheme(\Integers[S'])$ acts transitively on chambers and, in particular, cocompactly.
\end{proof}

It remains to provide the density statement used in the proof. It is known and a consequence of the Strong Approximation Theorem:

\begin{aplem}
\label{lem:density_of_s-arith_subgroup}
Let $\GlobalField$ be a global field and let $\GroupScheme$ be a $\GlobalField$-isotropic, connected, simply connected, almost simple $\GlobalField$-group. Let $S$ be a non-empty finite set of places and let $\Place \nin S$. Then $\GroupScheme(\Integers[S \union \{\Place\}])$ is dense in $\GroupScheme(\GlobalField_\Place)$.
\end{aplem}

\begin{proof}
For a place $\Place$ of $\GlobalField$ let $\GlobalField_\Place$ denote the local field at $\Place$ and $\Integers[\Place]$ the ring of integers in $\GlobalField_\Place$. For a finite set $S$ of places of $\GlobalField$ let $\A_S = \prod_{\Place \in S} \GlobalField_\Place \times \prod_{\Place \nin S} \Integers[\Place]$ denote the ring of $S$-adeles. Recall that the ring of adeles is $\A = \lim_{S} \A_S$ (see \cite{weil82}).

Note that $\GroupScheme_S \defeq \prod_{\Place \in S} \GroupScheme(\GlobalField_\Place)$ is non-compact by \cite[Proposition~2.3.6]{margulis}.

Recall that $\GlobalField_\Place$ embeds into $\A$ at $\Place$, and that $\GlobalField$ discretely embeds into $\A$ diagonally. With these identifications $\GroupScheme(\GlobalField) \cdot \GroupScheme_S$ is dense in $\GroupScheme(\A)$ by \cite[Theorem~A]{prasad77}, that is, if $U$ is an open subset of $\GroupScheme(\A)$, then $\GroupScheme(\GlobalField) \intersect U \GroupScheme_S \ne \emptyset$.

If $V$ is an open subset of $\GroupScheme(\GlobalField_\Place)$, then
\[
U = V \times \prod_{\Place' \in S} \GroupScheme(\GlobalField_{\Place'}) \times \prod_{\Place' \nin S \union \{\Place\}} \GroupScheme(\Integers[\Place'])
\]
is open in $\GroupScheme(\A)$. Hence there is a $g \in \GroupScheme(\GlobalField)$ with $g \in V$ and $g \in \GroupScheme(\Integers[S \union \{\Place\}])$ (where we now consider $\GroupScheme(\GlobalField)$ and $\GroupScheme(\Integers[S \union \{\Place\}])$ as subgroups of $\GroupScheme(\GlobalField_\Place)$). Thus $V \intersect \GroupScheme(\Integers[S \union \{\Place\}]) \ne \emptyset$ as desired.
\end{proof}

Theorem~\ref{thm:increasing_places_preserves_topfin} is the natural generalization to higher finiteness properties of Helmut Behr's Proposition~2 in \cite{behr98}, the proof of which is not given but attributed to Martin Kneser \cite{kneser64}. The main result of \cite{kneser64}, which applies to number fields, also has a natural generalization, namely the following Hasse principle proven by Andreas Tiemeyer \cite[Theorem~3.1]{tiemeyer97}:

\begin{apthm}
Let $\GlobalField$ be a global number field and let $S$ be a finite set of places. Let $\GroupScheme$ be a $\GlobalField$-group. Then $\GroupScheme(\Integers[S])$ is of type $F_n$ if and only if $\GroupScheme(\GlobalField_s)$ is of type $C_n$ for every non-Archimedean $s \in S$.
\end{apthm}

The properties $C_n$ generalize being compactly generated as the properties $F_n$ generalize being finitely generated. We do not go into the details here. However it is interesting to note, that the growth of the finiteness length implied by the Hasse principle is inverse to that of Theorem~\ref{thm:increasing_places_preserves_topfin}:

\begin{apcor}
Let $\GlobalField$ be a global number field and let $S$ be a finite set of places. Let $\GroupScheme$ be a $\GlobalField$-group. If $\GroupScheme(\Integers[S])$ is of type $F_n$ and $S' \subseteq S$, then $\GroupScheme(\Integers[S'])$ is also of type $F_n$.
\end{apcor}

\footerlevel{2}

\phantomsection
\pdfbookmark{Index of Symbols}{pdfbook:syms}
\printindex[xsyms]

\cleardoublepage
\phantomsection
\pdfbookmark{Bibliography}{pdfbook:bib}
\bibliographystyle{amsalpha}
\bibliography{../../local}

\end{document}